\documentclass[12pt]{article}

\usepackage[utf8]{inputenc}
\usepackage[T1]{fontenc}
\usepackage{geometry}
\usepackage[english]{babel}
\usepackage{amsmath,accents}
\usepackage{amssymb}
\usepackage{amsthm}
\usepackage[shortlabels]{enumitem}
\usepackage{mathtools}
\usepackage{xcolor}
\usepackage{booktabs}
\usepackage{array}
\usepackage{tabulary}
\usepackage{graphicx}
\graphicspath{ {images/} }

\usepackage[colorlinks,
pdfpagelabels,
pdfstartview = FitH,
bookmarksopen = true,
bookmarksnumbered = true,
linkcolor = black,
plainpages = false,
hypertexnames = false,
citecolor = black] {hyperref}

\newtheorem{satz}{Proposition}[subsection]
\newtheorem{lem}[satz]{Lemma}
\newtheorem{kor}[satz]{Corollary}
\newtheorem{theorem}[satz]{Theorem}

\theoremstyle	{definition}
\newtheorem{mydef}[satz]{Definition}
\newtheorem{bem}[satz]{Remark}
\newtheorem{bei}[satz]{Example}
\newtheorem{notation}[satz]{Notation}
\newtheorem{terminology}[satz]{Terminology}
\newtheorem{teano}[satz]{Terminology and Notation} \usepackage{tikz}
\usetikzlibrary{cd,babel,positioning,calc,backgrounds,fit}
\tikzset{
	math mode/.style = {execute at begin node=$, execute at end node=$},
	short/.style = {shorten >=#1, shorten <=#1},
	obj/.style = {minimum size=4pt, inner sep=0pt, outer sep=4pt, circle, fill},
	string/.style 2 args={
		r/.style={right=#1 ##1},
		rr/.style={right=2*#1 ##1},
		rrr/.style={right=3*#1 ##1},
		hr/.style={right=.5*#1 ##1},
		l/.style={left=#1 ##1},
		hl/.style={left=#1 ##1},
		a/.style={above=#2 ##1},
		ha/.style={above=.5*#2 ##1},
		ar/.style={above right=#2 and #1 ##1},
		ahr/.style={above right=#2 and .5*#1 ##1},
		al/.style={above left=#2 and #1 ##1},
		ahl/.style={above left=#2 and .5*#1 ##1},
		b/.style={below=#2 ##1},
		bb/.style={below=2*#2 ##1},
		bbb/.style={below=3*#2 ##1},
		hb/.style={below=.5*#2 ##1},
		br/.style={below right=#2 and #1 ##1},
		bhr/.style={below right=#2 and .5*#1 ##1},
		bl/.style={below left=#2 and #1 ##1},
		bbl/.style={below left=2*#2 and #1 ##1},
		bbbl/.style={below left=3*#2 and #1 ##1},
		bhl/.style={below left=#2 and .5*#1 ##1},
		hbhl/.style={below left=.5*#2 and .5*#1 ##1},
		bm/.style 2 args={at={($($(##1)!($.5*(##1)+.5*(##2)$)!($(##1)+(1,0)$)$)-(0,#2)$)}},
		bbm/.style 2 args={at={($($(##1)!($.5*(##1)+.5*(##2)$)!($(##1)+(1,0)$)$)-(0,2*#2)$)}},
		hbm/.style 2 args={at={($($(##1)!($.5*(##1)+.5*(##2)$)!($(##1)+(1,0)$)$)-(0,.5*#2)$)}},
		hbr/.style={below right=.5*#2 and #1 ##1},
		hbhr/.style={below right=.5*#2 and .5*#1 ##1},
		morphism/.style={circle, inner sep=0pt, minimum size=2ex, draw},
		mult/.style={circle, inner sep=0, outer sep=0, minimum size={.3*#1}, fill},
		unit/.style={circle, inner sep=0, minimum size={.3*#1}, draw},
		action/.style={circle, inner sep=0, minimum size=0, opacity=0},
		vertex/.style={inner sep=0, opacity=0, outer sep=0},
		rec/.style={rectangle, inner sep=0, minimum width=1.3*#1, minimum height=.5*#2, rounded corners=0, draw},
		in l/.style={out=-90, in=135},
		in r/.style={out=-90, in=45},
		out l/.style={out=-135, in=90},
		out r/.style={out=-45, in=90},
		d d/.style={out=-90, in=90},
		d l/.style={out=-90, in=0},
		d hl/.style={out=-90, in=45},
		d r/.style={out=-90, in=180},
		d hr/.style={out=-90, in=135},
		l d/.style={out=180, in=90},
		l u/.style={out=0, in=-90},
		l l/.style={out=180, in=0},
		l hl/.style={out=180, in=45},
		l r/.style={out=180, in=180},
		l hr/.style={out=180, in=135},
		r d/.style={out=0, in=90},
		r u/.style={out=0, in=-90},
		r l/.style={out=0, in=0},
		r hl/.style={out=0, in=45},
		r r/.style={out=0, in=180},
		r hr/.style={out=0, in=135},
		hl d/.style={out=-135, in=90},
		hl l/.style={out=-135, in=0},
		hl hl/.style={out=-135, in=45},
		hl r/.style={out=-135, in=180},
		hl hr/.style={out=-135, in=135},
		hr d/.style={out=-45, in=90},
		hr l/.style={out=-45, in=0},
		hr hl/.style={out=-45, in=45},
		hr r/.style={out=-45, in=180},
		hr hr/.style={out=-45, in=135},
		u d/.style={out=90, in=90},
		u l/.style={out=90, in=0},
		u hl/.style={out=90, in=45},
		u r/.style={out=90, in=180},
		u hr/.style={out=90, in=135},
		u u/.style={out=90, in=-90},
		on grid
	}
}

\newcommand{\cat}[1] {\mathtt{#1}}

\newcommand{\ob}[1]{\operatorname{Ob}(#1)}
\newcommand{\mor}[1]{\operatorname{1-Mor}(#1)}
\newcommand{\mmor}[1]{\operatorname{2-Mor}(#1)}
\newcommand{\mmmor}[1]{\operatorname{3-Mor}(#1)}
\newcommand{\nmor}[2]{\operatorname{#1-Mor}(#2)}
\newcommand{\set}[1]{\{#1\}}
\newcommand{\oper}[1]{\operatorname{#1}}

\newcommand{\hot}{\mathbin{\hat \otimes}}
\newcommand{\hs}{\mathbin{\hat *}}
\newcommand{\hc}{\circ}
\newcommand{\botimes}{\mathbin{\bar \otimes}}
\newcommand{\bstar}{\hs}
\newcommand{\bcirc}{\hc}
\newcommand{\fotimes}{\mathbin{\otimes ^{\scriptscriptstyle{F}}}}
\newcommand{\fstar}{\mathbin{* ^{\scriptscriptstyle{F}}}}
\newcommand{\fcirc}{\mathbin{\circ ^{\scriptscriptstyle{F}}}}

\newcommand{\adsq}{\scriptscriptstyle{\blacksquare}}
\newcommand{\globe}[1]{\operatorname{Glb}(#1)}
\newcommand{\barf}{\operatorname{bar}}
\newcommand{\ev}{\operatorname{ev}}
\newcommand{\ftilde}{\operatorname{tdl}}

\newcommand{\comseq}[2]{\prescript{#1}{}{#2}}

\newcommand{\bhat}{\cat{\widehat B}}
\newcommand{\bst}{\cat{B ^{st}}}

\newcommand{\that}{\cat{\widehat T}}
\newcommand{\ttil}{\cat{\widetilde T}}
\newcommand{\tbar}{\cat{\bar T}}
\newcommand{\tcu}{\cat{T ^{cu}}}
\newcommand{\tprime}{\cat{T'}}
\newcommand{\tst}{\cat{T ^{st}}}

\newcommand{\gset}[1]{\cat{gSet^{#1}}}
\newcommand{\magm}[1]{\cat{Magm^{#1}}}

\newcommand{\sms}[1]{\scriptscriptstyle{#1}}
\newcommand{\sfl}[1]{\dot{#1}}
\newcommand{\fc}[1]{#1^{\sms{F}}}

\newcommand{\bimod}[3]{\prescript{#1}{}{#2}^{#3}}

\makeatletter
\newcommand{\xRrightarrow}[2][]{\ext@arrow 0359\Rrightarrowfill@{#1}{#2}}
\newcommand{\Rrightarrowfill@}{\arrowfill@\equiv\equiv\Rrightarrow}
\newcommand{\xLleftarrow}[2][]{\ext@arrow 3095\Lleftarrowfill@{#1}{#2}}
\newcommand{\Lleftarrowfill@}{\arrowfill@\Lleftarrow\equiv\equiv}
\makeatother

\title{Tricategories and their diagrams}
\author{Peter Guthmann}
\date{}

\begin{document}

\nocite{Gurski2006}

\clearpage{}
\begin{titlepage}
      \textsf{
    \begin{center}
      \vspace*{1cm}
    \huge \textbf{The tricategory of formal composites and its strictification} \\
    \vspace{2cm}
    \LARGE\textbf{Master-Arbeit}\\[5mm]
    \Large zur Erlangung des Grades\\[5mm]
    \textbf{Master of Science (M.Sc.) \\[5mm]
    im Studiengang
    Mathematik\\[5mm]
    }
    am Department Mathematik der\\ Friedrich-Alexander-Universität
    Erlangen-Nürnberg\\[1cm]
    vorgelegt am 24.8.2018 \\[3mm]
    von \textbf{Peter Guthmann}
    \vfill
    \normalsize
    Betreuerin: Prof. Dr. Catherine Meusburger
    \\[2em]
    \end{center}
}
\end{titlepage}
\clearpage{}

\clearpage{}

\begin{abstract}
The results of this thesis allow one to replace calculations in tricategories with equivalent calculations in Gray categories (aka semistrict tricategories).
In particular the rewriting calculus for Gray categories as used for example by the online proof assistant globular \cite{Bar2016}, or equivalently the Gray-diagrams of \cite{Barrett2012} can then be used also in the case of a fully weak tricategory.

In order to achieve this we consider for an arbitrary tricategory \( \cat{T} \) a tricategory \( \that \) of "formal composites",
which distinguishes (formal) composites of 1- or 2-morphisms, even if they evaluate to the same 1-resp. 2-morphism in \( \cat{T}. \)
Then we show that \( \cat{ \that} \) is strictly triequivalent -that means triequivalent via a strict functor of tricategories- to a tricategory \( \cat{ \tst}. \)
This result, together with the observation that \( \cat{ \that} \) is strictly triequivalent to \( \cat{T}, \) is crucial for reducing calculations in tricategories to calculations in Gray categories.
\end{abstract}\clearpage{}

\tableofcontents

\clearpage{}

\section{Introduction}

Calculations in a 1-category can be performed by writing down equations between morphisms in 1-dimensional notation.
However, this kind of computation has the disadvantage that the types of the morphisms are not included in the notation.
Therefore computations are often performed with commuting diagrams, since those contain type-information which are often a valuable guidance in the calculation.
In higher category theory, there are even more types to keep track of and therefore 1-dimensional notation gets -with increasing dimension of the category- more and more inconvenient.
Therefore computations in higher categories are almost exclusively carried out in diagrammatic style.

In tricategories those diagrammatic computations are often cumbersome, since the occurring diagrams contain many constraint cells and thus become quite large.
Besides, the mass of constraint cells often obscures the core idea of the computation.
Thus the main aim of this thesis is to reduce calculations in general tricategories to calculations in Gray categories, which are tricategories in which all constraint cells except the interchangers are identities.
In order to achieve this one needs a method to replace equations of pasting diagrams in a local bicategory of a tricategory \( \cat{T} \) with equivalent equations of pasting diagrams in a suitable Gray-category.

At first sight one might think the fact that any tricategory \( \cat{T} \) is triequivalent to a Gray category \( \oper{Gr} \cat{T} \) (see \cite{Gurski2013} section 10.4) suffices for that purpose,
but this is not the case since the triequivalence \( F: \cat{T} \to \oper{Gr} \cat{T} \) (see \cite{Gurski2013} theorem 10.11) is not a \emph{strict} functor of tricategories.
Therefore composition is not preserved by \( F, \) and thus pasting diagrams in \( \cat{T} \) are not mapped to pasting diagrams in \( \oper{Gr} \cat{T}. \)
For instance the existence of the pasting diagram
\[
\begin{tikzcd} [column sep=huge]
f_1 \otimes f_2
\ar[r, bend left=40, "\alpha_1 \otimes \alpha_2" {name=s1}]
\ar[r, bend right=40, "\beta_1 \otimes \beta_2" {swap, name=t1}]
	& g_1 \otimes g_2
	\ar[r, bend left=40, "\gamma_1" {name=s2}]
	\ar[r, bend right=40, "\gamma_2" {swap, name=t2}]
		& h_1
\ar[r, from=s1, to=t1, Rightarrow, short=3.5mm, "\Lambda_1 \otimes \Lambda_2", xshift=-3mm]
\ar[r, from=s2, to=t2, Rightarrow, short=3mm, "\Lambda_3"]
\end{tikzcd}
\]
does not imply the existence of the pasting diagram
\[
\begin{tikzcd} [column sep=huge]
F f_1 \otimes F f_2
\ar[r, bend left=40, "F \alpha_1 \otimes F \alpha_2" {name=s1}]
\ar[r, bend right=40, "F \beta_1 \otimes F \beta_2" {swap, name=t1}]
	& F g_1 \otimes F g_2
	\ar[r, bend left=40, "F \gamma_1" {name=s2}]
	\ar[r, bend right=40, "F \gamma_2" {swap, name=t2}]
		& F h_1
\ar[r, from=s1, to=t1, Rightarrow, short=3.5mm, "F \Lambda_1 \otimes F \Lambda_2" {description, xshift=3mm}, xshift=-3mm]
\ar[r, from=s2, to=t2, Rightarrow, short=3mm, "F \Lambda_3"]
\end{tikzcd},
\]
since for example the source of \( F \gamma_1 \) is equal to \( F(g_1 \otimes g_2) \) but in general \( F(g_1 \otimes g_2) \neq F g_1 \otimes F g_2. \)

By the previous observations, it becomes clear that we are looking for a way to relate any tricategory \( \cat{T} \) to a suitable Gray category via a \emph{strict} triequivalence.
This is not directly possible, since in general a tricategory \( \cat{T} \) is not strictly triequivalent to a Gray category.
Therefore we will define in section \ref{sec_strictification_for_tricats} of this thesis a tricategory \( \that, \) which will be triequivalent to \( \cat{T} \) and to a Gray category \( \tst \) and will support strict triequivalences
\begin{align} \label{equ_that_equivalences}
\begin{tikzcd} [column sep=large, row sep=tiny, ampersand replacement=\&,]
 	\& \that
 	\ar[ld, "\ev" swap]
 	\ar[rd, "{[-]}"] \& \\
\cat{T}
		\& \& \tst.
\end{tikzcd}
\end{align}
Roughly speaking the "tricategory of formal composites" \( \that \) is quite similar to \( \cat{T}, \) only that the 1- and 2-morphisms of \( \that \) are defined as formal composites of 1- and 2-morphisms of \( \cat{T}. \)
The 3-morphisms of \( \that \) are just the 3-morphisms of \( \cat{T} \) plus a formal source and a formal target and the strict triequivalence \( \ev: \that \to \cat{T} \) simply evaluates formal composites.
Now it is clear that any composite 1- or 2-morphism in \( \cat{T} \) lifts along \( \ev \) to the 1- or 2-morphisms in \( \that \) which represents precisely the composite itself.
Similarly any pasting diagram in a local bicategory of \( \cat{T} \) lifts to the respective pasting diagram in \( \that \) (see the end of section \ref{sec_strictification_for_tricats} for an example of that principle).
Furthermore, since \( \ev \) is a strict triequivalence, two pasting diagrams in a local bicategory of \( \cat{T} \) are equal, if and only if the respective pasting diagrams in \( \that \) are equal.
And since \( [-]: \that \to \cat{T} \) is a strict triequivalence, two pasting diagrams in \( \that \) are equal, if and only if their images under \( [-] \) are equal.
Thus we have achieved our goal of "reducing calculations in tricategories to calculations in Gray categories",
since we are now able to check equations of pasting diagrams in \( \cat{T} \) by checking equivalent equations of pasting diagrams in \( \that. \)
Since most constraint cells in the latter are identities this is a tremendous simplification and has the additional advantage that pasting diagrams in Gray categories can be equivalently replaced by rewriting
diagrams as they are used for example by the online proof assistant globular \cite{Bar2016}.

The thesis is structured as follows:

Before we come to the just discussed tricategory case, chapter 2 treats the analog results for bicategories.

The first five sections of chapter \ref{sec_bicats} summarize some elementary results of 2-dimensional category theory that are required for this thesis.
In section \ref{sec_def_bicat} and \ref{sec_string_diagrams_bicats} we give the definition of a bicategory and introduce string diagrams for 2-categories.
Section \ref{sec_internalization} introduces adjunctions and equivalences internal to a bicategory \( \cat{B} \) and summarizes some standard results for the case \( \cat{B}=\cat{Cat}. \)
In section \ref{sec_bicat_transfors} we give the definitions of the bicategorical transfors, which are the morphisms and higher morphisms (i.e. morphisms between morphisms etc.) between bicategories.
Those will be needed later on in order to give the definition of a tricategory.
Section \ref{sec_coherence_for_bicats} summarizes -without a proof- the coherence theorem for bicategories from section 2.2 of \cite{Gurski2013}, which states that parallel coherence 2-morphisms in a free bicategory are equal.

In section \ref{sec_strictification_for_bicats} we discuss the well known fact that calculations in bicategories can be reduced to calculations in 2-categories.
This is usually implicitly assumed in literature, but it is worth taking a closer look at the details, since the overall strategy can be mimicked later on in the tricategory case.
Analogous to the discussion above it suffices to construct a pair \( \cat{B} \xleftarrow{\ev} \bhat \xrightarrow{[-]} \bst \) of strict biequivalences, where \( \bhat \) is "the bicategory of formal composites in \( \cat{B} \)" and \( \bst \) is a 2-category (see also diagram \eqref{equ_that_equivalences}).

Chapter \ref{sec_tricats} treats the tricategory case.

In section \ref{sec_def_tricat} we summarize the necessary background on tricategories from \cite{Gurski2013}.
This includes the definition of a tricategory, and some stricter versions of tricategories which are needed later on for our strictification result.
One of these stricter versions are Gray-categories, for which section \ref{sec_rewriting_diagrams} explains the graphical calculus of rewriting diagrams.

Sections \ref{sec_globular_sets_and_magmoids} and \ref{sec_tricats_as_magmoids} develop an alternative description of tricategories in terms of so called 3-magmoids.
The basic idea of a 3-magmoid is the following:
Algebraic structures like groups or rings have an underlying magma, which is obtained by forgetting everything about the algebra axioms or distinguished elements of the algebra, but remembering the operations and the underlying set of the algebra.
Analogously every tricategory has an underlying 3-magmoid, which is obtained by forgetting everything about the axioms and the constraint cells of the tricategory, but remembering composition operations and the types of the morphisms.
Section \ref{sec_tricats_as_magmoids} unpacks the definition of a tricategory in terms of transfors towards a definition in terms of a 3-magmoid with additional structure.
Although the definition of a tricategory in terms of transfors is more concise and conceptual, the magmoidal definition turns out to be useful to construct many of the tricategories appearing in section \ref{sec_strictification_for_tricats}.

Another advantage of the magmoidal definition of a tricategory is that it allows to give in section \ref{sec_virtually_strict_functor} a simple definition of a virtually strict functor as a morphism of 3-magmoids, which preserves constraint cells strictly.
We should mention that virtually strict functors and virtually strict triequivalences are in one-to-one correspondence with strict (tricategory) functors and strict triequivalences.
A different name is only chosen, since virtually strict functors compose in a different way than strict functors do.
In section \ref{sec_coherence_for_tricats} we state the coherence theorem of tricategories from \cite{Gurski2013}, which says that parallel coherence 3-morphisms in a free tricategory are equal.

Sections \ref{sec_quotiening_out_coherence} and \ref{sec_strictification_for_tricats} contain the main results of the thesis.
Section \ref{sec_quotiening_out_coherence} studies how one can "quotient out coherence" from tricategories.
The results from this section will be applied in \ref{sec_strictification_for_tricats},
where we construct the tricategory \( \that \) together with the strict triequivalences \( \ev: \that \to \cat{T} \) and \( [-]: \that \to \tst \) from \eqref{equ_that_equivalences}.
Towards the end of section \ref{sec_strictification_for_tricats} we discuss the implications of this construction for the diagrammatic calculus in tricategories.
As an example we demonstrate how the results of the thesis can be used to replace the axioms in the definition of a biadjunction by equivalent -but much simpler- axioms in a triequivalent Gray-category.
\clearpage{}

\clearpage{}

\section{Bicategories and their diagrams} \label{sec_bicats}

\subsection{The definition of a bicategory} \label{sec_def_bicat}

In this chapter we give the definition of a bicategory and of the bicategorical transfors (i.e. weak functors, transformations, modifications).
A neat collection of these definitions can be found for example in \cite{Leinster1998}.

The definition of a bicategory is obtained from the definition of a category by an instance of so called \emph{categorification}.
In order to understand the process of categorification we first observe that the definition of a category \( \cat{C} \) can be given purely in terms of sets and functions.
For example the existence of a distinguished (unit-)element \( 1_a \in \cat{C}(a,a) \) is equivalent to the existence of a (unit-)function \( I_a: * \to \cat{C}(a,a), \) where \( * \) denotes the terminal set.
Also the axioms of a category can be easily described in terms of equations of functions (which run between products of hom-sets and are built from composition- and unit-functions).

Categorifying the definition of a category then means to replace sets with categories (the terminal set gets replaced with the terminal category), functions with functors and equations of functions (i.e. axioms) with natural isomorphisms.
Note that during this process the axioms in the definition of a category became data (given in the form of natural isomorphisms) in the definition of a bicategory.
Therefore one imposes "new" axioms, which now come in form of equations between natural isomorphisms.

\begin{mydef} \label{def_bicat}
A bicategory \( (\cat{\cat{B}};{*},I;a,l,r) \) consists of the following data:

\begin{itemize}
\item
A collection of objects \( \ob{\cat{\cat{B}}}. \)

\item
For each pair \( (a,b) \) of objects of \( \cat{\cat{B}} \) a category \( \cat{\cat{B}}(a,b), \) called the local bicategory at \( (a,b). \)
The objects of a local category are called 1-morphisms or 1-cells of the bicategory \( \cat{\cat{B}}, \) and the morphism of a local bicategory are called 2-morphisms of the bicategory \( \cat{\cat{B}}. \)

\item
For each triple \( (a,b,c) \) of objects of \( \cat{\cat{B}} \) a functor
\[
{*} _{abc}: \cat{\cat{B}}(b,c) \times \cat{\cat{B}}(a,b) \to \cat{\cat{B}}(a,c).
\]

\item
For each object \( a \) of \( \cat{\cat{B}} \) a functor
\[
I_a: \cat{1} \to \cat{\cat{B}}(a,a),
\]
where \( \cat{1} \) denotes the terminal category.

\item
For each 4-tuple \( (a,b,c,d) \) of objects of \( \cat{\cat{B}} \) a natural isomorphism
\[
\begin{tikzcd}
\cat{\cat{B}}(c,d) \times \cat{\cat{B}}(b,c) \times \cat{\cat{B}}(a,b)
\ar[r, "{*} \times 1" {name=s1}]
\ar[d, "1 \times {*}"]
	& \cat{\cat{B}}(b,d) \times \cat{\cat{B}}(a,b)
	\ar[d, "{*}"] \\
\cat{\cat{B}}(c,d) \times \cat{\cat{B}}(a,c)
\ar[r, "{*}" {name=t1}]
	& \cat{\cat{B}}(a,d), \\
\ar[from=s1, to=t1, "a", Rightarrow, short=7pt]
\end{tikzcd}
\]
where we omit indices.

\item
For each pair \( (a,b) \) of objects of \( \cat{\cat{B}} \) natural isomorphism \( l \) and \( r \)
\[
\begin{tikzcd}[column sep=tiny]
	& \cat{\cat{B}}(b,b) \times \cat{\cat{B}}(a,b)
	\ar[rd,"{*}"]
	\ar[ld, leftarrow, "I_b \times 1" swap]
	& \\
\cat{\cat{B}}(a,b)
\ar[rr,"1" {name=t1, swap}]
		& & \cat{\cat{B}}(a,b)
		\ar[from=ul, to=t1, "l", Rightarrow, short=5pt]
\end{tikzcd}
\quad
\begin{tikzcd}[column sep=tiny]
	& \cat{\cat{B}}(a,b) \times \cat{\cat{B}}(a,a)
	\ar[rd,"{*}"]
	\ar[ld, leftarrow, "1 \times I_a" swap]
	& \\
\cat{\cat{B}}(a,b)
\ar[rr,"1" {name=t1, swap}]
		& & \cat{\cat{B}}(a,b)
		\ar[from=ul, to=t1, "r", Rightarrow, short=5pt].
\end{tikzcd}
\]
\end{itemize}

The data of \( \cat{\cat{B}} \) are subject to the following axioms:
\begin{enumerate}[(i)]
\item \label{list_pentagon_axiom}
For each 5-tuple \( (a,b,c,d,e) \) of objects of \( \cat{\cat{B}} \) the following equation of natural transformations holds:
(Here we write \( \cat{\cat{B}^4} = \cat{\cat{B}^4}(a,b,c,d,e) \) as an abbreviation for \( \cat{\cat{B}}(d,e) \times \cat{\cat{B}}(d,c) \times \cat{\cat{B}}(c,b) \times \cat{\cat{B}}(a,b), \) and similar for different powers of \( \cat{\cat{B}}^i. \))

\[
\begin{tikzcd}[column sep=small]
	& \cat{\cat{\cat{B}}^4}
	\ar[rr,"{*} \times 1 \times 1"]
	\ar[dr, "1 \times {*} \times 1" description]
	\ar[dl, "1 \times 1 \times {*}" swap]
			& & \cat{\cat{\cat{B}}^3}
			\ar[dr,"{*} \times 1"]
			\ar[dl, Rightarrow, short=8pt, "a \times 1" swap] & \\
\cat{\cat{\cat{B}}^3}
\ar[dr,"1 \times {*}" swap]
		& & \cat{\cat{\cat{B}}^3}
		\ar[rr,"{*} \times 1"]
		\ar[dl,"1 \times {*}" description]
		\ar[ll, Rightarrow, short=20pt, "1 \times a" swap]
				& & \cat{\cat{\cat{B}}^2}
				\ar[dl,"{*}"]
				\ar[dlll, Rightarrow, short=40pt, pos=.45, "a" swap] \\
	& \cat{\cat{\cat{B}}^2}
	\ar[rr, "{*}" swap]
			& & \cat{\cat{B}}
\end{tikzcd}
\phantom{xx} = \phantom{xx}
\begin{tikzcd}[column sep=small]
	& \cat{\cat{\cat{B}}^4}
	\ar[rr,"{*} \times 1 \times 1"]
	\ar[dl, "1 \times 1 \times {*}" swap]
			& & \cat{\cat{\cat{B}}^3}
			\ar[dr,"{*} \times 1"]
			\ar[dl, "1 \times {*}" description]
			\ar[dlll, phantom, "="] & \\
\cat{\cat{\cat{B}}^3}
\ar[dr,"1 \times {*}" swap]
\ar[rr,"{*} \times 1"]
		& & \cat{\cat{\cat{B}}^2}
		\ar[dr,"{*}" description]
		\ar[ld, Rightarrow, short=8pt, "a" swap]
				& & \cat{\cat{\cat{B}}^2}
				\ar[dl,"{*}"]
				\ar[ll, Rightarrow, short=20pt, "a" swap] \\
	& \cat{\cat{\cat{B}}^2}
	\ar[rr, "{*}" swap]
			& & \cat{\cat{B}}.
\end{tikzcd}
\]

\item \label{list_triangle_axiom}
For each triple of objects of \( \cat{\cat{B}} \) the following equation of natural transformations holds:
(For notation see previous item.)

\[
\begin{tikzcd} [sep=large]
\cat{B}^2
\ar[r, "1 \times I \times 1"]
\ar[rd, "1 \times 1" {swap, name=t1}]
	& \cat{B}^3
	\ar[r, "{*} \times 1"]
	\ar[d, "1 \times {*}"]
	\ar[to=t1, Rightarrow, short=1.5mm, "1 \times l" {xshift=-1mm, yshift=0mm}]
		& \cat{B}^2
		\ar[d, "{*}"]
		\ar[ld, Rightarrow, short=5mm, "a"] \\
	& \cat{B}^2
	\ar[r, "{*}"]
		& \cat{B}
\end{tikzcd}
\quad = \quad
\begin{tikzcd} [sep=large]
\cat{B}^2
\ar[r, "1 \times I \times 1"]
\ar[rd, "1 \times 1" {swap, name=t1}]
	& \cat{B}^3
	\ar[r, "{*} \times 1"]
	\ar[d, "{*} \times 1"]
	\ar[to=t1, Rightarrow, short=1.5mm, "r \times 1" {xshift=-1mm, yshift=0mm}]
		& \cat{B}^2
		\ar[d, "{*}"]
		\ar[ld, phantom, "="] \\
	& \cat{B}^2
	\ar[r, "{*}"]
		& \cat{B}
\end{tikzcd}
\]
\end{enumerate}
\end{mydef}

\begin{teano}~ \label{ter_bicat_stuff}
\begin{enumerate}[(i)]
\item
The local sets (usually called hom-sets) \( \cat{B}(a,b)(f,g) \) of the local categories are called 2-local sets of \( \cat{B}. \)

\item
If \( P \) is a property of a category, a bicategory \( \cat{B} \) is called locally \( P \) if all local categories of \( \cat{B} \) have the property \( P. \)

\item
We write respectively \( \ob{\cat{B}}, \mor{\cat{B}} \) and \( \mmor{\cat{B}} \) for the collection of objects (=0-morphisms), 1-morphisms and 2-morphisms of \( \cat{B}. \)

\item
We call \( a \) the source and \( b \) the target of a 1-morphism \( f \in \ob{\cat{B}(a,b)}. \)
Similarly we call \( f \) the source (1-morphism), \( g \) the target (1-morphism), \( a \) the source object (=source 0-morphism) and \( b \) the target object of a 2-morphism \( \alpha \in \cat{B}(a,b)(f,g). \)
We use \emph{type} as a collective term for source and target.
The types of an n-morphism \( x \) \( (n \in \set{1,2}) \) are well defined and we write \( s x \) resp. \( t x \) for the source- resp. target-(n-1)-morphism of \( x. \)
If for example \( \alpha \) is a 2-morphism in \( \cat{B}, \)  \( s^2 \alpha = s (s \alpha) \) denotes the source 0-morphism of \( \alpha. \)

When we want to specify all types of a 2-morphism \( \alpha \in \mmor{\cat{B}}, \) we often write instead of \( \alpha \in \cat{B}(a,b)(f,g) \) more suggestively \( \alpha: f \to g: a \to b. \)
Similarly we can specify the types of a 1-morphism \( f \in \mor{\cat{B}} \) via \( f: a \to b. \)

\item \label{list_parallel_morphisms}
Two n-morphisms \( x,y \) are called parallel if \( sx = sy \) and \( tx = ty. \)

\item
Let \( \alpha \in \cat{B}(a,b)(g,h) \) and \( \beta \in \cat{B}(a,b)(f,g), \) then we say \( \alpha \) and \( \beta \) are composable along 1-morphisms
and write \( \alpha \circ \beta \) for the composite of \( \alpha \) and \( \beta \) as morphisms of the category \( \cat{B}(a,b). \)

\item
Let \( f \in \ob{\cat{B}(b,c)} \) and \( g \in \ob{\cat{B}(a,b)}, \) then we say \( f \) and \( g \) are composable along objects and write \( f * g := {*} _{abc} (f,g). \)
Similarly if \( \alpha \in \mor{\cat{B}(b,c)} \) and \( \beta \in \mor{\cat{B}(a,b)}, \)
we call \( \alpha \) and \( \beta \) composable along objects and write \( \alpha * \beta := {*} _{abc} (\alpha,\beta). \)

\item
We denote the single object in the source category of \( I_a \) with \( \bullet \) and call
\[
1_a:= I_a (\bullet): a \to a
\]
the unit (1-morphism) at \( a. \)

The units of the local categories of \( \cat{B} \) are called unit 2-morphisms or 2-units of \( \cat{B}. \)

\item
Components of the natural isomorphisms \( (a,l,r) \) or inverses of such components are called constraint 2-cells or simply constraint cells.
I.e. any constraint cell is -for an appropriate choice of \( (f,g,h) \)- a 2-morphism from the following list:
\[
a _{fgh}, l_f, r_f, a ^{-1}_{fgh}, l ^{-1}_{f}, r ^{-1}_{f}.
\]

\item
The functorality of \( * \) implies
\[
(\alpha_1 * \beta_1) \circ (\alpha_2 * \beta_2) = (\alpha_1 \circ \alpha_2) * (\beta_1 \circ \beta_2),
\]
if the right-hand side is defined (this implies the left hand side is defined).
This fact is called the interchange law for bicategories.
\end{enumerate}
\end{teano}

\begin{bem} \label{rem_bicat_axioms_component_wise}
Above we gave the bicategory-axioms as two equations of natural transformations.
These equations hold if and only if they hold component wise, i.e. if the diagrams

\[
\begin{tikzcd} [column sep=huge]
((f*g)*h)*k
\ar[r, "a _{fgh} * 1_k"]
\ar[d, "a _{(f*g)hk}"]
	& (f*(g*h))*k
	\ar[r, "a _{f(g*h)k}"]
		& f*((g*h)*k)
		\ar[d, "1_f * a _{ghk}"] \\
(f*g)*(h*k)
\ar[rr, "a _{fg(h*k)}"]
		& & f*(g*(h*k))
\end{tikzcd}
\]

and

\[
\begin{tikzcd}
(f*1 _{sf})*g
\ar[rr, "a _{f1g}"]
\ar[dr, "r_f * 1_g" swap]
		& & f * (1 _{sf} *g)
		\ar[dl, "1_f * l_g"] \\
	& f*g.
\end{tikzcd}
\]

commute for all composable \( f,g,h,k. \)
\end{bem}

Later, we will see that any bicategory is equivalent to a so called strict bicategory.

\begin{mydef} \label{def_2_category}
A bicategory \( (\cat{B};*,I;a,l,r) \) is called a strict bicategory or a 2-category, if \( a,l,r \) are identity natural transformations.
\end{mydef}

\begin{bei}~ \label{ex_bicats}
\begin{enumerate}
\item \label{list_the_bicat_cat}
There is a 2-category \( \cat{Cat}, \) whose objects are categories, whose 1-morphisms are functors and whose 2-morphisms are natural transformations.

\item
A bicategory with one object is called a monoidal category.
Any category \( \cat{C} \) with finite products is a monoidal category:

In order to define horizontal composition, we choose for any pair \( (a,b) \) of objects of \( \cat{C} \) a product \( (a \times b), p_a: a \times b \to a, p_b: a \times b \to b. \)
This gives rise to a functor \( \times: \cat{C} \times \cat{C} \to \cat{C}, \) which sends a morphism \( (f,g):(a,b) \to (a',b') \) to the unique morphism \( f \times g: a \times b \to a' \times b', \)
which satisfies \( p _{a'} \circ (f \times g) = f \circ p_a \) and \( p _{b'} \circ (f \times g) = g \circ p_b. \)

The unit (at the only object of \( \cat{C} \)) is the terminal object \( 1 \in \ob{\cat{C}}, \) which exists since it is the empty product.

\( l_a: 1 \times a \to a \) and \( r_a: a \times 1 \to a \) are given by the projections \( p_a ,\) and \( a _{b_1 b_2 b_3} \) is defined in the obvious way through a universal property.

\item
Dually to the previous example any category with finite coproducts is a monoidal category.

\item
The category \( \cat{Ab} \) of abelian groups can be made into a monoidal category not only with the product or the coproduct,
but also with the usual tensor product \( \otimes \) of abelian groups, which "linearises bilinear maps".
Lots of categories which have a similar notion of bilinear maps can be given the structure of a monoidal category in an analogous way.
This can be made precise by means of so called multicategories.

\item
For any monoidal category \( \cat{V} \) there is a 2-category of \( \cat{V} \)-enriched categories (see \cite{Kelly1982} chapter 1.2.).
By setting \( \cat{V} = \cat{Set} \) we obtain example \ref{ex_bicats} (\ref{list_the_bicat_cat}) that way.

\item
There is a bicategory \( \cat{Bim} \) whose objects are rings, whose morphisms are bimodules \( \bimod{R}{M}{S}: S \to R, \) and whose 2-morphisms are homomorphisms of bimodules.

Horizontal composition is given as the tensor product over rings: \( \bimod{T}{N}{S} * \bimod{S}{M}{R} := \bimod{T}{(N \otimes_S M)}{R}. \)

The unit at \( S \) is the ring \( S \) considered as an \( (S,S) \)-bimodule.
\end{enumerate}
\end{bei}
\subsection{String diagrams for 2-categories} \label{sec_string_diagrams_bicats}

In a 1-category one can incorporate associativity in the notation of composition, by allowing unbracketed composites of morphisms, since any choice of a bracketing defines the same morphism.
This is sometimes called generalized associativity law.

Now let us consider notation for composites in a 2-category.
Since composition of 1-morphisms in a 2-category is strictly associative (in fact the 1-morphisms constitute a 1-category), one can adopt the convention from above and omit brackets in \( * \)-composites of 1-morphisms.
But the interchange law, which can be seen as a 2-dimensional associativity law for 2-morphisms, cannot be incorporated in a 1-dimensional notation nicely:
Although \( (\alpha * \beta) \circ (\alpha' * \beta') \) is equal to \( (\alpha \circ \alpha') * (\beta \circ \beta') \) one has to decide for one variant when writing down this morphism.
For large composites this makes it difficult to detect equality of 2-morphisms.
Writing composites in a 2-category in form of (2-dimensional) string diagrams (a short introduction can be found in \cite{Marsden2014}) addresses this problem of 1-dimensional notation.

In string diagram notation, objects are represented as regions, 1-morphisms as vertical strings between regions and 2-morphisms as points between strings.
More precisely an object \( a, \) a 1-morphism \( f: a \to b \) and a 2-morphism \( \alpha: f \to g: a \to b \) are denoted
\[
\begin{tikzpicture}[baseline = {([yshift=-.5ex]current bounding box.center)}, math mode]
\draw[densely dotted] node {\scriptstyle{a}} +(-.5,-.5) rectangle +(.5,.5);
\end{tikzpicture}
\quad
\begin{tikzpicture}[baseline = {([yshift=-.5ex]current bounding box.center)}, math mode]
\draw (0,0) -- node [pos=.7, auto, swap, inner sep = .1mm] {\scriptstyle{f}} (0,1);
\path (-.5,.5) node [yshift=-1mm] {\scriptstyle{b}} +(1,0) node [yshift=-1mm] {\scriptstyle{a}};
\draw[densely dotted] (-1,0) rectangle (1,1);
\end{tikzpicture}
\quad \text{and} \quad
\begin{tikzpicture}[string={1}{1}, baseline = {([yshift=-.5ex]current bounding box.center)}, math mode]
\draw (0,0) -- node [pos=1, above, swap, inner sep = .1mm] {\scriptstyle{f}}
	node [pos=0, below, swap, inner sep = .1mm] {\scriptstyle{f}}
	node [morphism, opacity=1, fill=white] {\scriptstyle{\alpha}} (0,1);
\path (-.5,.5) node [yshift=-1mm] {\scriptstyle{b}} +(1,0) node [yshift=-1mm] {\scriptstyle{a}};
\draw[densely dotted] (-1,0) rectangle (1,1);
\end{tikzpicture}.
\]

\( * \)-composites of 1-morphisms are denoted by gluing diagrams horizontally together, i.e.
\(
\begin{tikzpicture}[string={.4cm}{.5cm}, rounded corners=7, baseline={([yshift=-.5ex]current bounding box.center)}, math mode]
\node (s1) {}; \node (s2) [r=of s1] {};
\node (t1) [b=of s1] {}; \node (t2) [b=of s2] {};
\draw (s1.90) -- node [inner sep = .1mm, auto] {\scriptstyle{f}} (t1.270);
\draw (s2.90) -- node [inner sep = .1mm, auto] {\scriptstyle{g}} (t2.270);
\draw[densely dotted, rounded corners=0] ([xshift=2mm]current bounding box.south east) rectangle ([xshift=-2mm]current bounding box.north west);
\end{tikzpicture}
\)
(here and in the following we omit the labels of the areas) denotes the 1-morphism \( f * g. \)
Since composition of 1-morphisms in a 2-category is associative a diagram with more than two vertical strings like
\(
\begin{tikzpicture}[string={.4cm}{.5cm}, rounded corners=7, baseline={([yshift=-.5ex]current bounding box.center)}, math mode]
\node (s1) {}; \node (s2) [r=of s1] {}; \node (s3) [r=of s2] {};
\node (t1) [b=of s1] {}; \node (t2) [b=of s2] {}; \node (t3) [b=of s3] {};
\draw (s1.90) -- node [inner sep = .1mm, auto] {\scriptstyle{f}} (t1.270);
\draw (s2.90) -- node [inner sep = .1mm, auto] {\scriptstyle{g}} (t2.270);
\draw (s3.90) -- node [inner sep = .1mm, auto] {\scriptstyle{h}} (t3.270);
\draw[densely dotted, rounded corners=0] ([xshift=2mm]current bounding box.south east) rectangle ([xshift=-2mm]current bounding box.north west);
\end{tikzpicture}
\)
determines a unique 1-morphism.
The fact that in a 2-category \( 1 _{tf} * f = f = f * 1 _{sf}  \) holds can be incorporated in string diagram notation by not drawing unit 1-morphisms, i.e.
\(
\begin{tikzpicture}[string={.4cm}{.5cm}, rounded corners=7, baseline={([yshift=-.5ex]current bounding box.center)}, math mode]
\node (s1) {};
\node (t1) [b=of s1] {};
\path (s1.90) -- node [inner sep = .1mm] {\scriptstyle{a}} (t1.270);
\draw[densely dotted, rounded corners=0] ([xshift=2.8mm]current bounding box.south east) rectangle ([xshift=-2.8mm]current bounding box.north west);
\end{tikzpicture}
=
\begin{tikzpicture}[string={.4cm}{.5cm}, rounded corners=7, baseline={([yshift=-.5ex]current bounding box.center)}, math mode]
\node (s1) {};
\node (t1) [b=of s1] {};
\draw (s1.90) -- node [inner sep = .1mm, auto] {\scriptstyle{1_a}} (t1.270);
\draw[densely dotted, rounded corners=0] ([xshift=2mm]current bounding box.south east) rectangle ([xshift=-2mm]current bounding box.north west);
\end{tikzpicture}.
\)

Now one can allow the dots in the string diagrams to have more than one input and output string.
A 2-Morphism whose source or target morphism is a unit 1-morphism can be represented by a dot with no input resp. output wires.
For example, the diagrams
\[
\begin{tikzpicture}[string={.7cm}{.7cm}, rounded corners=7, baseline={([yshift=-.5ex]current bounding box.center)}, math mode]
\node (s1) {}; \node (s2) [r=of s1] {}; \node (s3) [r=of s2] {};
\node (mo1) [morphism, b=of s2] {\scriptstyle{\alpha}};
\node (t1) [bbm={s1}{s2}] {}; \node (t2) [bbm={s2}{s3}] {};
\draw (s1.90) to [d hr] node [inner sep = .3mm, left, pos=.4, xshift=-.7mm] {\scriptstyle{f}} (mo1) (mo1) to [hl d] node [inner sep = .3mm, left] {\scriptstyle{k}} (t1.270);
\draw (s2.90) -- node [inner sep=.3mm, auto] {\scriptstyle{g}} (mo1);
\draw (s3.90) to [d hl] node [inner sep = .3mm, auto, pos=.4] {\scriptstyle{h}} (mo1) (mo1) to [hr d] node [inner sep = .3mm, right] {\scriptstyle{l}} (t2.270);
\draw[densely dotted, rounded corners=0] ([xshift=2mm]current bounding box.south east) rectangle ([xshift=-2mm]current bounding box.north west);
\end{tikzpicture}
\quad \text{and} \quad
\begin{tikzpicture}[string={.7cm}{.7cm}, rounded corners=7, baseline={([yshift=-.5ex]current bounding box.center)}, math mode]
\node (s1) {}; \node (s2) [r=of s1] {}; \node (s3) [r=of s2] {};
\node (mo1) [morphism, b=of s2] {\scriptstyle{\beta}};
\node (t1) [bbm={s1}{s2}] {}; \node (t2) [bbm={s2}{s3}] {};
\draw(mo1) to [hl d] node [inner sep = .3mm, left] {\scriptstyle{f}} (t1.270);
\draw (mo1) to [hr d] node [inner sep = .3mm, right] {\scriptstyle{g}} (t2.270);
\draw[densely dotted, rounded corners=0] ([xshift=2mm]current bounding box.south east) rectangle ([xshift=-2mm]current bounding box.north west);
\end{tikzpicture}
=
\begin{tikzpicture}[string={.7cm}{.7cm}, rounded corners=7, baseline={([yshift=-.5ex]current bounding box.center)}, math mode, label distance=-1.5mm]
\node (s1) {}; \node (s2) [r=of s1] {}; \node (s3) [r=of s2] {};
\node (mo1) [vertex, b=of s2, label=90:\scriptstyle{\beta}] {};
\node (t1) [bbm={s1}{s2}] {}; \node (t2) [bbm={s2}{s3}] {};
\draw(mo1.center) to [l d] node [inner sep = .3mm, left, pos=.7] {\scriptstyle{f}} (t1.270);
\draw (mo1.center) to [r d] node [inner sep = .3mm, right, pos=.7] {\scriptstyle{g}} (t2.270);
\draw[densely dotted, rounded corners=0] ([xshift=2mm]current bounding box.south east) rectangle ([xshift=-2mm]current bounding box.north west);
\end{tikzpicture}
\]
represent the morphism \( \alpha: f * g * h \to k * l \) and \( \beta: 1 _{s^2 \beta} \to f * g. \)

\( * \)- resp. \( \circ \)-composites of 2-morphisms are denoted by gluing diagrams horizontally resp. vertically together, i.e.
\[
\begin{tikzpicture}[string={.7cm}{.7cm}, rounded corners=7, baseline={([yshift=-.5ex]current bounding box.center)}, math mode]
\node (s1) {}; \node (s2) [r=of s1] {};
\node (mo1) [morphism, b=of s1] {\scriptstyle{\alpha}}; \node (mo2) [morphism, b=of s2] {\scriptstyle{\beta}};
\node (t1) [bb=of s1] {}; \node (t2) [bb=of s2] {};
\draw (s1.90) -- node [inner sep = .1mm, auto] {\scriptstyle{f}} (mo1) -- node [inner sep = .1mm, auto] {\scriptstyle{f'}} (t1.270);
\draw (s2.90) -- node [inner sep = .1mm, auto] {\scriptstyle{g}} (mo2) -- node [inner sep = .1mm, auto] {\scriptstyle{g'}} (t2.270);
\draw[densely dotted, rounded corners=0] ([xshift=2mm]current bounding box.south east) rectangle ([xshift=-2mm]current bounding box.north west);
\end{tikzpicture}
\quad \text{resp.} \quad
\begin{tikzpicture}[string={.7cm}{.7cm}, rounded corners=7, baseline={([yshift=-.5ex]current bounding box.center)}, math mode]
\node (s1) {};
\node (mo1) [morphism, b=of s1] {\scriptstyle{\sigma}};
\node (mo2) [morphism, b=of mo1] {\scriptstyle{\sigma'}};
\node (t1) [b=of mo2] {};
\draw (s1.90) -- node [inner sep = .1mm, auto] {\scriptstyle{h}} (mo1) -- node [inner sep = .1mm, auto] {\scriptstyle{h'}} (mo2) -- node [inner sep = .1mm, auto] {\scriptstyle{h'' }} (t1.270);
\draw[densely dotted, rounded corners=0] ([xshift=2mm]current bounding box.south east) rectangle ([xshift=-2mm]current bounding box.north west);
\end{tikzpicture}
\]
denotes the horizontal composite \( \alpha * \beta \) of morphisms \( \alpha: f \to f' \) and \( \beta:g \to g', \)
resp. the vertical composite \( \sigma' \circ \sigma \) of morphisms \( \sigma: h \to h' \) and \( \sigma': h' \to h''. \)
The fact that \( 1 _{t \sigma} \circ \sigma = \sigma = \sigma \circ 1 _{s \sigma} \) can be incorporated in string diagram notation by not drawing unit 2-morphisms.

We emphasize once more that we think of string diagrams as a \emph{notation} for composites in a 2-category, just as strings of 1-morphisms are a notation for composites in a 1-category.
Analogous to the "generalized associativity law" for 1-categories, one now can show that all possible ways of interpreting a string diagram as a composite of 2-morphisms are equal.
(This statement is made rigorous and gets proven in \cite{Power1990}.
There this happen in the context of pasting diagrams, which are equivalent to string diagrams.)
For example the string diagram
\[
\begin{tikzpicture}[string={.7cm}{.7cm}, rounded corners=7, baseline={([yshift=-.5ex]current bounding box.center)}, math mode]
\node (s1) {}; \node (s2) [r=of s1] {};
\node (mo1) [morphism, b=of s1] {\scriptstyle{\alpha}}; \node (mo2) [morphism, b=of s2] {\scriptstyle{\beta}};
\node (mo3) [morphism, bb=of s1] {\scriptstyle{\alpha'}}; \node (mo4) [morphism, bb=of s2] {\scriptstyle{\beta'}};
\node (t1) [b=of mo3] {}; \node (t2) [b=of mo4] {};
\draw (s1) -- (mo1) -- (mo3) -- (t1);
\draw (s2) -- (mo2) -- (mo4) -- (t2);
\end{tikzpicture}
\]
can be interpreted as the composites
\[ (\alpha' * \beta') \circ (\alpha * \beta) \quad \text{or} \quad (\alpha' \circ \alpha) * (\beta' \circ \beta) \quad \text{or} \quad
\big( (\alpha' \circ \alpha) * \beta' \big) \circ (1 _{s \alpha} * \beta) \quad \text{or} \; \dots \]
which are all equal by the interchange law.

Note also that the string diagrams
\[
\begin{tikzpicture}[string={.7cm}{.7cm}, rounded corners=5, baseline={([yshift=-.5ex]current bounding box.center)}, math mode, label distance=-1.5mm]
\node (s1) {}; \node (s2) [r=of s1] {};
\node (mo1) [morphism, b=of s1] {\scriptstyle{\alpha}}; \node (mo2) [morphism, bb=of s2] {\scriptstyle{\beta}};
\node (t1) [bb=of mo1] {}; \node (t2) [b=of mo2] {};
\draw (s1) -- (mo1) -- (t1);
\draw (s2) -- (mo2) -- (t2);
\end{tikzpicture}
\quad \text{and} \quad
\begin{tikzpicture}[string={.7cm}{.7cm}, rounded corners=5, baseline={([yshift=-.5ex]current bounding box.center)}, math mode, label distance=-1.5mm]
\node (s1) {}; \node (s2) [r=of s1] {};
\node (mo1) [morphism, bb=of s1] {\scriptstyle{\alpha}}; \node (mo2) [morphism, b=of s2] {\scriptstyle{\beta}};
\node (t1) [b=of mo1] {}; \node (t2) [bb=of mo2] {};
\draw (s1) -- (mo1) -- (t1);
\draw (s2) -- (mo2) -- (t2);
\end{tikzpicture}
\quad \text{and} \quad
\begin{tikzpicture}[string={.7cm}{.7cm}, rounded corners=5, baseline={([yshift=-.5ex]current bounding box.center)}, math mode, label distance=-1.5mm]
\node (s1) {}; \node (s2) [r=of s1] {};
\node (mo1) [morphism, b=of s1] {\scriptstyle{\alpha}}; \node (mo2) [morphism, b=of s2] {\scriptstyle{\beta}};
\node (t1) [b=of mo1] {}; \node (t2) [b=of mo2] {};
\draw (s1) -- (mo1) -- (t1);
\draw (s2) -- (mo2) -- (t2);
\end{tikzpicture}
\]
all represent the same 2-morphism and are therefore considered as equal.
More general the height at which the dots in parallel strings are drawn does not matter, as long as we keep the order of the dots within a string.

String diagrams can be used offhandedly only for 2-categories, but many bicategories that appear in nature are non-strict bicategories.
Therefore section \ref{sec_strictification_for_bicats} will explain how string diagrams can also be employed in the non-strict case.
\subsection{Equivalences and adjunctions in bicategories} \label{sec_internalization}

Since categories, functors and natural transformations form a 2-category, one can generalize many definitions from 1-dimensional category theory to general bicategories.
We will do this now for a couple of concepts, which are needed throughout the thesis.

\begin{mydef} \label{def_equivalence_adjunction}
Let \( \cat{B} \) be a bicategory, \( f: a \to b \) and \( g: b \to a \) 1-morphisms in \( \cat{B} \) and \( \eta: 1_a \to g * f \) and \( \epsilon: f * g \to 1_b \) 2-morphisms in \( \cat{B}. \)
Then the 4-tuple \( (f,g,\eta,\epsilon) \)
\begin{itemize}
\item
is called an equivalence, if \( \eta \) and \( \epsilon \) are invertible.

\item \label{list_adjunction_axioms}
is called an adjunction, if the morphisms
\[
f \xrightarrow{r ^{-1}}
f*1 \xrightarrow{1 * \eta}
f*(g*f) \xrightarrow{a ^{-1}}
(f*g)*f \xrightarrow{\epsilon * 1}
1 * f \xrightarrow{l}
f
\]
and
\[
g \xrightarrow{l ^{-1}}
1 * g \xrightarrow{\eta * 1}
(g*f)*g \xrightarrow{a}
g*(f*g) \xrightarrow{1 * \epsilon}
g*1 \xrightarrow{r}
g
\]
are identities.

In this case we call \( \eta \) the unit and \( \epsilon \) the counit of the adjunction \( (f,g,\eta,\epsilon). \)

\item
is called an adjoint equivalence, if it is a equivalence and an adjunction.
\end{itemize}
\end{mydef}

A 1-morphism already gets a special name if it can be regarded as part of an equivalence or an adjunction, whereby the equivalence resp. adjunction need not be specified concretely.

\begin{mydef}~
\begin{enumerate}[(i)]
\item
A 1-morphism \( f \) in \( \cat{B} \) is called an \emph{equivalence}, if there exists an equivalence \( (f,g,\eta,\epsilon) \) in the sense of definition \ref{def_equivalence_adjunction}.
\item
Now let \( h,k \) be 1-morphisms in a bicategory \( \cat{B}. \)
\begin{itemize}[-]
\item
\( h \) is called a \emph{left adjoint}, if there exists an adjunction \( (h,g,\eta,\epsilon). \)
\item
\( h \) is called a \emph{right adjoint}, if there exists an adjunction \( (f,h,\eta,\epsilon). \)
\item
\( h \) is called \emph{left adjoint to \( k \)} (notation \( h \dashv k \)), if there exists an adjunction \( (h,k,\eta,\epsilon). \)
\end{itemize}
\end{enumerate}
\end{mydef}

An equivalence in \( \cat{Cat} \) has the  following alternative characterization.

\begin{lem} \label{lem_characterization_adjunction}
A functor \( F: \cat{C} \to \cat{D} \) is an equivalence if and only if it is essentially surjective and locally a bijection.
This means that for any \( d \in \ob{\cat{D}} \) there exists a \( c \in \ob{\cat{C}}, \) such that \( Fc \) is isomorphic to \( d, \)
and that the assignments \( F _{ab}: \cat{C}(a,b) \to \cat{D}(Fa,Fb) \) are bijections.
\end{lem}

An Adjunction in \( \cat{Cat} \) is already uniquely determined by a small part of the data of definition \ref{def_equivalence_adjunction}.
In order to make this precise we need the notion of an universal arrow.

\begin{mydef}
Let \( \cat{C} \) and \( \cat{D} \) be categories, \( S: \cat{D} \to \cat{C} \) a functor and \( c \in \ob{\cat{C}}. \)
A universal arrow from \( c \) to \( S \) is a tuple \( (r,u) \) consisting of an object \( r \in \ob{D} \) and a morphism \( u: c \to Sr, \)
such that for any \( d \in \ob{D} \) and any morphism \( f: c \to Sd \) there exists a unique morphism \( f': r \to d \) with \( Sf' \circ u = f. \)
\end{mydef}

\begin{satz} \label{prop_characterization_adjunctions_in_cat}.
Let \( G: \cat{D} \to \cat{C} \) be a functor and let for any \( c \in \ob{\cat{C}} \) \( ({F} c, \eta_c: c \to GF c) \) be  a universal arrow from \( c \) to \( G. \)
Then there exists a unique adjunction \( (F, G, \eta, \epsilon), \)
such that \( c \mapsto {F} c \) is the object-function of \( F \) and \( \eta_c \) are the components of the counit \( \eta. \)
\end{satz}

\begin{proof}
For a proof see \cite{MacLane2013} chapter IV theorem 2.
Here we just give the definition of the remaining data, namely the action of \( F \) on morphisms and the components of the natural isomorphism \( \epsilon. \)

For a morphism \( f:c \to c' \) a morphism in \( \cat{C}, \)
\( Ff: Fc \to Fc' \) is defined as the unique morphism (here we use the universality of \( \eta_c \)), which makes the diagram
\[
\begin{tikzcd}
c
\ar[r, "\eta_c"]
\ar[d, "f" swap]
	& G{F} c
	\ar[d,"G Ff "] \\
c'
\ar[r,"\eta _{c'}" swap]
	& G{F} c'
\end{tikzcd}
\]
commute.

For \( a \in \ob{D}, \) \( \epsilon_a: FG a \to a \) is the unique morphism, such that \( G \epsilon_a \circ \eta _{Ga} = 1_{Ga}. \)
\end{proof}
\subsection{Bicategorical transfors} \label{sec_bicat_transfors}

Similar as a bicategory is a categorification of a category (see the discussion at the beginning of this section), a weak functor is a categorification of a 1-categorical functor:
The local functions \( F _{ab} \) between local sets (=hom sets) get replaced by (local) functors between local categories.
The axioms of a 1-categorical functor get replaced by natural isomorphisms and thus become data in the definition of a weak functor.
Therefore new axioms are imposed in form of equations of natural isomorphisms.

\begin{mydef} \label{def_weak_functor}
A weak functor \( (F,\phi) \) between bicategories \( \cat{A} \) and \( \cat{B} \) consists of the following data:
\begin{itemize}
\item
An assignment \( F: \ob{\cat{A}} \to \ob{\cat{B}}. \)

\item
For each tuple \( (a,b) \) of objects of \( \cat{A} \) a functor \( F _{ab}: \cat{A}(a,b) \to \cat{B}(Fa,Fb). \)

\item
For each triple \( (a,b,c) \) of objects of \( \cat{A} \) a natural isomorphism
\[
\begin{tikzcd}[column sep = large]
\cat{A}(b,c) \times \cat{A}(a,b)
\ar[r, "* _{abc}"]
\ar[d, "F _{bc} \times F _{ab}" swap]
	& \cat{A}(a,c)
	\ar[d, "F _{ac}"] \\
\cat{B}(Fb,Fc) \times \cat{B}(Fa,Fb)
\ar[r, "* _{Fa Fb Fc}" swap]
\ar[ru, Rightarrow, "\phi _{abc}" swap, short = 5mm]
	& \cat{B}(Fa,Fc)
\end{tikzcd}
\]

\item
For each object \( a \) of \( \cat{A} \) a natural isomorphism

\[
\begin{tikzcd} [column sep=large]
\cat{1}
\ar[r, "I _{a}"]
\ar[dr, "I _{Fa}" {swap, name=s1}]
	& \cat{A}(a,a)
	\ar[d, "F _{aa}"]
	\ar[from=s1, Rightarrow, "\phi_a" {swap, pos=.8}, shorten <= 3mm] \\
	& \cat{B}(Fa,Fa)
\end{tikzcd}
\]
\end{itemize}

The components of these data have to make the following diagrams -in which we omit most indices- commute:
\begin{align} \label{diag_functor_axiom_asso}
\begin{tikzcd}[ampersand replacement=\&]
(Ff * Fg) * Fh
\ar[r, "\phi * 1"]
\ar[d, "a"]
	\& F(f*g)*Fh
	\ar[r, "\phi"]
		\& F((f*g)*h)
		\ar[d, "Fa"] \\
Ff*(Fg*Fh)
\ar[r, "1 * \phi"]
	\& Ff * F(g*h)
	\ar[r, "\phi"]
		\& F(f*(g*h))
\end{tikzcd}
\end{align}

\begin{align} \label{diag_functor_axiom_unit}
\begin{tikzcd}[ampersand replacement=\&]
1 _{F (tf)} * Ff
\ar[r, "l"]
\ar[d, "\phi * 1" swap]
	\& Ff \\
F 1 _{tf} * Ff
\ar[r, "\phi" swap]
	\& F(1 _{tf} * Ff)
	\ar[u, "Fl" swap]
\end{tikzcd}
\quad
\begin{tikzcd}[ampersand replacement=\&]
Ff * 1 _{F (sf)}
\ar[r, "r"]
\ar[d, "1 * \phi" swap]
	\& Ff \\
Ff * F 1 _{	sf}
\ar[r, "\phi" swap]
	\& F(f * 1 _{sf})
	\ar[u, "Fr" swap]
\end{tikzcd}
\end{align}

\end{mydef}

\begin{bem} \label{rem_bicat_1}
Bicategories and weak functors constitute -with the obvious composition - a 1-category \( \cat{Bicat_1}. \)
\end{bem}

\begin{terminology} \label{term_local_functor}
Let \( F \) be a weak functor:

If \( P \) is a property of a 1-categorical functor, \( F \) is called (1-)locally \( P \) if each (local) functor \( F _{ab} \) is \( P. \)

If \( P \) is a property of a function, \( F \) is called 2-locally \( P \) if each functor \( F_{ab} \) is locally \( P. \)
(A 1-categorical functor \( G: \cat{C} \to \cat{D} \) is called locally \( P, \) if the functions \( G _{ab}: \cat{C}(a,b) \to \cat{D}(Ga,Gb) \) are \( P \).)
\end{terminology}

\begin{mydef}
A strict functor of bicategories is a weak functor, for which the transformations \( \phi _{abc} \) and \( \phi _{a} \) are identity transformations.
\end{mydef}

\begin{mydef} \label{def_biequivalence}
A weak functor \( F: \cat{A} \to \cat{B} \) is called  biessentially surjective, if for every object \( b \in \cat{B} \) there exists an object \( a \in \cat{A} \) such that \( Fa \) is equivalent to \( b. \)

A weak functor is called a biequivalence if it is locally an equivalence and biessentially surjective.

A weak functor is called a strict biequivalence, if it is a strict functor and a biequivalence.
\end{mydef}

\begin{bem} \label{rem_biequivalence}
We can unpack the definition of a biequivalence by using the characterization of a equivalence in \( \cat{Cat} \) from lemma \ref{lem_characterization_adjunction}.
Proving that a weak functor \( F \) is a biequivalence then comes down to check two surjectivity and one bijectivity condition, namely
\begin{itemize}
\item
F is biessentially surjective
\item
F is locally essentially surjective
\item
F is 2-locally a bijection
\end{itemize}
\end{bem}

A pseudonatural transformation is the categorification of a (1-categorical) natural transformation.
To see this, one has to consider the components of a natural transformation as functions from the terminal set.
Then it is clear that in the definition of a pseudonatural transformation these functions get replaced by functors from the terminal category.
Similarly, naturality of 1-categorical transformations can be expressed -instead of the usual "naturality squares"- in terms of equations of functions.
Then these equations can be replaced by natural isomorphisms in the definition of a pseudonatural transformation.

\begin{mydef}
A (pseudonatural) transformation \( \sigma \) between weak functors \( F,G: \cat{A} \to \cat{B} \) consists of the following data:
\begin{itemize}
\item
For each object \( a \) in \( \cat{A} \) a functor \( \sigma_a: \cat{1} \to \cat{B}(Fa,Ga). \)
Notationally we make no difference between the functor \( \sigma_a \) and its only component \( \sigma_a (\bullet), \) which is also simply denoted \( \sigma_a. \)

\item
For each pair \( (a,b) \) of objects in \( \cat{A} \) a natural transformation
\[
\begin{tikzcd}
\cat{A}(a,b)
\ar[r, "\sigma_b \times F _{ab}"]
\ar[d, "G _{ab} \times \sigma_a" swap]
	& \cat{B}(Fb,Gb) \times \cat{B}(Fa,Fb)
	\ar[d, "*"]
	\ar[dl, Rightarrow, short=5mm, "\sigma _{ab}"] \\
\cat{B}(Ga,Gb) \times \cat{B}(Fa,Ga)
\ar[r,"*" swap]
	& \cat{B}(Fa,Gb)
\end{tikzcd}
\]

The components of these data have to make the following diagrams commute.

\[
\begin{tikzcd}
\sigma * (Ff * Fg)
\ar[r, "a ^{-1}"]
\ar[d, "1 * \phi ^{F}"]
	& (\sigma * Ff) * Fg
	\ar[r, "\sigma * 1"]
		& (Gf * \sigma) * Fg
		\ar[r, "a"]
			& Gf * (\sigma * Fg)
			\ar[d, "1 * \sigma"] \\
\sigma * F (f*g)
\ar[r, "\sigma"]
	& G (f*g) * \sigma
		& (Gf * Gg) * \sigma
		\ar[l, "\phi ^{G} * 1"]
			& Gf * (Gg * \sigma)
			\ar[l, "a ^{-1}"]
\end{tikzcd}
\]

\[
\begin{tikzcd}
\sigma * 1 _{Fa}
\ar[r, "r"]
\ar[d, "1 * \phi ^{F}" swap]
	& \sigma
	\ar[r, "l ^{-1}"]
		& 1 _{Ga} * \sigma
		\ar[d, "\phi ^{G} * 1"] \\
\sigma * F(1 _{a})
\ar[rr, "\sigma" swap]
		& & G(1 _{a}) * \sigma
\end{tikzcd}
\]
\end{itemize}
\end{mydef}

Roughly speaking (1-categorical) functors or weak functors carry n-cells of a category resp. bicategory to n-cells of another category resp. bicategory.
Similarly natural transformations or pseudonatural transformations carry n-cells to (n+1)-cells.
In a bicategory it makes also sense to carry n-cells to (n+2)-cells, i.e. objects to 2-morphisms.
A modification does just that, in a certain coherent way.

\begin{mydef} \label{def_modification}
A modification \( \Gamma \) between bicategorical transformations \( \sigma, \tau: F \to G: \cat{A} \to \cat{B} \) consists of an \( \ob{\cat{A}} \)-indexed collection of 2-morphisms
\( \Gamma_a: \sigma_a \to \tau_a: F_a \to G_a, \) subject to the the following axiom:
\[
\begin{tikzcd}
\sigma * Ff
\ar[r, "\sigma"]
\ar[d, "\Gamma * 1" swap]
	& Gf * \sigma
	\ar[d, "1 * \Gamma"] \\
\tau * Ff
\ar[r, "\tau" swap]
	& Gf * \tau
\end{tikzcd}
\]
\end{mydef}

\begin{bem}
To stay in line with previous definitions, one could have described a component \( \Gamma_a \) from definition \ref{def_modification} equivalently as
the only component of a natural transformation \( \Gamma_a: \sigma_a \to \tau_a: \cat{1} \to \cat{B}(Fa,Ga). \)
\end{bem}

The discussion before definition \ref{def_modification} motivates the introduction of the terminology of transfors (see \cite{nlab:transfor} for the rough idea and \cite{Crans2003} for a rigorous treatment).

\begin{terminology}~ \label{term_transfor}
\begin{itemize}[-]
\item
A 0-categorical 0-transfor is a function (between sets).
\item
A 1-categorical 0-transfor is a functor, a 1-categorical 1-transfor is a natural transformation.
\item
A 2-categorical 0-transfor is a functor, a 2-categorical 1-transfor is a pseudonatural transformation and a 2-categorical 2-transfor is a modification.
\end{itemize}
\end{terminology}

The 1-categorical transfors between categories \( \cat{C} \) and \( \cat{D} \) constitute a category.
Similarly the 2-categorical transfors between bicategories \( \cat{A} \) and \( \cat{B} \) constitute a bicategory.
Later, in the definition of a tricategory, adjoint equivalences in such "functor bicategories" will appear.

\begin{satz} \label{prop_functor_bicategory}
Let \( \cat{A} \) and \( \cat{B} \) be bicategories.
There is a bicategory \( \cat{Bicat}(\cat{A},\cat{B}) \) whose objects are weak functors from \( \cat{A} \) to \( \cat{B}, \)
whose 1-morphisms are pseudonatural transformations between such functors and whose 2-morphisms are modifications between such pseudonatural transformations.
\end{satz}

\subsection{Coherence for bicategories} \label{sec_coherence_for_bicats}

In this section we give the coherence theorem for bicategories.
Everything in this section is based on \cite{Gurski2013} section 2.2.

Stating the coherence theorem for bicategories requires the construction of free bicategories over category-enriched graphs.
Therefore we start by giving the definition of a graph, that is enriched over a category \( \cat{V}. \)

\begin{mydef} \label{def_V_graph}
Let \( \cat{V} \) be a category.

\begin{enumerate}
\item
A \( \cat{V} \)-enriched graph (short \( \cat{V} \)-graph) \( G \) consists of
\begin{itemize}
\item
a collection of objects \( \ob{G}. \)
\item
for each ordered pair \( (a,b) \) of objects of \( G \) an object \( G(a,b) \) of \( \cat{V}. \)
\end{itemize}
\item
A morphism \( f: G \to G' \) of \( \cat{V} \)-graphs \( G \) and \( G' \) consists of
\begin{itemize}
\item
an assignment \( f_0: \ob{G} \to \ob{G'}. \)
\item
for each ordered pair \( (a,b) \) of objects of \( G \) a morphism \( f _{ab}: G(a,b) \to G'(f_0 a, f_0 b) \) in \( \cat{V}. \)
\end{itemize}
\end{enumerate}

\( \cat{V} \)-graphs and morphisms of \( \cat{V} \)-graphs constitute a category \( \cat{V} \)-\( \cat{Gph}. \)
\end{mydef}

\begin{notation}
Let \( \cat{Cat}_1 \) denote the 1-category of categories and functors.
We call a \( \cat{(Cat_1)} \)-graph simply a \( \cat{Cat} \)-graph, and define \( \cat{Cat} \)-\( \cat{Gph} := \cat{(Cat_1)} \)-\( \cat{Gph}. \)
\end{notation}

\begin{bei} \label{ex_underlying_cat_graph}
Any bicategory has an underlying \( \cat{Cat} \)-graph \( U \cat{B}, \) defined via
\[
\ob{U \cat{B}}:=\ob{\cat{B}}
\quad \text{and} \quad
U \cat{B}(a,b):= \cat{B}(a,b).
\]

Every weak functor \( S:\cat{B} \to \cat{B'} \) has an underlying morphism of \( \cat{Cat} \)-graphs \( US: U \cat{B} \to U \cat{B'}, \) defined via
\[
(US)_0:=S_0
\quad \text{and} \quad
(US) _{ab}:= S _{ab}.
\]

This gives rise to a forgetful-functor \( U: \cat{Bicat_1} \to \cat{Cat \text{-} Gph}, \) where  \( \cat{Bicat_1} \) is the 1-category with objects bicategories and morphisms weak functors (see \ref{rem_bicat_1}).
\end{bei}

Now we want to define the free bicategory \( F G \) over a \( \cat{Cat} \)-graph \( G. \)
The morphisms of \( FG \) will be formal composites of adjunct units, adjunct constraint cells and cells of the "edge-categories" of \( G. \)
For example a generic 1-morphism of \( FG \) looks like
\[
f \fstar \big( (g \fstar \fc 1 _{c}) \fstar (h \fstar k) \big),
\]
where \( f \in G(d,e), g \in G(c,d), h \in G(b,c), k \in G(a,b) \) are objects in edge-categories and \( \fc 1 _{c} \) is an adjunct unit.
Finally one identifies the 2-morphisms of \( FG, \) that needs to be equal by the requirement that \( FG \) is a bicategory.

In order to define the n-morphisms of \( FG \) the concept of a typed set will be useful.

\begin{mydef}
Let \( s,t: M \to N \) be functions between sets \( M \) and \( N \).
Then we call the triple \( (M,s,t) \) an \emph{N-typed set}.
\end{mydef}

In universal algebra a congruence relation on an algebra is an equivalence relation compatible with the algebra-operations.
In this thesis there appear often "algebras", whose binary operations are only defined for certain ordered pairs of elements of the "algebra".
One can define a congruence relation on such "algebras" analogously to the congruence relation for algebras in the sense of universal algebra.

\begin{mydef}~ \label{def_congruence_rel}
\begin{enumerate}
\item
A \emph{partially defined n-ary operation} \( \square \) on a set \( M, \) is a function \( A \to M, \) where \( A \subseteq M ^{\times n} \) is a subset of the n-fold Cartesian product of \( M. \)
\item
Let \( \square \) be a partially defined n-ary operation on a set \( M. \)
An equivalence relation \( \sim \) on \( M \) is called \emph{compatible} with \( \square, \) if for \( x_i, y_i \in M, \, i \in \set{1 \dots n} \) with \( x_i \sim y_i \)
the existence of \( x_1 \square \dots \square x_n \) and \( y_1 \square \dots \square y_n \) implies \( x_1 \square \dots \square x_n \sim  y_1 \square \dots \square y_n. \)
\item
Let \( (\square_i) _{i \in \set{1 \dots n}} \) be partially defined operations on a set \( M. \)
An equivalence relation \( \sim \) on \( M \) is called a \emph{congruence relation} with respect to the partially defined operations \( (\square_i) _{i \in \set{1 \dots n}}, \)
or short an \( (\square_i) _{i \in \set{1 \dots n}} \)-congruence relation, if it is compatible with the operations \( (\square_i) _{i \in \set{1 \dots n}}. \)
\end{enumerate}
\end{mydef}

\begin{mydef}[free bicategory over a cat-graph] \label{def_free_bicat}
The \emph{free bicategory \( \oper{F} G \)} over a cat-graph \( G \) is defined as follows:

\begin{itemize}
\item
\( \ob{\oper{F}G} := \ob{G}. \)

\item
\( \mor{\oper{F}G} \) is an \( \ob{\oper{F}G} \)-typed collection and defined recursively as follows:

\( \mor{\oper{F}G} \) contains the so called \( F \)-basic 1-morphisms which are
\begin{itemize}[-]
\item
for each \( f \in \ob{G(a,b)} \) a 1-morphism \( f: a \to b, \)
\item
for each object \( a \in \ob{G} \) a 1-morphism \( \fc{1} _a: a \to a. \)
\end{itemize}

For \( v: b \to c \) and \( w: a \to b \) in \( \mor{\oper{F}G} \) we require the formal composite \( v \fstar w: a \to c \) to be a 1-morphism of \( FG. \)

\item
Composition of 1-morphisms is given by the \( \fstar \) operator and the unit at \( a \) is given as \( \fc 1_a. \)

\item
\( \mmor{\oper{F}G} \) is obtained by an auxiliary \( \mor{\oper{F}G} \)-typed collection \( \overline{\mmor{\oper{F}G}} \) by quotiening out a congruence relation.
\( \overline{\mmor{\oper{F}G}} \) is recursively defined as follows:

\( \overline{\mmor{\oper{F}G}} \) contains the following \( \bar F \)-basic 2-morphisms:

\begin{itemize}[-]
\item
For any \( \alpha \) in \( G(a,b)(f,g) \) a 2-morphism \( \alpha: f \to g. \)
\item
For each \( v \in \mor{\oper{F}G} \) 2-morphisms
\begin{gather*}
\fc 1_v: v \to v \\
\fc l_v: \fc 1 _{tv} \fstar v \to v \quad \text{and} \quad (\fc l_v) ^{-1}: v \to \fc 1 _{tv} \fstar v \\
\fc r_v: v \fstar \fc 1 _{sv} \to v \quad \text{and} \quad (\fc r_v) ^{-1}: v \to v \fstar \fc 1 _{sv}
\end{gather*}
\item
For each triple \( (v,w,x) \) of composable 1-morphisms, 2-morphisms
\begin{gather*}
\fc a _{vwx}: (v \fstar w) \fstar x \to v \fstar (w \fstar x) \\
\text{and} \quad (\fc a _{vwx})^{-1}: v \fstar (w \fstar x) \to (v \fstar w) \fstar x.
\end{gather*}
\end{itemize}

For \( \tau:w \to w' : b \to c, \; \sigma: v \to v' : a \to b \) and \( \sigma':v' \to v'': a \to b \) in \( \overline{\mmor{\oper{F}G}} \) we require the formal composites
\( \tau \fstar \sigma: w \fstar v \to w' \fstar v': a \to c \) and \( \sigma' \circ \sigma: v \to v'': a \to b \) to be contained in \( \overline{\mmor{\oper{F}G}}. \)

\item
Now we define the collection \( \mmor{\oper{F}G} := \overline{\mmor{\oper{F}G}} / \sim \) by taking equivalence classes with respect to the congruence relation \( \sim, \) defined in the following. The types of a \( [\sigma] \in \mmor{\oper{F}G} \) are then defined by \( s [\sigma] = [s \sigma] \) resp. \( t [\sigma] = [t \sigma], \) which gives \( \mmor{\oper{F}G} \) the structure of a \( \mor{\oper{F}G} \)-typed collection.

\( \sim \) is the \( (\fstar,\fcirc) \)-congruence relation (see definition \ref{def_congruence_rel}) generated by the following list of relations,
where \( v,w,x,y \in \mor{\oper{F}G}, \; \) \( \sigma,\sigma',\sigma'',\tau,\tau' \in \overline{\mmor{\oper{F}G}} \) and each relation holds whenever the composites in the relation are defined.
\begin{enumerate}
\item \label{list_associativity_local_cat}
\( (\sigma \fcirc \sigma') \fcirc \sigma'' \sim \sigma \fcirc (\sigma' \fcirc \sigma''). \)
\item \label{list_unitality_local_cat}
\( \fc 1 _{t \sigma} \fcirc \sigma \sim \sigma \) and \( \sigma \fcirc \fc 1 _{s \sigma} \sim \sigma. \)
\item \label{list_F_respects_local_comp}
\( \sigma \fcirc \tau \sim (\sigma \circ \tau) \) (if \( \sigma,\tau \) are composable morphisms in a category \( G(a,b). \))
\item \label{list_F_respects_local_units}
\( \fc 1_v \sim 1_v \) (if \( v \) is an object in a category \( G(a,b). \))
\item \label{list_functorality_fstar}
\( (\sigma \fstar \tau) \fcirc (\sigma' \fstar \tau') \sim (\sigma \fcirc \sigma') \fstar (\tau \fcirc \tau'). \)
\item \label{list_unitality_fstar}
\( \fc 1_v \fstar \fc 1_w \sim \fc 1 _{v \fstar w}. \)
\item \label{list_constraints_are_invertible}
\( \fc a _{vwx} \fcirc (\fc a _{vwx}) ^{-1} \sim 1 _{v \fstar (w \fstar x)}, \quad (\fc a _{vwx}) ^{-1} \fcirc \fc a _{vwx} \sim 1 _{(v \fstar w) \fstar x} \) and analogously for \( l \) and \( r. \)
\item \label{list_constraints_are_natural}
\( \fc a _{(t \sigma) (t \sigma') (t \sigma'')} \fcirc ((\sigma \fstar \sigma') \fstar \sigma'') \sim (\sigma \fstar (\sigma' \fstar \sigma'')) \fcirc \fc a _{(s \sigma) (s \sigma') (s \sigma'')} \)
and analogously for \( l \) and \( r. \)
\item \label{list_pentagon}
\( (\fc 1_v \fstar \fc a _{wxy}) \fcirc \fc a _{v(wx)y} \fcirc (\fc a _{vwx} \fstar \fc 1_y) = \fc a _{vw(xy)} \fcirc \fc a _{(vw)xy}. \)
\item \label{list_triangle}
\( (\fc 1_v \fstar \fc l_w) \fcirc \fc a _{v\fc 1_aw} = (\fc r_w \fstar \fc 1_v). \)
\end{enumerate}

\item
Composition along objects resp. 1-morphisms of 2-morphisms in \( \oper{F}G \) is defined by
\[
[\sigma] \fstar [\tau] := [\sigma \fstar \tau]
\quad \text{and} \quad
[\sigma] \fcirc [\sigma'] := [\sigma \fcirc \sigma'].
\]
2-units and constraint-2-cells of \( \oper{F}G \) are defined in the obvious way and carry an \( F \)-superscript.
For example the unit at \( v \in \mor{\oper{F}G} \) is given as \( \fc 1 _{v} :=  [\fc 1 _{v}], \)
or the associator of \( (v,w,x) \) is given as \( \fc a _{vwx} := [\fc a _{vwx}]. \)

We call a 2-morphism of \( FG \) \( F \)-basic, if it is the image of an \( \bar F \)-basic 2-morphism under the map \( [-]. \)
If no confusion is possible, we identify a 2-morphism of \( FG \) with its representative in \( \overline{\mmor{FG}}, \) i.e. we omit the \( [\,] \)-brackets.
\end{itemize}
\end{mydef}

\begin{satz}
\( \oper{F}G \) from definition \ref{def_free_bicat} constitutes a bicategory.
\end{satz}

\begin{proof}
We sketch the proof roughly.
\begin{itemize}[-]
\item
The elements of \( \mor{\oper{F}G} \) with types \( a \to b \) and the elements of \( \mmor{\oper{F}G} \) between such elements of \( \mor{\oper{F}G} \) constitute a local category \( \oper{F}G(a,b), \)
because of relation \eqref{list_associativity_local_cat} and \eqref{list_unitality_local_cat} in definition \ref{def_free_bicat}.
Moreover any 2-morphism \( \sigma \) of \( \oper{F}G \) is contained in such a local category, since \( s^2 \sigma = st \sigma \) and \( t^2 \sigma = ts \sigma \) holds by construction.
\item
Functorality of \( \fstar \) follows from relation \eqref{list_functorality_fstar} and \eqref{list_unitality_fstar}.
\item
Let \( \bullet \) be the single object of the terminal category \( \cat{1}, \) then
\begin{gather*}
I_a: \cat{1} \to \oper{F}G(a,b) \\
\bullet \mapsto \fc 1_a, \quad 1 _{\bullet} \mapsto \fc 1 _{(\fc 1_a)}
\end{gather*}
is the required unit-functor at \( a. \)
\item
\eqref{list_constraints_are_invertible} and \eqref{list_constraints_are_natural} ensure that \( a,l \) and \( r \) are natural isomorphisms.
\item
Due to \eqref{list_pentagon} and \eqref{list_triangle} the bicategory axioms are satisfied (see also remark \ref{rem_bicat_axioms_component_wise}).
\end{itemize}
\end{proof}

Let \( \cat{Bicat_1^s} \) be the category with bicategories as objects and strict functors as morphisms
and let \( U: \cat{Bicat_1^s} \to \cat{Cat \text{-} Gph} \) be the forgetful functor defined analogously to example \ref{ex_underlying_cat_graph}.
The next proposition shows that \( FG \) satisfies the universal property, that any free construction should satisfy.

\begin{satz} \label{prop_bicat_catgraph_adjunction}
For any cat-graph \( G \) the inclusion \( \iota_G: G \rightarrow UFG \) is a universal arrow from \( G \) to \( U: \cat{Bicat_1^s} \to \cat{Cat \text{-} Gph}: \)

For every bicategory \( \cat{B} \) and every morphism of \( \cat{Cat} \)-graphs \( S: G \to U \cat{B}, \) there is a unique strict functor \( S': FG \to \cat{B} \) with \( US' \circ \iota = S. \)
\end{satz}

\begin{proof}
First we show that \( \iota: G \to UFG \) is a morphism of \( \cat{Cat} \)-graphs.
To do this, we show that \( \iota _{ab}: G(a,b) \to UFG(a,b), \) which maps objects and 1-morphisms of the category \( G(a,b) \) to the objects and 1-morphisms in \( UFG(a,b) \) with the same name, is a functor.
This holds, because relation \eqref{list_F_respects_local_comp} in definition \ref{def_free_bicat} implies that \( \iota \) preserves composition
and relation \eqref{list_F_respects_local_units} implies that \( \iota \) preserves units.

Now we show that for any morphism of cat-graphs \( S: G \to U \cat{B}, \) there exists a unique strict functor \( S': FG \to \cat{B}, \) such that \( S=US' \circ \iota: \)

\begin{itemize}[-]
\item
We set \( S'(a) = S(a) \) on objects.

\item
On 1-morphism \( S' \) is defined inductively as follows:

For an \( F \)-basic 1-morphism \( f, \) we define
\[
S'(f) =
\begin{cases}
S(f)
	& \text{if } f \in \ob{G(a,b)} \text{ for some } a,b \in \ob{G}, \\
1 _{Sa}
	& \text{if } f= \fc 1 _{a}.
\end{cases}
\]

\( S' \) extends to general 2-morphisms via the induction rule
\[
S' (f \fstar g) = S'(f) * S'(g).
\]

\item
Next we inductively define an assignment \( \bar S': \overline{\mmor{FG}} \to \mmor{\cat{B}}: \)

For an \( \bar F \)-basic 2-morphism \( \sigma \) we define
\[
\bar S' (\sigma)=
\begin{cases}
S(\sigma)
	& \text{if } \sigma \in \mor{G(a,b)} \text{ for some } a,b \in \ob{G}, \\
1 _{S'(v)}
	& \text{if } \sigma = \fc 1 _{v}, \\
a _{S'v,S'w,S'x}
	& \text{if } \sigma = \fc a _{vwx}, \\
l _{S'v}
	& \text{if } \sigma = \fc l_v, \\
r _{S'v}
	& \text{if } \sigma = \fc r_v.
\end{cases}
\]

\( \bar S' \) extends to general 2-morphisms via the induction rules
\[
\bar S' (\sigma \fstar \tau) = \bar S'(\sigma) * \bar S'(\tau)
\quad \text{and} \quad
\bar S' (\sigma \fcirc \tau) = \bar S'(\sigma) \circ \bar S'(\tau).
\]

\item
\( \bar S' \) respects the congruence relation on \( \overline{\mmor{FG}} \) from definition \ref{def_free_bicat} (i.e. \( \sigma \sim \tau \implies S'(\sigma) = S'(\tau) \)),
since \( \bar S' \) respects all generators of this relation and preserves composition along objects and 1-morphisms.
Therefore we can define \( S \) on \( \mmor{FG} \) as \( S([\sigma]) = S'(\sigma). \)
\end{itemize}

\( S' \) is a strict functor, since by definition it strictly preserves composition \( (*, \circ) \) and constraint cells.
Furthermore by construction \( S' \) satisfies \( US' \circ \iota_G = S. \)

Now suppose \( P: FG \to \cat{B} \) is another strict functor, which satisfies \( UP \circ \iota_G = S. \)
As a strict functor, such a \( P \) is determined by its values on objects, \( F \)-basic 1-morphisms and \( F \)-basic 2-morphisms.
Due to \( S=UP \circ \iota, \) \( P=S' \) has to hold on objects and on those \( F \)-basic 1- and 2-morphisms of \( FG, \) which originate from \( G. \)
Because of the strictness of \( P \) we have
\begin{gather*}
P(\fc 1_a) = 1_{Pa} = 1_{S'a} = S'(\fc 1_a), \quad P(\fc a _{vwx}) = a _{(Pv)(Pw)(Px)} = a _{(S'v)(S'w)(S'x)} = S'(\fc a _{vwx}) \\
P(\fc l_v) = l _{Pv} = l _{S'v} = S'(\fc l_v) \quad \text{and} \quad  P(\fc r_v) = r _{Pv} = r _{S'v} = S'(\fc r_v).
\end{gather*}
Hence \( P=S' \) on the remaining \( F \)-basic 1- and 2-morphisms and therefore \( P=S'. \)
\end{proof}

The counit of the adjunction described in the following corollary is of great importance, when working with the coherence theorem for bicategories
(see propositions \ref{prop_coherence_bicats_free_version} and corollary \ref{cor_coherence_formally_defined_version}).

\begin{kor}
The functor \( U: \cat{Bicat_1^s} \to \cat{Cat \text{-} Gph} \) has a left adjoint \( F: \cat{Cat \text{-} Gph} \to \cat{Bicat_1^s}, \)
that sends a \( \cat{Cat} \)-graph \( G \) to the the free bicategory \( FG \) over that \( \cat{Cat} \)-graph.
The components of the unit of that adjunction are given by the inclusions \( \iota_G: G \to UFG \) from proposition \ref{prop_bicat_catgraph_adjunction}.
The components \( \epsilon _{\cat{B}}: FU \cat{B} \to \cat{B} \) of the counit \( \epsilon \) are given through evaluation.
\end{kor}

\begin{proof}
By proposition \ref{prop_characterization_adjunctions_in_cat} the collection of universal arrows in proposition \ref{prop_bicat_catgraph_adjunction} gives rise to an adjunction \( (F \dashv U, \eta, \epsilon). \)
The components of the counit \( \epsilon \) are obtained by - in the notation of the proof of proposition \ref{prop_characterization_adjunctions_in_cat} -
setting \( S = 1 _{U \cat{B}} \) and calculating \( S' = \epsilon _{\cat{B}}. \)
That calculation shows that \( \epsilon _{\cat{B}} \) is given through evaluation.
\end{proof}

Before we can state the coherence theorem for bicategories, we need to define the notion of a coherence morphism in a bicategory \( \cat{B}. \)

\begin{mydef}
A morphism in a bicategory \( \cat{B} \) is called a coherence (2-)morphism if it is a \( (*,\circ) \)-composite of units and constraint cells.
\end{mydef}

\begin{satz} \label{prop_coherence_bicats_free_version}
Let \( FG \) be the free tricategory over a \( \cat{Cat} \)-graph \( G. \)
Parallel coherence morphisms in \( FG \) are equal.
\end{satz}

\begin{proof}
This can be derived from theorem 2.13 of \cite{Gurski2013}, which is done in section 2.2.3 of \cite{Gurski2013}.
\end{proof}

It is important to note that in a non-free bicategory two parallel coherence 2-morphisms are not necessarily equal.
For example even if \( (f * g) * h = f * (g * h), \) the associator \( a_{fgh} \) might still be nontrivial.
Still the previous corollary can be slightly rephrased to refer to general bicategories:

\begin{terminology}
Let \( S: \cat{B} \to \cat{B'} \) be a weak functor and \( n \in \set{0,1,2}. \)
We say an \( n \)-morphism \( x' \in \cat{B'} \) lifts (along \( S \)) to an \( n \)-morphism \( x \in \cat{B} \) if \( S x = x'. \)
When it is clear from the context we might omit the functor \( S \).
(For example if we lift a morphism \( x \in \cat{B} \) to a morphism \( y \in FU \cat{B} \) we usually do not mention explicitly that the lift takes place along \( \epsilon. \))
\end{terminology}

\begin{kor} \label{cor_coherence_formally_defined_version}
Two coherence morphisms in a bicategory \( \cat{B} \) are equal, if they lift to parallel coherence morphisms in the free bicategory \( FU \cat{B}. \)
\end{kor}

\subsection{Strictification for the bicategory of formal composites} \label{sec_strictification_for_bicats}

In this section we will consider for an arbitrary bicategory \( \cat{B} \) the bicategory of "formal composites" \( \bhat, \)
that distinguishes (formal) composites of 1- morphisms, even if they evaluate to the same 1-morphism in \( \cat{B}. \)
The bicategory \( \bhat \) is -along a strict biequivalence \( \ev ,\) given by evaluation- biequivalent to \( \cat{B}, \) but has the advantage that parallel coherence morphisms in \( \bhat \) are equal.
This fact can be used to construct a 2-category \( \bst, \) whose 1-morphisms are the 1-morphisms of \( \bhat \) modulo coherence, and whose 2-morphisms are the 2-morphisms of \( \bhat \) up to conjugation with coherence morphisms.
Furthermore the process of "quotiening out coherence" is given by a strict biequivalence \( [-]: \bhat \to \bst. \)

So far we have strict biequivalences \( \ev: \bhat \to \cat{B} \) and \( [-]: \bhat \to \bst \) as in the diagram
\begin{align} \label{equ_bhat_equivalences}
\begin{tikzcd} [column sep=large, row sep=tiny, ampersand replacement=\&,]
 	\& \bhat
 	\ar[ld, "\ev" swap]
 	\ar[rd, "{[-]}"] \& \\
\cat{B}
		\& \& \bst.
\end{tikzcd}
\end{align}
Towards the end of the section we will explain how the situation in \eqref{equ_bhat_equivalences} allows us to replace equations between composite-2-morphisms in \( \cat{B} \) by equivalent equations of composite-2-morphisms in \( \bst. \)
Since \( \bst \) is a 2-category, the latter equations can be handled with string diagrams, which simplifies calculations tremendously as we demonstrate in example \ref{ex_adjunction_strictified}.

We start by defining the bicategory of formal composites \( \bhat. \)

\begin{satz} \label{prop_bhat_is_bicat}
Let \( \cat{B} \) be a bicategory.
Then the following defines a bicategory \( \bhat \) together with a strict biequivalence \( \ev: \bhat \to \cat{B}: \)

\begin{itemize}
\item
\( \ob{\bhat} = \ob{\cat{B}} \) and \( \ev \) acts on objects as an identity.

\item
The 1-morphisms of \( \bhat \) are defined inductively as follows:

Any 1-morphism in \( \cat{B} \) is a 1-morphism in \( \bhat \) as well (with the same types).
These 1-morphism of \( \bhat \) are called \( \bhat \)-basic 1-morphisms.

For any two 1-morphisms in \( \bhat \) (not necessarily \( \bhat \)-basic ones) \( \hat f: b \to c \) and \( \hat g: a \to b, \)
we require the formal composite \( \hat f \hs \hat g: a \to c \) to be a 1-morphism of \( \bhat. \)

Composition of 1-morphisms along objects is given by the \( \hs \)-operator.

\item
The action of \( \ev \) on 1-morphisms is defined inductively:

For \( \bhat \)-basic 1-morphisms \( \ev \) is the identity and \( \ev(\hat f \hs \hat g) := \ev(\hat f) * \ev(\hat g). \)

\item
The 2-morphisms of \( \bhat \) are triples \( (\alpha, \hat f, \hat g): \hat f \to \hat g \) where \( \alpha: \ev \hat f \to \ev \hat g \) is a 2-morphism in \( \cat{B}. \)

Composition of 2-morphisms in \( \tbar \) is inherited from composition of 2-morphisms in \( \cat{B}: \)
\begin{align*}
(\alpha, \hat f_1, \hat f_2) \hs (\beta, \hat g_1, \hat g_2) & := (\alpha * \beta, \hat f_1 \hs \hat g_1, \hat f_2 \hs \hat g_2) \\
(\alpha, \hat g_2, \hat g_3) \hc (\beta, \hat g_1, \hat g_2) & := (\alpha \circ \beta, \hat g_1, \hat g_3)
\end{align*}

\item
\( \ev \) acts on 2-morphisms via \( \ev(\alpha, \hat f, \hat g) = \alpha. \)

\item
The units in \( \bhat \) are given by
\[
\hat 1_a := 1_a \quad \text{and} \quad \hat 1 _{\hat f} := (1 _{\ev (\hat f)}, \hat f, \hat f).
\]
The constraint-cells of \( \bhat \) are given by

\begin{gather*}
\hat a _{\hat f \hat g \hat h} = \big(a _{\ev(\hat f) \ev(\hat g) \ev(\hat h)}, (\hat f \hs \hat g) \hs \hat h, \hat f \hs (\hat g \hs \hat h) \big) \\
\hat l _{\hat f} = (l _{\ev(\hat f)}, \hat 1 _{t \hat f} \hs \hat f, \hat f) \quad \text{and} \quad \hat r _{\hat f} = (r _{\ev(\hat f)}, \hat f \hs \hat 1 _{s \hat f}, \hat f)
\end{gather*}
\end{itemize}
\end{satz}

\begin{proof}
Showing that \( \bhat \) is a bicategory is based on two simple observations.
Firstly it follows directly from its definition, that \( \ev \) strictly preserves constraint cells and units as well as composition along 0- and 1-morphisms.
Secondly it follows from the definition of the 2-morphisms in \( \bhat, \)
that for parallel \( \hat \alpha, \hat \beta \in \mmor{ \bhat} \) we have \( \hat \alpha = \hat \beta \) if and only if \( \ev \hat \alpha = \ev \hat \beta. \)

Now the functoriality of \( \hs _{ab}: \bhat(b,c) \times \bhat(a,b) \to \bhat(a,c) \) follows from
\begin{align*}
\ev \big( (\hat \alpha_1 \hs \hat \beta_1) \hc (\hat \alpha_2 \hs \hat \beta_2) \big)
= \big( \ev (\hat \alpha_1) * \ev (\hat \beta_1) \big) \circ \big( \ev (\hat \alpha_2) * \ev (\hat \beta_2) \big) \\
= \big( \ev (\hat \alpha_1) \circ \ev (\hat \alpha_2) \big) * \big( \ev (\hat \beta_1) \circ \ev (\hat \beta_2) \big)
= \ev \big( (\hat \alpha_1 \hc \hat \alpha_2) \hs (\hat \beta_1 \hc \hat \beta_2) \big)
\end{align*}
and
\begin{align*}
\ev (\hat 1 _{\hat f} \hs \hat 1 _{\hat g})
= 1 _{\ev \hat f} \hs 1 _{ \ev \hat g}
= 1 _{ \ev (\hat f) * \ev (\hat g)}
= \ev (\hat 1 _{\hat f \hs \hat g}).
\end{align*}

Similarly one shows functorality of the unit-functors \( I_a, \) naturality of \( a,l \) and \( r \) and the bicategory axioms \ref{list_pentagon_axiom}, \ref{list_triangle_axiom} of definition \ref{def_bicat}.

This establishes that \( \ev \) is a strict functor, since \( \bhat \) is a bicategory and \( \ev \) strictly preserves composition and constraint cells.
Finally, according to remark \ref{rem_biequivalence}, \( \ev \) is also a biequivalence, since it is surjective, 1-locally surjective and 2-locally a bijection.
\end{proof}

\begin{lem} \label{lem_unique_coherence_in_bhat}
Parallel coherence 2-morphisms in \( \bhat \) are equal.
\end{lem}

\begin{proof}
For any \( \hat f \in \mor{\bhat}, \) we inductively define a 2-morphism \( \hat f ^{\oper{f}} \in \mor{F \bhat} \) as follows:
For a \( \bhat \)-basic 2-morphism \( f \) we define
\[
f ^{\oper{f}} :=
\begin{cases}
1_a ^{\sms{F}}\text{ (the unit at \( a \) in \( \oper{F} \bhat \)),}
	& \text{if } f = 1_a \\
f \text{ (as a morphism in \( \oper{F} \bhat \)),}
	& \text{if } f \text{ is not a unit}.
\end{cases}
\]
Then we recursively define \( (\hat f \hs \hat g) ^{\oper{f}} := \hat f ^{\oper{f}} \fstar \hat g ^{\oper{f}}. \)

Now in order to show that parallel coherence 2-morphisms in \( \bhat \) are equal, it suffices to show the following claim:
\emph{Any coherence 2-morphism \( \gamma: \hat f \to \hat g \) in \( \bhat \) lifts to the coherence 2-morphism \( \hat f ^{\oper{f}} \to \hat g ^{\oper{f}} \) in \( F \bhat. \)}

Since \( \gamma \) is a coherence 2-morphism, there exists a natural number \( N, \)
such that \( \gamma \) is a \( (*, \circ) \)-composite of \( N \) 2-morphisms, each of which being either a 2-unit or a constraint 2-cell.
We show the claim via induction over \( N. \)

If \( N = 1 \) \( \gamma \) is -for an appropriate choice of \( h, h_1, h_2, h_3 \)- one of the following 2-morphism or its inverse
\[
\hat 1 _{\hat h}, \quad \hat a _{\hat h_1 \hat h_2 \hat h_3}, \quad \hat l _{\hat h}, \quad \hat r _{\hat h}.
\]
It is easy to check that in these cases \( \gamma \) lifts to the coherence \(( s \gamma) ^{\oper{f}} \to (t \gamma) ^{\oper{f}}. \)

If \( N > 1 \) either \( \gamma = \gamma_1 \hs \gamma_2 \) or \( \gamma = \gamma_1 \hc \gamma_2, \)
where \( \gamma_1 \) and \( \gamma_2 \) can be written as composites of less than \( N \) constraint cells or units.
Therefore we can assume by induction that \( \gamma_1 \) resp. \( \gamma_2 \) lift to the coherence 2-morphisms
\( \mu_1: (s \gamma_1) ^{\oper{f}} \to (t \gamma_1) ^{\oper{f}} \) resp. \( \mu_2: (s \gamma_2) ^{\oper{f}} \to (t \gamma_2) ^{\oper{f}}. \)
Now let us assume \( \gamma = \gamma_1 \hs \gamma_2. \)
Then \( \mu_1 \fstar \mu_2 \) is the required lift of \( \gamma, \) as the following equations show:
\begin{gather*}
\epsilon(\mu_1 \fstar \mu_2) = \epsilon(\mu_1) \hs \epsilon(\mu_2) = \gamma_1 \hs \gamma_2 = \gamma \\
s (\mu_1 \fstar \mu_2) = s \mu_1 \fstar s \mu_2 = (s \gamma_1) ^{\oper{f}} \fstar (s \gamma_2) ^{\oper{f}} = (s \gamma_1 \hs s \gamma_2) ^{\oper{f}} = (s \gamma) ^{\oper{f}} \\
t (\mu_1 \fstar \mu_2) = t \mu_1 \fstar t \mu_2 = (t \gamma_1) ^{\oper{f}} \fstar (t \gamma_2) ^{\oper{f}} = (t \gamma_1 \hs t \gamma_2) ^{\oper{f}} = (t \gamma) ^{\oper{f}}.
\end{gather*}

Now let us assume \( \gamma = \gamma_1 \hc \gamma_2. \)
Then it is shown similar, that \( \mu_1 \fcirc \mu_2 \) is the required lift of \( \gamma. \)
\end{proof}

\begin{lem} \label{lem_coherence_congruence_relation}
Let \( \cat{B} \) be a bicategory, and \( \bhat \) be the associated bicategory described in proposition \ref{prop_bhat_is_bicat}.
Then
\[
\hat f \sim_1 \hat g
\iff
\text{a coherence 2-morphism } \hat \gamma: \hat f \to \hat g \text{ exists.}
\]
defines a \( \hs \)-congruence relation (see definition \ref{def_congruence_rel}) on \( \mor{\bhat} \) and
\[
\hat \alpha \sim_2 \hat \alpha'
\iff
\text{coherence 2-morphisms } \hat \gamma, \hat \gamma' \text{ exist, such that } \hat \gamma \circ \hat \alpha = \hat \alpha' \circ \hat \gamma'
\]
defines a \( (\hs,\hc) \)-congruence relation on \( \mmor{\bhat}. \)
\end{lem}

\begin{proof}
\( \sim_1 \) is reflexive, because identities are coherence morphisms;
\( \sim_1 \) is symmetric, because coherence morphisms are isomorphisms;
\( \sim_1 \) is transitive, because \( \hc \)-composites of coherence morphisms are by definition coherence morphisms;
\( \sim_1 \) is a congruence, because \( \hs \)-composites of coherence morphisms are by definition coherence morphisms.
Therefore we have shown \( \sim_1 \) is a congruence relation.

\( \sim_2 \) is a equivalence relation which is compatible with \( \hs \) for the same reasons as for \( \sim_1. \)
For showing compatibility with \( \hc, \) we consider 2-morphisms \( \hat \alpha_1 \sim_2 \hat \alpha_1' \) and \( \hat \alpha_2 \sim_2 \hat \alpha_2', \) where \( \hat \alpha_1 \) and \( \hat \alpha_2 \) as well as \( \hat \alpha_1' \) and \( \hat \alpha_2' \) are composable.
By definition of \( \sim_2 \) there exist coherence 2-morphisms \( \hat \gamma_1, \hat \gamma'_1, \hat \gamma_2, \hat \gamma'_2, \)
such that \( \hat \gamma_1 \circ \hat \alpha_1 = \hat \alpha_1' \circ \hat \gamma'_1 \) and \( \hat \gamma_2 \circ \hat \alpha_2 = \hat \alpha_2' \circ \hat \gamma'_2. \)
Since parallel coherence 2-morphisms in \( \bhat \) are equal by lemma \( \ref{lem_unique_coherence_in_bhat} \) and since \( s \hat \alpha_1 = t \hat \alpha_2 \) and \( s \hat \alpha_1' = t \hat \alpha_2', \)
the types of \( \hat \gamma'_1: s \hat \alpha_1 \to s \hat \alpha_1' \) and \( \hat \gamma_2: t \hat \alpha_2 \to t \hat \alpha_2' \) imply \( \hat \gamma'_1 = \hat \gamma_2. \)
Now
\[
\hat \gamma_1 \hc \hat \alpha_1 \hc \hat \alpha_2 =
\hat \alpha_1' \hc \hat \gamma'_1 \hc \hat \alpha_2 =
\hat \alpha_1' \hc \hat \gamma_2 \hc \hat \alpha_2 =
\hat \alpha_1' \hc \hat \alpha_2' \hc \hat \gamma'_2
\]
and hence \( \hat \alpha_1 \circ \hat \alpha_2 \sim_2 \hat \alpha_1' \circ \hat \alpha_2'. \)
\end{proof}

Now we can strictify \( \bhat, \) by quotiening out the congruence relations from lemma \ref{lem_unique_coherence_in_bhat}.

\begin{theorem} \label{prop_Bst_is_strict_bicat}~
\begin{enumerate}
\item
For any bicategory \( \cat{B} \) there is a 2-category \( \bst \) defined as follows:
\begin{itemize}
\item
The objects of \( \bst \) are the objects of \( \bhat. \)
\item
The 1-morphisms of \( \bst \) are the equivalence classes of \( \mor{\bhat} \) with respect to \( \sim_1. \)

The types of the 1-morphisms of \( \bst \) are defined by \( s[\hat f] = s \hat f \) and \( t[\hat f] = t \hat f. \)

\( * \)-composition in \( \mor{\bst} \) is given by \( [\hat f] * [\hat g] = [\hat f \hs \hat g]. \)

\item
The 2-morphisms of \( \bst \) are the equivalence classes of \( \mmor{\bhat} \) with respect to \( \sim_2. \)

The types of the 2-morphisms of \( \bst \) are defined by \( s[\hat \alpha] = [s \hat \alpha] \) and \( t[\hat \alpha] = [t \hat \alpha]. \)

\( * \)-composition in \( \mmor{\bst} \) is given by \( [\hat \alpha] * [\hat \beta] = [\hat \alpha \hs \hat \beta]. \)

If \( [\hat \alpha_1] \) and \( [\hat \alpha_2] \) are \( \circ \)-composable 2-morphisms in \( \bst, \)
there exist \( \hc \)-composable 2-morphisms \( \hat \alpha'_1 \) and \( \hat \alpha'_2 \) in \( \that, \)
such that \( \hat \alpha_1 \sim_2 \hat \alpha'_1 \) and \( \hat \alpha_2 \sim_2 \hat \alpha'_2. \)
Then we define \( [\hat \alpha_1] \circ [\hat \alpha_2] := [\hat \alpha'_1 \hc \hat \alpha'_2]. \)

\item
The units in \( \bst \) are given by \( 1_a = [1_a] \) and \( 1 _{[\hat f]} := [\hat 1 _{\hat f}]. \)
\end{itemize}

\item
There is a strict biequivalence \( \bhat \to \bst, \)
which is the identity on objects and sends each 1-morphism \( \hat f \) and 2-morphisms \( \hat \alpha \) to the equivalence classes \( [\hat f] \) and \( [\hat \alpha]. \)
\end{enumerate}
\end{theorem}

\begin{proof}
\begin{enumerate}
\item
The types of the 1-morphisms of \( \bhat \) are well defined, since \( \sim_1 \)-related 1-morphisms in \( \bhat \) are parallel.
The types of the 2-morphisms of \( \bhat \) are well defined, since for \( \sim_2 \)-related \( \alpha, \alpha' \in \mmor{\bhat} \)
\( s \alpha \sim_1 s \alpha' \) and \( t \alpha \sim_1 t \alpha' \) holds.
The compositions \( * \) and \( \circ \) are well defined, because \( \sim_1 \) and \( \sim_2 \) are congruence relations.

For showing functorality of \( *, \) let \( [\hat \alpha_1], [\hat \alpha_2], [\hat \beta_1], [\hat \beta_2] \in \mmor{\bst} \) such that \( ([\hat \alpha_1] \circ [\hat \alpha_2]) * ([\hat \beta_1] \circ [\hat \beta_2]) \) exists.
Without loss of generality we can assume that the representatives \( \hat \alpha_1 \) and \( \hat \alpha_2 \) resp. \( \hat \beta_1 \) and \( \hat \beta_2 \) are \( \circ \)-composable.
Then
\begin{align*}
([\hat \alpha_1] * [\hat \beta_1]) \circ ([\hat \alpha_2] * [\hat \beta_2])
= [(\hat \alpha_1 \hs \hat \beta_1) \hc (\hat \alpha_2 \hs \hat \beta_2)] \\
= [(\hat \alpha_1 \hc \hat \alpha_2) \hs (\hat \beta_1 \hc \hat \beta_2)]
=([\hat \alpha_1] \circ [\hat \alpha_2]) * ([\hat \beta_1] \circ [\hat \beta_2])
\end{align*}
\[
1 _{[\hat f]} * 1 _{[\hat g]} = [\hat 1 _{\hat f}] * [\hat 1 _{\hat g}] = [\hat 1 _{\hat f} \hs \hat 1 _{\hat g}] = [\hat 1 _{\hat f \hs \hat g}] = 1 _{[\hat f \hs \hat g]} = 1 _{[\hat f] * [\hat g]}
\]
which shows the functorality of \( *. \)

The associators \( a _{abcd} \) of \( \bst \) are identity natural transformations, since
\[
([\hat f] * [\hat g]) * [\hat h] = [(\hat f \hs \hat g) \hs \hat h] = [\hat f \hs (\hat g \hs \hat h)] = [\hat f] * ([\hat g] * [\hat h])
\]
\begin{gather*}
([\hat \alpha_1] * [\hat \alpha_2]) * [\hat \alpha_3] = [(\hat \alpha_1 \hs \hat \alpha_2) \hs \hat \alpha_3]
= [\hat a _{t \hat \alpha_1, t \hat \alpha_2, t \hat \alpha_3} \hc \big( (\hat \alpha_1 \hs \hat \alpha_2) \hs \hat \alpha_3 \big)] \\
= [\big( \hat \alpha_1 \hs (\hat \alpha_2 \hs \hat \alpha_3) \big) \hc \hat a _{s \hat \alpha_1, s \hat \alpha_2, s \hat \alpha_3}]
= [\hat \alpha_1 \hs (\hat \alpha_2 \hs \hat \alpha_3)] = [\hat \alpha_1] * ([\hat \alpha_2] * [\hat \alpha_3]).
\end{gather*}

Similarly one shows that left and right units \( l_a, r_a \) of \( \bst \) are identity natural transformations.

The bicategory axioms are trivially satisfied since all constraints of \( \bst \) are identities.

\item
By definition \( [-] \) strictly preserves composition and units.
That \( [\hat a _{\hat f \hat g \hat h}],[\hat l _{\hat f}],[\hat r _{\hat f}] \) are all identities, means that constraint cells are strictly preserved as well.
Furthermore \( [-] \) is surjective, locally surjective and 2-locally a bijection.
Therefore \( [-] \) is a strict biequivalence.
\end{enumerate}
\end{proof}

Since we have constructed strict biequivalences \( \ev: \bhat \to \cat{B} \) and \( [-]: \bhat \to \bst, \) we have in particular shown that \( \cat{B} \) is biequivalent to \( \bst. \)
(Note that this biequivalence is not given by a \emph{strict} functor, though).

\begin{kor} \label{cor_bicat_2_cat_equivalence}
Every bicategory is biequivalent to a 2-category.
\end{kor}

Now we will explain how the results of this section can be used to reduce calculations in a bicategory \( \cat{B} \) to calculations in the 2-category \( \bst. \)
More precisely we describe how a commutative diagram in a local category of \( \cat{B} \) can be replaced with an equivalent commutative diagram in a local category of \( \bst. \)

First we note that any composite 1-morphism in \( \cat{B} \) lifts along \( \ev \) to the 1-morphism in \( \bhat, \) that constitutes precisely that composite.
For example the composite \( (f_1 * f_2) * f_3 \) in \( \cat{B} \) lifts to the 1-morphism \( (f_1 \hs f_2) \hs f_3 \) in \( \bhat. \)
Similarly a diagram in a local category of \( \cat{B} \) lifts along \( \ev \) to a corresponding diagram in a local category of \( \bhat. \)
Since \( \ev \) is a strict biequivalence, the lifted diagram in \( \bhat \) commutes if and only if the original diagram in \( \cat{B} \) commutes.
Finally the diagram in \( \bhat \) gets mapped under the strict biequivalence \( [-]: \bhat \to \bst \) to a diagram in \( \bst, \)
which commutes if and only if the diagram in \( \bhat \) does so.
Therefore, in total we have achieved our goal of replacing a diagram in \( \cat{B} \) with an equivalent diagram in \( \bst. \)
(Here equivalent means that the diagram in \( \cat{B} \) commutes if and only if the diagram in \( \bst \) commutes.)

The following example will clarify the whole procedure and its benefits.

\begin{bei} \label{ex_adjunction_strictified}
Recall from definition \ref{def_equivalence_adjunction} that an adjunction in \( \cat{B} \) consists of 1-morphisms \( f: a \to b, \, g: b \to a \) and 2-morphisms \( \eta:1_a \to g*f, \, \epsilon: f*g \to 1_b, \)
satisfying the equations
\begin{align} \label{equ_snake_1}
f \xrightarrow{r ^{-1}}
f*1 \xrightarrow{1 * \eta}
f*(g*f) \xrightarrow{a ^{-1}}
(f*g)*f \xrightarrow{\epsilon * 1}
1 * f \xrightarrow{l}
f = f \xrightarrow{1} f
\end{align}
\begin{align} \label{equ_snake_2}
g \xrightarrow{l ^{-1}}
1 * g \xrightarrow{\eta * 1}
(g*f)*g \xrightarrow{a}
g*(f*g) \xrightarrow{1 * \epsilon}
g*1 \xrightarrow{r}
g = g \xrightarrow{1} g.
\end{align}

Now we define the 2-morphisms \( \hat \eta:=(\eta, 1_a, g \hs f) \) and  \( \hat \epsilon: (\epsilon: f \hs g, 1_b) \) in \( \bhat \) and observe that,
since \( ev:\bhat \to \cat{B} \) is a strict biequivalence, the equations
\begin{align} \label{equ_snake_11}
f \xrightarrow{\hat r ^{-1}}
f \hs 1 \xrightarrow{\hat 1 \hs \hat \eta}
f \hs (g \hs f) \xrightarrow{\hat a ^{-1}}
(f \hs g)\hs f \xrightarrow{\hat \epsilon \hs \hat 1}
1 \hs f \xrightarrow{\hat l}
f = f \xrightarrow{\hat 1} f
\end{align}
\begin{align} \label{equ_snake_22}
g \xrightarrow{\hat l ^{-1}}
1 \hs g \xrightarrow{\hat \eta \hs \hat 1}
(g \hs f)\hs g \xrightarrow{\hat a}
g \hs (f \hs g) \xrightarrow{\hat 1 \hs \hat \epsilon}
g \hs 1 \xrightarrow{\hat r}
g = g \xrightarrow{\hat 1} g.
\end{align}
in \( \bhat \) hold if and only if equations \ref{equ_snake_1} and \ref{equ_snake_2} hold.

Since \( [-]: \bhat \to \bst \) is a strict biequivalence, equations \ref{equ_snake_11} and \ref{equ_snake_22} hold, if and only if the following equations, which are written as string diagrams, hold in \( \bst: \)
\begin{align} \label{equ_snake_identities}
\begin{tikzpicture}[string={.7cm}{1.5cm}, rounded corners=7, baseline={([yshift=-.5ex]current bounding box.center)}, math mode, label distance=-1.5mm]
\node (s1) {}; \node (s2) [rr=of s1, label=90:{[\hat \eta]}] {};
\node (t1) [br=of s1, label=-90:{[\hat \epsilon]}] {}; \node (t2) [br=of s2] {};
\draw (s1.90) to [d r] node [inner sep=.5mm, left, pos=.4]  {\scriptstyle{[f]}} (t1.center)
	(t1.center) to [r r] node [inner sep=.5mm, right]  {\scriptstyle{[g]}} (s2.center)
	(s2.center) to [r d] node [inner sep=.5mm, right, pos=.6]  {\scriptstyle{[f]}} (t2.270);
\end{tikzpicture}
=
\begin{tikzpicture}[string={.7cm}{1.5cm}, rounded corners=7, baseline={([yshift=-.5ex]current bounding box.center)}, math mode]
\node (s1) {};
\node (t1) [b=of s1] {};
\draw (s1.90) -- node[inner sep=.5mm, right]  {\scriptstyle{[f]}} (t1.-90);
\end{tikzpicture}
\quad \text{and} \quad
\begin{tikzpicture}[string={.7cm}{1.5cm}, rounded corners=7, baseline={([yshift=-.5ex]current bounding box.center)}, math mode, label distance=-1.5mm]
\node (s1) [label=90:{[\hat \eta]}] {}; \node (s2) [rr=of s1] {};
\node (t1) [bl=of s1] {}; \node (t2) [bl=of s2, label=-90:{[\hat \epsilon]}] {};
\draw (t1.-90) to [d r] node [inner sep=.5mm, left, pos=.6]  {\scriptstyle{[g]}} (s1.center)
	(s1.center) to [r r] node [inner sep=.5mm, right]  {\scriptstyle{[f]}} (t2.center)
	(t2.center) to [r d] node [inner sep=.5mm, right, pos=.4]  {\scriptstyle{[g]}} (s2.90);
\end{tikzpicture}
=
\begin{tikzpicture}[string={.7cm}{1.5cm}, rounded corners=7, baseline={([yshift=-.5ex]current bounding box.center)}, math mode]
\node (s1) {};
\node (t1) [b=of s1] {};
\draw (s1.90) -- node[inner sep=.5mm, right]  {\scriptstyle{[g]}} (t1.-90);
\end{tikzpicture}
\end{align}

Therefore we have shown that we can replace the axioms \ref{equ_snake_1} and \ref{equ_snake_2} of an adjunction, by the slightly simpler axioms \ref{equ_snake_identities}, which are called snake identities.
So far we have only saved a few constraint cells, but this pays off significantly when it comes to proving statements about adjunctions.
To demonstrate this, we will prove the following basic lemma:

\begin{lem} [proposition 13 in \cite{Baez2004}]
Let \( (f,g,\eta,\epsilon) \) be an equivalence in a bicategory \( \cat{B}. \)
Then the equivalence \( (f,g,\eta,\epsilon) \) satisfies both equations \ref{equ_snake_1} and \ref{equ_snake_2} if it satisfies one of it.
\end{lem}

\begin{proof}
Let us assume equation \ref{equ_snake_1} holds.
According to the discussion previous to the lemma it suffices to show that the second equation in \ref{equ_snake_identities} holds:

\begin{align*}
\begin{tikzpicture}[string={.7cm}{.7cm}, rounded corners=7, baseline={([yshift=-.5ex]current bounding box.center)}, math mode, label distance=-1.5mm]
\node (s1) [label=90:\scriptstyle{[\hat \eta]}] {}; \node (s2) [rr=of s1] {};
\node (t1) [bbbl=of s1] {}; \node (t2) [bbbl=of s2, label=-90:\scriptstyle{[\hat \epsilon]}] {};
\draw (t1.-90) to [d r] (s1.center)
	(s1.center) to [r r] (t2.center)
	(t2.center) to [r d] (s2.90);
\end{tikzpicture}
\quad = \quad
\begin{tikzpicture}[string={.7cm}{.7cm}, rounded corners=7, baseline={([yshift=-.5ex]current bounding box.center)}, math mode, label distance=-1.5mm]
\node (s1) [label=90:\scriptstyle{[\hat \eta]}] {}; \node (s2) [rr=of s1] {};
\node (t1) [bbbl=of s1] {}; \node (t2) [bbbl=of s2, label=-90:\scriptstyle{[\hat \epsilon]}] {};
\node (s3) [br=of s2, label=90:\scriptstyle{[(\hat \epsilon) ^{-1}]}] {}; \node (s4) [b=of s3, label=-90:\scriptstyle{[\hat \epsilon]}] {};
\draw (s3.center) [r l] to (s4.center) (s4.center) [l r] to (s3.center);
\draw (t1.-90) to [d r] (s1.center)
	(s1.center) to [r r] (t2.center)
	(t2.center) to [r d] (s2.90);
\end{tikzpicture}
\quad = \quad
\begin{tikzpicture}[string={.7cm}{.7cm}, rounded corners=7, baseline={([yshift=-.5ex]current bounding box.center)}, math mode, label distance=-1.5mm]
\node (s1) [label=90:\scriptstyle{[\hat \eta]}] {}; \node (s2) [rr=of s1] {};
\node (t1) [bbbl=of s1] {}; \node (t2) [bbbl=of s2, label=-90:\scriptstyle{[\hat \epsilon]}] {};
\node (s3) [br=of s2, label=90:\scriptstyle{[(\hat \epsilon) ^{-1}]}] {}; \node (s4) [b=of s3, label=-90:\scriptstyle{[\hat \epsilon]}] {};
\node (c1) [hbhl=of s3] {};
\draw (s3.center) [r l] to (s4.center)
	(s3.center) [l l] to (c1.90)
	(s4.center) [l l] to (c1.-90);
\draw (t1.-90) to [d r] (s1.center)
	(s1.center) to [r r] (t2.center)
	(t2.center) to [r r] (c1.-90)
	(c1.90) to [l d] (s2.90);
\end{tikzpicture}
= \\
\begin{tikzpicture}[string={.7cm}{.7cm}, rounded corners=7, baseline={([yshift=-.5ex]current bounding box.center)}, math mode, label distance=-1.5mm]
\node (s1) {}; \node (s2) [rrr= of s1] {};
\node (m1) [br=of s1] {}; \node (m2) [b=of m1] {}; \node (m3) [rr=of m2] {};
\node (t1) [bl=of m2] {}; \node (t2) [rr=of t1] {}; \node (t3) [rr=of t2] {};
\draw (s1.90) [d r] to (m1.-90)
	(m1.-90) [r r] to (s2.center)
	(s2.center) [r l] to (t3.center)
	(t3.center) [l l] to (m3.center)
	(m3.center) [l l] to (t2.center)
	(t2.center) [l l] to (m2.center)
	(m2.center) [l d] to (t1.-90);
\end{tikzpicture}
\quad = \quad
\begin{tikzpicture}[string={.7cm}{.7cm}, rounded corners=7, baseline={([yshift=-.5ex]current bounding box.center)}, math mode, label distance=-1.5mm]
\node (s1) {}; \node (s2) [rrr= of s1] {};
\node (m1) [br=of s1] {}; \node (m2) [b=of m1] {}; \node (m3) [rr=of m2] {};
\node (t1) [bl=of m2] {}; \node (t2) [rr=of t1] {}; \node (t3) [rr=of t2] {};
\draw
	(t2.center) [l r] to (s2.center)
	(s2.center) [r l] to (t3.center)
	(t3.center) [l l] to (m3.center)
	(m3.center) [l l] to (t2.center);
\draw (s1.90) -- (t1.-90);
\end{tikzpicture}
\quad = \quad
\begin{tikzpicture}[string={.7cm}{.7cm}, rounded corners=7, baseline={([yshift=-.5ex]current bounding box.center)}, math mode, label distance=-1.5mm]
\node (s1) {};
\node (t1) [bbb=of s1] {};
\draw (s1.90) -- (t1.-90);
\draw (.8,-1) circle [radius=.5cm];
\end{tikzpicture}
\quad = \quad
\begin{tikzpicture}[string={.7cm}{.7cm}, rounded corners=7, baseline={([yshift=-.5ex]current bounding box.center)}, math mode, label distance=-1.5mm]
\node (s1) {};
\node (t1) [bbb=of s1] {};
\draw (s1.90) -- (t1.-90);
\end{tikzpicture}
\end{align*}.
\end{proof}

\end{bei}
 \clearpage{}

\clearpage{}

\section{Tricategories and their diagrams} \label{sec_tricats}

\subsection{The definition of a tricategory} \label{sec_def_tricat}

This chapter summarizes some standard definitions of three dimensional category theory.
We start by giving Nick Gurski's algebraic definition of a tricategory.
Then we introduce some versions of tricategories with varying degrees of strictness, namely 3-categories, cubical tricategories and Gray-categories.

In section \ref{sec_def_bicat} the definition of a bicategory was stated in terms of 1-categorical transfors (see terminology \ref{term_transfor}).
Similarly the definition of a tricategory is stated most concisely in terms of 2-categorical transfors.
The additional level of transfors (i.e. modifications), can be used to replace the equations of 1-categorical transformations (in the bicategory-definition) by modifications (in the tricategory definition).
One then can impose new appropriate axioms in terms of equations between these modifications.

\begin{mydef}[Definition 4.1. in \cite{Gurski2013}] \label{def_tricat}
A Tricategory \( (\cat{T}; \otimes,I; a,l,r; \pi,\mu,\lambda,\rho) \) consists of the following data:
\begin{itemize}
\item
A collection of objects \( \ob{\cat{T}}. \)

\item
For each ordered pair \( (a,b) \) of objects of \( \cat{T} \), a bicategory \( \cat{T}(a,b), \) called the local bicategory at \( (a,b). \)
The n-cells of a local bicategory are called n+1-cells of the tricategory \( T, \) and the collection of n-cells of \( \cat{T} \) is denoted \( \nmor{n}{\cat{T}} \).

\item
For any triple of objects \( (a,b,c) \) a weak functor \( \otimes: \cat{T}(b,c) \times \cat{T}(a,b) \to \cat{T}(a,c), \) called composition.

\item
For any object \( a \) of \( \cat{T} \) a weak functor \( I_a: \cat{1} \to \cat{T}(a,a), \) where \( \cat{1} \) denotes the bicategory with one object, one 1-morphism and one 2-morphism.

\item
For each 4-tuple \( (a,b,c,d) \) of objects of \( \cat{T} \) an adjoint equivalence \( (a \dashv a ^{\adsq}, \eta ^{a}, \epsilon ^{a}) \)
\[
\begin{tikzcd}
\cat{T}(c,d) \times \cat{T}(b,c) \times \cat{T}(a,b)
\ar[r, "\otimes \times 1" {name=s1}]
\ar[d, "1 \times \otimes"]
	& \cat{T}(b,d) \times \cat{T}(a,b)
	\ar[d, "\otimes"] \\
\cat{T}(c,d) \times \cat{T}(a,c)
\ar[r, "\otimes" {name=t1}]
	& \cat{T}(a,d) \\
\ar[from=s1, to=t1, "a", Rightarrow, short=7pt]
\end{tikzcd}
\]
in \( \cat{Bicat} \big( \cat{T}(c,d) \times \cat{T}(b,c) \times \cat{T}(a,b), \cat{T}(a,d) \big). \)

\newpage

\item
For each ordered pair \( (a,b) \) of objects of \( \cat{T} \) adjoint equivalences \( (l \dashv l ^{\adsq}, \eta ^{l}, \epsilon ^{l}) \) and \( (r \dashv r ^{\adsq}, \eta ^{r}, \epsilon ^{r}) \)
\[
\begin{tikzcd}[column sep=tiny]
	& \cat{T}(b,b) \times \cat{T}(a,b)
	\ar[rd,"\otimes"]
	\ar[ld, leftarrow, "I_b \times 1" swap]
	& \\
\cat{T}(a,b)
\ar[rr,"1" {name=t1, swap}]
		& & \cat{T}(a,b)
		\ar[from=ul, to=t1, "l", Rightarrow, short=5pt]
\end{tikzcd}
\quad
\begin{tikzcd}[column sep=tiny]
	& \cat{T}(a,b) \times \cat{T}(a,a)
	\ar[rd,"\otimes"]
	\ar[ld, leftarrow, "1 \times I_a" swap]
	& \\
\cat{T}(a,b)
\ar[rr,"1" {name=t1, swap}]
		& & \cat{T}(a,b)
		\ar[from=ul, to=t1, "r", Rightarrow, short=5pt]
\end{tikzcd}
\]
in \( \cat{Bicat} \big( \cat{T}(a,b), \cat{T}(a,b) \big). \)

\item
For each 5 tuple \( (a,b,c,d,e) \) of objects of \( \cat{T} \) an invertible modification \( \pi \)

\[
\begin{tikzcd}[column sep=small]
	& \cat{\cat{T}^4}
	\ar[rr,"\otimes \times 1 \times 1"]
	\ar[dr, "1 \times \otimes \times 1" description]
	\ar[dl, "1 \times 1 \times \otimes" swap]
			& & \cat{\cat{T}^3}
			\ar[dr,"\otimes \times 1"]
			\ar[dl, Rightarrow, short=8pt, "a \times 1" swap] & \\
\cat{\cat{T}^3}
\ar[dr,"1 \times \otimes" swap]
		& & \cat{\cat{T}^3}
		\ar[rr,"\otimes \times 1"]
		\ar[dl,"1 \times \otimes" description]
		\ar[ll, Rightarrow, short=20pt, "1 \times a" swap]
				& & \cat{\cat{T}^2}
				\ar[dl,"\otimes"]
				\ar[dlll, Rightarrow, short=40pt, pos=.45, "a" swap] \\
	& \cat{\cat{T}^2}
	\ar[rr, "\otimes" swap]
			& & \cat{T}
\end{tikzcd}
\xRrightarrow{\phantom{x} \pi \phantom{x}}
\begin{tikzcd}[column sep=small]
	& \cat{\cat{T}^4}
	\ar[rr,"\otimes \times 1 \times 1"]
	\ar[dl, "1 \times 1 \times \otimes" swap]
			& & \cat{\cat{T}^3}
			\ar[dr,"\otimes \times 1"]
			\ar[dl, "1 \times \otimes" description]
			\ar[dlll, phantom, "="] & \\
\cat{\cat{T}^3}
\ar[dr,"1 \times \otimes" swap]
\ar[rr,"\otimes \times 1"]
		& & \cat{\cat{T}^2}
		\ar[dr,"\otimes" description]
		\ar[ld, Rightarrow, short=8pt, "a" swap]
				& & \cat{\cat{T}^2}
				\ar[dl,"\otimes"]
				\ar[ll, Rightarrow, short=20pt, "a" swap] \\
	& \cat{\cat{T}^2}
	\ar[rr, "\otimes" swap]
			& & \cat{T}
\end{tikzcd}
\]
in the bicategory \( \cat{Bicat}\big(\cat{T^4}(a,b,c,d,e),\cat{T}(a,e)\big), \)
where \( \cat{T^4} = \cat{T^4}(a,b,c,d,e) \) is for example an abbreviation for \( \cat{T}(d,e) \times \cat{T}(c,d) \times \cat{T}(b,c) \times \cat{T}(a,b). \)

\item
For each triple of objects of \( \cat{T} \) invertible modifications

\[
\begin{tikzcd}[column sep=huge]
\cat{\cat{T}}^2
\ar[rrd, bend left=30, "1" {name=s1}]
\ar[rd,"1 \times I \times 1" description]
\ar[rdd, bend right=30, "1" {swap, name=t2}] &[-10pt] & \\[-10pt]
	& \cat{\cat{T}}^3
	\ar[r,"\otimes \times 1"]
	\ar[d,"1 \times \otimes"]
	\ar[from=s1, Rightarrow, short=10pt, "r ^{\adsq} \times 1"]
	\ar[to=t2, Rightarrow, short=10pt, "1 \times l"]
		& \cat{\cat{T}}^2
		\ar[d,"\otimes"]
		\ar[dl, Rightarrow, short=25pt, "a"] \\
	& \cat{\cat{T}}^2
	\ar[r,"\otimes" swap]
		& \cat{T}
\end{tikzcd}
\xRrightarrow{\phantom{x} \mu \phantom{x}}
\begin{tikzcd}[column sep=huge]
\cat{\cat{T}}^2
\ar[rrd, bend left=30, "1" , "" {name=s1, pos=.7}]
\ar[rdd, bend right=30, "1" swap, "" {name=t1, pos=.85}] &[-10pt] & \\[-10pt]
		& & \cat{\cat{T}}^2
		\ar[d,"\otimes"] \\
	& \cat{\cat{T}}^2
	\ar[r, "\otimes" swap]
		& \cat{T}
\ar[from=s1, to=t1, Rightarrow, short=30pt, "1"]
\end{tikzcd}
\]

\[
\begin{tikzcd} [column sep=huge]
	& \cat{\cat{T}}^3
	\ar[rd, "\otimes \times 1"] & \\[-.9em]
\cat{\cat{T}}^2
\ar[ur, "I \times 1 \times 1"]
\ar[rr, bend right=12, "1" {name=t1, swap}]
\ar[from=ur, to=t1, Rightarrow, short=5pt, "l \times 1"]
\ar[d, "\otimes" swap]
\ar[drr, phantom, "=", yshift=-.8em]
		& & \cat{\cat{T}}^2
		\ar[d, "\otimes"] \\[+.9em]
\cat{T}
\ar[rr, "1" swap]
		& & \cat{T}
\end{tikzcd}
\xRrightarrow{\phantom{x} \lambda \phantom{x}}
\begin{tikzcd} [column sep=huge, row sep=.7em]
	& \cat{\cat{T}}^3
	\ar[dd, "1 \times \otimes"]
	\ar[rd, "\otimes \times 1"] & \\
\cat{\cat{T}}^2
\ar[dr, phantom, "="]
\ar[ur, "I \times 1 \times 1"]
\ar[dd, "\otimes" swap]
		& & \cat{\cat{T}}^2
		\ar[dl, Rightarrow, short=2em, "a"]
		\ar[dd, "\otimes"] \\[-.5em]
	& \cat{\cat{T}}^2
	\ar[dr, "\otimes"] & \\[+.5em]
\cat{T}
\ar[ur, "I \times 1"]
\ar[rr, "1" {swap, name=t1}]
\ar[from=ur, to=t1, Rightarrow, short=8pt, "l" {xshift=.2em}]
		& & \cat{T}
\end{tikzcd}
\]

\[
\begin{tikzcd} [column sep=huge]
	& \cat{\cat{T}}^3
	\ar[rd, "1 \times \otimes"] & \\[-.9em]
\cat{\cat{T}}^2
\ar[ur, "1 \times 1 \times I"]
\ar[rr, bend right=12, "1" {name=t1, swap}]
\ar[from=t1, to=ur, Rightarrow, short=5pt, "1 \times r ^{\adsq}"]
\ar[d, "\otimes" swap]
\ar[drr, phantom, "=", yshift=-.8em]
		& & \cat{\cat{T}}^2
		\ar[d, "\otimes"] \\[+.9em]
\cat{T}
\ar[rr, "1" swap]
		& & \cat{T}
\end{tikzcd}
\xRrightarrow{\phantom{x} \rho \phantom{x}}
\begin{tikzcd} [column sep=huge, row sep=.7em]
	& \cat{\cat{T}}^3
	\ar[dd, "\otimes \times 1"]
	\ar[rd, "1 \times \otimes"] & \\
\cat{\cat{T}}^2
\ar[dr, phantom, "="]
\ar[ur, "1 \times 1 \times I"]
\ar[dd, "\otimes" swap]
		& & \cat{\cat{T}}^2
		\ar[dl, Leftarrow, short=2em, "a"]
		\ar[dd, "\otimes"] \\[-.5em]
	& \cat{\cat{T}}^2
	\ar[dr, "\otimes"] & \\[+.5em]
\cat{T}
\ar[ur, "1 \times I"]
\ar[rr, "1" {swap, name=t1}]
\ar[from=ur, to=t1, Leftarrow, short=8pt, "r ^{\adsq}" {xshift=.2em}]
		& & \cat{T}
\end{tikzcd}
\]
\end{itemize}

The data of \( \cat{T} \) are subject to the following axioms:

\begin{itemize}
\item
For each 6-tuple \( (a,b,c,d,e,f) \) of objects of \( \cat{T} \) the following equation of 2-cells in the local bicategory \( T(a,f) \) holds.
Here for compactness we replace \( \otimes \) by juxtaposition and drop indices.

\[
\begin{tikzcd}[cells={font=\scriptsize}, sep=tiny]
	& & \node (1-3) [yshift=-.5cm] {(k(j(hg)))f}; \\
	\node (2-1) {(((k(jh))g)f}; & \node (2-2) [yshift=1cm] {((k((jh)g))f}; & \node (2-3) {k(((jh)g)f)}; & \node (2-4) [yshift=1cm] {k((j(hg))f)}; & \node (2-5) {k(j((hg))f))}; \\
	\node (3-1) {(((kj)h)g)f}; & \node (3-2) [yshift=.7cm] {(k(jh))(gf)}; & & \node (3-4) [yshift=.7cm] {k((jh)(gf))}; & \node (3-5) {k(j(h(gf)))}; \\
	& \node (4-2) {((kj)h)(gf)}; & & \node (4-4) {(kj)(h(gf))}; \\
\\ \\
	& & \node (1-3b) [yshift=-.5cm] {(k(j(hg)))f}; \\
	\node (2-1b) {((k(jh))g)f}; & \node (2-2b) [yshift=1cm] {(k((jh)g))f}; & & \node (2-4b) [yshift=1cm] {k((j(hg))f)}; & \node (2-5b) {k(j((hg))f))}; \\
	\node (3-1b) {(((kj)h)g)f}; & \node (3-2b) [xshift=.8cm, yshift=1cm] {((kj)(hg))f}; & & \node (3-4b) [xshift=-1cm, yshift=.7cm] {(kj)((hg)f)}; & \node (3-5b) {k(j(h(gf)))}; \\
	& \node (4-2b) {((kj)h)(gf)}; & & \node (4-4b) {(kj)(h(gf))}; \\
\ar[from=3-1, to=2-1, "(a1)1"]
\ar[from=2-1, to=2-2, "a1"]
\ar[from=2-2, to=1-3, "(1a)1"]
\ar[from=1-3, to=2-4, "a"]
\ar[from=2-4, to=2-5, "1a" {name=s3}]
\ar[from=2-5, to=3-5, "1(1a)"]
	\ar[from=3-1, to=4-2, "a" {name=tcong, swap}]
	\ar[from=4-2, to=4-4, "a" {name=t2, swap}]
	\ar[from=4-4, to=3-5, "a" swap]
		\ar[from=2-1, to=3-2, "a" {name=scong}]
		\ar[from=4-2, to=3-2, "a(11)" swap]
		\ar[from=3-2, to=3-4, "a" {name=s2}]
		\ar[from=3-4, to=3-5, "1a"]
		\ar[from=2-2, to=2-3, "a" {name=s1}]
		\ar[from=2-3, to=2-4, "1(a1)"]
		\ar[from=2-3, to=3-4, "1a"]
			\ar[from=1-3, to=2-3, "a", Rightarrow, short=3.5mm]
			\ar[from=scong, to=tcong, "a" swap, Rightarrow, yshift=-2mm, xshift=-1mm, short=3.5mm]
			\ar[from=s1, to=3-2, "\pi", Rightarrow, short=10pt]
			\ar[from=s2, to=t2, "\pi", Rightarrow, short=10pt]
			\ar[from=s3, to=3-4, "1 \pi", Rightarrow, short=10pt]
\ar[from=3-1b, to=2-1b, "(a1)1"]
\ar[from=2-1b, to=2-2b, "a1"]
\ar[from=2-2b, to=1-3b, "(1a)1"]
\ar[from=1-3b, to=2-4b, "a" {name=s1b}]
\ar[from=2-4b, to=2-5b, "1a" {name=s3}]
\ar[from=2-5b, to=3-5b, "1(1a)"]
	\ar[from=3-1b, to=4-2b, "a" {name=tcong, swap}]
	\ar[from=4-2b, to=4-4b, "a" {name=t2b, swap}]
	\ar[from=4-4b, to=3-5b, "a" {name=tcong, swap}]
		\ar[from=3-1b, to=3-2b, "a1"]
		\ar[from=3-2b, to=3-4b, "a" {name=t1b}]
		\ar[from=3-4b, to=4-4b, "(11)a" {swap, pos=.35}]
		\ar[from=3-2b, to=1-3b, "a1"]
		\ar[from=3-4b, to=2-5b, "a" {name=scong}]
			\ar[from=2-2b, to=3-2b, Rightarrow, short=10pt, "\pi 1" {swap, pos=.4}]
			\ar[from=s1b, to=t1b, Rightarrow, short=18pt, swap, "\pi"]
			\ar[from=t1b, to=t2b, Rightarrow, short=10pt, swap, "\pi"]
			\ar[from=scong, to=tcong, Rightarrow, "a ^{-1}", short=5mm]
				\ar[from=t2, to=1-3b, phantom, "\scalebox{2}{=}"]
\end{tikzcd}
\]

\item
For each 4-tuple \( (a,b,c,d) \) of objects of \( \cat{T} \) the following equations of 2 cells in \( T(a,d) \) hold.
The cells labeled with \( \cong \) are given by unique coherence isomorphisms in the local bicategory \( T(a,d). \)

\[
\begin{tikzcd}[row sep=small, cells={font=\scriptsize}]
		& & (h(Ig)f
		\ar[dd, "a"]
		\ar[dr, "(1l)1"] &
\\
	& ((hI)g)f
	\ar[dddd, Rightarrow, xshift=-1cm, short=10mm, "a" swap]
	\ar[ur, "a1"]
	\ar[dd, "a"]
			& & (hg)f
			\ar[dddd, "a"]
			\ar[dl, Rightarrow, "a", short=2mm, yshift=-3mm]
\\[-0.8em]
		& & h((Ig)f)
		\ar[dd, "1a" swap]
		\ar[dddr, bend left=10, "1(l1)" {description, xshift=3mm, pos=.3}, "" {name=s1}]
		\ar[dl, Rightarrow, short=7pt, yshift=4mm, "\pi"] &
\\[-.3em]
(hg)f
\ar[ddr, "a" swap]
\ar[uur, "(r ^{\adsq}1)1"]
	& (hI)(fg)
	\ar[dr, "a"]
	 \\[-.5em]
		& & h(I(gf))
		\ar[from=s1, Rightarrow, short=.2em, shift={(-3pt,-3pt)}, "1 \lambda"]
		\ar[dr,"1l" swap]
\\[-.5em]
	& h(gf)
	\ar[uu, "r ^{\adsq} (11)" swap]
	\ar[rr, bend right=20, "1" {name=t2, swap}]
	\ar[from=ur, to=t2, Rightarrow, short=.5em, "\mu"]
			& & h(gf)
\\[+3em]
		& & (h*Ig))f
		\ar[from=t2, phantom, "\scalebox{2}{=}" description]
		\ar[dr, "(1l)1"]
\\
	&((hI)g)f
	\ar[ur, "a1"]
		& & (hg)f
		\ar[dd, "a"]
\\
(hg)g
\ar[ur, "(r ^{\adsq}1)1"]
\ar[rrru, "11" swap, "" {pos=.75, name=st}]
\ar[rd, "a" swap]
\\[-.8em]
	& h(gf)
	\ar[rr, "1" {name=t, swap}]
			& & h(gf)
			\ar[from=uuul, to=st, Rightarrow, short=1mm, "\mu 1" {swap, yshift=2mm}]
			\ar[from=st, to=t, Rightarrow, phantom, "\cong" ]
\end{tikzcd}
\]

and

\[
\begin{tikzcd}[row sep=small, cells={font=\scriptsize}]
		& & h((gI)f)
		\ar[dd, leftarrow, "a"]
		\ar[dr, leftarrow, "1 (r ^{\adsq} 1)"] &
\\
	& h(g(If))
	\ar[dddd, Leftarrow, xshift=-1cm, short=10mm, "a" swap]
	\ar[ur, leftarrow, "1a"]
	\ar[dd, leftarrow, "a"]
			& & h(gf)
			\ar[dddd, leftarrow, "a"]
			\ar[dl, Leftarrow, "a", short=2mm, yshift=-3mm]
\\[-0.8em]
		& & (h(gI)f
		\ar[dd, leftarrow, "a1" swap]
		\ar[dddr, leftarrow, bend left=10, "(1 r ^{\adsq}) 1" {description, xshift=3mm, pos=.3}, "" {name=s1}]
		\ar[dl, Leftarrow, short=7pt, yshift=4mm, "\pi"] &
\\[-.3em]
h(gf)
\ar[ddr, leftarrow, "a" swap]
\ar[uur, leftarrow, "1(1l)"]
	& (hg)(Ig)
	\ar[dr, leftarrow, "a"]
\\[-.5em]
		& & ((hg)I)f
		\ar[from=s1, Leftarrow, short=.2em, shift={(-3pt,-3pt)}, "\rho 1"]
		\ar[dr, leftarrow,"r ^{\adsq}1" swap]
\\[-.5em]
	& h(gf)
	\ar[uu, leftarrow, "(11)l" swap]
	\ar[rr, leftarrow, bend right=20, "1" {name=t2, swap}]
	\ar[from=ur, to=t2, Leftarrow, short=.5em, "\mu"]
			& & (hg)f
\\[+3em]
		& & h((gI)f)
		\ar[from=t2, leftarrow, phantom, "\scalebox{2}{=}" description]
		\ar[dr, leftarrow, "1 (r ^{\adsq}1)"]
\\
	& h(g(If)
	\ar[ur, leftarrow, "1a"]
		& & h(gf)
		\ar[dd, leftarrow, "a"]
\\
h(gf)
\ar[ur, leftarrow, "1(1l)"]
\ar[rrru, leftarrow, "11" swap, "" {pos=.75, name=st}]
\ar[rd, leftarrow, "a" swap]
\\[-.8em]
	& (hg)f
	\ar[rr, leftarrow, "1" {name=t, swap}]
			& & (hg)f
			\ar[from=uuul, to=st, Leftarrow, short=1mm, "1 \mu" {swap, yshift=2mm}]
			\ar[from=st, to=t, Leftarrow, phantom, "\cong" ]
\end{tikzcd}
\]

\end{itemize}

\end{mydef}

\begin{teano}~
\begin{enumerate}[(i)]
\item
The categories \( \cat{T}(a,b)(f,g), \) i.e. the local categories of the local 2-categories of \( \cat{T} \), are called \emph{2-local categories} of \( \cat{T}. \)
Similarly, the sets \( \cat{T}(a,b)(f,g)(\alpha,\beta) \) are called \emph{3-local sets}.
\item
If \( P \) is a property of a bicategory, a tricategory \( \cat{T} \) is called locally \( P \) if all local bicategories of \( \cat{T} \) have the property \( P. \)
\item
We denote the single object in the source bicategory of \( I_a \) with \( \bullet, \) and call
\[
1_a := I_a(\bullet): a \to a
\]
the unit of \( a. \) Furthermore we define
\[
i_a := I_a(1_\bullet): 1_a \to 1_a
\]
and abbreviate the single component of the unit constraint of \( I_a \) as
\[
\phi_a := \phi ^{I_a} _{\bullet}: 1 _{(1_a)} \to i
\]
\item
Any cell of \( \cat{T} \) which appears as a component of any 1- or 2- transfor which is involved in the definition of a tricategory is called a constraint cell.
These are
\begin{itemize}[-]
\item
components of the constraints of the local bicategories, which are called \emph{local constraints} and denoted \( a ^{loc}, l ^{loc}, r ^{loc}. \)
\item
components of the constraints of the weak functors \( \otimes, I. \)
\item
components of the pseudonatural transformations \( a, a ^{\adsq}, l, l ^{\adsq},r, r ^{\adsq}. \)
\item
components of the modifications \( \epsilon ^{a}, \eta ^{a},\epsilon ^{l}, \eta ^{l},\epsilon ^{r}, \eta ^{r}, \pi, \mu, \lambda, \rho. \)
\end{itemize}
\item
The constraint cells \( \phi ^{\otimes} _{(\alpha,\beta)(\alpha',\beta')} \) are called interchangers of the tricategory \( \cat{T}. \)
\item
Any 2- resp. 3- morphism in \( \cat{T} \) which is a composite of constraint-cells and units is called a \emph{coherence} 2- resp. 3-morphism.
\end{enumerate}
\end{teano}

A 2-category (see definition \ref{def_2_category}) is equivalently defined as a \( \cat{Cat_1} \)-enriched category, where \( (\cat{Cat_1}, \times) \) is the monoidal 1-category of categories and functors.
Similarly a 3-category is a \( (\cat{2\text{-}Cat}_1) \)-enriched category, where \( (\cat{2\text{-}Cat}_1, \times) \) is the monoidal 1-category of 2-categories and strict functors.
Equivalently one can define a 3-category as a tricategory, which is "as strict as possible".

\begin{mydef} \label{def_strict_tricat}
A 3-category is a tricategory \( \cat{T} \) for which
\begin{itemize}
\item
the local bicategories \( \cat{T}(a,b) \) are strict,
\item
\( \otimes \) and \( I \) are strict functors,
\item
\( a,l,r \) are identity adjoint equivalences,
\item
the modifications \( \pi,\mu,\lambda,\rho \) are identities.
\end{itemize}
\end{mydef}

Every bicategory is biequivalent to a 2-category (see corollary \ref{cor_bicat_2_cat_equivalence}).
However not every tricategory is triequivalent to a 3-category, but only to a so called Gray-category, that is a tricategory in which everything is strict except the interchangers.
The definition of a Gray-category relies on the concept of a cubical functor.

\begin{mydef} [Definition 4.1. in \cite{Gordon1995}]
Let \( A_i \) and \( B \) be 2-categories.
A weak functor \( F: A_1 \times A_2 \times \dots \times A_n \to B \) that strictly preserves 1-units is called \emph{cubical} if the following condition holds:

If \( (f_1,f_2,\dots,f_n) \) and \( (g_1,g_2,\dots,g_n) \) are composable morphisms in the 2-category \( A_1 \times A_2 \times \dots \times A_n \) such that for all \( i>j, \) either \( g_i \) or \( f_j \) is a 1-unit, the constraint cell
\[
F(f_1,f_2,\dots,f_n) * F(g_1,g_2,\dots,g_n)
\xrightarrow{\phi}
F \big( (f_1,f_2,\dots,f_n) * (g_1,g_2,\dots,g_n) \big)
\]
is an identity.
\end{mydef}

\begin{bem}~
\begin{enumerate}[(i)]
\item
For \( n=1 \) a cubical functor is just a strict functor between 2-categories.
\item
In \cite{Gordon1995} the requirement that \( F \) strictly preserves 1-units is omitted, since it is erroneously assumed that this is already implied by the remaining conditions.
See \cite{MO_question} for a counterexample.
\end{enumerate}
\end{bem}

\begin{mydef}[Definition 8.1 and Theorem 8.12 in \cite{Gurski2013}]~ \label{def_cubical_tricat_Gray_cat}
\begin{enumerate}
\item
A tricategory is called cubical, if it is locally strict, and the functors \( (\otimes, I) \) are cubical functors.
\item
A cubical tricategory is called a Gray-category, if \( a,l,r \) are identity adjoint equivalences and \(\pi,\mu,\lambda,\rho \) are identity modifications.
\end{enumerate}
\end{mydef}

\subsection{Rewriting diagrams for Gray categories} \label{sec_rewriting_diagrams}

Similarly to the 2-dimensional string diagrams for bicategories, there exist 3-dimensional diagrams for Gray-categories,
where objects are represented by regions, 1-morphisms by planes between such regions, 2-morphisms by lines between such planes and 3- morphisms by points on lines.
These diagrams are rigorously defined in \cite{Barrett2012}.

Another approach to carry out calculations in Gray-categories graphically are the rewriting diagrams used for example by the online proof assistant \emph{globular} \cite{Bar2016}.
A rewriting diagram for a Gray category \( \cat{T} \) is just a diagram in a 2-local category of \( \cat{T}, \)
in which objects -which are 2-morphisms in \( \cat{T} \)- are represented via string diagrams.
These string diagrams for Gray categories work almost the same way, as string diagrams for 2-categories do.
The only difference is that, since the interchangers in \( \cat{T} \) are in general non-trivial, the height at which the disc of a 2-morphism is drawn matters.
For example the diagrams
\[
\begin{tikzpicture}[string={.7cm}{.7cm}, rounded corners=5, baseline={([yshift=-.5ex]current bounding box.center)}, math mode, label distance=-1.5mm]
\node (s1) {}; \node (s2) [r=of s1] {};
\node (mo1) [morphism, b=of s1] {\scriptstyle{\alpha}}; \node (mo2) [morphism, bb=of s2] {\scriptstyle{\beta}};
\node (t1) [bb=of mo1] {}; \node (t2) [b=of mo2] {};
\draw (s1) -- (mo1) -- (t1);
\draw (s2) -- (mo2) -- (t2);
\end{tikzpicture}
\quad \text{and} \quad
\begin{tikzpicture}[string={.7cm}{.7cm}, rounded corners=5, baseline={([yshift=-.5ex]current bounding box.center)}, math mode, label distance=-1.5mm]
\node (s1) {}; \node (s2) [r=of s1] {};
\node (mo1) [morphism, bb=of s1] {\scriptstyle{\alpha}}; \node (mo2) [morphism, b=of s2] {\scriptstyle{\beta}};
\node (t1) [b=of mo1] {}; \node (t2) [bb=of mo2] {};
\draw (s1) -- (mo1) -- (t1);
\draw (s2) -- (mo2) -- (t2);
\end{tikzpicture}
\]
do not represent the same 2-morphism in \( \cat{T}, \) since \( (1 \otimes \beta) \circ (\alpha \otimes 1) \) and \( (\alpha \otimes 1) \circ (1 \otimes \beta) \) are in general not equal.

Globular incorporates the non-strictness of the interchangers, by allowing only one disc per height and evaluating diagrams always from bottom to top.
We will adopt this convention, only that we reverse the vertical order and evaluate string diagrams from top to bottom in order to stay in line with section \ref{sec_string_diagrams_bicats}.
For example, the string-diagram
\[
\begin{tikzpicture}[string={.7cm}{.7cm}, rounded corners=5, baseline={([yshift=-.5ex]current bounding box.center)}, math mode]
\node (s1) {}; \node (s2) [r=of s1] {}; \node (s3) [r=of s2] {};
\node (mo1) [morphism, b=of s3] {\scriptstyle{\alpha}}; \node (mo2) [morphism, bb=of s1] {\scriptstyle{\beta}};
\node (mo3) [morphism, bbm={mo1}{s2}] {\scriptstyle{\gamma}};
\node (t1) [bb=of mo2] {}; \node (t2) [b=of mo3] {};
\draw (s1) -- (mo2) -- (t1);
\draw (s2) |- (mo3) -- (t2);
\draw (s3) |- (mo1) |- (mo3);
\end{tikzpicture}
\]
represents the morphism \( (1 \otimes \gamma) * (\beta \otimes 1 \otimes 1) * \ (1 \otimes 1 \otimes \alpha). \)

One might add that it is not a strict necessity to disallow more than 1-disc on the same height (in \cite{Bartlett2014} this is not required),
but this convention makes life slightly easier and still allows one to represent all composite 2-morphisms of \( \cat{T} \) via string diagrams.
For example the composite \( (\alpha * \beta) \otimes \gamma \) is - since \( \otimes \) is cubical -
equal to \( (1 \otimes \gamma) * (\alpha \otimes 1) * (\beta \otimes 1) \) and is therefore represented by the string diagram
\[
\begin{tikzpicture}[string={.7cm}{.7cm}, rounded corners=5, baseline={([yshift=-.5ex]current bounding box.center)}, math mode]
\node (s1) {}; \node (s2) [r=of s1] {};
\node (mo1) [morphism, b=of s1] {\scriptstyle{\alpha}};
\node (mo2) [morphism, bb=of s1] {\scriptstyle{\beta}};
\node (mo3) [morphism, bbb=of s2] {\scriptstyle{\gamma}};
\node (t1) [bb=of mo2] {}; \node (t2) [b=of mo3] {};
\draw (s1) -- (mo1) -- (mo2) -- (t1);
\draw (s2) -- (mo3) -- (t2);
\end{tikzpicture}.
\]

Note that composites of 2-morphisms, which make use of the strict equality of certain 1-morphisms in \( \cat{T}, \) can also be represented by string diagrams.
For example the composite \( f \xrightarrow{\alpha} g_1 \otimes g_2 = h_1 \otimes h_2 \xrightarrow{\beta} k \) is represented by
\[
\begin{tikzpicture}[string={.7cm}{.7cm}, rounded corners=5, baseline={([yshift=-.5ex]current bounding box.center)}, math mode]
\node (s1) {};
\node (mo1) [morphism, b=of s1] {\scriptstyle{\alpha}};
\node (mo2) [rec, b=of mo1] {\scriptstyle{=}};
\node (mo3) [morphism, b=of mo2] {\scriptstyle{\beta}};
\node (t1) [b=of mo3] {};
\draw (s1) -- (mo1) to [r d] node [right, pos=.7, inner sep=.4mm] {\scriptstyle{g2}} (mo2.27) (mo2.-27) to [d l] node [right, pos=.3, inner sep=.4mm] {\scriptstyle{h2}} (mo3) (mo3) -- (t1);
\draw (mo1) to [l d] node [left, pos=.7, inner sep=.4mm] {\scriptstyle{g1}} (mo2.153) (mo2.-153) to [d r] node [left, pos=.3, inner sep=.4mm] {\scriptstyle{h1}} (mo3);
\end{tikzpicture}.
\]

Now a 3-morphism in \( \cat{T} \) can be drawn as a morphism in a 2-local category of \( \cat{T} \), whose types are represented by string diagrams.
For example the rewriting diagram
\[
\begin{tikzpicture}[string={.7cm}{.7cm}, rounded corners=5, baseline={([yshift=-.5ex]current bounding box.center)}, math mode]
\node (s1) {}; \node (s2) [r=of s1] {};
\node (mo1) [morphism, b=of s1] {\scriptstyle{\alpha}};
\node (mo2) [morphism, bm={mo1}{s2}] {\scriptstyle{\beta}};
\node (t1) [b=of mo2] {};
\draw (s1) -- (mo1) to [d r] (mo2) (mo2) -- (t1);
\draw (s2) |- (mo2) (mo2) -- (t1);
\end{tikzpicture}
\xrightarrow{\phantom{xxx} \Pi \phantom{xxx}}
\begin{tikzpicture}[string={.7cm}{.7cm}, rounded corners=5, baseline={([yshift=-.5ex]current bounding box.center)}, math mode]
\node (s1) {}; \node (s2) [r=of s1] {};
\node (mo1) [morphism, b=of s2] {\scriptstyle{\alpha'}};
\node (mo2) [morphism, bm={mo1}{s1}] {\scriptstyle{\beta'}};
\node (t1) [b=of mo2] {};
\draw (s2) -- (mo1) to [d l] (mo2) (mo2) -- (t1);
\draw (s1) |- (mo2) (mo2) -- (t1);
\end{tikzpicture}
\]
represents the 3-morphism \( \Pi: \beta * (\alpha \otimes 1) \to (1 \otimes \alpha') * \beta'. \)

If a 3-morphism "acts only on a part of a diagram", one usually highlights the relevant subdiagrams and labels the 3-morphism as if it acted on the source subdiagram.
For example on writes
\[
\begin{tikzpicture}[string={.7cm}{.7cm}, rounded corners=5, baseline={([yshift=-.5ex]current bounding box.center)}, math mode]
\node (s1) {}; \node (s2) [r=of s1] {}; \node (s3) [r=of s2] {};
\node (mo1) [morphism, b=of s1] {\scriptstyle{\alpha}};
\node (mo2) [morphism, bm={mo1}{s2}] {\scriptstyle{\beta}};
\node (mo3) [morphism, bbb=of s3] {\scriptstyle{\gamma}};
\node (t1) [bb=of mo2] {}; \node (t2) [b=of mo3] {};
\draw (s1) -- (mo1) to [d r] (mo2) (mo2) -- (t1);
\draw (s2) |- (mo2) (mo2) -- (t1);
\draw (s3) -- (mo3) -- (t2);
\node [draw=blue, fit=(mo1) (mo2), rounded corners=0, inner sep=3mm] {};
\end{tikzpicture}
\xrightarrow{\phantom{xxx} \Pi \phantom{xxx}}
\begin{tikzpicture}[string={.7cm}{.7cm}, rounded corners=5, baseline={([yshift=-.5ex]current bounding box.center)}, math mode]
\node (s1) {}; \node (s2) [r=of s1] {}; \node (s3) [r=of s2] {};
\node (mo1) [morphism, b=of s2] {\scriptstyle{\alpha'}};
\node (mo2) [morphism, bm={mo1}{s1}] {\scriptstyle{\beta'}};
\node (mo3) [morphism, bbb=of s3] {\scriptstyle{\gamma}};
\node (t1) [bb=of mo2] {}; \node (t2) [b=of mo3] {};
\draw (s2) -- (mo1) to [d l] (mo2) (mo2) -- (t1);
\draw (s1) |- (mo2) (mo2) -- (t1);
\draw (s3) -- (mo3) -- (t2);
\node [draw=blue, fit=(mo1) (mo2), rounded corners=0, inner sep=3mm] {};
\end{tikzpicture}
\quad
\text{instead of}
\quad
\begin{tikzpicture}[string={.7cm}{.7cm}, rounded corners=5, baseline={([yshift=-.5ex]current bounding box.center)}, math mode]
\node (s1) {}; \node (s2) [r=of s1] {}; \node (s3) [r=of s2] {};
\node (mo1) [morphism, b=of s1] {\scriptstyle{\alpha}};
\node (mo2) [morphism, bm={mo1}{s2}] {\scriptstyle{\beta}};
\node (mo3) [morphism, bbb=of s3] {\scriptstyle{\gamma}};
\node (t1) [bb=of mo2] {}; \node (t2) [b=of mo3] {};
\draw (s1) -- (mo1) to [d r] (mo2) (mo2) -- (t1);
\draw (s2) |- (mo2) (mo2) -- (t1);
\draw (s3) -- (mo3) -- (t2);
\end{tikzpicture}
\xrightarrow{(1 \otimes 1_{\gamma}) * (\Pi \otimes 1)}
\begin{tikzpicture}[string={.7cm}{.7cm}, rounded corners=5, baseline={([yshift=-.5ex]current bounding box.center)}, math mode]
\node (s1) {}; \node (s2) [r=of s1] {}; \node (s3) [r=of s2] {};
\node (mo1) [morphism, b=of s2] {\scriptstyle{\alpha'}};
\node (mo2) [morphism, bm={mo1}{s1}] {\scriptstyle{\beta'}};
\node (mo3) [morphism, bbb=of s3] {\scriptstyle{\gamma}};
\node (t1) [bb=of mo2] {}; \node (t2) [b=of mo3] {};
\draw (s2) -- (mo1) to [d l] (mo2) (mo2) -- (t1);
\draw (s1) |- (mo2) (mo2) -- (t1);
\draw (s3) -- (mo3) -- (t2);
\end{tikzpicture}
\]

The discs in a string diagram have a fixed order from top to bottom.
Boxes need to be drawn in such a way that they contain only consecutive discs, since otherwise their content cannot be interpreted as a sub-composite of the composite represented by the string diagram.
For instance the box in the string diagram on the left is not admissible.
Instead one has to insert appropriate interchangers that rearrange the order, as shown on the right side.
\[
\begin{tikzpicture}[string={.7cm}{.7cm}, rounded corners=5, baseline={([yshift=-.5ex]current bounding box.center)}, math mode]
\node (s1) {}; \node (s2) [r=of s1] {}; \node (s3) [r=of s2] {};
\node (mo1) [morphism, bb=of s1] {\scriptstyle{\alpha}};
\node (mo2) [morphism, b=of s2] {\scriptstyle{\beta}};
\node (mo3) [morphism, bbb=of s3] {\scriptstyle{\gamma}};
\node (t1) [bb=of mo1] {}; \node (t2) [r=of t1] {}; \node (t3) [r=of t2] {};
\draw (s1) -- (mo1) -- (t1);
\draw (s2) -- (mo2) -- (t2);
\draw (s3) -- (mo3) -- (t3);
\node [draw=blue, fit=(mo2) (mo3), rounded corners=0, inner sep=2mm] {};
\end{tikzpicture}
\rightarrow \dots
\qquad \qquad \qquad
\begin{tikzpicture}[string={.7cm}{.7cm}, rounded corners=5, baseline={([yshift=-.5ex]current bounding box.center)}, math mode]
\node (s1) {}; \node (s2) [r=of s1] {}; \node (s3) [r=of s2] {};
\node (mo1) [morphism, bb=of s1] {\scriptstyle{\alpha}};
\node (mo2) [morphism, b=of s2] {\scriptstyle{\beta}};
\node (mo3) [morphism, bbb=of s3] {\scriptstyle{\gamma}};
\node (t1) [bb=of mo1] {}; \node (t2) [r=of t1] {}; \node (t3) [r=of t2] {};
\draw (s1) -- (mo1) -- (t1);
\draw (s2) -- (mo2) -- (t2);
\draw (s3) -- (mo3) -- (t3);
\node [draw=green, fit=(mo1) (mo2), rounded corners=0, inner sep=2mm] {};
\end{tikzpicture}
\rightarrow
\begin{tikzpicture}[string={.7cm}{.7cm}, rounded corners=5, baseline={([yshift=-.5ex]current bounding box.center)}, math mode]
\node (s1) {}; \node (s2) [r=of s1] {}; \node (s3) [r=of s2] {};
\node (mo1) [morphism, b=of s1] {\scriptstyle{\alpha}};
\node (mo2) [morphism, bb=of s2] {\scriptstyle{\beta}};
\node (mo3) [morphism, bbb=of s3] {\scriptstyle{\gamma}};
\node (t1) [bbb=of mo1] {}; \node (t2) [r=of t1] {}; \node (t3) [r=of t2] {};
\draw (s1) -- (mo1) -- (t1);
\draw (s2) -- (mo2) -- (t2);
\draw (s3) -- (mo3) -- (t3);
\node [draw=green, fit=(mo1) (mo2), rounded corners=0, inner sep=2mm] {};
\node [draw=blue, fit=(mo2) (mo3), rounded corners=0, inner sep=2mm] {};
\end{tikzpicture}
\rightarrow \dots
\]

Rewriting diagrams make no rigorous sense in general tricategories.
But at the end of this thesis we will explain how the results from section \ref{sec_strictification_for_tricats} can be used to replace equations in tricategories by equivalent equations in Gray-categories.
This will make rewriting diagrams employable also in the case of general tricategories.
\subsection{n-globular sets and n-magmoids} \label{sec_globular_sets_and_magmoids}

The language of transfors is the most concise way to give the definition of a tricategory.
But sometimes it is convenient to describe tricategories in more elementary terms and this section introduces the algebraic structures which are necessary for doing so.

First we formalize what is obtained when one forgets everything about composition and constraint-cells in a tricategory and keeps only the information of the types of the morphisms.

\begin{mydef}[Definition 1.4.5 in \cite{Leinster2004}]~
\begin{enumerate}
\item
A \emph{n-globular set} \( X \) consists of sets
\[
\nmor{0}{X}, \nmor{1}{X}, \dots, \nmor{n}{X}
\]
and maps
\[
s _{k}: \nmor{(k+1)}{X} \to \nmor{k}{X}
\quad \text{and} \quad
t _{k}: \nmor{(k+1)}{X} \to \nmor{k}{X}
\]
for all \( k \in \set{0, \dots, n-1}, \) subject to the \emph{globular identities} (here and in the following we drop the indices)
\[
ss = st
\quad \text{and} \quad
ts = tt.
\]
\item
A morphism \( f:X \to Y \) of n-globular sets \( X \) and \( Y \) consists of maps
\[
f_k: \nmor{k}{X} \to \nmor{k}{Y}
\quad \text{for all } k \in \{0 \dots n\}
\]
subject to
\[
f s = s f \quad \text{and} \quad f t = t f.
\]
\item
We write \( \gset{n} \) for the category of \( n \)-globular sets.
\end{enumerate}
\end{mydef}

\begin{terminology}
We call the elements of \( \nmor{k}{X} \) the k-morphisms of the n-globular set \( X. \)
We call two k-morphisms \( x,y \) parallel if \( sx=sy \) and \( tx=ty \) and by convention arbitrary 0-morphisms are called parallel as well.
\end{terminology}

\begin{bei} \label{ex_underlying_globe}
Every tricategory \( \cat{T} \) has an underlying 3-globular set, which we denote either \( \globe{\cat{T}} \) or, if no confusion is possible, simply \( \cat{T}. \)
\end{bei}

Similar to a tricategory having 1-local bicategories, 2-local categories or 3-local sets, an n-globular set has m-local (n-m)-globular-sets.

\begin{mydef}[m-local globular sets]~
Let \( X \) be a n-globular set and \( x,y \in \nmor{(m-1)}{X} \) parallel.

We define a (n-m)-globular set \( X(x,y) \) whose p-morphisms are given by those \( z \in \nmor{(p+m)}{X} \) which satisfy
\[
s ^{p+1} z = x \quad \text{and} \quad t ^{p+1} z = y,
\]
and call it the \emph{m-local} globular set of \( X \) at \( (x,y). \)

Furthermore if \( F:X \to Y \) is a morphism of n-globular sets, we define the \emph{m-local} morphism of globular sets
\[
F _{x,y} := F \restriction _{X(x,y)}: X(x,y) \to Y(Fx,Fy).
\]
\end{mydef}

\begin{bem}
Note that
\[
X(x,y) = X(sx,tx)(x,y) = X(sy,ty)(x,y).
\]
\end{bem}

Products in the category of n-globular sets exist and are given as follows:

\begin{mydef}[product of globular sets]
If \( X \) and \( Y \) are n-globular sets the \emph{product globular set} \( X \times Y \) is the globular set with
\[
\nmor{k}{X \times Y} = \nmor{k}{X} \times \nmor{k}{Y}
\]
and
\[
s ^{(X \times Y)} = s ^X \times s^Y.
\]
\end{mydef}

The concept of a 3-globular set is sufficient to describe the cells of a tricategory and their types, but it lacks compositions operations in order to state the axioms of a tricategory.
Therefore we introduce n-magmoids, which are roughly speaking globular sets together with composition operations.

\begin{mydef}[category of n-magmoids]~
\begin{enumerate}
\item
An n-magmoid \( M \) is an n-globular set \( M \) together with morphisms of (n-k-1)-globular sets
\[
\otimes ^{k} _{x,y,z}: M(y,z) \times M(x,y) \to M(x,z)
\]
for all \( k<n \) and every triple of parallel k-morphisms \( (x,y,z). \)

\item
A morphism of n-magmoids \( F:M \to N \) is a morphism of globular sets, such that
\[
F(v \otimes ^{k} _{x,y,z} w) = F(v) \otimes ^{k} _{Fx,Fy,Fz} F(w)
\]
for all \( k < n, \) every triple of parallel k-morphisms \( (x,y,z) \) and all \( (v,w) \) contained in \( M(y,z) \times M(x,y). \)

\item
We write \( \magm{n} \) for the category of \( n \)-magmoids.
\end{enumerate}
\end{mydef}

\begin{bei}
\begin{enumerate}
\item
Apart from the absence of units, a 1-magmoid consist of the same data as a category, but does not impose any axioms such as associativity of composition.
Therefore any 1-category has an underlying 1-magmoid.
\item
Similarly any tricategory has an underlying 3-magmoid.
\end{enumerate}
\end{bei}

The local structure of the underlying globular set of a magmoid is compatible with the magmoid structure, which gives rise to local magmoids and local morphisms of magmoids.

\begin{mydef}
Let \( F:M \to N \) be a morphism of n-magmoids and \( x,y \) parallel (m-1)-morphisms in \( M. \) Then there is an
\begin{enumerate}
\item
m-local (n-m)-magmoid \( M(x,y) \) whose underlying globular set is just the m-local (n-m)-globular set \( M(x,y). \)
\item
m-local morphism of (n-m)-magmoids
\[
F _{x,y}: M(x,y) \to N(Fx,Fy),
\]
whose underlying globular map is just the m-local (n-m)-globular map
\(
F _{x,y}.
\)
\end{enumerate}
\end{mydef}

\begin{terminology} [local properties] \label{term_local_properties}
Let P be a property of an (n-m)-magmoid.
We say an n-magmoid is m-locally P, if its m-local (n-m)-magmoids all have the property P.
Similarly if P is a property of a map of (n-m)-magmoids,
we say a map of n-magmoids is m-locally P, if its m-local (n-m)-magmoid-morphisms all have this property P.
\end{terminology}

\subsection{Tricategories as 3-magmoids with additional structure} \label{sec_tricats_as_magmoids}

The tricategories appearing in section \ref{sec_strictification_for_tricats} are most conveniently defined by constructing first their underlying 3-magmoids and then specifying their units and constraint cells.
Therefore in this section we give the "unpacked" definition of a tricategory as a 3-magmoid together with additional data and axioms.

In the following definition \( a \) is an object, \( f,g,h,k \) are 1-morphisms, \( \alpha, \beta, \gamma \) are 2-morphisms and \( \Lambda,\Pi \) are 3-morphisms in the 3-magmoid \( \cat{T}. \)
Each constraint cell is defined for any choice of indices which makes the types of the constraint cell well defined.
Each axiom holds for any choice of indices, for which the diagram or the equation that expresses the axiom is well defined.

\begin{mydef}[magmoidal definition of a tricategory] \label{def_tricat_as_magmoid}
A tricategory is a 3-magmoid \( \cat{(T, \otimes, *, \circ)} \) together with the following data:

\begin{itemize}
\item
The 2-local categories give rise to
\begin{itemize}[-]
\item
(Unit) 3-cells \( 1_\alpha: \alpha \to \alpha. \)
\end{itemize}

\item
The local bicategories give rise to
\begin{itemize}[-]
\item
(Unit) 2-cells \( 1_f: f \to f, \)
\item
\( a ^{loc} _{\alpha \beta \gamma}: (\alpha * \beta) * \gamma \to \alpha * (\beta * \gamma), \)
\item
\( l ^{loc} _{\alpha}: 1 _{t \alpha} * \alpha \to \alpha, \)
\item
\( r ^{loc} _{\alpha}: \alpha * 1 _{s \alpha} \to \alpha. \)
\end{itemize}

\item
The weak functors \( I_a \) from definition \ref{def_tricat} give rise to
\begin{itemize}[-]
\item
(Unit) 1-cells \( 1_a: a \to a, \)
\item
2-cells \( i_a: 1_a \to 1_a, \)
\item
3-cells \( \phi_a: 1 _{1_a} \to i_a. \)
\end{itemize}

\item
The weak functors \( \otimes \) from definition \ref{def_tricat} give rise to
\begin{itemize}[-]
\item
\( \phi ^{\otimes} _{(\alpha,\beta)(\alpha',\beta')}: (\alpha \otimes \beta) * (\alpha' \otimes \beta') \to (\alpha * \alpha') \otimes (\beta * \beta'). \)
\item
\( \phi ^{\otimes} _{(f,g)}: 1 _{f \otimes g} \to 1_f \otimes 1_g. \)
\end{itemize}

\item
The adjoint equivalences \( (a \dashv a ^{\adsq}, \eta ^{a}, \epsilon ^{a}) \) give rise to
\begin{itemize}[-]
\item
components of the 2-categorical transformations \( a \) and \( a ^{\adsq} \)
\[
a _{f,g,h}: (f \otimes g) \otimes h \to f \otimes (g \otimes h)
\]
\[
a _{\alpha \beta \gamma}: a _{t \alpha, t \beta, t \gamma} * (\alpha \otimes \beta) \otimes \gamma \to \alpha \otimes (\beta \otimes \gamma) * a _{s \alpha, s \beta, s \gamma}.
\]
\[
a ^{\adsq} _{f,g,h}: f \otimes (g \otimes h) \to (f \otimes g) \otimes h
\]
\[
a ^{\adsq} _{\alpha \beta \gamma}: a ^{\adsq} _{t \alpha, t \beta, t \gamma} * \alpha \otimes (\beta \otimes \gamma)
\to (\alpha \otimes \beta) \otimes \gamma * a ^{\adsq} _{s \alpha, s \beta, s \gamma}.
\]
\item
components of the modifications \( \eta ^{a} \) resp. \( \epsilon ^{a}: \)
\[
\eta ^{a} _{fgh}: 1 _{(f \otimes g) \otimes h} \xrightarrow{\sim} a ^{\adsq} _{fgh} * a _{fgh}
\]
\[
\epsilon ^{a} _{fgh}: a _{fgh} * a ^{\adsq} _{fgh} \xrightarrow{\sim} 1 _{f \otimes (g \otimes h)}.
\]
\end{itemize}

\item
The adjoint equivalences \( (l \dashv l ^{\adsq}, \eta ^{l}, \epsilon ^{l}) \) give rise to
\begin{itemize}[-]
\item
components of the 2-categorical transformations \( l \) and \( l ^{\adsq}: \)
\[
l _{f}: 1 _{tf} \otimes f \to f
\]
\[
l _{\alpha}: l _{t \alpha} * (i _{tt \alpha} \otimes \alpha) \to \alpha * l _{s \alpha}
\]
\[
l ^{\adsq} _{f}: f \to 1 _{tf} \otimes f
\]
\[
l ^{\adsq} _{\alpha}: l ^{\adsq} _{t \alpha} * \alpha \to (i _{tt \alpha} \otimes \alpha) * l ^{\adsq} _{s \alpha}
\]

\item
components of the modifications \( \eta ^{l} \) and \( \epsilon ^{l}: \)
\[
\eta ^{l}_f: 1 _{f} \to l ^{\adsq} _{f} * l _{f}
\]
\[
\epsilon ^{l}_f: l _{f} * l ^{\adsq} _{f} \to 1 _{tf \otimes f}
\]
\end{itemize}

\item
The adjoint equivalences \( (r \dashv r ^{\adsq}, \eta ^{r}, \epsilon ^{r}) \) give rise to
\begin{itemize}[-]
\item
components of the 2-categorical transformations \( r \) and \( r ^{\adsq} \):
\[
r _{f}: f \otimes 1 _{sf} \to f
\]
\[
r _{\alpha}: r _{t \alpha} * (\alpha \otimes i _{ss \alpha}) \to \alpha * r _{s \alpha}
\]
\[
r ^{\adsq} _{f}: f \to f \otimes 1 _{sf}
\]
\[
r ^{\adsq} _{\alpha}: r ^{\adsq} _{t \alpha} * \alpha \to (\alpha \otimes i _{ss \alpha}) * r ^{\adsq} _{s \alpha}
\]

\item
invertible components of the modifications \( \epsilon ^{r} \) and \( \eta ^{r}: \)
\[
\eta ^{r}_f: 1 _{f} \to r ^{\adsq} _{f} * r _{f}
\]
\[
\epsilon ^{r}_f: r _{f} * r ^{\adsq} _{f} \to 1 _{tf \otimes f}
\]
\end{itemize}

\item
The modifications \( \pi, \mu, \lambda, \rho, \) give rise to 3-cells
\[
\pi _{fghk}: 1_f * a _{ghk} * a _{f(g \otimes h)k} * (a _{fgh} \otimes 1_k)
\]
\[
\mu _{fg}: (1 \otimes l _{g}) * a _{f(1 _{sf}) g} * (r ^{\adsq} _{f} \otimes 1)
\]
\[
\lambda _{fg}: l _{f} \otimes 1_g \to  l _{(f \otimes g)} * a _{(1 _{tf})fg}
\]
\[
\rho _{fg}: 1_f \otimes r ^{\adsq} _{g} \to a _{fg(1 _{sg})} * r ^{\adsq} _{f \otimes g}.
\]
\end{itemize}

\vspace{1.5cm}

This data has to satisfy the following axioms:

\begin{itemize}
\item
Axioms expressing that \( T \) is 2-locally a category:
\[
(\Lambda_1 \circ \Lambda_2) \circ \Lambda_3 = \Lambda_1 \circ (\Lambda_2 \circ \Lambda_3),
\]
\[
1 _{t \Lambda} \circ \Lambda = \Lambda \text{ and } \Lambda \circ 1 _{s \Lambda} = \Lambda.
\]

\item
Any constraint 3-cell \( X \) is required to be invertible, i.e. we require the existence of a 3-cell \( X ^{-1} \) with \( X \circ X ^{-1} = X ^{-1} \circ X. \)

\item
Axioms expressing that \( T \) is locally a 2-category:
\begin{itemize}[-]
\item
Functoriality of \( * \) is equivalent to the equations
\[
(\Pi_1 \circ \Pi_2) * (\Lambda_1 \circ \Lambda_2) = (\Pi_1 * \Lambda_1) \circ (\Pi_2 * \Lambda_2)
\]
\[
1 _{\alpha} * 1_\beta = 1 _{\alpha*\beta.}
\]

\item
Naturality of \( a ^{loc} \) is equivalent to the commutativity of:
\[
\begin{tikzcd}[column sep = large]
(\alpha_1 * \alpha_2) * \alpha_3
\ar[r, "a ^{loc} _{\alpha_1 \alpha_2 \alpha_3}"]
\ar[d, "(\Pi_1 * \Pi_2) * \Pi_3" swap]
	& \alpha_1 * (\alpha_2 * \alpha_3)
	\ar[d, "\Pi_1 * (\Pi_2 * \Pi_3)"] \\
(\alpha'_1 * \alpha'_2) * \alpha'_3
\ar[r, "a ^{loc} _{\alpha'_1 \alpha'_2 \alpha'_3}"]
	& \alpha'_1 * (\alpha'_2 * \alpha'_3) \\
\end{tikzcd}
\]

\item
Naturality of \( l ^{loc} \) is equivalent to the commutativity of.
\[
\begin{tikzcd}
1 _{t \alpha} * \alpha
\ar[r, "l ^{loc} _{\alpha}"]
\ar[d, "1 _{(1 _{t \alpha})} * \Pi" swap]
	& \alpha
	\ar[d, "\Pi"] \\
1 _{t \alpha} * \alpha'
\ar[r, "l ^{loc} _{\alpha'}"]
	& \alpha'
\end{tikzcd}
\]

\item
Naturality of \( r ^{loc} \) is equivalent to the commutativity of.
\[
\begin{tikzcd}
\alpha * 1 _{s \alpha}
\ar[r, "r ^{loc} _{\alpha}"]
\ar[d, "\Pi * 1 _{(1 _{s \alpha})}" swap]
	& \alpha
	\ar[d, "\Pi"] \\
\alpha' * 1 _{s \alpha}
\ar[r, "r ^{loc} _{\alpha'}"]
	& \alpha'
\end{tikzcd}
\]

\item
The bicategory axioms are given by the commutative diagrams in remark \ref{rem_bicat_axioms_component_wise}.
\end{itemize}

\item
Axioms expressing that \( \otimes \) is a weak functor:
\begin{itemize}[-]
\item
That \( \otimes \) is locally given by 1-categorical functors corresponds to the equations
\begin{gather*}
(\Pi_1 \circ \Pi_2) \otimes (\Lambda_1 \circ \Lambda_2) = (\Pi_1 \otimes \Lambda_1) \circ (\Pi_2 \otimes \Lambda_2) \\
1 _{\alpha} \otimes 1_\beta = 1 _{\alpha \otimes \beta.}
\end{gather*}

\item
Naturality of the constraints \( \phi ^{\otimes} \) corresponds to the commutativity of the following diagram:
\[
\begin{tikzcd}
(\alpha_1 \otimes \beta_1) * (\alpha_2 \otimes \beta_2)
\ar[r, "\phi ^{\otimes}"]
\ar[d, "{(\Pi_1 \otimes \Lambda_1) * (\Pi_2 \otimes \Lambda_2)}" swap]
	& (\alpha_1 * \alpha_2) \otimes (\beta_1 * \beta_2)
	\ar[d, "{(\Pi_1 * \Pi_2) \otimes (\Lambda_1 * \Lambda_2)}"] \\
 (\alpha'_1 \otimes \beta'_1) * (\alpha'_2 \otimes \beta'_2)
\ar[r, "\phi ^{\otimes}"]
	& (\alpha'_1 * \alpha'_2) \otimes (\beta'_1 * \beta'_2) \\
\end{tikzcd}
\]

\item
The weak functor axioms \ref{diag_functor_axiom_asso} and \ref{diag_functor_axiom_unit} yield the commutativity of the following diagrams:

\begin{gather*}
\begin{tikzcd}[ampersand replacement = \&]
\big( (\alpha_1 \otimes \beta_1) * (\alpha_2 \otimes \beta_2) \big) * (\alpha_3 \otimes \beta_3)
\ar[r, "a ^{loc}"]
\ar[d, "\phi ^{\otimes} *1" swap]
	\& (\alpha_1 \otimes \beta_1) * \big( (\alpha_2 \otimes \beta_2) * (\alpha_3 \otimes \beta_3) \big)
	\ar[d, "1 * \phi ^{\otimes}"] \\
\big( (\alpha_1 * \alpha_2) \otimes (\beta_1 * \beta_2) \big) * (\alpha_3 \otimes \beta_3)
\ar[d, "\phi ^{\otimes}" swap]
	\& (\alpha_1 \otimes \beta_1) * \big( (\alpha_2 * \alpha_3) \otimes (\beta_2 * \beta_3) \big)
	\ar[d, "\phi ^{\otimes}"] \\
\big( (\alpha_1 * \alpha_2) * \alpha_3 \big) \otimes \big( (\beta_1 * \beta_2) * \beta_3 \big)
\ar[r, "a ^{loc}" swap]
	\& \big( \alpha_1 * (\alpha_2 * \alpha_3) \big) \otimes \big( \beta_1 * (\beta_2 * \beta_3) \big) \\
\end{tikzcd}
\\
\begin{tikzcd}[ampersand replacement = \&]
1 _{t \alpha \otimes t \beta} * (\alpha \otimes \beta)
\ar[r, "l ^{loc}"]
\ar[d, "\phi ^{\otimes} * 1" swap]
	\& \alpha \otimes \beta \\
(1 _{t \alpha} \otimes 1 _{t \beta}) * (\alpha \otimes \beta)
\ar[r, "\phi ^{\otimes}" swap]
	\& (1 _{t \alpha} * \alpha) \otimes (1 _{t \beta} * \beta)
	\ar[u, "l ^{loc} \otimes l ^{loc}" swap]
\end{tikzcd}
\\
\begin{tikzcd}[ampersand replacement = \&]
(\alpha \otimes \beta) * 1 _{s \alpha \otimes s \beta}
\ar[r, "r ^{loc}"]
\ar[d, "1 * \phi ^{\otimes}" swap]
	\& \alpha \otimes \beta \\
(\alpha \otimes \beta) * (1 _{s \alpha} \otimes 1 _{s \beta})
\ar[r, "\phi ^{\otimes}" swap]
	\& (\alpha * 1 _{s \alpha}) \otimes (\beta * 1 _{t \beta})
	\ar[u, "r ^{loc} \otimes r ^{loc}" swap]
\end{tikzcd}
\end{gather*}
\end{itemize}

\item
Axioms that express that \( 1_a \) is a weak functor:
\begin{itemize}[-]
\item
The weak functor axioms \eqref{diag_functor_axiom_unit} correspond to the commutativity of
\[
\begin{tikzcd}[column sep = large]
i_a * i_a
\ar[r, "\phi_a ^{-1} * 1"]
\ar[d, "1 * \phi_a ^{-1}" swap]
	& 1 _{(1_a)} * i_a
	\ar[d, "{l _{(i_a)}}"] \\
i_a * 1 _{(1_a)}
\ar[r, "{r _{(i_a)}}" swap]
	& i_a
\end{tikzcd}
\]
\item
When we write \( X \) for the diagonal in the previous diagram, the weak functor axiom \eqref{diag_functor_axiom_asso} corresponds to the commutativity of
\[
\begin{tikzcd}
(i_a * i_a) * i_a
\ar[r, "{X * 1}"]
\ar[d, "a" swap]
	& i_a * i_a
	\ar[r, "{X}"]
		& i_a
		\ar[d, equal] \\
i_a * (i_a * i_a)
\ar[r, "{1 * X}" swap]
	& i_a * i_a
	\ar[r, "{X}" swap]
		& i_a
\end{tikzcd}
\]
\end{itemize}
\item
Axioms that express \( (a, a ^{\adsq}, \eta, \epsilon) \) is an adjoint equivalence:
\begin{itemize}[-]
\item
The pseudonaturality of \( a \) corresponds to the commutativity of
\[
\begin{tikzcd}[column sep = large]
a _{t \alpha_1, t \alpha_2, t \alpha_3} * \big( (\alpha_1 \otimes \alpha_2) \otimes \alpha_3 \big)
\ar[d, "1 * \big( (\Lambda_1 \otimes \Lambda_2) \otimes \Lambda_3 \big)" swap]
\ar[r,"a _{\alpha_1 \alpha_2 \alpha_3}"]
	& \big( \alpha_1 \otimes (\alpha_2 \otimes \alpha_3) \big) * a _{s \alpha_1, s \alpha_2, s \alpha_3}
	\ar[d, "\big( \Lambda_1 \otimes (\Lambda_2 \otimes \Lambda_3) \big) * 1"] \\
a _{t \alpha_1', t \alpha_2', t \alpha_3'} * \big( (\alpha_1' \otimes \alpha_2') \otimes \alpha_3' \big)
\ar[r,"a _{\alpha_1' \alpha_2' \alpha_3'}" swap]
	& \big( \alpha_1' \otimes (\alpha_2' \otimes \alpha_3') \big) * a _{s \alpha_1', s \alpha_2', s \alpha_3'},
\end{tikzcd}
\]
and similar for \( a ^{\adsq}. \)

\item
That \( \epsilon \) is a modification corresponds to the following diagram, where we omit indices and abbreviate \( (\alpha \otimes \beta) \otimes \gamma \) with \( X \) resp.
\( \alpha \otimes (\beta \otimes \gamma) \) with \( X'. \)
\[
\begin{tikzcd}[font=\scriptsize, cramped]
(a ^{\adsq} * a) * X
\ar[d,"\epsilon * 1"]
\ar[r, "a ^{loc}"]
	& a ^{\adsq} * ( a * X )
	\ar[r,"a"]
		& a ^{\adsq} * ( X' * a )
		\ar[r, "(a ^{loc}) ^{-1}"]
			& ( a ^{\adsq} * X') * a
			\ar[dr, "a ^{\adsq}"] \\
1 * X
\ar[r,"l"]
	& X
	\ar[r,"r ^{-1}"]
		& X * 1
			& X * (a ^{\adsq} * a)
			\ar[l,"1 * \epsilon"]
				& ( X * a ^{\adsq}) * a
				\ar[l, "a ^{loc}"]
\end{tikzcd}
\]

\item
That \( \epsilon \) is a modification corresponds to a similar diagram.

\item
The axioms in the definition of an adjoint equivalence (see definition \ref{def_equivalence_adjunction}) yield the commutativity of following diagrams (again indices are omitted)
\[
\begin{tikzcd}[cramped]
a ^{\adsq}
\ar[r,"l ^{-1}"]
\ar[d,"1" swap]
	& 1 * a ^{\adsq}
	\ar[r, "\eta * 1"]
		& (a ^{\adsq} * a) * a ^{\adsq}
		\ar[d, "a ^{loc}"] \\
a ^{\adsq}
	& a ^{\adsq} * 1
	\ar[l, "r"]
		& a ^{\adsq} * (a * a ^{\adsq})
		\ar[l, "1 * \epsilon"]
\end{tikzcd}
\qquad
\begin{tikzcd}[cramped]
a
\ar[r, "r ^{-1}"]
\ar[d,"1" swap]
	& a * 1
	\ar[r, "1 * \eta"]
		& a * (a ^{\adsq} * a)
		\ar[d, "a ^{loc}"] \\
a
	& 1 * a
	\ar[l, "l"]
		& (a * a ^{\adsq}) * a
		\ar[l, "\epsilon * 1"].
\end{tikzcd}
\]
\end{itemize}

\item
The axioms that express that \( (l, l ^{\adsq}, \eta ^{l}, \epsilon ^{l}) \) and \( (r, r ^{\adsq}, \eta ^{r}, \epsilon ^{r}) \) are adjoint equivalences are analogous.

\item
That \( \lambda \) is a modification corresponds to the equality of the following two composites:
\begin{align*}
& (l _{t \alpha} \otimes 1 _{t \beta}) * \big( (i _{t^2 \alpha} \otimes \alpha) \otimes \beta \big) \\
\xrightarrow{\mathmakebox[1.3cm]{\lambda _{t \alpha t \beta} * 1}} \;
& (l _{t \alpha * t \beta} * a _{1 _{t^2 \alpha}, t \alpha, t \beta}) * \big( (i _{t^2 \alpha} \otimes \alpha) \otimes \beta \big) \\
\xrightarrow{\mathmakebox[1.3cm]{a ^{loc}}} \;
& l _{t \alpha * t \beta} * \big( a _{1 _{t^2 \alpha}, t \alpha, t \beta} * ((i _{t^2 \alpha} \otimes \alpha) \otimes \beta) \big) \\
\xrightarrow{\mathmakebox[1.3cm]{1 * a _{i  \alpha \beta}}} \;
& l _{t \alpha * t \beta} * \big( (i _{t^2 \alpha} \otimes (\alpha \otimes \beta)) * a _{1 _{s^2 \alpha}, s \alpha, s \beta} \big) \\
\xrightarrow{\mathmakebox[1.3cm]{a ^{loc}}} \;
& \big( l _{t \alpha * t \beta} * (i _{t^2 \alpha} \otimes (\alpha \otimes \beta)) \big) * a _{1 _{s^2 \alpha}, s \alpha, s \beta} \\
\xrightarrow{\mathmakebox[1.3cm]{l _{\alpha \otimes \beta}*1}} \;
& \big( (\alpha \otimes \beta) * l _{s \alpha * s \beta} \big) * a _{1 _{s^2 \alpha}, s \alpha, s \beta} \\
\xrightarrow{\mathmakebox[1.3cm]{a ^{loc}}} \;
& (\alpha \otimes \beta) * ( l _{s \alpha * s \beta} * a _{1 _{s^2 \alpha}, s \alpha, s \beta} ) \\
& \phantom{xxxxxxxxxx} = \\
& (l _{t \alpha} \otimes 1 _{t \beta}) * \big( (i _{t^2 \alpha} \otimes \alpha) \otimes \beta \big) \\
\xrightarrow{\mathmakebox[1.3cm]{\phi}} \;
& \big( (l _{t \alpha} * (i _{t^2 \alpha} \otimes \alpha) \big) \otimes (1 _{t \beta} * \beta) \\
\xrightarrow{\mathmakebox[2.5cm]{l _{\alpha} \otimes \big((r _{\beta} ^{loc})^{-1} \circ l ^{loc} _{\beta} \big) }} \;
& (\alpha * l _{s \alpha}) \otimes (\beta * 1 _{s \beta}) \\
\xrightarrow{\mathmakebox[1.3cm]{\phi ^{-1}}} \;
& (\alpha \otimes \beta) * (l _{s \alpha} \otimes 1 _{s \beta}) \\
\xrightarrow{\mathmakebox[1.3cm]{1 * \lambda _{s \alpha s \beta}}} \;
& (\alpha \otimes \beta) * ( l _{s \alpha * s \beta} * a _{1 _{s^2 \alpha}, s \alpha, s \beta} ) \\
\end{align*}

\item
That \( \pi,\mu,\rho \) are modifications yields similar equations.
\end{itemize}

\end{mydef}

\subsection{Virtually strict triequivalences} \label{sec_virtually_strict_functor}

The following notion of a virtually strict functor is in one to one correspondence to the usual notion of a strict functor of tricategories (definition 4.13 in \cite{Gurski2013}).
The advantage of working with virtually strict functors is that one gets rid of the trivial constraint data in the usual definition of a strict tricategory-functor and therefore composition becomes much easier, which in turn enables one to define a 1-category of tricategories and virtually strict functors.
(See also the discussion at the beginning of chapter 6.2 in \cite{Gurski2013}).

\begin{mydef} [definition 6.6 in \cite{Gurski2013}]
Let \( \cat{T} \) and \( \cat{T'} \) be tricategories.
A morphism of magmoids \( F:\cat{T} \to \cat{T'} \) is called a virtually strict functor, if it \emph{(strictly) preserves units and constraint cells}:
\[
F (1_a) = 1 _{Fa}, \;
F(1_f) = 1 _{Ff}, \;
F(1_\alpha) = 1 _{F \alpha}, \;
F (i_a) = i _{Fa}, \;
F (\phi_a) = \phi _{Fa}, \;
\dots
\]
(See the data in definition \ref{def_tricat_as_magmoid} for a complete list of constraints).
Composition of virtually strict functors is given by composition of morphisms of magmoids.
\end{mydef}

\begin{bem}
One can show (one direction is not completely trivial, though) that these virtually strict functors are the same as the ones from definition 6.2 in \cite{Gurski2013}.
The slightly unwieldy name "virtually strict functor" stems from the fact that virtually strict functors are strictly speaking no strict functors in the usual sense, but equivalent to those.
\end{bem}

\begin{satz} [definition 6.7 in \cite{Gurski2013}]
There is a 1-category \( \cat{Tricat_1} \) with tricategories as objects and virtually strict functors as morphisms.
\end{satz}

\begin{proof}
The units of \( \cat{Tricat_1} \) are given by units of \( \magm{3}. \)
Note that the composite of two constraint-preserving magmoid morphisms is a constraint-preserving magmoid morphism as well.
Then the category axioms in \( \cat{Tricat_1} \) hold, since they hold in \( \magm{3}. \)
\end{proof}

\begin{bem} \label{rem_virtual_functor_locally}
Locally a virtually strict functor is a morphism of 2-magmoids, which preserves constraint cells. And that in turn is just a strict functor of bicategories.
\end{bem}

\begin{mydef}
A virtually strict functor \( F: T \to T' \) is called \emph{triessentially surjective} if for any \( a' \in \ob{\cat{T'}} \) there exists an object \( a \in \ob{\cat{T}} \) together with a biequivalence \( F a \to a'. \)
\end{mydef}

\begin{mydef}[virtually strict triequivalence]
We call a virtually strict functor \( F: \cat{T} \to \cat{T'} \) a virtually strict triequivalence, if it is triessentially surjective and locally a strict biequivalence (see also remark \ref{rem_virtual_functor_locally}).
\end{mydef}

\begin{bem} \label{rem_characterization_virtually_strict_triequivalence}
We can further unpack this definition using the characterizations of biequivalence and equivalence.
Proving that a virtually strict functor \( F \) is a triequivalence then comes down to check three surjectivity and one bijectivity condition, namely
\begin{itemize}
\item
\( F \) is triessentially surjective
\item
\( F \) is locally biessentially surjective
\item
\( F \) is 2-locally essentially surjective
\item
\( F \) is 3-locally a bijection.
\end{itemize}
\end{bem}

\subsection{Coherence for tricategories} \label{sec_coherence_for_tricats}

In order to state the coherence theorem for tricategories we need to introduce free tricategories.
The underlying data from which the free construction starts is a \( \cat{Cat \text{-} Gph} \)-enriched graph (for \( \cat{V} \)-enriched graph see definition \ref{def_V_graph}).

\begin{mydef}~ \label{def_cat_2_graph}
We call the category \( \cat{(Cat\text{-}Gph)\text{-}Gph} \) the category of \emph{category-enriched 2-graphs} and denote it \( \cat{Cat\text{-}2Gph}. \)
\end{mydef}

Now the construction of a free tricategory over a category-enriched 2-graph works analogously to the construction of a free bicategory over a \( \cat{Cat} \)-graph.

\begin{mydef}[(compare definition 6.12 in \cite{Gurski2013}] \label{def_free_tricat}
Let \( X \) be a category enriched 2-graph.
Then the free tricategory \( F X \) over \( X \) is defined as follows:
\begin{itemize}
\item
\( \ob{F X}=\ob{X}. \)
\item
The 1-morphisms of \( FX \) are defined inductively as follows:

\( \mor{\oper{F}X} \) contains the so called \( F \)-basic 1-morphisms which are
\begin{itemize}[-]
\item
for each \( f \in \ob{X(a,b)} \) a 1-morphism \( f: a \to b, \)
\item
for each object \( a \in \ob{X} \) a unit 1-morphism \( \fc{1} _a: a \to a. \)
\end{itemize}

For \( v: b \to c \) and \( w: a \to b \) in \( \mor{\oper{F}G} \) we require the formal composite \( v \fstar w: a \to c \) to be a 1-morphism of \( FG. \)

\item
The 2-morphisms of \( FX \) are defined inductively as follows:

There are the following \( F \)-basic 2-morphisms:
\begin{itemize}[-]
\item
For any \( \alpha \in \ob{X(a,b)(f,g)} \) a 2-morphism \( \alpha: f \to g: a \to b, \)
\item
unit 2-morphisms \( 1_v: v \to v \) for all \( v \in \mor{FX}, \)
\item
2-cells \( i_a: 1_a \to 1_a \) for all \( a \in \ob{X}, \)
\item
constraint cells
\( l_v: 1 _{tv} \otimes v \to v \) and \( l_v ^{\adsq} : v \to 1 _{tv} \otimes v \) for all \( v \in \mor{FX}, \)
\item
constraint cells
\( r_v: v \otimes 1 _{sv} \to v \) and \( r_v ^{\adsq} : v \to v \otimes 1 _{sv} \) for all \( v \in \mor{FX}, \)
\item
constraint cells
\( a _{v_1 v_2 v_3}: (v_1 \otimes v_2) \otimes v_3 \to v_1 \otimes (v_2 \otimes v_3) \) and \( a _{v_1 v_2 v_3} ^{\adsq}: v_1 \otimes (v_2 \otimes v_3) \to (v_1 \otimes v_2) \otimes v_3 \).
\end{itemize}

For \( \tau:w \to w' : b \to c, \; \sigma: v \to v' : a \to b \) and \( \sigma':v' \to v'': a \to b \) in \( {\mmor{F X}} \) we require the formal composites
\( \tau \fstar \sigma: w \fstar v \to w' \fstar v': a \to c \) and \( \sigma' \circ \sigma: v \to v'': a \to b \) to be 2-morphisms of \( FX. \)

\item
The construction of the 3-morphisms of \( FX \) works analogously to the construction of the 2-morphisms in definition \ref{def_free_bicat}:
The basic 3-morphisms are the morphisms in the categories \( X(a,b)(f,g) \) and adjunct 3-cells corresponding to the data in the magmoidal definition of a tricategory \ref{def_tricat_as_magmoid}.
The general 3-cells are then built inductively from the basic 3-cells by composing formally along objects, 1-morphisms and 2-morphisms.
Finally one quotients out the congruence relation generated by the composition laws of the categories \( X(a,b)(f,g) \) and by the tricategory axioms from definition \ref{def_tricat_as_magmoid}.
\end{itemize}
\end{mydef}

When we write \( U: \cat{Tricat_1} \to \cat{Cat\text{-}2Gph} \) for the forgetful functor from the 1-category of tricategories and virtually strict functors to the category of category enriched 2-graphs,
\( FG \) satisfies the usual universal property for free constructions.

\begin{satz}[proposition 6.8 in \cite{Gurski2013}] \label{prop_universal_property_free_tricat}
Let \( X \) be a Cat-2-Graph.
The morphism of Cat-2-Graphs \( \iota_X: X \to UFX, \) which sends any cell of \( X \) to the cell in \( UFX \) with the same name is a universal arrow from \( X \) to \( U. \)

This means for any morphism of Cat-2-Graphs \( S: X \to UT, \) where \( T \in \cat{Tricat_1} \) there is a unique virtually strict functor \( S' :FX \to T \) such that \( S = US' \circ \iota. \)
\end{satz}

Another important part of the adjunction \( F \dashv U: \cat{Tricat_1} \to \cat{Cat\text{-}2Gph} \) proposition \ref{prop_universal_property_free_tricat} gives rise to
(see \ref{prop_characterization_adjunctions_in_cat}), is its counit, which acts through evaluation.

\begin{satz}
The components
\(
\epsilon _{\cat{T}}: FU \cat{T} \to \cat{T}
\)
of the counit of the adjunction \( F \dashv U: \cat{Tricat_1} \to \cat{Cat\text{-}2Gph} \) from proposition \ref{prop_universal_property_free_tricat} are given through evaluation.
\end{satz}

Now we can state the coherence theorem for tricategories.

\begin{terminology}
A category enriched 2-graph \( X \) is called 2-locally discrete, if each \( X(a,b)(f,g) \) is a discrete category.
\end{terminology}

\begin{satz}[corollary 10.6 in \cite{Gurski2013}] \label{prop_tricat_coherence_theorem}
Let \( X \) be a 2-locally discrete category enriched 2-graph. Then any parallel 3-morphisms in \( FX \) are equal.
\end{satz}

Let \( X \) be a \( \cat{Cat} \)-enriched 2-graph, and \( A \subseteq X \) be the \( \cat{Cat} \)-enriched sub-2-graph with
\( \ob{A} = \ob{X}, \; A(a,b) = X(a,b) \) and \( A(a,b)(f,g) \) is the the discrete category with the same objects as \( X(a,b)(f,g). \)
Then \( FA \subseteq FX \) is a sub-tricategory whose 3-morphisms are precisely the coherence 3-morphisms of \( FX. \)
This yields the following corollary.

\begin{kor}
Parallel coherence 3-morphisms in a free tricategory are equal.
\end{kor}

We can also rephrase the coherence theorem to refer to general tricategories:

\begin{kor}
Two coherence 3-morphisms in a tricategory \(  \cat{T} \) are equal, if they lift to parallel coherence 3-morphisms in \( FU \cat{T}. \)
\end{kor}

The objects, 1- and 2-morphisms of a tricategory form a algebra, that reminds of a bicategory, only that the equations which hold in a bicategory are now replaced by coherence 3-morphisms.
Therefore the coherence theorem for bicategories, which states that parallel constraint-2-morphisms in a free bicategory are equal, implies that in a free tricategory parallel constraint 2-morphisms are isomorphic via a coherence 3-morphism. And due to the coherence-theorem for tricategories this 3-coherence is unique.

\begin{kor} \label{cor_FX_is_semi_free}
Any two parallel constraint 2-morphisms in a free tricategory are uniquely isomorphic via a coherence 3-morphism.
\end{kor}

\subsection{Quotienting out coherence} \label{sec_quotiening_out_coherence}

In this section we make precise how one can "quotient out coherence" from a tricategory.
This can be done by either quotienting out a system of coherence-3-morphisms \ref{prop_coh_3_system}, or by quotienting out coherence 2- and 3- morphisms simultaneously \ref{prop_coherence_23_system}.
Both results will be needed in section \ref{sec_strictification_for_tricats}.

\begin{mydef}[Coh-3-system]
Let \( \cat{T} \) be a tricategory.
A subset \( C \subseteq \mmmor{T} \) is called a coherence 3-system of \( \cat{T} \), if it satisfies the following conditions:
\begin{enumerate}
\item
\( C \) is closed under \( \otimes, *, \circ. \)
\item
\( 1_\alpha \in C \) for all \( \alpha \in \mmor{T}. \)
\item
Any 3-morphism in \( C \) is invertible, and
\[
\Gamma \in C \implies \Gamma ^{-1} \in C.
\]
\item
\( \Gamma, \Gamma' \in C, \text{ parallel } \implies \Gamma = \Gamma'. \)
\end{enumerate}
\end{mydef}

\begin{lem} \label{lem_coherence_3_congruence_relation}
Let \( C \subseteq \mmmor{\cat{T}} \) be a coherence-3-system of a tricategory \( \cat{T}, \) then
\[
\alpha \sim_2 \beta \iff \exists \Gamma \in C: \alpha \to \beta
\]
defines a \( (*,\circ) \)-congruence relation (see definition \ref{def_congruence_rel}) on \( \mmor{\cat{T}} \) and
\[
\Pi \sim_3 \Lambda \iff \exists \Gamma, \Gamma' \in C: \Gamma \circ \Pi \circ \Gamma' = \Lambda
\]
defines a \( (\otimes, *, \circ) \) congruence relation on \( \mmmor{\cat{T}}. \)
\end{lem}

\begin{proof} ~
\begin{itemize}[-]
\item
\( \sim_2 \) is reflexive, symmetric and transitive, because \( C \) contains units, inverses and is closed under \( \circ. \)
\( \sim_2 \) is compatible with \( \otimes \) and \( *, \) because \( C \) is closed under \( \otimes \) and \( *. \)
Therefore \( \sim_2 \) is a congruence relation.
\item
Again \( \sim_3 \) is reflexive, symmetric and transitive, because \( C \) contains units, inverses and is closed under \( \circ. \)

To show that \( \sim_3 \) is compatible with \( \otimes \) let \( \Lambda_1 = \Gamma_1 \circ \Lambda'_1 \circ \Gamma'_1 \) and \( \Lambda_2 = \Gamma_2 \circ \Lambda'_2 \circ \Gamma'_2, \) then
\begin{align*}
\Lambda_1 \otimes \Lambda_2 = (\Gamma_1 \circ \Lambda'_1 \circ \Gamma'_1) \otimes (\Gamma_2 \circ \Lambda'_2 \circ \Gamma'_2) \\
= (\Gamma_1 \otimes \Gamma_2) \circ (\Lambda_1' \otimes \Lambda_2') \circ (\Gamma'_1 \otimes \Gamma'_2).
\end{align*}
Similarly one shows compatibility with \( *. \)

For compatibility with \( \circ \) let \( \Lambda_1, \Lambda'_1, \Lambda_2, \Lambda'_2 \) be as before,
with \( \Lambda_1 \) and \( \Lambda_2 \) resp. \( \Lambda'_1 \) and \( \Lambda'_2 \) being \( \circ \)- composable.
Then we have
\[
\Lambda_1 \circ \Lambda_2 = \Gamma_1 \circ \Lambda'_1 \circ \Gamma'_1 \circ \Gamma_2 \circ \Lambda'_2 \circ \Gamma'_2 \\
= \Gamma_1 \circ \Lambda'_1 \circ \Lambda'_2 \circ \Gamma'_2,
\]
where we used that \( \Gamma'_1 \circ \Gamma_2 = 1 _{t \Lambda'_2}, \) because parallel 3-morphisms in \( C \) are equal.
\end{itemize}
\end{proof}

\begin{satz}~ \label{prop_coh_3_system}
Let \( C \) be a coherence-3-system of a tricategory \( \cat{T}. \)
\begin{enumerate}
\item
There is a tricategory \( \cat{T}/\scriptstyle{C} \), whose 0- and 1-morphisms are the 0- and 1-morphisms of \( \cat{T} \)
and whose 2- and 3-morphisms are equivalence classes with respect to the relations \( \sim_2 \) and \( \sim_3 \) from lemma \ref{lem_coherence_3_congruence_relation}.

We write \( [\alpha] \) resp. \( [\Pi] \) for the equivalence class of a 2- resp. 3-morphism of \( \cat{T} \)
and we may also write \( [a] \) resp. \( [f], \) if we consider \( a \) resp. \( f \) as an object resp. 1-morphism of \( \cat{T}/\scriptstyle{C}. \)
The remaining data of \( \cat{T}/\scriptstyle{C} \) are given as follows:

\begin{itemize}
\item
The types of the 2- and 3-morphisms of \( \cat{T}/\scriptstyle{C} \) are given as
\[
s[\alpha] = s \alpha, \quad
t[\alpha] = t \alpha, \quad
s[\Pi] = [s \Pi], \quad
t[\Pi] = [t \Pi].
\]

\item
Composition for 2-morphisms in \( \cat{T}/\scriptstyle{C} \) is defined by
\[
[\alpha] \otimes [\beta] = [\alpha \otimes \beta]
\quad \text{and} \quad
[\alpha] * [\beta] = [\alpha * \beta].
\]
The composition operations \( (\otimes, *) \) of 3-morphisms are defined by
\[
[\Pi] \otimes [\Lambda] = [\Pi \otimes \Lambda]
\quad \text{and} \quad
[\Pi] * [\Lambda] = [\Pi * \Lambda].
\]
To define \( (\circ) \) we observe, that \( s [\Pi] = t [\Lambda] \) implies the existence of composable \( \Pi', \Lambda' \) with \( [\Pi'] = [\Pi] \) and \( [\Lambda'] = [\Lambda]. \)
Then we define
\[
[\Pi] \circ [\Lambda] = [\Pi' \circ \Lambda'].
\]

\item
The constraints of \( \cat{T}/\scriptstyle{C} \) are given as follows:

\[
1 _{[a]} = [1 _{a}], 1 _{[f]} = [1 _{f}], a _{[\alpha][\beta][\gamma]} = [a _{\alpha \beta \gamma}], \rho _{[f],[g]} = [\rho _{f,g}], \dots.
\]
\end{itemize}

\item
Mapping morphisms in \( \cat{T} \) to their respective equivalence classes in \( \cat{T}/\scriptstyle{C} \) defines a virtually strict triequivalence
\(
[-]: \cat{T} \to \cat{T}/\scriptstyle{C}.
\)
\end{enumerate}
\end{satz}

\begin{proof}~
\begin{itemize}[-]
\item
All types in \( \cat{T}/\scriptstyle{C} \) are well-defined and satisfy the globularity conditions \( st = ss \) and \( ts = tt, \) therefore \( \cat{T}/\scriptstyle{C} \) caries the structure of a globular set.
\item
Because \( \sim_2 \) and \( \sim_3 \) are congruence relations, the operations \( (\otimes, *, \circ) \) are well-defined.
And since the operations \( (\otimes,*,\circ) \) are also compatible with the globular structure, the 3-magmoid structure of \( \cat{T}/\scriptstyle{C} \) is established.
\item
It follows directly from the definition of composition in \( \cat{T}/\scriptstyle{C} \) that \( [-]: \cat{T} \to \cat{T}/\scriptstyle{C} \) is a morphism of magmoids.
\item
It is clear that the units and constraints subscripted with 0- and 1-morphisms are well-defined.
Constraints subscripted with 2-morphisms (see the data in section \ref{sec_tricats_as_magmoids} for a list of such constraint cells) are well-defined, due to the naturality of the corresponding constraint morphisms in \( \cat{T}. \)
For example let
\[
\Gamma_1: \alpha_ 1 \to \alpha'_1, \quad \Gamma_2: \alpha_ 2 \to \alpha'_2, \quad \Gamma_3: \alpha_ 3 \to \alpha'_3
\]
be \( \otimes \)-composable 3-morphisms in \( C. \)
Then
\[
[a _{\alpha_1 \alpha_2 \alpha_3}] = [a _{\alpha'_1 \alpha'_2 \alpha'_3}]
\]
because of the commutativity of the square
\[
\begin{tikzcd}[column sep = huge, row sep = large]
a _{t \alpha_1, t \alpha_2, t \alpha_3} * \big( (\alpha_1 \otimes \alpha_2) \otimes \alpha_3 \big)
\ar[d, "{1 * \big( (\Gamma_1 \otimes \Gamma_2) \otimes \Gamma_3 \big)}" swap]
\ar[r, "a _{\alpha_1 \alpha_2 \alpha_3}"]
	& \big( \alpha_1 \otimes (\alpha_2 \otimes \alpha_3) \big) * a _{s \alpha_1, s \alpha_2, s \alpha_3}
	\ar[d, "{\big( \Gamma_1 \otimes (\Gamma_2 \otimes \Gamma_3) \big) * 1}"] \\
a _{t \alpha'_1, t \alpha'_2, t \alpha'_3} * \big( (\alpha'_1 \otimes \alpha'_2) \otimes \alpha'_3 \big)
\ar[r, "a _{\alpha'_1 \alpha'_2 \alpha'_3}" swap]
	& \big( \alpha'_1 \otimes (\alpha'_2 \otimes \alpha'_3) \big) * a _{s \alpha'_1, s \alpha'_2, s \alpha'_3}.
\end{tikzcd}
\]

The types of the constraint-cells are in line with what is required by section (\ref{sec_tricats_as_magmoids}).
This is the case because of the way the constraint cells are defined and because \( [-] \) is a morphism of 3-magmoids.
For example

\begin{align*}
s (a _{[\alpha_1][\alpha_2][\alpha_3]}) = s [a _{\alpha_1 \alpha_2 \alpha_3}] = [s (a _{\alpha_1 \alpha_2 \alpha_3})] \\
= [a _{t \alpha_1, t \alpha_2, t \alpha_3} * \big( (\alpha_1 \otimes \alpha_2) \otimes \alpha_3 \big)] \\
= a _{t [\alpha_1], t [\alpha_2], t [\alpha_3]} * \big( ([\alpha_1] \otimes [\alpha_2]) \otimes [\alpha_3] \big).
\end{align*}

\item
Any axiom for \( \cat{T}/\scriptstyle{C} \) is the image of an axiom of \( \cat{T} \) under \( [-]. \)
For example
\[
\begin{tikzcd}[column sep = huge, row sep = large]
a _{t [\alpha_1], t [\alpha_2], t [\alpha_3]} * \big( ([\alpha_1] \otimes [\alpha_2]) \otimes [\alpha_3] \big)
\ar[d, "{1 * \big( ([\Pi_1] \otimes [\Pi_2]) \otimes [\Pi_3] \big)}" swap]
\ar[r, "a _{[\alpha_1] [\alpha_2] [\alpha_3]}"]
	& \big( [\alpha_1] \otimes ([\alpha_2] \otimes [\alpha_3]) \big) * a _{s [\alpha_1], s [\alpha_2], s [\alpha_3]}
	\ar[d, "{\big( [\Pi_1] \otimes ([\Pi_2] \otimes [\Pi_3]) \big) * 1}"] \\
a _{t [\alpha'_1], t [\alpha'_2], t [\alpha'_3]} * \big( ([\alpha'_1] \otimes [\alpha'_2]) \otimes [\alpha'_3] \big)
\ar[r, "a _{[\alpha'_1] [\alpha'_2] [\alpha'_3]}" swap]
	& \big( [\alpha'_1] \otimes ([\alpha'_2] \otimes [\alpha'_3]) \big) * a _{s [\alpha'_1], s [\alpha'_2], s [\alpha'_3]} \\
\end{tikzcd}
\]
commutes, because it is the image of
\[
\begin{tikzcd}[column sep = huge, row sep = large]
a _{t \alpha_1, t \alpha_2, t \alpha_3} * \big( (\alpha_1 \otimes \alpha_2) \otimes \alpha_3 \big)
\ar[d, "{1 * \big( (\Pi_1 \otimes \Pi_2) \otimes \Pi_3 \big)}" swap]
\ar[r, "a _{\alpha_1 \alpha_2 \alpha_3}"]
	& \big( \alpha_1 \otimes (\alpha_2 \otimes \alpha_3) \big) * a _{s \alpha_1, s \alpha_2, s \alpha_3}
	\ar[d, "{\big( \Pi_1 \otimes (\Pi_2 \otimes \Pi_3) \big) * 1}"] \\
a _{t \alpha'_1, t \alpha'_2, t \alpha'_3} * \big( (\alpha'_1 \otimes \alpha'_2) \otimes \alpha'_3 \big)
\ar[r, "a _{\alpha'_1 \alpha'_2 \alpha'_3}" swap]
	& \big( \alpha'_1 \otimes (\alpha'_2 \otimes \alpha'_3) \big) * a _{s \alpha'_1, s \alpha'_2, s \alpha'_3} \\
\end{tikzcd}
\]
under \( [-]. \)
Analogously one shows the other axioms from section \ref{sec_tricats_as_magmoids} hold.

\item
We already observed that \( [-] \) is a morphism of magmoids.
That \( [-] \) preserves units and constraints as well, follows directly from their definition in \( \cat{T}/\scriptstyle{C}. \)
Therefore \( [-] \) is a strict functor of tricategories.

\item
Finally we show that \( [-] \) is a virtually strict triequivalence.
We start by showing that \( [-] \) is 3-locally a bijection.
3-local surjectivity is trivially satisfied and 3-local injectivity follows because uniqueness of parallel 3-morphisms in \( C \) implies \( \Pi = \Pi' \) if
\(
\Pi, \Pi': \alpha \to \beta \) and
\(
[\Pi] = [\Pi'].
\)
Since \( [-] \) is 2-locally surjective and 3-locally a bijection, it is 2-locally an equivalence.
Then since \( [-] \) is 1-locally surjective and 2-locally an equivalence, it is 1-locally a biequivalence.
And finally because \( [-] \) is 0-locally surjective and 1-locally a biequivalence, it is a triequivalence.
\end{itemize}
\end{proof}

\begin{lem} \label{lem_C_under_quotient}
Let \( C \) be a coherence 3-system in a tricategory \( \cat{T}. \)
Then for \( \Gamma \in C \) we have
\[
s [\Gamma] = t [\Gamma]
\quad \text{and} \quad
[\Gamma] = 1 _{s [\Gamma]} = 1 _{t [\Gamma]}.
\]
\end{lem}

\begin{proof}
\[
[\Gamma] = [1 _{t \Gamma} \circ \Gamma \circ \Gamma ^{-1}] = [1 _{t \Gamma}] = 1 _{[t \Gamma]} = 1 _{t [\Gamma]} = 1 _{s [\Gamma]}.
\]
\end{proof}

We have seen that uniqueness of parallel 3-morphisms in \( C \) is a necessary requirement, for quotienting out a coherence 3-system \( C \subseteq \mmmor{\cat{T}} \) from a tricategory \( \cat{T}. \)
(Otherwise \( \circ \)-composition would not be well-defined in \( \cat{T}/\scriptstyle{C}. \))
Similarly parallel coherence 2-morphisms need to be equal up to unique coherence, if we want to "quotient out coherence 2-morphisms" in a tricategory.
That motivates the following definition:

\begin{mydef}
A tricategory is called \emph{semi-free}, if (parallel) coherence 2-morphisms are isomorphic via a uniquely determined coherence 3-morphism.
\end{mydef}

\begin{mydef}[Coherence-(2,3)-system] \label{def_coherence_23_system}
Let \( \cat{T}	\) be a semi-free tricategory, \( C_2 \) the collection of all coherence 2-morphisms of \( \cat{T} \) and \( C_3 \subseteq \mmmor{T}. \)

Then \( C_3 \) is called a \emph{coherence-(2,3)-system} if the following conditions are satisfied:
\begin{enumerate}
\item
\( C_3 \) is a coherence-3-system.
\item \label{list_C_3_is_C_2_complete}
If \( \gamma,\gamma' \) are parallel 2-morphism in \( C_2, \) the unique coherence \( \Gamma: \gamma \to \gamma' \) is contained in \( C_3. \)
\item \label{list_interchanger_in_23_coherence_system}
All interchangers of the form
\begin{align*}
& \phi _{(\gamma,\gamma')(\alpha,\alpha')}: (\gamma \otimes \gamma') * (\alpha \otimes \alpha') \to (\gamma * \alpha) \otimes (\gamma' * \alpha') \quad \text{or} \\
& \phi _{(\alpha,\alpha')(\gamma,\gamma')}: (\alpha \otimes \alpha') * (\gamma \otimes \gamma') \to (\alpha * \gamma) \otimes (\alpha' * \gamma')
\end{align*}
with \( \gamma,\gamma' \in C_2 \) and \( \alpha,\alpha' \in \mmor{\cat{T}}, \) are contained in \( C_3. \)
\item \label{list_local_constraints_in_C_3}
All local constraints, i.e. components of \( a ^{loc}, l ^{loc}, r ^{loc}, \) are contained in \( C_3. \)
\end{enumerate}
\end{mydef}

\begin{bem}~
\begin{enumerate}
\item
If \( \gamma: f \to g \) and \( \gamma ^{\adsq}: g \to f \) are morphisms in \( C_2, \)
it follows that \( \gamma, \gamma ^{\adsq} \) are part of an equivalence whose unit and counit are morphisms in \( C_3. \)
The weak inverse \( \gamma ^{\adsq} \in C_2 \) is not unique, but once it is chosen, unit and counit are unique as morphisms of \( C_3. \)
\item
From \eqref{list_C_3_is_C_2_complete} in definition \ref{def_coherence_23_system} it follows that all constraint 3-cells,
which are indexed exclusively with 2-morphisms from \( C_2 \) are contained in \( C_3. \)
Also all constraint 3-cells which are components of a modification are contained in \( C_3. \)
\end{enumerate}
\end{bem}

\begin{lem} \label{lem_coherence_23_congruence_relation}
Let \( C_3 \) be a coherence-(2,3)-system of a tricategory \( \cat{T}. \)
Then
\[
f \sim_1 g \iff \exists \gamma: f \to g \in C_2.
\]
defines a \( \otimes \)-congruence relation (see definition \ref{def_congruence_rel}) on \( \mor{\cat{T}}, \)
\[
\alpha \sim_2 \alpha' \iff \exists \gamma,\gamma' \in C_2 \text{ and } \Gamma: \gamma * \alpha \to \alpha' * \gamma' \in C_3.
\]
defines a \( (\otimes, *) \)-congruence relation on \( \mmor{\cat{T}} \) and
\[
\Lambda \sim_3 \Lambda' \iff \exists \gamma, \gamma' \in C_2 \text{ and } \Gamma,\Gamma' \in C_3: \Gamma \circ (1_\gamma * \Lambda)= (\Lambda' * 1 _{\gamma'}) \circ \Gamma'
\]
defines a \( (\otimes, *, \circ) \)-congruence relation on \( \mmmor{\cat{T}}. \)

\end{lem}

Trying to establish a relation \( \alpha \sim_2 \alpha' \) resp. \( \Lambda \sim_3 \Lambda' \) includes choosing morphisms \( \gamma, \gamma' \in C_2. \)
The following lemma states that any such choice, for which \( \gamma * \alpha \) and \( \alpha' * \gamma' \) resp. \( \gamma * s \Lambda \) and \( s \Lambda' * \gamma' \) are parallel, is eligible to establish such a relation.

\begin{lem}[free choice of \( C_2 \)-conjugators]~ \label{lem_free_choice_of_C2_conjugators}
\begin{enumerate}
\item
Let \( \alpha \sim_2 \alpha' \) and \( \gamma, \gamma' \in C_2 \) with \( \gamma: t \alpha \to t \alpha' \) and \( \gamma': s \alpha' \to s \alpha. \)
Then there exists a unique \( \Gamma \in C_3 \) s.t. \( \Gamma: \gamma * \alpha \to \alpha' * \gamma'. \)
\item
Let \( \Lambda \sim_3 \Lambda' \) and \( \gamma, \gamma' \in C_2 \) with \( \gamma: t^2 \Lambda \to t^2 \Lambda' \) and \( \gamma': s^2 \Lambda' \to s^2 \Lambda. \)
Then there exist unique \( \Gamma, \Gamma' \in C_3 \) s.t. \( \Gamma \circ (1_ \gamma * \Lambda) = (\Lambda' * 1_ {\gamma'}) \circ \Gamma'. \)
\end{enumerate}
\end{lem}

\begin{proof}
We prove the slightly more difficult second case.
So let \( \Lambda \sim_3 \Lambda'. \)
Then there are \( \gamma_2, \gamma'_2 \in C_2 \) and \( \Gamma_2, \Gamma'_2 \in C_3 \) with \( \Gamma_2 \circ (1 _{\gamma_2} * \Lambda) = (\Lambda'  * 1_ {\gamma'_2}) \circ \Gamma'_2. \)
Since \( \gamma,\gamma_2 \) and \( \gamma',\gamma'_2 \) are respectively parallel there exist \( C_3 \)-morphisms
\( \Gamma _{\gamma,\gamma_2}: \gamma \to \gamma_2 \) and \( \Gamma _{\gamma'_2,\gamma'}: \gamma'_2 \to \gamma'. \)
Now the desired morphisms \( \Gamma \) and \( \Gamma' \) are given uniquely by the vertical parts in the commutative square

\[
\begin{tikzcd}
\gamma * s \Lambda
\ar[r, "1 * \Lambda"]
\ar[d, "\Gamma _{\gamma,\gamma_2} * 1" swap]
	& \gamma * t \Lambda
	\ar[d, "\Gamma _{\gamma,\gamma_2} * 1"] \\
\gamma_2 * s \Lambda
\ar[r, "1 * \Lambda"]
\ar[d, "\Gamma'_2 " swap]
	& \gamma_2 * t \Lambda
	\ar[d, "\Gamma_2"] \\
s \Lambda' * \gamma'_2
\ar[r, "\Lambda' * 1"]
\ar[d, "1 * \Gamma _{\gamma'_2,\gamma'}" swap]
	& t \Lambda' * \gamma'_2
	\ar[d, "1 * \Gamma _{\gamma'_2,\gamma'}"] \\
s \Lambda' * \gamma'
\ar[r, "\Lambda' * 1"]
	& t \Lambda' * \gamma'
\end{tikzcd}
\]

where the top and the bottom square commute because of the interchange law in the local bicategory at hand, and the middle square commutes by assumption.
\end{proof}

A coherence-(2,3)-system is in particular a coherence-3-system, therefore one could examine relatedness of 2- and 3-morphism with either of the relations of
lemma \ref{lem_coherence_3_congruence_relation} or lemma \ref{lem_coherence_23_congruence_relation}.
The following lemma states that coherence-3-relatedness implies coherence-(2,3)-relatedness.

\begin{lem}[optionality of \( C_2 \) conjugation] \label{lem_optionality_C2_conjugation}
For the relations of lemma \ref{lem_coherence_23_congruence_relation} the following holds.
\begin{enumerate}
\item \label{list_optional_C2_conj_for_2_morphs}
For \( \alpha,\alpha' \in \mmor{\cat{T}} \) the existence of \( \Gamma \in C_3 \) with \( \alpha \xrightarrow{\Gamma} \alpha' \) implies \( \alpha \sim_2 \alpha'. \)
\item \label{list_optional_C2_conj_for_3_morphs}
For \( \Lambda, \Lambda' \in \mmmor{\cat{T}} \) the existence of \( \Gamma, \Gamma' \in C_3 \) with \( \Gamma \circ \Lambda = \Lambda' \circ \Gamma' \) implies \( \Lambda \sim_3 \Lambda'. \)
\end{enumerate}
\end{lem}

\begin{proof}
Both statements follow easily from the naturality of \( l ^{loc} \) and \( r ^{loc} \) plus the fact that \( l ^{loc} \) and \( r ^{loc} \) are contained in \( C_3. \)
\end{proof}

Before proving lemma \ref{lem_coherence_23_congruence_relation}, we mention some conventions which are used in the following.
For any \( \gamma \in C_2 \) the notation for a (non unique) adjoint equivalence will be \( (\gamma, \gamma ^{\adsq}, \epsilon ^{\gamma}, \eta ^{\gamma}) \) and the superscripts of unit and counit are usually dropped.
Furthermore we will usually not give (all) types in the definition of morphisms. If these morphisms appear later on in calculations, the types are assumed in such a way that the equations make sense.
For example if we define \( \Lambda \) and \( \Lambda' \) to be 3-morphisms in \( \cat{T} \) and later on talk about the composite \( \Lambda * \Lambda' \) this implies \( s^2 \Lambda = t ^2 \Lambda'. \)
When working in local bicategories we usually omit local constraints (i.e. components of \( a ^{loc}, l ^{loc}, r ^{loc}) \) and brackets.

\begin{proof}[Proof of lemma \ref{lem_coherence_23_congruence_relation}] ~

\begin{itemize}
\item
\( \sim_1 \) is reflexive, symmetric and transitive, because \( C_2 \) contains units and weak inverses and because \( C_2 \) is closed under \( *. \)
\( \sim_1 \) is compatible with \( \otimes, \) because \( C_2 \) is closed under \( \otimes. \)

\item
\( \sim_2 \) is reflexive, because
\[
1 _{t \alpha} * \alpha
\xrightarrow{l_\alpha}
\alpha
\xrightarrow{r_\alpha ^{\adsq}}
\alpha * 1 _{s \alpha}
\]
establishes the relation \( \alpha \sim_1 \alpha. \)
For 3-morphisms
\[
\gamma_1 * \alpha \xrightarrow{\Gamma_1} \alpha' * \gamma'_1
\quad \text{and} \quad
\gamma_2 * \alpha'  \xrightarrow{\Gamma_2} \alpha ^{\prime \prime} * \gamma'_2
\]
we have the following 3-morphisms in \( C_3 \)
\begin{align*}
& \gamma_1 ^{\adsq} * \alpha'
\xrightarrow{1 * 1 * \epsilon ^{-1}}
\gamma_1 ^{\adsq} * \alpha' * \gamma'_1 * \gamma_1 ^{\prime \adsq}
\xrightarrow{1 * \Gamma ^{-1} * 1}
\gamma_1 ^{\adsq} * \gamma_1 * \alpha * \gamma ^{\prime \adsq}
\xrightarrow{\eta ^{-1} * 1 * 1}
\alpha * \gamma ^{\prime \adsq},
\\
& \gamma_2 * \gamma_1 * \alpha
\xrightarrow{1 * \Gamma_1}
\gamma_2 * \alpha' * \gamma'_1
\xrightarrow{\Gamma_2 * 1}
\alpha ^{\prime \prime} * \gamma'_2 * \gamma'_1
\end{align*}
and therefore \( \sim_2 \) is symmetric and transitive.

\item
Suppose that \( \alpha_1 \sim_2 \alpha'_1 \) and \( \alpha_2 \sim_2 \alpha'_2. \)
Then there are 3-morphisms
\[
\gamma_1 * \alpha_1 \xrightarrow{\Gamma_1} \alpha_1' * \gamma'_1
\quad \text{and} \quad
\gamma_2 * \alpha_2 \xrightarrow{\Gamma_2} \alpha_2' * \gamma'_2.
\]
We assume \( \alpha_1, \alpha_2 \) are \( \otimes \)-composable and show \( \alpha_1 \otimes \alpha_2 \sim_2 \alpha'_1 \otimes \alpha'_2 \) by considering the morphism
\begin{alignat*}{2}
& (\gamma_1 \otimes \gamma_2) * (\alpha_1 \otimes \alpha_2)
& \; \xrightarrow{\phi_{(\gamma_1,\gamma_2)(\alpha_1,\alpha_2)}} \;
& (\gamma_1 * \alpha_1) \otimes (\gamma_2 * \alpha_2) \\
\xrightarrow{\Gamma_1 \otimes \Gamma_2} \;
& (\alpha'_1 * \gamma'_1) \otimes (\alpha'_2 * \gamma'_2)
& \; \xrightarrow{\phi _{(\alpha'_1,\alpha'_2)(\gamma'_1,\gamma'_2)}} \;
& (\alpha'_1 \otimes \alpha'_2) * (\gamma'_1 \otimes \gamma'_2).
\end{alignat*}
which is in \( C_3 \) due to requirement (\ref{list_interchanger_in_23_coherence_system}) in definition \ref{def_coherence_23_system}.

Now let us assume \( \alpha_1, \alpha_2 \) and \( \alpha'_1, \alpha'_2 \) are \( * \)-composable, i.e. \( s \alpha_1 = t \alpha_2 \) and \( s \alpha'_1 = t \alpha'_2. \)
Then \( \gamma'_1 \) and \( \gamma_2 \) are parallel and because of lemma \ref{lem_free_choice_of_C2_conjugators} we can assume \( \gamma'_1 = \gamma_2. \)
That \( \alpha_1 * \alpha_2 \sim_2 \alpha'_1 * \alpha'_2 \) then follows by considering
\[
\gamma_1 * \alpha_1 * \alpha_2
\xrightarrow{\Gamma_1 * 1}
\alpha'_1 * \gamma'_1 * \alpha_2
= \alpha'_1 * \gamma_2 * \alpha_2
\xrightarrow{1 * \Gamma_2}
\alpha'_1 * \alpha'_2  * \gamma'_2.
\]

\item
To show reflexivity for \( \sim_3 \)	we observe that for arbitrary \( \Lambda \in \mmmor{\cat{T}} \) the diagram
\[
\begin{tikzcd}
1 * s \Lambda
\ar[r, "1 * \Lambda"]
\ar[d, "l"l]
	& 1 * t \Lambda
	\ar[d, "l"] \\
s \Lambda
\ar[r,"\Lambda"]
\ar[d, "r"]
	& t \Lambda
	\ar[d, "r"] \\
s \Lambda * 1
\ar[r, "\Lambda * 1"]
	& t \Lambda * 1
\end{tikzcd}
\]
commutes, due to naturality of the local left and right units.

Next we show symmetry for \( \sim_3 .\) If \( \Lambda \sim_3 \Lambda' \) we have 2-morphisms \( \gamma,\gamma' \) and 3-morphisms \( \Gamma, \Gamma' \) such that the following diagram commutes:
\[
\begin{tikzcd} [column sep = large]
\gamma * s \Lambda
\ar[r, "1 * \Lambda"]
\ar[d, "\Gamma' "]
	& \gamma * t \Lambda
	\ar[d," \Gamma "] \\
s \Lambda' * \gamma'
\ar[r, " \Lambda' * 1 "]
	& t \Lambda'  * \gamma'.
\end{tikzcd}
\]

Then the exterior square in
\[
\begin{tikzcd}[column sep = large]
\gamma ^{\adsq} * s \Lambda'
\ar[r, "1 * \Lambda' "]
\ar[d,"1 * 1 * \epsilon ^{-1}" swap]
	& \gamma ^{\adsq} * t \Lambda'
	\ar[d,"1 * 1 * \epsilon ^{-1}"] \\
\gamma ^{\adsq} * s \Lambda' * \gamma' * \gamma ^{\prime \adsq}
\ar[r, "1 * \Lambda' * 1 * 1"]
\ar[d, "1 * (\Lambda') ^{-1} * 1" swap]
	& \gamma ^{\adsq} * t \Lambda' * \gamma' * \gamma ^{\prime \adsq}
	\ar[d, "1 * \Lambda ^{-1} * 1"] \\
\gamma ^{\adsq} * \gamma * s \Lambda * \gamma ^{\prime \adsq}
\ar[r, "1 * 1 * \Lambda * 1"]
\ar[d, "\eta ^{-1} * 1 * 1" swap]
	& \gamma ^{\adsq} * \gamma * t \Lambda * \gamma ^{\prime \adsq}
	\ar[d, "\eta ^{-1} * 1 * 1"] \\
s \Lambda * \gamma ^{\prime \adsq}
\ar[r, "\Lambda * 1"]
	& s \Lambda * \gamma ^{\prime \adsq}.
\end{tikzcd}
\]
commutes, because all small squares commute due to the local interchange law and the assumed commutativity of the initial diagram.
Therefore we have \( \Lambda' \sim_3 \Lambda \) and symmetry is shown.
A similar calculation shows transitivity of \( \sim_3. \)

\item
It remains to show that \( \sim_3 \) is a congruence relation.
For that suppose \( \Lambda_1 \sim \Lambda_1' \) and \( \Lambda_2 \sim \Lambda_2'. \) Then we have commutative squares
\[
\begin{tikzcd} [column sep = large]
\gamma_1 * s \Lambda_1
\ar[r, "1 * \Lambda_1"]
\ar[d, "\Gamma_1' "]
	& \gamma_1 * t \Lambda_1
	\ar[d," \Gamma_1"] \\
s \Lambda'_1 * \gamma'_1
\ar[r, " \Lambda'_1 * 1 "]
	& t \Lambda'_1  * \gamma'_1
\end{tikzcd}
\; \text{and} \;
\begin{tikzcd} [column sep = large]
\gamma_2 * s \Lambda_2
\ar[r, "1 * \Lambda_2"]
\ar[d, "\Gamma_2' "]
	& \gamma_2 * t \Lambda_2
	\ar[d," \Gamma_2"] \\
s \Lambda'_2 * \gamma'_2
\ar[r, " \Lambda'_2 * 1 "]
	& t \Lambda'_2  * \gamma'_2.
\end{tikzcd}
\]
Now suppose \( \Lambda_1, \Lambda_2 \) are \( \otimes \)-composable (this implies \( \Lambda'_1, \Lambda'_2 \) are \( \otimes \)-composable).
Then \( \Lambda_1 \otimes \Lambda_2 \sim_3 \Lambda'_1 \otimes \Lambda' _2, \) follows from the commutativity of
\[
\begin{tikzcd}[column sep = huge]
(\gamma_1 \otimes \gamma_2) * (s \Lambda_1 \otimes s \Lambda_2)
\ar[r,"1 * (\Lambda_1 \otimes \Lambda_2)"]
\ar[d, "\phi ^{\otimes}" swap]
	& (\gamma_1 \otimes \gamma_2) * (t \Lambda_1 \otimes t \Lambda_2)
	\ar[d, "\phi ^{\otimes}"] \\
(\gamma_1 * s \Lambda_1) \otimes (\gamma_2 * s \Lambda_2)
\ar[r,"(1 * \Lambda_1) \otimes (1 * \Lambda_2)"]
\ar[d,"\Gamma'_1 \otimes \Gamma'_2" swap]
	& (\gamma_1 * t \Lambda_1) \otimes (\gamma_2 * t \Lambda_2)
	\ar[d,"\Gamma_1 \otimes \Gamma_2"] \\
(\Lambda'_1 * \gamma'_1) \otimes (s \Lambda'_2 * \gamma'_2)
\ar[r,"(\Lambda_1 * 1) \otimes (\Lambda_2 * 1)"]
\ar[d, "(\phi ^{\otimes}) ^{-1}" swap]
	& (t \Lambda'_1 * \gamma'_1) \otimes (t \Lambda'_2 * \gamma'_2)
	\ar[d, "(\phi ^{\otimes}) ^{-1}"] \\
(s \Lambda'_1 \otimes s \Lambda'_2) * (\gamma'_1 \otimes \gamma'_2)
\ar[r,"(\Lambda'_1 \otimes \Lambda'_2) * 1"]
	& (t \Lambda'_1 \otimes t \Lambda'_2) * (\gamma'_1 \otimes \gamma'_2).
\end{tikzcd}
\]

Now suppose \( \Lambda_1, \Lambda_2 \) and \( \Lambda'_1, \Lambda'_2 \) are \( * \)-composable, i.e. \( s^2 \Lambda_1 = t^2 \Lambda_2 \) and \( s^2 \Lambda'_1 = t^2 \Lambda'_2. \)
Then \( \gamma'_1 \) and \( \gamma_2 \) are parallel morphisms and by lemma \ref{lem_free_choice_of_C2_conjugators} we can assume \( \gamma'_1 = \gamma_2. \)
Now \( \Lambda_1 * \Lambda_2 \sim_3 \Lambda'_1 * \Lambda'_2 \) follows from the commutativity of
\[
\begin{tikzcd}[column sep = huge]
\gamma_1 * s \Lambda_1 * s \Lambda_2
\ar[r,"1 * \Lambda_1 * \Lambda_2"]
\ar[d, "\Gamma'_1 * 1" swap]
	& \gamma_1 * t \Lambda_1 * t \Lambda_2
	\ar[d, "\Gamma_1 * 1"] \\
s \Lambda'_1 * \gamma'_1 * s \Lambda_2
\ar[r, "\Lambda'_1 * 1 * \Lambda_2"]
\ar[d, "1 * \Gamma'_2 " swap]
	& t \Lambda'_1 * \gamma'_1 * t \Lambda_2
	\ar[d, "1 * \Gamma_2"] \\
s \Lambda'_1 * s \Lambda'_2 * \gamma'_2
\ar[r, "\Lambda'_1 * \Lambda'_2 * 1"]
	& s \Lambda'_1 * s \Lambda'_2 * \gamma'_2.
\end{tikzcd}
\]

Now suppose \( \Lambda_1, \Lambda_2 \) and \( \Lambda'_1 , \Lambda'_2 \) are \( \circ \)-composable, i.e. \( s \Lambda_1 = t \Lambda_2 \) and \( s \Lambda' _1 = t \Lambda' _2. \)
Then in particular \( s^2 \Lambda_2 = t^2 \Lambda_1 \) and \( s^2 \Lambda'_2 = t^2 \Lambda'_1, \)
and because of lemma \ref{lem_free_choice_of_C2_conjugators} we can assume \( \gamma_1 = \gamma_2 \) and \( \gamma'_1 = \gamma'_2. \)
This in turn implies that \( \Gamma'_1 \) and \( \Gamma_2 \) are parallel morphisms, and therefore \( \Gamma'_1 = \Gamma_2. \)
Now \( \Lambda_1 \circ \Lambda_2 \sim_3 \Lambda'_1 \circ \Lambda'_2 \) follows from the commutativity of
\[
\begin{tikzcd} [column sep = large]
\gamma_2 * s \Lambda_2
\ar[r, "1 * \Lambda_2"]
\ar[d," \Gamma'_2"]
\ar[rrr, "1 * (\Lambda_1 \circ \Lambda_2)", bend left=20]
	& \gamma_2 * t \Lambda_2
	\ar[r,equal, "" {name=s1}]
		&[-30pt] \gamma_1 * s \Lambda_1
		\ar[r,"1 * \Lambda_1"]
			& \gamma_1 * t \Lambda_1
			\ar[d, "\Gamma_1"] \\
s \Lambda'_2 * \gamma'_2
\ar[r, " \Lambda'_2 * 1 "]
\ar[rrr, "(\Lambda'_1 \circ \Lambda'_2) * 1", bend right=20]
	& t \Lambda'_2  * \gamma'_2
	\ar[r, equal, "" {name=t1}]
		& s \Lambda'_1 * \gamma'_1
		\ar[r,"\Lambda'_1 * 1"]
			& t \Lambda'_1 * \gamma'_1.
\ar[from=s1, to=t1, "\Gamma_2 = \Gamma'_1" description, shorten <=3pt]
\end{tikzcd}
\]
\end{itemize}
\end{proof}

\begin{satz} \label{prop_coherence_23_system}
Let \( C_3 \) be a coherence-(2,3)-system of a tricategory \( \cat{T}. \)
\begin{enumerate}
\item
There is a tricategory \( \cat{T}/\scriptstyle{C _3} \), whose objects are the objects of \( \cat{T} \)
and whose 1-,2- and 3-morphisms are equivalence classes with respect to the congruence relations \( \sim_1, \sim_2 \) and \( \sim_3 \) from lemma \ref{lem_coherence_23_congruence_relation}.

We write \( [f] \) resp. \( [\alpha] \) resp. \( [\Pi] \) for the equivalence class of a 1-resp. 2- resp. 3-morphism of \( \cat{T} \)
and we may also write \( [a], \) if we consider \( a \) as an object of \( \cat{T}/\scriptstyle{C}. \)
The remaining data of \( \cat{T}/\scriptstyle{C _3} \) is given as follows:
\begin{itemize}
\item
The types of the 1-,2- and 3-morphisms are given by
\[
s[f]=sf, \;
t[f]=tf, \;
s[\alpha] = [s \alpha], \;
t[\alpha] = [t \alpha], \;
s[\Lambda] = [s \Lambda], \;
t[\Lambda] = [t \Lambda]
\]
\item
Composition along objects is given by
\[
[f] \otimes [g] = [f \otimes g], \quad [\alpha] \otimes [\beta] = [\alpha \otimes \beta], \quad \text{and} \quad [\Lambda] \otimes [\Pi] = [\Lambda \otimes \Pi].
\]
\item
Composition along 1-morphisms is given for 2-morphisms by
\[
[\alpha]*[\beta] = [\alpha' * \beta'],
\]
where \( \alpha', \beta' \in \mmor{\cat{T}} \) are \( * \)-composable with \( [\alpha] = [\alpha'] \) and \( [\beta] = [\beta'], \)
and for 3-morphisms by
\[
[\Lambda]*[\Pi] = [\Lambda' * \Pi'],
\]
where \( \Lambda', \Pi' \in \mmmor{\cat{T}} \) are \( * \)-composable with \( [\Lambda] = [\Lambda'] \) and \( [\Pi] = [\Pi']. \)

\item
Composition along 2-morphisms is given by
\[
[\Lambda] \circ [\Pi] = [\Lambda' \circ \Pi'],
\]
where \( \Lambda', \Pi' \in \mmmor{\cat{T}} \) are \( \circ \)-composable with \( [\Lambda] = [\Lambda'] \) and \( [\Pi] = [\Pi']. \)
\item
The constraints of \( \cat{T}/\scriptstyle{C _3} \) are given as follows:
\[
1 _{[a]} = [1 _{a}], 1 _{[f]} = [1 _{f}], a _{[\alpha][\beta][\gamma]} = [a _{\alpha \beta \gamma}], \rho _{[f],[g]} = [\rho _{f,g}], \dots,
\]
where for
\[
a ^{loc} _{[\alpha_1][\alpha_2][\alpha_3]} = [a ^{loc} _{\alpha_1 \alpha_2 \alpha_3}]
\quad \text{and} \quad
\phi ^{\otimes} _{([\alpha_1][\beta_1])([\alpha_2][\beta_2])} = [\phi ^{\otimes} _{(\alpha_1 \beta_1)(\alpha_2 \beta_2)}]
\]
we choose the representatives of the indices on the left hand side, such that the right hand side is defined.
\end{itemize}

\item
Mapping morphisms in \( \cat{T} \) to their respective equivalence classes in \( \cat{T}/\scriptstyle{C} \) defines a virtually strict triequivalence
\(
[-]: \cat{T} \to \cat{T}/\scriptstyle{C}.
\)
\end{enumerate}
\end{satz}

\begin{proof}
Everything is proven literally as in proposition \ref{prop_coh_3_system}, except the well-definedness of the constraints and the fact that \( [-] \) is a virtually strict triequivalence.

We start by proving that units and constraints of \( \cat{T}/\scriptstyle{C _3} \) are well-defined.
Because the objects of \( \cat{T}/\scriptstyle{C_3} \) are the objects of \( \cat{T}, \) the constraints which are indexed over objects
- namely \( 1 _{[a]}, i _{[a]} \; \text{and} \; \phi _{[a]} \) - are well-defined.
The remaining units are also well-defined:

Firstly if \( f \sim_1 f' \) there is a 2-morphism \( f \xrightarrow{\gamma} f' \) in \( C_2 \) and the 3-morphism
\[
\gamma * 1_f
\xrightarrow{r ^{loc} _{\gamma}}
\gamma
\xrightarrow{(l ^{loc} _{\gamma}) ^{-1}}
1 _{f'} * \gamma,
\]
is in \( C_3 \) by requirement \eqref{list_local_constraints_in_C_3} in definition \ref{def_coherence_23_system}.
Thus we have \( 1_f \sim_2 1 _{f'}. \)

Secondly if \( \alpha \sim_2 \alpha' \) there is a 3-morphism \( \gamma * \alpha \xrightarrow{\Gamma} \alpha' * \gamma' \) and the diagram
\[
\begin{tikzcd}
\gamma * \alpha
\ar[r, "1 * 1_\alpha"]
\ar[d, "\Gamma"]
	& \gamma * \alpha
	\ar[d, "\Gamma"] \\
\alpha' * \gamma'
\ar[r, "1 _{\alpha'} * 1"]
	& \alpha' * \gamma'
\end{tikzcd}
\]
commutes trivially. Thus we have \( 1_\alpha \sim_3 1 _{\alpha'}. \)

To show the well-definedness of \( \phi ^{\otimes} _{[f_1],[f_2]} \) assume we have \( f_i \sim_1 f'_i \) via \( \gamma_i * f_i \xrightarrow{\Gamma_i} f'_i * \gamma'_i \) for \( i \in \set{1,2}. \)
Because \( \sim_2 \) is a congruence relation and because we already have shown that \( h \sim_1 h' \) implies \( 1_h \sim_2 1 _{h'}, \)
we have \( 1 _{f_1 \otimes f_2} \sim_2 1 _{f'_1 \otimes f'_2} \) and \( 1 _{f_1} \otimes 1 _{f_2} \sim_2 1 _{f' _1} \otimes 1 _{f' _2}. \)
Therefore we can conclude together with lemma \ref{lem_free_choice_of_C2_conjugators}, that there are \( C_3 \)-morphisms
\begin{align*}
& \Gamma': (\gamma_1 \otimes \gamma_2) * 1 _{f_1 \otimes f_2} \to 1 _{f'_1 \otimes f'_2} * (\gamma'_1 \otimes \gamma'_2)
\quad \text{and} \\
& \Gamma: (\gamma_1 \otimes \gamma_2) * (1 _{f_1} \otimes 1 _{f_2}) \to (1 _{f'_1} \otimes 1 _{f'_2}) * (\gamma'_1 \otimes \gamma'_2).
\end{align*}
Now, due to the uniqueness of \( C_3 \)-morphisms, the following diagram commutes
\[
\begin{tikzcd}[column sep = large]
(\gamma_1 \otimes \gamma_2) * 1 _{f_1 \otimes f_2}
\ar[r, "1 * \phi ^{\otimes} _{(f_1,f_2)}"]
\ar[d, "\Gamma' " swap]
	& (\gamma_1 \otimes \gamma_2) * (1 _{f_1} \otimes 1 _{f_2})
	\ar[d, "\Gamma"] \\
1 _{f'_1 \otimes f'_2} * (\gamma'_1 \otimes \gamma'_2)
\ar[r, "\phi ^{\otimes} _{(f'_1,f'_2)} * 1" swap]
	& (1 _{f'_1} \otimes 1 _{f'_2}) * (\gamma'_1 \otimes \gamma'_2),
\end{tikzcd}
\]
and this shows that \( \phi ^{\otimes} _{[f_1],[f_2]} \) is well-defined.

To show the well-definedness of \( \phi _{([\alpha_1],[\alpha_2])([\alpha_3],[\alpha_4])} \)
let for \( i \in \set{1,2,3,4} \) be \( \alpha_i \sim_2 \alpha'_i \) via \( \gamma_i * \alpha_i \xrightarrow{\Gamma_i} \alpha'_i * \gamma' \) such that
\[
(\alpha_1 * \alpha_3) \otimes (\alpha_2 * \alpha_4)
\quad \text{and} \quad
(\alpha'_1 * \alpha'_3) \otimes (\alpha'_2 * \alpha'_4)
\]
exist.
Due to lemma \ref{lem_free_choice_of_C2_conjugators}, we can assume
\( \gamma'_1 = \gamma_3 =: \gamma _{1',3} \) and \( \gamma'_2 = \gamma_4 =: \gamma _{2',4}. \)
Now the exterior of the following diagram commutes, since all small squares commute either by coherence or by naturality of \( \phi ^{\otimes}. \)
\[
\begin{tikzcd}
(\gamma_1 \otimes \gamma_2) * (\alpha_1 \otimes \alpha_2) * (\alpha_3 \otimes \alpha_4)
\ar[r, "1 * \phi ^{\otimes}"]
\ar[d, "\phi ^{\otimes} * 1"]
	& (\gamma_1 \otimes \gamma_2) * \big( (\alpha_1 * \alpha_3) \otimes (\alpha_2 * \alpha_4) \big)
	\ar[d, "\phi ^{\otimes}"] \\
\big( (\gamma_1 * \alpha_1) \otimes (\gamma_2 * \alpha_2) \big) * (\alpha_3 \otimes \alpha_4)
\ar[r, "\phi ^{\otimes}"]
\ar[d, "(\Gamma_1 \otimes \Gamma_2) * 1"]
	& (\gamma_1 * \alpha_1 * \alpha_3) \otimes (\gamma_2 * \alpha_2 * \alpha_4)
	\ar[d, "(\Gamma_1 * 1) \otimes ( \Gamma_2 * 1)"] \\
\big( (\alpha'_1 * \gamma'_1) \otimes (\alpha'_2 * \gamma'_2) \big) * (\alpha_3 \otimes \alpha_4)
\ar[r, "\phi ^{\otimes}"]
\ar[d, "(\phi ^{\otimes}) ^{-1} * 1"]
	& (\alpha'_1 * \gamma'_1 * \alpha_3) \otimes (\alpha'_2 * \gamma'_2 * \alpha_4)
	\ar[dd, equal] \\
(\alpha'_1 \otimes \alpha'_2) * (\gamma'_1 \otimes \gamma'_2) * (\alpha_3 \otimes \alpha_4)
\ar[d,"1 * \phi ^{\otimes}"] \\
(\alpha'_1 \otimes \alpha'_2) * \big( (\gamma _{1',3} * \alpha_3) \otimes (\gamma _{2',4} * \alpha_4) \big)
\ar[r,"\phi ^{\otimes}"]
\ar[d,"1 * (\Gamma_3 \otimes \Gamma_4)"]
	& (\alpha'_1 * \gamma _{1',3} * \alpha_3) \otimes (\alpha'_2 * \gamma _{2',4} * \alpha_4)
	\ar[d,"(1 * \Gamma_3) \otimes (1 * \Gamma_4)"] \\
(\alpha'_1 \otimes \alpha'_2) * \big( (\alpha'_3 * \gamma'_3) \otimes (\alpha'_4 * \gamma'_4) \big)
\ar[r,"\phi ^{\otimes}"]
\ar[d,"1 * (\phi ^{\otimes}) ^{-1}"]
	& (\alpha'_1 * \alpha'_3 * \gamma'_3) \otimes (\alpha'_2 * \alpha'_4 * \gamma'_4)
	\ar[d, "(\phi ^{\otimes}) ^{-1}"] \\
(\alpha'_1 \otimes \alpha'_2) * (\alpha'_3 \otimes \alpha'_4) * (\gamma'_3 \otimes \gamma'_4)
\ar[r, "\phi ^{\otimes} * 1"]
	& \big( (\alpha'_1 * \alpha'_3) \otimes (\alpha'_2 * \alpha'_4) \big) * (\gamma'_3 \otimes \gamma'_4)
\end{tikzcd}
\]
Since all interchangers which appear as vertical arrows are contained in \( C_3 \) by requirement \ref{list_interchanger_in_23_coherence_system} of definition \ref{def_coherence_23_system},
the vertical parts of the diagram are contained in \( C_3. \)
Thus we have shown the well-definedness of \( \phi _{([\alpha_1],[\alpha_2])([\alpha_3],[\alpha_4])}. \)

To show the well-definedness of \( a ^{loc} _{[\alpha_1][\alpha_2][\alpha_3]}, \) let for \( i \in \set{1,2,3} \) be \( \alpha_i \sim_2 \alpha'_i \) via
\( \gamma_i * \alpha_i \xrightarrow{\Gamma_i} \alpha'_i * \gamma'_i, \) such that the composites
\[
\alpha_1 * \alpha_2 * \alpha_3
\quad \text{and} \quad
\alpha'_1 * \alpha'_2 * \alpha'_3
\]
exist.
Because \( \sim_2 \) is a congruence relation there are \( \gamma,\gamma' \in C_2 \) and a \( C_3 \)-morphism
\[
\Gamma: \gamma * \alpha_1 * \alpha_2 * \alpha_3
\to
\alpha'_1 * \alpha'_2 * \alpha'_3 * \gamma'.
\]
Then the diagram
\[
\begin{tikzcd} [column sep = huge]
\gamma_1 * \alpha_1 * \alpha_2 * \alpha_3
\ar[r, "1 * a ^{loc} _{\alpha_1 * \alpha_2 * \alpha_3}"]
\ar[d, "\Gamma" swap]
	& \gamma * \alpha_1 * \alpha_2 * \alpha_3
	\ar[d, "\Gamma"] \\
\alpha'_1 * \alpha'_2 * \alpha'_3 * \gamma'
\ar[r, "a ^{loc} _{\alpha'_1 * \alpha'_2 * \alpha'_3} * 1"]
	& \alpha'_1 * \alpha'_2 * \alpha'_3 * \gamma'
\end{tikzcd}
\]
commutes, because it consists of two parallel morphisms in \( C_3 \) and we have shown that \( a ^{loc} _{\alpha_1 \alpha_2 \alpha_3} \sim_3 a ^{loc} _{\alpha'_1 \alpha'_2 \alpha'_3}. \)

The well-definedness of the remaining local constraints \( l ^{loc} \) and \( r ^{loc} \) is shown analogously.
If \( \alpha \sim_2 \alpha' \) via \( \gamma' * \alpha \xrightarrow{\Gamma} \alpha' * \gamma' \) the existence of \( \Gamma': \gamma * 1 * \alpha \to 1 * \alpha' * \gamma' \) in \( C_3 \)
follows from the fact that \( \sim_2 \) is a congruence relation and from lemma \ref{lem_free_choice_of_C2_conjugators}.
Then we can construct the diagram
\[
\begin{tikzcd}
\gamma * 1 * \alpha
\ar[r, "1 * l ^{loc} _\alpha"]
\ar[d, "\Gamma' " swap]
	& \gamma * \alpha
	\ar[d, "\Gamma"] \\
1 * \alpha' * \gamma'
\ar[r, "l ^{loc} _{\alpha'} * 1"]
	& \alpha' * \gamma
\end{tikzcd}
\]
which commutes because it consists of two parallel morphisms in \( C_3 \) and \( l ^{loc} _{\alpha} \sim_2 l ^{loc} _{\alpha'} \) is shown.
The proof for \( r ^{loc} \) is analogous.

Next we show that the bottom-level components of the 2-categorical transformations \( a, l, r \) are are well-defined.
For that purpose let \( f_i \sim_1 f_i' \) via \( f_i \xrightarrow{\gamma_i} f'_i. \)
Since the 3-cell
\[
\begin{tikzcd}[column sep = large]
(f_1 \otimes f_2) \otimes f_3
\ar[r,"a _{f_1 f_2 f_3}"]
\ar[d,"(\gamma_1 \otimes \gamma_2) \otimes \gamma_3" swap]
	& f_1 \otimes (f_2 \otimes f_3)
	\ar[d,"\gamma_1 \otimes (\gamma_2 \otimes \gamma_3)"]
	\ar[dl, Rightarrow, "a _{\gamma_1 \gamma_2 \gamma_3} ^{-1}" description, short=5pt] \\
(f'_1 \otimes f'_2) \otimes f'_3
\ar[r,"a _{f'_1 f'_2 f'_3}" swap]
	& f'_1 \otimes (f'_2 \otimes f'_3) \\
\end{tikzcd}
\]
is a coherence morphism between coherence 2-morphisms it is contained in \( C_3 \) by requirement \ref{list_C_3_is_C_2_complete} of definition \ref{def_coherence_23_system}.
This proves that \( a _{f_1 f_2 f_3} \sim_2 a _{f'_1 f'_2 f'_3}. \)

Similarly if \( f \sim_1 f' \) via \( f \xrightarrow{\gamma} f' \) we have \( l _{f} \sim_2 l _{f'} \) and \( r _{f} \sim_2 r _{f'} \) because
\[
\gamma * l _{f}
\xrightarrow{l _{\gamma} ^{-1}}
l _{f'} * (i _{t f} \otimes \gamma)
\quad \text{and} \quad
\gamma * r _{f}
\xrightarrow{r _{\gamma} ^{-1}}
r _{f'} * (\gamma \otimes i _{s f})
\]
are morphism in \( C_3. \)

Now we show that the top level components of the 2-categorical transformations \( a,l,r \) are well-defined.
For that let \( \alpha_i  \sim_2 \alpha'_i \) via \( \gamma_i * \alpha_i \xrightarrow{\Gamma_i} \alpha'_i * \gamma'_i. \)
Now we consider the following diagram, where the arrows labeled with \( \sim \) are coherence 3-morphisms that are contained in \( C_3. \)
\[
\begin{tikzcd}
\big( \gamma_1 \otimes (\gamma_2 \otimes \gamma_3) \big) * a * \big( (\alpha_1 \otimes \alpha_2) \otimes \alpha_3 \big)
\ar[r, "1 * a"]
\ar[d, "\sim" {sloped, anchor=north}]
	& \big( \gamma_1 \otimes (\gamma_2 \otimes \gamma_3) \big) * \big( \alpha_1 \otimes (\alpha_2 \otimes \alpha_3) \big) * a
	\ar[d, "\sim" {sloped, anchor=south}] \\
a * \big( \big( (\gamma_1 * \alpha_1) \otimes (\gamma_2 * \alpha_2) \big) \otimes (\gamma_3 * \alpha_3) \big)
\ar[r, "a"]
\ar[d, "1 * \big( (\Gamma_1 \otimes \Gamma_2) \otimes \Gamma_3 \big)" swap]
	& \big( (\gamma_1 * \alpha_1) \otimes \big( (\gamma_2 * \alpha_2) \otimes (\gamma_3 * \alpha_3) \big) \big) * a
	\ar[d, "\big( \Gamma_1 \otimes (\Gamma_2 \otimes \Gamma_3) \big) * 1"] \\
a * \big( \big( (\alpha'_1 * \gamma'_1) \otimes (\alpha'_2 * \gamma'_2) \big) \otimes (\alpha'_3 * \gamma'_3) \big)
\ar[r, "a"]
\ar[d, "\sim" {sloped, anchor=north}]
	& \big( (\alpha'_1 * \gamma'_1) \otimes \big( (\alpha'_2 * \gamma'_2) \otimes (\alpha'_3 * \gamma'_3) \big) \big) * a
	\ar[d, "\sim" {sloped, anchor=south}] \\
a * \big( (\alpha'_1 \otimes \alpha'_2) \otimes \alpha'_3 \big) * \big( (\gamma'_1 \otimes \gamma'_2) \otimes \gamma'_3 \big)
\ar[r, "a * 1"]
	& \big( \alpha'_1 \otimes (\alpha'_2 \otimes \alpha'_3) \big) * a * \big( (\gamma'_1 \otimes \gamma'_2) \otimes \gamma'_3 \big).
\end{tikzcd}
\]
The top and bottom square commute commute because of the coherence theorem for tricategories
and the middle square commutes because of the naturality of the top level components of \( a. \)
Therefore we have shown
\(
a _{\alpha_1 \alpha_2 \alpha_3} \sim_3 a _{\alpha'_1 \alpha'_2 \alpha'_3}.
\)

Similarly \( l_\alpha \sim_3 l_\alpha' \) follows from the commutativity of the diagram
\[
\begin{tikzcd}
\gamma * l * (i \otimes \alpha)
\ar[r, "1 * l_\alpha"]
\ar[d, "\sim" {sloped, anchor=north}]
	& \gamma * \alpha * l
	\ar[d, equal] \\
l * \big( i \otimes (\gamma * \alpha) \big)
\ar[r, "l _{\gamma * \alpha}"]
\ar[d, "1 * (1 \otimes \Gamma)" swap]
	& \gamma * \alpha * l
	\ar[d, "\Gamma * 1"] \\
l * \big( i \otimes (\alpha' * \gamma') \big)
\ar[r, "l _{\alpha' * \gamma'}"]
\ar[d, "\sim" {sloped, anchor=north}]
	& \alpha' * \gamma' * l
	\ar[d, "\sim" {sloped, anchor=south}] \\
l * (i \otimes \alpha') * (i \otimes \gamma')
\ar[r, "l _{\alpha'} * 1"]
	& \alpha' * l * (i \otimes \gamma'),
\end{tikzcd}
\]
where again the arrows labeled with \( \sim \) are coherence 3-morphisms from \( C_3. \)
That \( r_\alpha \sim_3 r_\alpha' \) follows analogously.

That the components of the modifications \( \epsilon^l ,\eta^l, \epsilon^r, \eta^r, \epsilon^a, \eta^a, \pi, \lambda, \mu \) and \( \rho \) are well-defined follows directly from the coherence theorem for tricategories.
If \( f \sim_1 f' \) via \( f \xrightarrow{\gamma} f' \), we have \( \epsilon ^{l}_f \sim_3 \epsilon ^{l} _{f'} \) because the diagram
\[
\begin{tikzcd}
\gamma * l_f * l_f ^{\adsq}
\ar[r, "1 * \epsilon ^{l} _f"]
\ar[d, "\sim" {sloped, anchor=north}]
	& \gamma * 1_f
	\ar[d, "\sim" {sloped, anchor=south}] \\
l _{f'} * l _{f'} ^{\adsq} * \gamma
\ar[r, "\epsilon ^{l} _{f'} * 1"]
	& 1 _{f'} * \gamma
\end{tikzcd}
\]
commutes by coherence.
Again the arrow labeled with \( \sim \) is a coherence 3-morphism and moreover contained in \( C_3. \)

The proof that the components of \( \eta^l, \epsilon^r, \eta^r, \epsilon^a, \eta^a, \pi, \lambda, \mu \) and \( \rho \) are well-defined is analogous.

Now we show that \( [-] \) is a virtually strict triequivalence by checking that it satisfies the conditions in remark \ref{rem_characterization_virtually_strict_triequivalence}.
It is trivial that \( [-] \) is surjective and locally surjective.

For 2-local surjectivity we have to show that for any \( [\alpha']: [f] \to [g] \) there is a 2-morphism \( \alpha: f \to g \) with \( \alpha \sim_2 \alpha'. \)
Suppose \( \alpha': f' \to g' \) is a representative of \( [\alpha']. \) That means there are coherence 2-morphisms \( \gamma_f: f \to f' \) and \( \gamma_g: g \to g'. \)
We claim that \( \alpha := \gamma_g  ^{\adsq} * \alpha' * \gamma_f \) is the desired morphism.
To show that, we consider the following diagram, in which the 3-cell in the left triangle and the 3-cell in the middle square are given by local constraints
and are therefore contained in \( C_3 \) by requirement \eqref{list_local_constraints_in_C_3} in definition \ref{def_coherence_23_system}.
The 3-cell in the right triangle is the unique coherence 3-morphism \( \gamma_g * \gamma ^{\adsq}_g \to 1'_g \) (note that \( T \) is semi-free)
and therefore contained in \( C_3 \) by requirement \eqref{list_C_3_is_C_2_complete} in definition \ref{def_coherence_23_system}.
Therefore the whole diagram is a 3-cell contained in \( C_3 \) and we have shown \( \alpha \sim_2 \alpha'. \)
\[
\begin{tikzcd}
f
\ar[r, "\gamma_f"]
\ar[d, "\gamma_f" {swap, name=t1}]
	& f'
	\ar[r, "\alpha' " {name=s2}]
	\ar[ld, "1_f' " {pos=.35}]
		&g'
		\ar[r, "\gamma_g ^{\adsq}"]
		\ar[dr, "1_g' " {swap, name=t3}]
			& g
			\ar[d, "\gamma_g"] \\
f'
\ar[rrr, "\alpha' " {swap, name=t2}]
			& & & g'.
\ar[from=ull, to=t1, Rightarrow, short=4mm]
\ar[from=s2, to=t2, Rightarrow, short=5mm]
\ar[from=u, to=t3, Rightarrow, short=2mm]
\end{tikzcd}
\]

Before we show 3-local bijectivity, we observe that for \( \Lambda,\Lambda' \in \mmmor{\cat{T}} \) the existence of \( \gamma, \gamma' \in C_2 \) and \( \Gamma,\Gamma' \in C_3 \) such that
\( \Gamma \circ (1 _{\gamma} * \Lambda * 1 _{\gamma'}) \circ \Gamma' = \Lambda' \) implies \( \Lambda \sim \Lambda'. \)
This follows from the commutativity of the diagram
\[
\begin{tikzcd} [column sep = large]
\gamma * s \Lambda
\ar[r, "1 * \Lambda"]
\ar[d, "1 * 1 * \epsilon ^{-1}"]
	& \gamma * t \Lambda
	\ar[d,"1 * 1 * \epsilon ^{-1}"] \\
\gamma * s \Lambda * \gamma' * \gamma ^{\prime \adsq}
\ar[r, "1 * \Lambda * 1 * 1"]
\ar[d, "(\Gamma') ^{-1} * 1"]
	& \gamma * t \Lambda * \gamma' * \gamma ^{\prime \adsq}
	\ar[d, "\Gamma * 1"] \\
s \Lambda' * \gamma ^{\prime \adsq}
\ar[r, "\Lambda' * 1"]
	& t \Lambda' * \gamma ^{\prime \adsq}.
\end{tikzcd}
\]

For 3-local surjectivity we have to show that for parallel \( \alpha, \beta \in \mmor{\cat{T}} \) and any \( [\Lambda']: [\alpha] \to [\beta], \)
there is a 3-morphism \( \Lambda: \alpha \to \beta \) such that \( [\Lambda]=[\Lambda']. \)
Suppose \( \Lambda': \alpha' \to \beta' \) is a representative of \( [\Lambda']. \)
Then there are \( \gamma_\alpha, \gamma'_\alpha, \gamma_\beta, \gamma'_\beta \in C_2 \) and \( \Gamma_\alpha, \Gamma_\beta \in C_3 \)
with \( \Gamma_\alpha: \gamma_\alpha * \alpha \to \alpha' * \gamma'_\alpha \) and \( \Gamma_\beta: \gamma_\beta * \beta \to \beta' * \gamma'_\beta. \)
Since \( \gamma_\alpha, \gamma_\beta \) and \( \gamma'_\alpha, \gamma'_\beta \) are parallel,
we can assume by lemma \ref{lem_free_choice_of_C2_conjugators} \( \gamma_\alpha = \gamma_\beta := \gamma \) and \( \gamma'_\alpha = \gamma'_\beta := \gamma'. \)
Then the morphism
\begin{align*}
& \alpha \xrightarrow{\eta * 1} \gamma ^{\adsq} * \gamma * \alpha
\xrightarrow{1 * \Gamma_\alpha} \gamma ^{\adsq} * \alpha' * \gamma' \\
\xrightarrow{1 * \Lambda' * 1} & \gamma ^{\adsq} * \beta' * \gamma'
\xrightarrow{1 * \Gamma_\beta ^{-1}} \gamma ^{\adsq} * \gamma * \beta
\xrightarrow{ \eta ^{-1} * 1} \beta
\end{align*}
is of the form \( \Gamma \circ (1 _{\gamma ^{\adsq}} * \Lambda' * 1 _{\gamma'}) \circ \Gamma', \)
with \( \Gamma = (\eta ^{-1} * 1) \circ (1 * \Lambda ^{-1} _{\beta}) \) and \( \Gamma' = (1 * \Lambda _{\alpha}) \circ (\eta * 1), \)
and therefore \( \sim_3 \)-related to \( \Lambda'. \)
That means we have found the desired morphism \( \Lambda. \)

For showing 3-local injectivity suppose we have parallel 3-morphisms \( \Lambda, \Lambda': \alpha \to \beta \) with \( [\Lambda]=[\Lambda']. \)
Then we consider the following diagram:
\[
\begin{tikzcd}
1 * s \Lambda
\ar[r, "1 * \Lambda"]
\ar[d, "l ^{loc}" swap]
	& 1 * t \Lambda
	\ar[d, "l ^{loc}"] \\
s \Lambda
\ar[r, "\Lambda"]
\ar[d, equal]
	& t \Lambda
	\ar[d, equal] \\
s \Lambda'
\ar[r, "\Lambda' "]
\ar[d, "(r ^{loc})^{-1}" swap]
	& t \Lambda'
	\ar[d, "(r ^{loc})^{-1}"] \\
s \Lambda' * 1
\ar[r, "\Lambda' * 1"]
	& t \Lambda' * 1
\end{tikzcd}
\]
By naturality of \( l ^{loc}, r ^{loc} \) the top and bottom square commute and because of lemma \ref{lem_free_choice_of_C2_conjugators} the exterior diagram commutes.
Therefore the middle square commutes and we have shown \( \Lambda = \Lambda'. \)

\end{proof}

\begin{satz} \label{prop_T_mod_C23}
Let \( C_3 \) be a coherence-(2,3)-system of a semi-free tricategory \( \cat{T}. \)
Then for any \( \gamma \in C_2 \) and \( \Gamma \in C_3 \) we have
\[
[\gamma] = 1 _{s[\gamma]} = 1_{t[\gamma]}
\quad \text{and} \quad
[\Gamma] = 1 _{s[\Gamma]} = 1 _{t[\Gamma]}.
\]
\end{satz}

\begin{proof}
Since
\(
1 _{s \gamma} * 1 _{s \gamma}
\xrightarrow{l ^{loc}}
1 _{s \gamma}
\xrightarrow{\eta}
\gamma * \gamma ^{\adsq}
\)
is contained in \( C_3 \) it follows that \( 1 _{s \gamma} \sim_2 \gamma \) and therefore \( 1 _{s[\gamma]} = 1 _{[s \gamma]}= [1 _{s \gamma}] = [\gamma]. \)

With lemma \ref{lem_optionality_C2_conjugation} and the identity \( 1 _{s \Gamma} \circ 1 _{s \Gamma} = \Gamma ^{-1} \circ \Gamma \) it follows that
\( 1 _{s \Gamma} \sim_3 \Gamma \) and therefore \( 1 _{s[\Gamma]} = 1 _{[s \Gamma]}= [1 _{s \Gamma}] = [\Gamma]. \)
\end{proof}

\begin{kor}
The free tricategory \( FX \) over a category enriched 2-graph \( X \) is triequivalent to a strict tricategory via a virtually strict triequivalence.
\end{kor}

\begin{proof}
According to corollary \ref{cor_FX_is_semi_free} \( FX \) is a semi-free tricategory.
Furthermore the collection of all coherence 3-morphisms in \( FX \) is trivially a coherence-(2,3)-system \( C_3. \)
Then it follows from proposition \ref{prop_T_mod_C23} that \( (FX)/C_3 \) is a strict tricategory and \( [-] \) is the desired virtually strict triequivalence.
\end{proof}

\subsection{Strictification for the tricategory of formal composites} \label{sec_strictification_for_tricats}

The overall strategy in this section is analogous to section \ref{sec_strictification_for_bicats}:
First we will define for an arbitrary tricategory \( \cat{T} \) the tricategory of "formal composites" \( \that \) together with a virtually strict triequivalence \( \ev: \that \to \cat{T}, \) given by evaluation.
Then we strictify the tricategory \( \that \) along a sequence of virtually strict triequivalences, until we finally arrive at a Gray category \( \tst. \)
At the end of the section we explain how the strict triequivalences \( \that \to \cat{T} \) and \( \that \to \tst \) can be used to reduce calculations in tricategories to much simpler calculations in Gray categories.

\begin{satz}
Let \( \cat{T} \) be a tricategory.
Then the following defines a tricategory \( \that \) and a virtually strict triequivalence \( \ev: \that \to \cat{T}: \)
\begin{itemize}
\item
The objects of \( \that \) are the same as the objects of \( \cat{T}, \) and \( \ev \) acts on objects as the identity.
\item
The 1-morphisms of \( \that \) are built inductively as follows:

Any 1-morphism in \( \cat{T} \) is a 1-morphisms in \( \that \) as well (with the same types), these 1-morphisms of \( \that \) are also called \( \that \)-basic 1-morphisms.

For any two 1-morphisms in \( \that \) (not necessarily \( \that \)-basic ones) \( \hat f: b \to c \) and \( \hat g: a \to b, \)
we require the formal composite \( \hat f \hs \hat g: a \to c \) to be a 1-morphism of \( \that. \)

Composition of 1-morphisms (along objects) is given by the \( \hot \)-operator.

\item
The action of \(\ev \) on 1-morphisms is defined via induction:

On basic 1-morphisms \( \ev \) is the identity, and \( \ev (\hat f \hot \hat g) := \ev(\hat f) \otimes \ev(\hat g). \)
\item
Now we can inductively define \( \mmor{\that}: \)

The \( \that \)-basic 2-morphisms of \( \that \) are triples
\[ (\alpha,\hat f,\hat g): \hat f \to \hat g, \quad \text{where } \hat f, \hat g \in \mor{ \that} \text{ and } \alpha: \ev(\hat f) \to \ev(\hat g). \]

For 2-morphisms
\[
\hat \beta: \hat g_1 \to \hat g_2: b \to c, \quad \hat \alpha_1: \hat f_1 \to \hat f_2: a \to b \quad \text{and} \quad \hat \alpha_2: \hat f_2 \to \hat f_3: a \to b
\]
we require the formal composites
\[
\hat \beta \hot \hat \alpha_1: \hat g_1 \hot \hat f_1 \to \hat  g_2 \hot \hat  f_2 : a \to c,
\quad \text{and} \quad \hat \alpha_2 \hs \hat \alpha_1: f_1 \to f_3: a \to b.
\]
to be 2-morphisms of \( \that. \)

Composition of 2-morphisms along objects resp. 1-morphisms is given by the \( \hot \)- resp. the \( \hs \)-operator.

\item
On 2-morphisms \( \ev \) is inductively defined as follows:
\begin{gather*}
\ev(\alpha, \hat f, \hat g) = \alpha, \\
\ev(\hat \alpha \hot \hat \beta) = \ev(\hat \alpha) \otimes \ev(\hat \beta) \quad \text{and} \quad
\ev(\hat \alpha \hs \hat \beta) = \ev(\hat \alpha) * \ev(\hat \beta).
\end{gather*}

\item
The 3-morphisms of \( \that \) are triples \( (\Pi, \hat \alpha, \hat \beta) : \hat \alpha \to \hat \beta, \)
where \( \Pi: \ev(\hat \alpha) \to \ev(\hat \beta) \) is a 3-morphism in \( \cat{T}. \)

Composition of 3-morphisms in \( \that \) is inherited from composition in \( \cat{T}: \)
\begin{align*}
(\Pi, \hat \alpha_1, \hat \alpha_2) \hot (\Lambda, \hat \beta_1, \hat \beta_2) & := (\Pi \otimes \Lambda, \hat \alpha_1 \hot \hat \beta_1, \hat \alpha_2 \hot \hat \beta_2), \\
(\Pi, \hat \alpha_1, \hat \alpha_2) \hs (\Lambda, \hat \beta_1, \hat \beta_2) & := (\Pi * \Lambda, \hat \alpha_1 \hs \hat \beta_1, \hat \alpha_1 \hs \hat \beta_2,) \\
(\Pi, \hat \alpha_2, \hat \alpha_3) \hc (\Lambda, \hat \alpha_1, \hat \alpha_2) & := (\Pi \circ \Lambda, \hat \alpha_1, \hat \alpha_3).
\end{align*}

\item
\( \ev \) acts on 3-morphisms via \( \ev(\Pi, \hat \alpha, \hat \beta) = \Pi. \)

\item
The unit- and constraint-cells of \( \that \) are essentially constraint cells for \( \cat{T}, \) only that their types are now considered as formal composites.
The units are given by
\[
\hat 1_a = 1_a, \quad \hat 1 _{\hat f} = (1 _{\ev(\hat f)}, \hat f, \hat f) \quad \text{and} \quad 1 _{\hat \alpha} =
(1 _{\ev(\hat \alpha)}, \hat \alpha, \hat \alpha)
\]
The constraints are given by
\begin{align*}
& \hat a _{\hat f \hat g \hat h}= \big( a _{\ev(f) \ev(g) \ev(h)}, (\hat f \hot \hat g) \hot \hat h, \hat f \hot (\hat g \hot \hat h) \big)
\\
& \hat \alpha _{\hat \alpha, \hat \beta, \hat \gamma} =
\big( a _{\ev (\hat \alpha) \ev (\hat \beta) \ev (\hat \gamma)}, \: \;
a _{{t \hat \alpha, t \hat \beta, t \hat \gamma}} * ((\hat \alpha \hot \hat \beta) \hot \hat \gamma), \: \;
(\hat \alpha \hot (\hat \beta \hot \hat \gamma)) * a _{{s \hat \alpha, s \hat \beta, s \hat \gamma}} \big)
\end{align*}
and analogously for the remaining constraint cells (see the data from definition \ref{def_tricat_as_magmoid} for a list of all constraint cells).
\end{itemize}

\end{satz}

\begin{proof}
First we observe that \( \that \) is a 3-magmoid and \( \ev: \that \to \cat{T} \) is a morphism of 3-magmoids.
Furthermore by definition \( \ev \) strictly preserves units and constraint cells.

Next we show, that the axioms in the description of a tricategory as a 3-magmoid (see definition \ref{def_tricat_as_magmoid}) are satisfied for \( \that. \)
Because \( \ev \) preserves composition and constraint cells strictly, any diagram which expresses an axiom of definition \ref{def_tricat_as_magmoid} is mapped to the respective diagram for \( \cat{T}. \)
This already shows the claim, because for parallel \( \hat \Pi, \hat \Lambda \in \that \) we have \( \hat \Pi = \hat \Lambda \) if and only if \( \ev(\hat \Pi) = \ev(\hat \Lambda). \)

By now we established that \( \that \) is a 3-category and \( \ev: \that \to \cat{T} \) a virtually strict functor.
According to remark \ref{rem_characterization_virtually_strict_triequivalence} \( \ev \) is even a virtually strict triequivalence,
since it is surjective, 1-locally surjective, 2-locally surjective and 3-locally a bijection.
\end{proof}

It is often useful to describe a formal binary composite as a string of morphisms \( (f_i)_{i \in \set{1 \dots n}} \) together with a chosen bracketing.
To formalize this description we need the following definition.

\begin{mydef} \label{def_bracketing}
Let \( [n] \) denote the graph
\[
\langle 0 \rangle \xrightarrow{1} \langle 1 \rangle \xrightarrow{2} \dots \xrightarrow{n} \langle n \rangle.
\]
and let \( F_m [n] \) denote the free 1-magmoid over this graph.
\begin{enumerate}
\item
A morphism \( \sigma: \langle 0 \rangle \to \langle n \rangle \) in \( F_m [n] \) is called a \emph{bracketing} of length n.
\item \label{list_bracketed_sequence}
Let \( \sigma \) be a bracketing of length n and \( (f_i) _{i \in \set{1 \dots n}} \) a sequence of composable morphism in a magmoid \( \cat{M}. \)
We write
\[
(f_i) _{i \in \set{1 \dots n}} ^{\sigma}
\]
for the image of \( \sigma \) under the unique morphism of magmoids \( F_m [n] \to \cat{M}, \) which sends \( i \) to \( f_i. \)
\end{enumerate}
\end{mydef}

\begin{notation} \label{not_comseq}
We use the notation of definition \ref{def_bracketing} (\ref{list_bracketed_sequence}) for sequences of 1- and 2-morphisms.
For the latter we clarify, whether \( (\alpha_i) _{i \in \set{1 \dots n}} \) (often written shortly as \( (\alpha_i)_n \)) is regarded as a sequence of 2-morphisms
in the 1-magmoid with operation \( \otimes \) or the 1-magmoid with operation \( *, \) by either writing
\[
\comseq{\otimes}{(\alpha_i)} _{i \in \set{1 \dots n}} ^{\sigma}
\quad \text{or} \quad
\comseq{*}{(\alpha_i)} _{i \in \set{1 \dots n}} ^{\sigma}.
\]
\end{notation}

\begin{bem} \label{rem_non_lifting_coherence_3_morphisms_in_that}
Parallel coherence 2-morphism in the bicategory \( \bhat \) from proposition \ref{prop_bhat_is_bicat}) lift to parallel coherence 2-morphisms in \( \oper{F} \bhat \)
(see the proof of lemma \ref{lem_unique_coherence_in_bhat}).
The same holds for parallel coherence 2-morphism in \( \that \) as we will see in the proof of proposition \ref{prop_that_semifree}.
However not every coherence 3-morphism in \( \that \) lifts to a coherence 3-morphism in \( \oper{F} \that. \)
Examples (see \ref{ex_non_lifting_coherence}) of such non-lifting coherence 3-morphisms arise from two different phenomena:
\begin{enumerate}
\item \label{list_that_FT_have_different_basics_rem}
If the types of an \( \oper{F} \)-basic 2-morphism \( \mu \in \mmor{\oper{F} \that}\) are not \( \oper{F} \)-basic, it follows that \( \mu \) is a coherence 2-morphism.
But if the types of a \( \that \)-basic 2-morphism \( \hat \alpha \in \mmor{ \that} \) are not \( \that \)-basic, \( \hat \alpha \) need not be a coherence 2-morphism.
That means, for \( f,g,h \in \mor{\cat{T}}, \) \( \that \) contains basic 2-morphisms of the form \( f \hot g \to h, \) but \( \oper{F} \that \) does not contain basic 2-morphisms of the form \( f \fotimes g \to h. \)

\item \label{1_unit_lifts_non_unique_to_FT_rem}
A unit \( \hat 1_a \in \mor{ \that} \) has two different lifts, namely the constraint cell \(1 ^{\sms{F}} _a \in \mor{\oper{F} \that} \) and the \( \oper{F} \)-basic 1-morphism \( \hat 1_a. \)
\end{enumerate}
\end{bem}

Now we construct examples of non-lifting coherence 3-morphisms, which arise from the two phenomena described in remark \ref{rem_non_lifting_coherence_3_morphisms_in_that}.

\begin{bei} [examples for remark \ref{rem_non_lifting_coherence_3_morphisms_in_that}]~ \label{ex_non_lifting_coherence}
\begin{enumerate}
\item \label{list_that_FT_have_different_basics_ex}
Let \( \hat \alpha_1, \hat \alpha_2, \hat \alpha_3, \hat \beta \) be \( \that \)-basic 2-morphisms in \( \that, \) such that the following composite exists:
\[
\hat \beta \hs \big( \hat a _{t \hat \alpha_1, t \hat \alpha_2, t \hat \alpha_3} \hs ((\hat \alpha_1 \hot \hat \alpha_2) \hot \hat \alpha_3)) \big)
\xrightarrow{\hat 1 \hs \hat a _{\hat \alpha_1 \hat \alpha_2 \hat \alpha_3}}
\hat \beta \hs \big( (\hat \alpha_1 \hot (\hat \alpha_2 \hot \hat \alpha_3)) \hs \hat a _{s \hat \alpha_1, s \hat \alpha_2, s \hat \alpha_3} \big).
\]
Furthermore let us assume \( \hat \beta \) is not a constraint-cell.
A lift \( \mu \) of \( \hat a _{\hat \alpha_1 \hat \alpha_2 \hat \alpha_3} \) to a coherence 3-morphism in \( \oper{F} \that \) needs to have target 1-morphism \( t \hat \alpha_1 \fotimes (t \hat \alpha_2 \fotimes t \hat \alpha_3), \)
but a lift \( \nu \) of \( \hat \beta \) will have source 1-morphism \( s \hat \beta = t \hat \alpha_1 \hot (t \hat \alpha_2 \hot t \hat \alpha_3). \)
Therefore \( \nu \) and \( \mu \) are not \( \fstar \)- composable in \( \oper{F} \that \) and \( \hat 1 _{\hat \beta} \hs \hat a  _{\hat \alpha_1 \hat \alpha_2 \hat \alpha_3} \) does not lift to a coherence 3-morphism in \( \oper{F} \that. \)

\item
Consider a 2-morphism \( \hat \alpha: \hat f \to \hat 1_a \in \that \) with \( \hat \alpha \neq  \hat 1 _{(\hat 1_a)}, \hat \alpha \neq \hat i_a \)
and another 2-morphism \( \hat \beta: \hat g \to \hat g' \) in \( \that \) such that the following composite exists.
\[
\big( \hat l _{\hat g'} \hs ( \hat i _a \hot \hat \beta) \big) \hs (\hat \alpha \hot \hat 1 _{g})
\xrightarrow{\hat l _{\beta} \hs 1}
(\hat \beta \hs \hat l _{\hat g}) \hs (\hat \alpha \hot \hat 1 _{g}).
\]
Similarly to \eqref{list_that_FT_have_different_basics_ex} one can show that lifts of \( \hat l_\beta \) and \( 1 _{\hat \alpha \hot \hat 1 _{g}} \) to coherence 3-morphisms in \( \oper{F} \that \) are not \( * \)-composable.
Therefore there exists no lift of \( \hat l _{\beta} \hs 1 _{\hat \alpha \hot \hat 1 _{g}} \) to a coherence 3-morphism in \( \oper{F} \that. \)
\end{enumerate}
\end{bei}

Because of the previous remark it is clear that the coherence theorem for tricategories does not imply that parallel constraint 3-morphisms in \( \that \) are equal.
But one can show that at least parallel constraint 3-morphisms between coherence 2-morphisms are equal:

\begin{satz} \label{prop_that_semifree}
\( \that \) is a semi-free tricategory.
\end{satz}

\begin{proof}
For a \( \that \)-basic 1-morphism \( f \) we define \( f ^{\oper{f}} \) as the following 1-morphism in \( \oper{F} \that: \)
\[
f ^{\oper{f}} :=
\begin{cases}
1_a ^{\sms{F}}\text{ (the unit at \( a \) in \( \oper{F} \that \)),}
	& \text{if } f = 1_a \\
f \text{ (as a morphism in \( \oper{F} \that \)),}
	& \text{if } f \text{ is not a unit}.
\end{cases}
\]
For a general 1-morphism \( \hat f = \comseq{\hot}{(f_i)} _{i \in 1 \dots n} ^{\sigma} \in \mor{ \that}, \) we define
\[
\hat f ^{\oper{f}} := \comseq{\fotimes}{(f_i ^{\oper{f}})} _{i \in 1 \dots n} ^{\sigma} \in \mor{\oper{F} \that},
\]
where \( \fotimes \) denotes composition along objects in the free tricategory \( \oper{F} \that. \)

If \( \hat \gamma: \hat f \to \hat g \) in \( \that \) is a coherence 2-morphism, it is either a constraint 2-cell or a unit,
or there exist unique 2-morphisms \( \gamma_1 \) and \( \gamma_2 \) and a uniquely determined composition symbol (\( \otimes \) or *),
such that \( \gamma \) is the composite of \( \gamma_1 \) and \( \gamma_2 \) with respect to that composition symbol.
Furthermore \( \gamma_1 \) and \( \gamma_2 \) are necessarily coherence 2-morphisms.
Therefore we can define a coherence 2-morphism \( \hat \gamma ^{\oper{f}}: \hat f ^{\oper{f}} \to \hat g ^{\oper{f}} \) in \( \oper{F} \that \) inductively as follows:
If \( \hat \gamma \) is a constraint cell or a unit, \( \hat \gamma ^{\oper{f}} \) is the unique coherent lift of \( \hat \gamma \) with \( s \hat \gamma ^{\oper{f}} = \hat f ^{\oper{f}}. \)
Otherwise \( \hat \gamma \) is either of the form \( \hat \gamma_1 \hot \hat \gamma_2 \) or \( \hat \gamma_1 \hs \hat \gamma_2, \) with coherence 2-morphisms \( \hat \gamma_1 \) and \( \hat \gamma_2. \)
Then we can define \( \hat \gamma:= \hat \gamma_1 ^{\oper{f}} \fotimes \hat \gamma_2 ^{\oper{f}} \) or \( \hat \gamma := \hat \gamma_1 ^{\oper{f}} \fstar \hat \gamma_2 ^{\oper{f}}, \)
since the respective composites exist in \( \oper{F} \that. \)

Now we claim that for coherence 2-morphisms \( \hat \gamma, \hat \gamma' \in \mmor{ \that}, \) a coherence 3-morphism \( \hat \Gamma: \hat \gamma \to \hat \gamma' \)
lifts to the unique coherence 3-morphisms \( \hat \gamma ^{\oper{f}} \to \hat \gamma ^{\prime \oper{f}} \) in \( \oper{F} \that. \)
Since any coherence 3-morphism \( \hat \Gamma \) is a \( (\hot, \hs, \hc) \)-composite of constraint 3-cells and units, existence can be shown via induction:
It is easy to check that if \( \hat \Gamma \) is a single constraint cell a lift with the required types does exist.
Otherwise \( \hat \Gamma \) is of the form \( \hat \Gamma_1 \hot \hat \Gamma_2, \) resp. \( \hat \Gamma_1 \hs \hat \Gamma_2, \) resp. \( \hat \Gamma_1 \hc \hat \Gamma_2, \) with coherence 2-morphisms \( \hat \Gamma_1 \) and \( \hat \Gamma_2. \)
By the induction hypothesis we can assume \( \hat \Gamma_1 ^{\oper{f}}, \hat \Gamma_2 ^{\oper{f}} \) are the required lifts of \( \hat \Gamma_1, \hat \Gamma_2. \)
Moreover it is easy to check that \( \hat \Gamma_1 ^{\oper{f}}, \hat \Gamma_2 ^{\oper{f}} \) are \( \fotimes \)- resp. \( \fstar \)- resp. \( \fcirc \)-composable if
\( \hat \Gamma_1, \hat \Gamma_2 \) are \( \hot \)- resp. \( \hs \)- resp. \( \hc \)-composable.
Therefore \( \hat \Gamma_1 ^{\oper{f}} \fotimes \hat \Gamma_2 ^{\oper{f}}, \) resp. \( \hat \Gamma_1 ^{\oper{f}} \fstar \hat \Gamma_2 ^{\oper{f}}, \)
resp. \( \hat \Gamma_1 ^{\oper{f}} \fcirc \hat \Gamma_2 ^{\oper{f}} \) is the desired lift of \( \hat \Gamma. \)

Now if \( \hat \Gamma, \hat \Gamma': \hat \gamma \to \hat \gamma' \) are parallel coherence-3-morphisms between coherence-2-morphisms \( \hat \gamma, \hat \gamma', \)
the lifts \( \hat \Gamma ^{\oper{f}} \) and \( \hat \Gamma ^{\prime \oper{f}} \) are parallel coherence 3-morphisms and therefore equal.
Thus \( \hat \Gamma \) and \( \hat \Gamma' \) are equal.
\end{proof}

In the following we will strictify \( \that \) along a sequence
\[
\that \xrightarrow{\ftilde} \ttil \xrightarrow{\barf} \tbar \xrightarrow{[-]} \tcu \xrightarrow{[-]} \tprime \xrightarrow{[-]} \tst
\]
of virtually strict triequivalences, until we finally arrive at a Gray-category \( \tst. \)
The tricategories \( \ttil \) and \( \tbar \) are constructed on sub-globular sets of \( \that. \)
More precisely we have
\[
\tbar \subseteq \ttil \subseteq \that \quad \text{as globular sets}
\]
and even
\[
\ttil \subseteq \that \quad \text{as 3-magmoids}.
\]
The remaining tricategories are then obtained by "quotienting out coherence" by means of proposition \ref{prop_coh_3_system} and \ref{prop_coherence_23_system}.

The first step of our strictification maps \( \that \) to a slightly simpler tricategory \( \ttil, \) whose underlying 3-magmoid is a sub-3-magmoid of \( \that. \)
Basically this is achieved by sending units on non-\( \that \)-basic 1-morphisms to \( \otimes \)-composites of units on \( \that \)-basic 1-morphisms
(e.g. \( \hat 1 _{(f \hot g) \hot h} \mapsto (\hat 1_f \hot \hat 1_g) \hot \hat 1_h \)) and by sending cells \( i_a \) to units \( 1 _{(1_a)}. \)
Therefore we define \( \ttil \) as the following sub-magmoid of \( \that: \)

\begin{mydef}
For any tricategory \( \cat{T} \) we define a 3-magmoid \( \ttil, \) as the following sub-3-magmoid of \( \that: \)
\begin{itemize}
\item
\( \ob{\that} = \ob{\ttil} \) and \( \mor{\that} = \mor{\ttil}. \)

\item
A \( \that \)-basic 2-morphism is called \( \ttil \)-basic, if it is not of the form \( \hat 1 _{\hat f} \) or \( \hat i_a, \)
where \( \hat f \) is a non-basic 1-morphism of \( \that \) and \( a \) an arbitrary object of \( \that. \)

The 2-morphisms of \( \ttil \) are the smallest collection of 2-morphisms in \( \that \) which contains all \( \ttil \)-basic 2-morphisms and is closed under \( \hot \) and \( \hs. \)

\item
The 3-morphisms of \( \ttil \) are the 3-morphisms of \( \that \) whose types are in \( \mmor{\ttil}. \)
\end{itemize}
\end{mydef}

\begin{satz}
There is a morphism of magmoids \( \ftilde: \that \to \ttil \) defined as follows:
\begin{itemize}
\item
On objects and 1-morphisms \( \ftilde \) is the identity.
\item
On 2-morphisms \( \ftilde \) is defined inductively. On \( \that \)-basic 2-morphisms we define
\[
\ftilde (\alpha, \hat f, \hat g) :=
\begin{cases}
\hat 1 _{ \hat f} & \text{if } \hat f = \hat g = \hat 1_a \text{ and } \alpha = i_a \\
\comseq{\otimes}{(\hat 1 _{f_i})} _{n} ^{\sigma} & \text{if } \hat f = (f _i) _n ^{\sigma} \text{ and } \alpha = 1 _{\ev \hat f} \\
(\alpha, \hat f, \hat g) & \text{else}.
\end{cases}
\]
We extend \( \ftilde \) on general 2-morphisms via the recursion rules
\begin{align*}
\ftilde (\hat \alpha \hot \hat \beta) & = \ftilde (\hat \alpha) \hot \ftilde (\hat \beta) \\
\ftilde (\hat \alpha \hs \hat \beta) & = \ftilde (\hat \alpha) \hs \ftilde (\hat \beta).
\end{align*}

\item
On 3-morphisms \( \ftilde \) is defined as
\[
\ftilde \hat \Pi := \hat \Xi _{t \hat \Pi} \hc \hat \Pi \hc \hat \Xi ^{-1} _{\hat s \hat \Pi},
\]
where
\[
\hat \Xi _{\hat \alpha}: \hat \alpha \to \ftilde{\hat \alpha}
\]
is inductively defined as follows:
\[
\hat \Xi _{(\alpha, \hat f, \hat g)} :=
\begin{cases}
\hat \phi_a ^{-1} & \text{if } \hat f = \hat g = \hat 1_a \text{ and } \alpha = i_a \\
\text{unique coherence } \hat \Gamma: \hat 1 _{\hat f} \xrightarrow{\sim} \ftilde ( \hat 1 _{\hat f}) & \text{if } (\alpha, \hat f, \hat g) = \hat 1 _{\hat f} \\
1 _{(\alpha, \hat f, \hat g)} & \text{else}
\end{cases}
\]
\[
\hat \Xi_{\hat \alpha \hot \hat \beta} := \hat \Xi _{\hat \alpha} \hot \hat \Xi _{\hat \beta}
\qquad \text{and} \qquad
\hat \Xi_{\hat \alpha \hs \hat \beta} := \hat \Xi _{ \hat \alpha} \hs \hat \Xi _{\hat \beta}
\]
\end{itemize}
\end{satz}

\begin{proof}
\begin{itemize}[-]
\item
\( \ftilde \) is clearly a morphism of globular sets.
\item
\( \ftilde \) preserves composition of 1-morphisms since we have by definition
\(
\ftilde (\hat f \hot \hat g) = \hat f \hot \hat g = \ftilde \hat f \hot \ftilde \hat g.
\)
\item
\( \ftilde \) preserves composition of 2-morphisms along objects and along 1-morphisms, which follows directly from the inductive definition of \( \ftilde. \)
\item
That \( \ftilde \) preserves composition of 3-morphisms along objects follows from the inductive definition of \( \hat \Xi \) together with the functoriality of \( \otimes. \)
For a detailed calculation compare with the proof of proposition \ref{prop_bar_is_magmoid_morph}, where an analogous (slightly more complicated) situation appears.
\item
That \( \ftilde \) preserves composition of 3-morphisms along 1-morphisms follows from the inductive definition of \( \hat \Xi \) together with the interchange laws for the local bicategories.
(For details compare with the proof of proposition \ref{prop_bar_is_magmoid_morph}.)
\item
That \( \ftilde \) preserves composition of 3-morphisms along 2-morphisms follows from the invertibility of \( \hat \Xi. \)
(For details compare with the proof of proposition \ref{prop_bar_is_magmoid_morph}.)
\end{itemize}
\end{proof}

It follows immediately from its definition, that \( \ftilde \) acts trivial on morphisms in \( \ttil. \)
Therefore we record:

\begin{satz}
Considered as a magmoid-endomorphism on \( \that, \) \( \ftilde \) is a projection (\( \ftilde \circ \ftilde = \ftilde \)) with image \( \ttil. \)
\end{satz}

\begin{satz} \label{prop_ttil_is_tricat_tdl_is_triequ}
We can equip \( \ttil \) with units and constraints,
\[
\tilde 1 _{a} = \ftilde (\hat 1_a), \; \tilde i_a = \ftilde (\hat i _a), \; \tilde a _{\tilde \alpha, \tilde \beta, \tilde \gamma} = \ftilde (\hat a _{\tilde \alpha, \tilde \beta, \tilde \gamma}), \dots
\]
which makes \( \ttil \) a tricategory and \( \ftilde: \that \to \ttil \) a strict triequivalence.
\end{satz}

\begin{proof}
We omit the proof, since it is a slightly simpler version (due to the simpler definition of \( \hat \Xi \) in comparison with \( \tilde \Theta \)) of the proof of proposition \ref{prop_tbar_is_a_tricat}.
\end{proof}

The tricategory \( \ttil \) is already slightly stricter as \( \that, \) since some of its constraint cells are identities:

\begin{lem} \label{lem_phi_a_and_phi_fg_are_identities}
Let \( a \) be an object of \( \ttil \) and \( \hat f, \hat g \) be composable 1-morphisms of \( \ttil. \)
Then the constraint cells \( \tilde \phi_a: \tilde 1 _{(1_a)} \to \tilde i_a \) and \( \tilde \phi_{(\hat f, \hat g)}: \tilde 1 _{\hat f \hot \hat g} \to \tilde 1 _{\hat f} \hot \tilde 1 _{\hat g} \) are identities.
\end{lem}

\begin{proof}
When we unravel definitions we realize that \( \tilde \phi_a \) as well as \( \tilde \phi_{(\hat f, \hat g)} \) is given as a coherence automorphism on a coherence 2-morphism.
Therefore \( \tilde \phi_a \) and \( \tilde \phi_{(\hat f, \hat g)} \) are identities, because \( \that \) is semi-free.
\end{proof}

The tricategory \( \ttil \) is still too complicated, to further strictify it by applying propositions \ref{prop_coh_3_system} and \ref{prop_T_mod_C23}.
Therefore we show that \( \ttil \) is triequivalent to a "more structured" tricategory \( \tbar, \) in which any 2-morphism is a \( * \)-composite of so called \( \tbar \)-basic 2-morphisms.
Such a \( \tbar\)-basic 2-morphism is a \( \hot \)-composite of precisely one \( \ttil \)-basic 2-morphisms and arbitrary many unit 2-morphism in \( \ttil. \)
Any 2-morphism \( \tilde \alpha \in \mmor{\ttil} \) can be "deformed" to a 2-morphism in \( \tbar \) by continually applying interchangers and local unit constraints.
This is essentially what the virtually strict triequivalence \( \barf: \ttil \to \tbar \) does.

We start by defining \( \tbar \) as a 3-globular set \( \tbar \subseteq \ttil: \)

\begin{mydef}
\( \tbar \) is the following 3-globular subset of \( \ttil: \)
\begin{itemize}
\item
\( \ob{\tbar} = \ob{\ttil} \) and \( \mor{\tbar} = \mor{\ttil}. \)

\item
A \( \tbar \)-basic 2-morphism is a 2-morphism of the form
\[
(\tilde 1 _{f_1} \hot \dots \hot \tilde 1 _{f_m} \hot \tilde \alpha \hot \tilde 1 _{f _{m+1}} \hot \dots \hot \tilde 1 _{f _n}) ^{\sigma},
\]
where \( m \leq n \) are natural numbers, \( f_i \) are \( \that \)-basic 1-morphisms and \( \tilde \alpha \) is a \( \ttil \)-basic 2-morphism.

\( \mmor{\tbar} \) is the smallest collection of 2-morphisms in \( \ttil, \) which contains all \( \tbar \)-basic 2-morphism and is closed under \( \hs. \)

\item
The 3-morphisms of \( \tbar \) are the 3-morphisms of \( \that \) whose types are contained in \( \mmor{\tbar}. \)
\end{itemize}
\end{mydef}

\begin{bem}
With notation \ref{not_comseq} we obtain that any 2-morphism in \( \tbar \) is of the form \( \comseq{\hs}{(\bar \alpha_i)} ^{\sigma} _{n}, \)
where \( \bar \alpha_i \) are \( \tbar \)-basic 2-morphisms.
\end{bem}

Since \( \mmor{\tbar} \) is not closed under \( \hot, \) \( \tbar \) is not a sub-3-magmoid of 	\( \ttil \) or \( \that. \)
Therefore we define a magmoidal structure on \( \tbar. \)

\begin{mydef} \label{def_magmoid_structure_of_tbar}
\( \tbar \) carries the structure of a 3-magmoid \( (\tbar, \botimes, \bstar, \bcirc) \), where
\begin{itemize}
\item
\( \bstar \) and \( \bcirc \) are the restrictions of \( \hs \) and \( \hc \) on \( \tbar \subseteq \that. \)
\item
On 1-morphisms \( \botimes = \hot. \)
\item
Let \( \bar \alpha = \comseq{\hs}{(\bar \alpha_i)} ^{\sigma} _{n} \) and \( \bar \beta = \comseq{\hs}{(\bar \beta_j)} ^{\tau} _{m} \) be arbitrary 2-morphisms in \( \tbar, \)
where \( \bar \alpha_i \) and \( \bar \beta_j \) are \( \tbar \)-basic 2-morphisms.

If \( \bar \alpha \hot \bar \beta \) is contained in \( \mmor{ \that}, \) we define
\[
\bar \alpha \botimes \bar \beta := \bar \alpha \hot \bar \beta.
\]

Otherwise we define
\[
\bar \alpha \botimes \bar \beta
=
\comseq{\hs}{(\tilde 1 _{t \bar \alpha} \hot \bar \beta_j)} _{m}^{\tau}
\hs
\comseq{\hs}{(\bar \alpha_i \hot \tilde 1 _{s \bar \beta})} _{n}^{\sigma}.
\]

\item
Next we define a 3-morphism
\[
\tilde \Theta _{\bar \alpha, \bar \beta}: \bar \alpha \hot \bar \beta \to \bar \alpha \botimes \bar \beta.
\]

If \( \bar \alpha \hot \bar \beta \) is contained in \( \mmor{ \tbar}, \) we define
\[
\tilde \Theta _{\bar \alpha, \bar \beta} = \tilde 1 _{\bar \alpha \hot \bar \beta}
\]

Otherwise \( \tilde \Theta _{\bar \alpha, \bar \beta} \) is obtained by applying the counit \( \epsilon: \oper{F} \ttil \to \ttil \) to the following unique coherence 3-morphisms in \( \oper{F} \ttil: \)
\[
\comseq{\fstar}{(\bar \alpha_i)} _{n}^{\sigma} \fotimes \comseq{\fstar}{(\bar \beta_j)} _{m}^{\tau}
\rightarrow
\comseq{\fstar}{(1 ^{\sms{F}} _{t \bar \alpha} \fotimes \bar \beta_j)} _{m}^{\tau}
\fstar
\comseq{\fstar}{(\bar \alpha_i \fotimes 1 ^{\sms{F}} _{s \bar \beta})} _{n}^{\sigma}.
\]

\item
On 3-morphisms \( \botimes \) is defined by
\[
\bar{\Pi} \botimes \bar{\Lambda} := \tilde \Theta _{t \Pi, t \Lambda} \hc (\bar{\Pi} \hot \bar{\Lambda}) \hc \tilde \Theta _{s \Pi, s \Lambda}.
\]
\end{itemize}
\end{mydef}

\begin{satz} \label{prop_bar_is_magmoid_morph}
There is a morphism of 3-magmoids \( \barf: \ttil \to \tbar, \) defined as follows:

\begin{itemize}
\item
On objects and 1-morphisms \( \barf \) is the identity.
\item
On \( \ttil \)-basic 2-morphisms \( \barf \) is the identity.
On the remaining 2-morphisms of \( \ttil, \) \( \barf \) is defined via the recursion rules
\[
\barf (\tilde \alpha \hs \tilde \beta) := \barf (\tilde \alpha) \hs \barf (\tilde \beta) \qquad
\barf (\tilde \alpha \hot \tilde \beta) := \barf (\tilde \alpha) \botimes \barf (\tilde \beta).
\]

\item
On 3-morphisms \( \barf \) is defined as follows:
\[
\barf \tilde \Pi := \tilde \Theta _{t \tilde \Pi} \hc \tilde \Pi \hc \tilde \Theta ^{-1} _{s \tilde \Pi},
\]
where
\(
\tilde \Theta _{\tilde \alpha}: \tilde \alpha \to \barf \tilde \alpha,
\)
is a 3-morphism inductively defined as follows:

For a \( \ttil \)-basic 2-morphism \( \tilde \alpha \) we define \( \tilde \Theta _{\tilde \alpha} = 1 _{\tilde \alpha}. \)

The definition of \( \tilde \Theta _{\tilde \alpha} \) extends to general \( \tilde \alpha \in \mmor{\ttil} \) via recursion:
\[
\tilde \Theta _{\tilde \alpha \hs \tilde \beta} := \tilde \Theta _{ \tilde \alpha} \hs \tilde \Theta _{\tilde \beta}
\qquad
\tilde \Theta _{\tilde \alpha \hot \tilde \beta} := \tilde \Theta _{\barf \tilde \alpha,\barf \tilde \beta} \circ \tilde \Theta _{\tilde \alpha} \hot \tilde \Theta _{\tilde \beta}
\]
\end{itemize}
\end{satz}

\begin{proof}
It is easy to see that \( \barf \) is a morphism of globular sets.

By definition \( \barf \) strictly preserves \( \otimes \)-composition of 1- and 2-morphisms and \( * \)-composition of 2-morphisms.

A direct calculation shows that \( \barf \) strictly preserves \( \otimes \)-composition of 3-morphism:
\begin{align*}
& \barf ( \tilde \Pi \hot  \tilde \Lambda) \\
={} & \tilde \Theta _{t  \tilde \Pi \hot t  \tilde \Lambda} \circ ( \tilde \Pi \hot \tilde \Lambda) \circ \tilde \Theta ^{-1} _{s  \tilde \Pi \hot s \tilde \Lambda} \\
={} & \tilde \Theta _{\barf (t  \tilde \Pi), \barf (t  \tilde \Lambda)} \circ ( \tilde \Theta _{t  \tilde \Pi} \hot \tilde \Theta _{t \tilde \Lambda}) \circ ( \tilde \Pi \hot \tilde \Lambda)
\circ (\tilde \Theta ^{-1} _{s  \tilde \Pi} \hot \tilde \Theta ^{-1} _{t  \tilde \Lambda}) \circ \tilde \Theta ^{-1} _{\barf (s  \tilde \Pi), \barf (s  \tilde \Lambda)} \\
={} & \tilde \Theta _{\barf (t  \tilde \Pi), \barf (t  \tilde \Lambda)} \circ (\barf  \tilde \Pi \hot \barf  \tilde \Lambda) \circ \tilde \Theta ^{-1} _{\barf (s  \tilde \Pi), \barf (s  \tilde \Lambda)} \\
={} & \barf  \tilde \Pi \botimes \barf  \tilde \Lambda.
\end{align*}
Also \( * \)-composition of 3-morphisms is strictly preserved:
\begin{align*}
\barf (\tilde \Pi \hs \tilde \Lambda)
={} & \tilde \Theta _{t (\tilde \Pi \hs \tilde \Lambda)} \circ (\tilde \Pi \hs \tilde \Lambda) \circ \tilde \Theta ^{-1} _{s (\tilde \Pi \hs \tilde \Lambda)} \\
={} & (\tilde \Theta _{t \tilde \Pi} \hs \tilde \Theta _{t \tilde \Lambda}) \circ (\tilde \Pi \hs \tilde \Lambda) \circ (\tilde \Theta ^{-1} _{s \tilde \Pi} \hs \tilde \Theta ^{-1} _{s \tilde \Lambda}) \\
={} & (\tilde \Theta _{t \tilde \Pi} \circ \tilde \Pi \circ \tilde \Theta ^{-1} _{s \tilde \Pi}) \hs (\tilde \Theta _{t \tilde \Lambda} \circ \tilde \Lambda \circ \tilde \Theta ^{-1} _{s \tilde \Lambda}) \\
={} & \barf \tilde \Pi \bstar \barf \tilde \Lambda.
\end{align*}
Finally \( \circ \)-composition of 3-morphisms is strictly preserved as well:
\begin{align*}
\barf (\tilde \Pi  \circ \tilde \Lambda)
={} & \tilde \Theta _{t \tilde \Pi} \circ \tilde \Pi \circ \tilde \Lambda \circ \tilde \Theta ^{-1} _{s \tilde \Lambda} \\
={} & \tilde \Theta _{t \tilde \Pi} \circ \tilde \Pi \circ \tilde \Theta ^{-1} _{s \tilde \Pi} \circ \Theta _{ t \tilde \Lambda} \circ \tilde \Lambda \circ \tilde \Theta ^{-1} _{s \tilde \Lambda} \\
={} & \barf \tilde \Pi \circ \barf \tilde \Lambda.
\end{align*}
\end{proof}

\begin{bem}~
If we consider \( \ttil \) and \( \tbar \) as globular sets and \( \barf: \ttil \to \that \) as a morphism of globular sets,
we observe \( \barf \) acts trivially on \( \tbar \subseteq \ttil \) and that \( \barf \) is a projection (\( \barf \circ \barf = \barf. \))
\end{bem}

\begin{satz} \label{prop_tbar_is_a_tricat}
The 3-magmoid \(\tbar \) carries a unique tricategory structure, for which \( \barf: \ttil \to \tbar \) becomes a virtually strict triequivalence.
\end{satz}

\begin{proof}
Since \( \barf \) is required to be a strict functor, the units of \( \tbar \) need to be
\[
\bar 1 _{a} = \barf (\tilde 1_a), \;
\bar 1 _{\hat f} = \barf (\tilde 1 _{\hat f}), \;
\bar 1 _{\bar \alpha} = \barf (\tilde 1 _{\bar \alpha}),
\]
the constraints of \( \tbar \) need to be
\[
\bar{i} _{a} := \barf (\tilde i _{a}), \;
\bar{a} _{\hat f \hat g \hat h} := \barf (\tilde a _{\hat f \hat g \hat h}), \;
\bar{a} _{\bar{\alpha} \bar{\beta} \bar{\gamma}} := \barf (\tilde a _{\bar{\alpha} \bar{\beta} \bar{\gamma}}),
\]
and analogous for the remaining constraints (see definition \ref{def_tricat_as_magmoid} for a complete list of constraints).
Now we show that these constraint data equip \( \tbar \) with the structure of a tricategory:

First we check that the types of the constraint-cells are as they need to be according to definition \ref{def_tricat_as_magmoid}.
This is the case because of the way the constraints of \( \tbar \) are defined, and because \( \barf \) is a morphism of magmoids.
For example
\[
s (\bar{i} _{a}) = s (\barf \tilde i _{a}) = \barf s (\tilde i _{a}) = \barf \tilde 1_a = \bar{1}_a,
\quad \text{or}
\]
\begin{align*}
s (\bar{a} _{\bar{\alpha} \bar{\beta} \bar{\gamma}}) ={}&
s (\barf \tilde a _{\bar{\alpha}, \bar{\beta}, \bar{\gamma}}) =
\barf s (\tilde a _{\bar{\alpha}, \bar{\beta}, \bar{\gamma}}) \\ ={}&
\barf \big(\tilde a _{t \alpha, t \beta, t \gamma} \hs \big( (\bar{\alpha} \hot \bar{\beta}) \hot \bar{\gamma} \big) \big) =
\bar{a} _{t \alpha, t \beta, t \gamma} \hs \big( (\bar{\alpha} \botimes \bar{\beta}) \botimes \bar{\gamma} \big),
\end{align*}
and analogously for the remaining constraint cells.

Now we check that the axioms of definition \ref{def_tricat_as_magmoid} are satisfied for \( \tbar. \)

This follows, because any axiom for \( \tbar \) is the image of an axiom for \( \ttil \) under \( \barf. \)

For example consider 3-morphisms
\[
\bar{\Pi} _{1}: \bar{\alpha} \to \bar{\alpha}', \;
\bar{\Pi} _{2}: \bar{\beta} \to \bar{\beta}', \;
\bar{\Pi} _{3}: \bar{\gamma} \to \bar{\gamma}'
\]
then
\[
\begin{tikzcd}
\bar{a} _{t \bar{\alpha}, t \bar{\beta}, t \bar{\gamma}} \hs \big( (\bar{\alpha} \botimes \bar{\beta}) \botimes \bar{\gamma} \big)
\ar[r, "\bar{a} _{\bar{\alpha},\bar{\beta},\bar{\gamma}}"]
\ar[d, "{1 \hs \big( (\bar{\Pi}_1 \botimes \bar{\Pi}_2) \botimes \bar{\Pi}_3 \big)}" swap]
	& \big( \bar{\alpha} \botimes (\bar{\beta} \botimes \bar{\gamma}) \big) \hs \bar{a} _{s \bar{\alpha}, s \bar{\beta}, s \bar{\gamma}}
	\ar[d, "{\big( \bar{\Pi}_1 \botimes (\bar{\Pi}_2 \botimes \bar{\Pi}_3) \big) \hs 1}"] \\
\bar{a} _{t \bar{\alpha}', t \bar{\beta}', t \bar{\gamma}'} \hs \big( (\bar{\alpha}' \botimes \bar{\beta}') \botimes \bar{\gamma}' \big)
\ar[r, "\bar{a} _{\bar{\alpha}',\bar{\beta}',\bar{\gamma}'}"]
	& \big( \bar{\alpha}' \botimes (\bar{\beta}' \botimes \bar{\gamma}') \big) \hs \bar{a} _{s \bar{\alpha}', s \bar{\beta}', s \bar{\gamma}'} \\
\end{tikzcd}
\]
commutes, because it is the image of the following diagram under \( \barf. \)
\[
\begin{tikzcd}
\tilde{a} _{t \bar{\alpha}, t \bar{\beta}, t \bar{\gamma}} \hs \big( (\bar{\alpha} \hot \bar{\beta}) \hot \bar{\gamma} \big)
\ar[r, "\tilde{a} _{\bar{\alpha},\bar{\beta},\bar{\gamma}}"]
\ar[d, "{1 \hs \big( (\bar{\Pi}_1 \hot \bar{\Pi}_2) \hot \bar{\Pi}_3 \big)}" swap]
	& \big( \bar{\alpha} \hot (\bar{\beta} \hot \bar{\gamma}) \big) \hs \tilde a _{s \bar{\alpha}, s \bar{\beta}, s \bar{\gamma}}
	\ar[d, "{\big( \bar{\Pi}_1 \hot (\bar{\Pi}_2 \hot \bar{\Pi}_3) \big) \hs 1}"] \\
\tilde{a} _{t \bar{\alpha}', t \bar{\beta}', t \bar{\gamma}'} \hs \big( (\bar{\alpha}' \hot \bar{\beta}') \hot \bar{\gamma}' \big)
\ar[r, "\tilde{a} _{\bar{\alpha}',\bar{\beta}',\bar{\gamma}'}"]
	& \big( \bar{\alpha}' \hot (\bar{\beta}' \hot \bar{\gamma}') \big) \hs \tilde a _{s \bar{\alpha}', s \bar{\beta}', s \bar{\gamma}'} \\
\end{tikzcd}
\]

In proposition \ref{prop_bar_is_magmoid_morph} we have already shown that \( \barf \) is a morphism of magmoids.
For showing \( \barf \) is a virtually strict functor it remains to show, that \( \barf \) strictly preserves units and constraints.

Since \( \barf \) is the identity on objects and 1-morphisms of \( \ttil, \) all units or constraints indexed by objects or 1-morphisms are strictly preserved.
Those are
\[
\tilde 1_a, \, \tilde i_a, \, \tilde \phi _a, \,
\tilde a _{\hat f \hat g \hat h}, \, \tilde l _{\hat f}, \, \tilde r _{\hat f}, \,
\tilde a ^{\adsq} _{\hat f \hat g \hat h}, \, \tilde l  ^{\adsq} _{\hat f}, \, \tilde r  ^{\adsq} _{\hat f}, \,
\tilde \eta ^{a} _{\hat f \hat g \hat h}, \, \tilde \epsilon ^{a} _{\hat f \hat g \hat h}, \,
\tilde \eta ^{l} _{\hat f}, \, \tilde \epsilon ^{l} _{\hat f}, \,
\tilde \eta ^{r} _{\hat f}, \, \tilde \epsilon ^{r} _{\hat f}, \,
\tilde \lambda _{\hat f \hat g}, \, \tilde \mu _{\hat f \hat g}, \, \tilde \rho _{\hat f \hat g}, \, \tilde \pi _{\hat f \hat g \hat h \hat k}.
\]

3-units are preserved due to the identity
\[
\barf (\tilde 1 _{\tilde \alpha}) = \tilde \Theta _{\tilde \alpha} \circ \tilde 1 _{\tilde \alpha} \circ \tilde \Theta ^{-1} _{\tilde \alpha}
= \tilde 1 _{t(\tilde \Theta _{\tilde \alpha})} = \bar{1} _{\barf (\tilde \alpha).}
\]

The associators from the local bicategories are preserved, as can be seen as follows:
Both morphisms
\[
\tilde a ^{loc} _{\barf \tilde \alpha, \barf \tilde \beta, \barf \tilde \gamma} = \bar a ^{loc} _{\barf \tilde \alpha, \barf \tilde \beta, \barf \tilde \gamma}
\quad \text{and} \quad
\barf (\tilde a ^{loc} _{\tilde \alpha \tilde \beta \tilde \gamma})
\]
make the following diagram commute and are therefore equal.
\[
\begin{tikzcd}[column sep = huge]
(\tilde \alpha \hs \tilde \beta) \hs  \tilde \gamma
\ar[r, "\tilde a ^{loc} _{\tilde \alpha \tilde \beta \tilde \gamma}"]
\ar[d, "{(\tilde \Theta _{\tilde \alpha} \hs \tilde \Theta _{\tilde \beta}) \hs \tilde \Theta _{\tilde \gamma}}" swap]
	& \tilde \alpha \hs (\tilde \beta \hs \tilde \gamma)
	\ar[d, "{\tilde \Theta _{\tilde \alpha} \hs (\tilde \Theta _{\tilde \beta} \hs \tilde \Theta _{\tilde \gamma})}"]
	\\
(\barf \tilde \alpha \hs \barf \tilde \beta) \hs \barf \tilde \gamma
\ar[r]
	& \barf \tilde \alpha \hs (\barf \tilde \beta \hs \barf \tilde \gamma)
\end{tikzcd}
\]

To show that the local unit-constraints are preserved we consider the following diagram:
\[
\begin{tikzcd}[column sep = huge]
\tilde 1 _{t \tilde \alpha} \hs \tilde \alpha
\ar[r, "\tilde l ^{loc} _{\tilde \alpha}"]
\ar[d, "\tilde 1 \hs \tilde \Theta _{\tilde \alpha}" swap]
	& \tilde \alpha
	\ar[d, "\tilde \Theta _{\tilde \alpha}"] \\
\tilde 1 _{t \tilde \alpha} \hs \barf \tilde \alpha
\ar[r, "\tilde l ^{loc} _{\barf \tilde \alpha}"]
\ar[d, "\tilde \Theta _{\tilde 1 _{t \tilde \alpha}} * 1" swap]
	& \barf \tilde \alpha
	\ar[d, "1"] \\
\barf \tilde 1 _{t \tilde \alpha} \hs \barf \tilde \alpha
\ar[r]
	& \barf \tilde \alpha
\end{tikzcd}
\]
The upper square commutes because of of the naturality of the local unit constraints.
The lower square commutes for
\(
\barf (\tilde l ^{loc} _{\barf \tilde \alpha}) = \bar l ^{loc} _{\barf \tilde \alpha}.
\)
Then the exterior square commutes as well.
Since the latter is the case if and only if the bottom morphism is equal to
\(
\barf \tilde l ^{loc} _{\tilde \alpha}
\)
it follows that
\(
\barf \tilde l ^{loc} _{\tilde \alpha} = \bar l ^{loc} _{\barf \tilde \alpha}.
\)
A similar argument shows that
\(
\barf \tilde r ^{loc} _{\tilde \alpha} = \bar r ^{loc} _{\barf \tilde \alpha}.
\)

To show that the interchangers are preserved we consider the following diagram:
\[
\begin{tikzcd}[column sep = huge, row sep = large]
(\tilde \alpha \hot \tilde \beta) \hs (\tilde \alpha' \hot \tilde \beta')
\ar[d, "{(\tilde \Theta _{\tilde \alpha} \hot \tilde \Theta _{\tilde \beta}) \hs (\tilde \Theta _{\tilde \alpha'} \hot \tilde \Theta _{\tilde \beta'})}" swap]
\ar[r, "\tilde \phi ^{\otimes}"]
	& (\tilde \alpha \hs \tilde \alpha') \hot (\tilde \beta \hs \tilde \beta')
	\ar[d, "{(\tilde \Theta _{\tilde \alpha} \hs \tilde \Theta _{\tilde \alpha'}) \hot (\tilde \Theta _{\tilde \beta} \hs \tilde \Theta _{\tilde \beta'})}"] \\
(\barf \tilde \alpha \hot \barf \tilde \beta) \hs (\barf \tilde \alpha' \hot \barf \tilde \beta')
\ar[d, "\tilde \Theta _{\barf \tilde \alpha, \barf \tilde \beta} \hs \tilde \Theta _{\barf \tilde \alpha', \barf \tilde \beta'}" swap]
\ar[r, "\tilde \phi ^{\otimes}"]
	& (\barf \tilde \alpha \hs \barf \tilde \alpha') \hot (\barf \tilde \beta \hs \barf \tilde \beta')
	\ar[d, "\tilde \Theta _{\barf \tilde \alpha \hs \barf \tilde \alpha', \tilde \beta \hs \barf \tilde \beta'}"] \\
(\barf \tilde \alpha \botimes \barf \tilde \beta) \hs (\barf \tilde \alpha' \botimes \barf \tilde \beta')
\ar[r]
	& (\barf \tilde \alpha \hs \barf \tilde \alpha') \botimes (\barf \tilde \beta \hs \barf \tilde \beta') \\
\end{tikzcd}
\]

The upper square commutes because of the naturality of the interchanger.
The lower square commutes for the 3-morphism
\(
\bar{\phi} _{(\barf \tilde \alpha,\barf \tilde \beta)(\barf \tilde \alpha',\barf \tilde \beta')}.
\)
Then the exterior square commutes as well.
Since the latter is the case if and only if the bottom morphism is equal to
\(
\barf(\tilde \phi _{(\tilde \alpha, \tilde \beta)(\tilde \alpha', \tilde \beta')})
\)
it follows that
\(
\barf(\tilde \phi _{(\tilde \alpha, \tilde \beta)(\tilde \alpha', \tilde \beta')}) =
\bar{\phi} _{(\barf \tilde \alpha,\barf \tilde \beta)(\barf \tilde \alpha',\barf \tilde \beta')}.
\)

To show that left units are preserved we consider the following diagram, where the subscripts of \( \tilde \Theta \) are omitted.
\[
\begin{tikzcd}
\tilde l _{t \tilde \alpha} \hs (\tilde i _{t^2 \tilde \alpha} \hot \tilde \alpha)
\ar[r, "\tilde l _{\tilde \alpha}"]
\ar[d, "\tilde 1 \hs ( \tilde 1 \hot \tilde \Theta _{\tilde \alpha})"]
\ar[dd, start anchor=west, end anchor=west, "\tilde \Theta" swap, bend right=40]
	& \tilde \alpha \hs \tilde l _{s \tilde \alpha}
	\ar[d, "\tilde \Theta _{\tilde \alpha} \hs \tilde 1"]
	\ar[dd, start anchor=east, end anchor=east, "\tilde \Theta", bend left=40] \\
\tilde l _{t \tilde \alpha} \hs (\tilde i _{t^2 \tilde \alpha} \hot \barf \tilde \alpha)
\ar[r, "\tilde l _{\barf \tilde \alpha}"]
\ar[d, "\tilde \Theta"]
	& \barf \tilde \alpha \hs \tilde l _{s \tilde \alpha}
	\ar[d, "\tilde \Theta"] \\
\barf \big( \tilde l _{t \tilde \alpha} \hs (\tilde i _{t^2 \tilde \alpha} \hot \tilde \alpha) \big)
\ar[r, "\bar l _{ \barf \tilde \alpha}"]
	& \barf \big( \tilde \alpha \hs \tilde l _{s \tilde \alpha} \big) \\
\end{tikzcd}
\]
The upper square in the middle commutes by naturality, the lower one by definition of \( \bar l _{\barf \tilde \alpha}. \)
The following calculation shows that the left triangle commutes, and an analogous calculation shows that the right triangle commutes.
\begin{align*}
\tilde \Theta _{\tilde l \hs (\tilde i \hot \tilde \alpha)} ={}& \tilde \Theta _{\tilde l} \hs \tilde \Theta _{\tilde i \hot \tilde \alpha} \\ ={}&
\tilde \Theta _{\tilde l} \hs \big( \tilde \Theta _{\barf \tilde i, \barf \tilde \alpha} \hc (\tilde \Theta _{\tilde i} \hot \tilde \Theta _{\tilde \alpha}) \big) \\ ={}&
\tilde \Theta _{\tilde l} \hs \big( \tilde \Theta _{\barf \tilde i, \barf \tilde \alpha} \hc (\tilde \Theta _{\tilde i} \hot \tilde 1 _{\barf \tilde \alpha})
\hc (\tilde 1 _{\tilde i} \hot \tilde \Theta _{\tilde \alpha}) \big) \\ ={}&
(\tilde \Theta _{\tilde l} \hs \tilde \Theta _{\barf \tilde i, \barf \tilde \alpha}) \hc \big( \tilde 1 _{\tilde l} \hs (\tilde \Theta _{\tilde i} \hot \tilde 1 _{\barf \tilde \alpha}) \big)
\hc \big( \tilde 1 _{\tilde l} \hs (\tilde 1 _{\tilde i} \hot \tilde \Theta _{\tilde \alpha} ) \big) \\ ={}&
\big( \tilde \Theta _{\tilde l} \hs (\tilde \Theta _{\barf \tilde i, \barf \tilde \alpha} \hc (\tilde \Theta _{\tilde i} \hot \tilde 1 _{\barf \tilde \alpha})) \big)
\hc \big( \tilde 1 _{\tilde l} \hs (\tilde 1 _{\tilde i} \hot \tilde \Theta _{\tilde \alpha} ) \big) \\ ={}&
\big( \tilde \Theta _{\tilde l} \hs (\tilde \Theta _{\barf \tilde i, \barf \tilde \alpha} \hc (\tilde \Theta _{\tilde i} \hot \tilde \Theta _{\barf \tilde \alpha})) \big)
\hc \big( \tilde 1 _{\tilde l} \hs (\tilde 1 _{\tilde i} \hot \tilde \Theta _{\tilde \alpha} ) \big) \\ ={}&
\tilde \Theta _{\tilde l \hs (\tilde i \hot \barf \tilde \alpha)} \hc \big( \tilde 1 _{\tilde l} \hs (\tilde 1 _{\tilde i} \hot \tilde \Theta _{\tilde \alpha} ) \big).
\end{align*}
It follows that the exterior diagram commutes as well, but since this is the case if and only if the bottom morphism of the exterior square is \( \barf \tilde l _{\tilde \alpha} \) we have shown
\[
\barf \tilde l _{\tilde \alpha} = \bar l _{ \barf \tilde \alpha}.
\]
An analogous argument shows that components of the right unit \( r \) and the associator \( a \) are preserved.

Finally we show \( \barf \) is a triequivalence.
Note that \( \barf \) is 3-locally a bijection, because \( \barf \) acts on a 3-local set through conjugation with invertible 3-morphisms.
Moreover \( \barf \) is 0-, 1- and 2-locally surjective and therefore a triequivalence by remark \ref{rem_characterization_virtually_strict_triequivalence}.
\end{proof}

In a cubical tricategory a certain kind of interchangers is given by identities (see definition \ref{def_cubical_tricat_Gray_cat}).
Those interchangers are not yet identities in \( \tbar, \) but they are given as local constraint cells as the following lemma shows.
This is important, since the next step in our strictification will be to "quotient out" the local constraints of \( \that. \)

\begin{lem} \label{lem_interchanger_in_tbar}
In \( \tbar \) interchangers of the form \( \phi ^{\otimes} _{(\bar \alpha,\bar \beta)(\bar \gamma, \bar 1)} \) or \( \phi ^{\otimes} _{(\bar 1,\bar \alpha)(\bar \beta, \bar \gamma)} \)
are given as \( (\hs,\hc) \)-composites of the local constraint cells \( a ^{loc}, l ^{loc}, r ^{loc}. \)
\end{lem}

\begin{proof}
We consider an interchanger \( \bar \phi ^{\otimes}_{(\bar \alpha,\bar \beta)(\bar \gamma, \bar 1)} \) with
\( \bar \alpha = \comseq{\hs}{(\bar \alpha_i)}_n ^{\sigma}, \) \( \bar \beta = \comseq{\hs}{(\bar \beta_i)}_m ^{\tau} \) and \( \bar \gamma = \comseq{\hs}{(\bar \gamma)}_l ^{\mu} \)
for \( \tbar \)-basic 2-morphisms \( \bar \alpha_i, \bar \beta_i, \bar \gamma_i. \)
Furthermore we assume neither \( \bar \alpha \hot \bar \beta, \) nor \( \bar \gamma \hot \bar 1 \) are contained in \( \mmor{ \tbar} \) (these cases are proven similarly).

Now we consider the following diagram in \( \oper{F} \ttil, \) where for compactness we wrote \( \otimes, * \) and \( 1 \) to mean \( \fotimes, \fstar \) and \( 1 ^{\sms{F}}. \)
All morphisms in the diagram are given by coherence, and therefore the diagram commutes.
\[
\begin{tikzcd}[column sep=tiny, cramped, font=\scriptsize]
\big( \comseq{*}{(1 \otimes \bar \beta_i)}^{\tau}_m *
\comseq{*}{(\bar \alpha_i \otimes 1)}^{\sigma}_n \big) *
\big( (1 \otimes 1)*
\comseq{*}{(\bar \gamma_i \otimes 1)}^{\mu}_l \big)
\ar[d]
\ar[r]
	& \big( \comseq{*}{(1 \otimes \bar \beta_i)}^{\tau}_m *
	\comseq{*}{(\bar \alpha_i \otimes 1)}^{\sigma}_n \big) *
	\big( 1*
	\comseq{*}{(\bar \gamma_i \otimes 1)}^{\mu}_l \big)
	\ar[d] \\
\big( \comseq{*}{(\bar \alpha_i)}^{\sigma}_n \otimes \comseq{*}{(\bar \beta_i)}^{\tau}_m \big) *
\big( \comseq{*}{(\bar \gamma_i)}^{\mu}_l \otimes 1 \big)
\ar[d]
	& \big( \comseq{*}{(1 \otimes \bar \beta_i)}^{\tau}_m
	* (1) \big) *
	\big( \comseq{*}{(\bar \alpha_i \otimes 1)}^{\sigma}_n
	* \comseq{*}{(\bar \gamma_i \otimes 1)}^{\mu}_l \big)
	\ar[d] \\
\big( \comseq{*}{(\bar \alpha_i)}^{\sigma}_n * \comseq{*}{(\bar \gamma_i)}^{\mu}_l \big) \otimes
\big( \comseq{*}{(\bar \beta_i)}^{\tau}_m * 1 \big)
\ar[r]
	& \big( \comseq{*}{(1 \otimes \bar \beta_i)}^{\tau}_m
	* (1 \otimes 1) \big) *
	\big( \comseq{*}{(\bar \alpha_i \otimes 1)}^{\sigma}_n
	* \comseq{*}{(\bar \gamma_i \otimes 1)}^{\mu}_l \big)
\end{tikzcd}
\]

The first and last morphism in the left path of the diagram, are lifts of \( \tilde \Theta ^{-1} \) resp. \( \tilde \Theta. \)
The second morphism in the left path is a lift of the interchanger \( \tilde \phi ^{\otimes}_{(\bar \alpha,\bar \beta)(\bar \gamma,1)}. \)
Therefore the left path constitutes a lift of \( \bar \phi ^{\otimes}_{(\bar \alpha,\bar \beta)(\bar \gamma,1)}. \)

The first and the last morphism in the right path are identities under \( \epsilon: \oper{F} \ttil \to \ttil, \)
since the unit comparison cells of \( \otimes \) are identities in \( \ttil \) by lemma \ref{lem_phi_a_and_phi_fg_are_identities}.
The second morphism in the right path is a \( (\fstar, \fcirc) \)-composite of local constraint cells.
Therefore under \( \epsilon \) the right path as a whole is given as a \( (\hs, \hc) \)-composite of local constraint cells and so is \( \bar \phi ^{\otimes}_{(\bar \alpha,\bar \beta)(\bar \gamma,1)}. \)
\end{proof}

\begin{satz} \label{prop_coh_3_system_tbar}
Let \( C \subseteq \mmmor{\tbar} \) be the collection of 3-morphisms which are coherence morphisms for a local bicategory of \( \tbar. \)
In other words a morphism in \( C \) is a \( (\hs,\hc) \)-composite of unit-3-morphisms and components of \( \bar{a} ^{loc}, \bar{l} ^{loc}, \bar{r} ^{loc}. \)
Then \( C \) is a coherence-3-system.
\end{satz}

\begin{proof}
For a \( \tbar \)-basic 2-morphism \( \bar \alpha, \) we define
\[
\bar \alpha ^{\sms{L}} :=
\begin{cases}
1 _{\hat f}^{\sms{F}} \text{ (as a unit in \( \oper{F} \ttil \)), }
	& \text{if } \bar \alpha = \bar 1_f \\
\bar \alpha \text{ (as a \( \oper{F} \)-basic morphism in \( \oper{F} \ttil \)), }
	& \text{else}.
\end{cases}
\]
For an arbitrary 2-morphism \( \bar \beta = \comseq{\hs}{(\bar \beta_i)}_n ^{\sigma} \) in \( \tbar \) we define
\[
\bar \beta ^{\sms{L}} = \comseq{\fstar}{(\bar \beta_i ^{\sms{L}})}_n ^{\sigma}.
\]

Now we show that \( C \) is a coherence-3-system.
\begin{itemize}[-]
\item
We show that parallel 3-morphisms in \( C \) are equal.
For that consider an arbitrary 3-morphism \( \bar \Gamma: \comseq{\hs}{(\bar \alpha_i)} ^{\sigma}_{n} \to \comseq{\hs}{(\bar \alpha'_j)} ^{\sigma}_{m} \) in \( C. \)
It suffices to show that \( \bar \Gamma \) lifts to the coherence 3-morphism
\(
\comseq{\fstar}{(\bar \alpha_i ^{\sms{L}})} ^{\sigma}_{n} \to \comseq{\fstar}{(\bar \alpha_j ^{\prime \sms{L}})} ^{\sigma}_{m}
\)
in \( \oper{F} \ttil. \)
To show that, suppose \( \bar \Gamma \) is a \( (\hs, \circ) \)-composite of \( k \) constraint cells and units.
Then the claim follows easily by induction over \( k \).
(see the proof of proposition \ref{prop_that_semifree} for a similar induction.)

\item
To show that \( C \) is closed under \( \botimes \) let
\[
\bar{\alpha} = \comseq{\hs}{(\bar{\alpha}_i)}^{\sigma}_n \xrightarrow{\bar \Gamma_1} \comseq{\hs}{(\bar{\alpha}'_i)}^{\sigma'}_{n'} = \bar{\alpha}'
\quad \text{and} \quad
\bar{\beta} = \comseq{\hs}{(\bar{\beta}_i)}^{\tau}_m \xrightarrow{\bar \Gamma_2} \comseq{\hs}{(\bar{\beta}'_i)}^{\tau'}_{m'} = \bar{\beta}'
\]
be 3-morphisms in \( C. \)
For a composite \( \nu_1 \fotimes \nu_2 \) of \( F \)-basic 2-cells \( \nu_1, \nu_2 \) we introduce the notation
\[
U (\nu_1 \fotimes \nu_2) :=
\begin{cases}
1 ^{\sms{F}} _{(s \nu_1 \fotimes s \nu_2)}
	& \text{if } \nu_1 = 1 ^{\sms{F}} _{s \nu_1} \text{ and } \nu_2 = 1 ^{\sms{F}} _{s \nu_2} \\
\nu_1 \fotimes \nu_2
	& \text{else}.
\end{cases}
\]
Now we consider the following diagram in \( \oper{F} \ttil, \) where for compactness we write \( \otimes \) and \( * \) instead of \( \fotimes \) and \( \fstar. \)
\[
\begin{tikzcd}[column sep = huge]
\comseq{*}{(\bar{\alpha}_i^{\sms{L}})}^{\sigma}_n \otimes \comseq{*}{(\bar{\beta}_i^{\sms{L}})}^{\tau}_m
\ar[r]
	& \comseq{*}{(\bar{\alpha}^{\prime \sms{L}}_i)}^{\sigma'}_{n'} \otimes \comseq{*}{(\bar{\beta}^{\prime \sms{L}}_i)}^{\tau'}_{m'}
	\ar[d] \\
\comseq{*}{(\bar{\alpha}_i ^{\sms{L}} \otimes 1^{\sms{F}} )} ^{\sigma}_n * \comseq{*}{(1^{\sms{F}}  \otimes \bar{\beta}_i^{\sms{L}})} ^{\tau}_m
\ar[d]
\ar[u]
	& \comseq{*}{(\bar{\alpha}^{\prime \sms{L}}_i \otimes 1^{\sms{F}} )} ^{\sigma'}_{n'} * \comseq{*}{(1^{\sms{F}}  \otimes \bar{\beta}^{\prime \sms{L}}_i)} ^{\tau'}_{m'} \\
\comseq{*}{\big( U(\bar{\alpha}_i ^{\sms{L}} \otimes 1^{\sms{F}} ) \big)} ^{\sigma}_n * \comseq{*}{\big( U(1^{\sms{F}} \otimes \bar{\beta}_i^{\sms{L}}) \big)} ^{\tau}_m
\ar[r]
	& \comseq{*}{\big( U(\bar{\alpha}^{\prime \sms{L}}_i \otimes 1^{\sms{F}} ) \big)} ^{\sigma'}_{n'} * \comseq{*}{\big( U(1^{\sms{F}}  \otimes \bar{\beta}^{\prime \sms{L}}_i) \big)} ^{\tau'}_{m'} 		\ar[u] \\
\end{tikzcd}
\]
The arrows in the diagram are (unique) coherence 3-morphisms.
The first and the third morphism of the upper path are lifts of \( (\tilde \Theta _{\bar \alpha \hot \bar \beta}) ^{-1} \) resp. \( \tilde \Theta _{\bar \alpha' \hot \bar \beta'}. \)
The second morphism in the upper path is a lift of \( \bar \Gamma_1 \hot \bar \Gamma_2. \)
Therefore the upper path is a lift of \( \bar \Gamma_1 \botimes \bar \Gamma_2. \)

The first and the third morphism of the lower path are identities under \( \epsilon: \oper{F} \ttil \to \ttil \) according to lemma \ref{lem_interchanger_in_tbar}.
The second morphism in the lower path is given as a \( (\fstar, \fcirc) \)-composite of local constraint cells.
Therefore the image of the lower path as a whole under \( \epsilon \) is a \( (\hs, \hc) \)-composite of local constraint cells of \( \tbar. \)
And since the diagram commutes by coherence, the same holds for \( \bar \Gamma_1 \botimes \bar \Gamma_2. \)

\item
The remaining requirements for \( C \) being a coherence-3-system are trivially satisfied, since \( C \) is closed under \( * \) and \( \circ \) and contains units and inverses.
\end{itemize}
\end{proof}

\begin{satz} \label{prop_def_tcu}
Let \( T \) be a tricategory and \( C \subseteq \mmmor{\tbar} \) the coherence 3-system from proposition \ref{prop_coh_3_system_tbar}.
Then \( \tcu := \tbar/\scriptstyle{C} \) is a cubical tricategory.
\end{satz}

\begin{proof}~
\begin{itemize}[-]
\item
By definition of the constraint cells of \( \tcu \) (compare proposition \ref{prop_coh_3_system}) and by lemma \ref{lem_C_under_quotient} we have
\( l ^{loc} _{[\bar \alpha]} = [\bar l ^{loc} _{\bar \alpha}] = 1 _{[\bar \alpha]}. \)
Similarly all components of \( r ^{loc} \) and \( a ^{loc} \) are identities in \( \tcu. \)
Thus the local bicategories of \( \tcu \) are strict.

\item
Since \( \bar \phi_a = \barf \tilde \phi_a \) is an identity in \( \tbar \) by lemma \ref{lem_phi_a_and_phi_fg_are_identities}, \( \phi_{[a]} = [\bar \phi_a] \) is an identity in \( \cat{\hat T} ^{cu}. \)

The compatibility-with-composition constraint for the weak functor \( I_{[a]} \) consists of a single component \( i_{[a]} * i_{[a]} \to i_{[a]}. \)
It follows from the functor axioms \eqref{diag_functor_axiom_unit} in definition \ref{def_weak_functor} that this component is given by
\[
i_{[a]} * i_{[a]}
\xrightarrow{\phi_{[a]} ^{-1} * 1}
1 _{(1_{[a]})} * i_a
\xrightarrow{l ^{loc} _{i_{[a]}}}
i_{[a]} .
\]
The first morphism in this composite is an identity, since all components of \( l ^{loc} \) are identities.
The second one is an identity, because \( \phi_{[a]} \) is an identity.
Therefore the composite is an identity and \( I_{[a]} \) is cubical.

\item
That the comparison-cells for units of the weak functors \( \otimes \) are identities follows by unraveling definitions and applying lemma \( \ref{lem_phi_a_and_phi_fg_are_identities}: \)
\[
\phi ^{\otimes} _{([\hat f], [\hat g])} = [\bar \phi ^{\otimes} _{(\hat f, \hat g)}]
 = [\barf \tilde \phi ^{\otimes} _{(\hat f, \hat g)}] = [\barf \tilde 1 ^{\otimes} _{(\hat f, \hat g)}]
 = [\bar 1 ^{\otimes} _{(\hat f, \hat g)}] = 1 _{[\hat f, \hat g]}.
\]

Interchangers of the form
\(
\phi _{([\bar{1}], [\bar{\beta}])([\bar{\alpha}'],[\bar{\beta}'])}
\)
and
\(
\phi _{([\bar{\alpha}], [\bar{\beta}])([\bar{\alpha}'],[\bar{1}])}
\)
are identities, because
\(
\bar \phi _{(\bar{1}, \bar{\beta})(\bar{\alpha}',\bar{\beta}')}
\)
and
\(
\bar \phi _{(\bar{\alpha}, \bar{\beta})(\bar{\alpha}',\bar{1})}
\)
are contained in \( C \) by lemma \ref{lem_interchanger_in_tbar}.
\end{itemize}
\end{proof}

Before we can continue with our strictification, we need a more explicit description of the tricategory \( \tcu. \)
Since \( \tcu \) was obtained by quotienting out local constraints from \( \tbar, \)
the 2-morphisms of \( \tcu \) are obtained from the 2-morphisms of \( \tbar, \) by omitting 2-units and brackets in \( * \)-composites.
This is made precise in the following.

As \( [-] \) acts injectivly on \( \tbar \)-basic 2-morphisms we can consider the set of \( \tbar \)-basic 2-morphisms a subset of \( \mmor{\tcu}. \)
Therefore the following convention is reasonable.

\begin{teano}
We call the image of a \( \tbar \)-basic 2-morphism \( \bar \alpha \in \mmor{\tbar} \) under \( [-] \) a \( \tbar \)-basic 2-morphism of \( \tcu \)
and write simply \( \bar \alpha \in \mmor{\tcu} \) instead of \( [\bar \alpha]. \)
\end{teano}

Since any 2-morphism in \( \tbar \) is a \( \hs \)-composite of \( \tbar \)-basic 2-morphisms, any 2-morphism in \( \tcu \) is a \( * \)-composite of \( \tbar \)-basic 2-morphisms as well.
Because \( \tcu \) is locally strict, we can omit brackets in such a \( * \)-composite.
Thus if \( \alpha \) is a 2-morphism in \( \tbar, \) there are \( \tbar \)-basic 2-morphisms \( \bar \alpha_1, \dots, \bar \alpha_n \) such that
\( \alpha = \bar \alpha_1 * \dots * \bar \alpha_n = \comseq{*}{(\bar \alpha_i)_n}. \)
In the following lemma we clarify to what extent the sequence \( (\bar \alpha_i)_n \) is unique.

\begin{lem} [decomposing 2-morphisms in \( \tcu \)]~ \label{lem_tcu_2_mor_decomp}
\begin{enumerate}
\item
Let	 \( (\bar \alpha_i)_n \) and \( (\bar \beta_j)_m \) be two sequences of \( \tbar \)-basic 2-morphisms in \( \tcu. \)
Then \( \comseq{*}{(\bar \alpha_i)}_n = \comseq{*}{(\bar \beta_j)}_m \) if and only if the maximal subsequences of \( (\bar \alpha_i)_n \) and \( (\bar \beta_j)_m \) which contain only non-units are equal.
\item
For any 2-morphism \( \alpha \in \mmor{\tcu} \) there exists a unique \( n \in \mathbb{N} \) and a unique sequence of \( \tbar \)-basic 2-morphisms \( (\bar \alpha_i)_n \)
such that \( n \) is minimal and \( \alpha = \comseq{*}{(\bar \alpha_i)}_n. \)
\item \label{list_C2_partition}
For any 2-morphism \( \alpha \in \mmor{ \tcu} \) there exists a unique odd \( n \in \mathbb{N} \) and a sequence \( (\alpha_i)_{i \in \set{1 \dots n}} \) in \( \mmor{\tcu} \) such that
\begin{enumerate}[(i)]
\item
\( \alpha_i \in \mmor{\tcu} \) is a coherence 2-morphism if \( i \) is odd,
\item
\( \alpha_i \) is a \( \tbar \)-basic 2-morphism if \( i \) is even,
\item
\( \alpha = \comseq{*}{(\bar \alpha_i)}_n. \)
\end{enumerate}
\end{enumerate}
\end{lem}

\begin{proof}
For arbitrary bracketings \( \sigma,\tau \) the 2-morphism \( \comseq{\hs}{(\bar \alpha_i)}_n^{\sigma} \) and \( \comseq{\hs}{(\bar \beta_i)}_{m}^{\tau} \) in \( \tbar \)
are lifts of the 2-morphisms \( \comseq{*}{(\bar \alpha_i)}_n \) and \( \comseq{*}{(\bar \beta_j)}_m \) from (1).
Then \( \comseq{*}{(\bar \alpha_i)}_n = \comseq{*}{(\bar \beta_j)}_m \) if and only if
\( \comseq{\hs}{(\bar \alpha_i)}_n^{\sigma} \) is related to \( \comseq{\hs}{(\bar \beta_i)}_{m}^{\tau} \) in \( \tbar \) morphisms in \( C, \) i.e. via local constraints.
This is the case if and only if the maximal subsequences of \( (\bar \alpha_i)_n \) and \( (\bar \beta_j)_m \) which contain only non-units are equal.
Thus we have shown (1).
(2) and (3) follow easily from (1).
\end{proof}

Next we want to quotient out the coherence 3-morphisms in \( \tcu, \) which run between parallel coherence 2-morphisms.
Therefore one would like to define a coherence 3-system \( C \subseteq \mmmor{\tcu} \), which contains all \( (\otimes, *, \circ) \)-composites of such coherence 3-morphisms and units.
But for this definition of \( C \) it is quite hard to show that parallel 3-morphisms in \( C \) are equal.
Thus we restrict \( C \) further, by additionally requiring a certain lifting property for its morphisms.
This additional condition will make it easy to show uniqueness of parallel morphisms in \( C, \) for the price that it is no longer obvious that \( C \) is closed under composition.
In order to state this lifting property, we need to introduce some terminology, which is based on the previous lemma.

\begin{terminology}
We call the sequence \( (\alpha_i)_n \) from lemma \ref{lem_tcu_2_mor_decomp} \eqref{list_C2_partition} the \emph{ \( C_2 \)-block partition} of \( \alpha. \)
\end{terminology}

\begin{terminology} \label{term_standard_lift}
Let \( \alpha = \bar \alpha_1 * \dots * \bar \alpha_n \) be a general 2-morphism in \( \tcu, \) where \( \bar \alpha_i \) are \( \tbar \)-basic 2-morphisms and \( n \) is minimal.
We call the left bracketed morphism
\[
\big( \dots \big( (\bar \alpha_1 \hs \bar \alpha_2) \hs \bar \alpha_3 \big) \dots \hs \bar \alpha_n \big) \in \mmor{\tbar}
\]
the \emph{standard lift} of \( \alpha. \)
\end{terminology}

\begin{satz} \label{prop_coherence_3_system_tcu}
Let \( \cat{T} \) be a tricategory and \( \tcu \) the cubical tricategory from proposition \ref{prop_def_tcu}.
We consider a subset \( C \subseteq \mmmor{\tcu} \) defined as follows:

Let \( \Pi: \alpha \to \beta \) be a 3-morphism in \( \tcu \) and let \( (\alpha_i)_n \) resp. \( (\beta_i)_m \) be the \( C_2 \)-block-partition of \( \alpha \) resp. \( \beta. \)
Then \( \Pi \) is contained in \( C \) if and only if \( n=m \) and there exists a sequence \( (\Pi_i: \alpha_i \to \beta_i) _{i \in \set{1 \dots n}} \) such that
\begin{enumerate}[(i)]
\item \label{list_coh_type}
For \( i \) odd, \( \Pi_i: \alpha_i \to \beta_i \) lifts to the unique coherence in \( \that \) between the standard lifts of \( \alpha_i \) and \( \beta_i. \)
\item \label{list_id_type}
For \( i \) even \( \Pi_i: \alpha_i \to \beta_i \) is an identity.
\item
\( \Pi = \comseq{*}{(\Pi_i)}_n. \)
\end{enumerate}

Then \( C \) is a coherence 3-system.
\end{satz}

\begin{proof}
In the following we will refer to the sequence \( (\Pi_i)_n \) from the proposition as the \( C \)-partition of \( \Pi. \)
(By definition any 3-morphism in \( C \) has such a \( C \)-partition.)
\begin{itemize}[-]
\item
Parallel morphisms in \( C \) are equal, since the \( C \)-partition of a morphism in \( C \) is uniquely determined by its types.
\item
It is easy to show that \( C \) is closed under \( \circ, \) by using the interchange law for local bicategories
and observing that composites of 3-morphisms of type \ref{list_coh_type} resp. type \ref{list_id_type} (in the list of the proposition) are again of type \ref{list_coh_type} resp. type \ref{list_id_type}.
\item
It is straightforward to show that units and inverses are contained in \( C. \)
\item
Now we show that \( C \) is closed under \( *. \)
For this, let \( \Pi \) and \( \Lambda \) be \( * \)-composable 3-morphisms in \( C \) and let \( (\Pi_i: \alpha_i \to \alpha'_i)_n \) and \( (\Lambda_j: \beta_j \to \beta'_j)_m \) be their \( C \)-partitions.
We claim that
\[
(\Pi_1, \dots, \Pi _{n-1}, \Pi_n * \Lambda_1,  \Lambda_2, \dots, \Lambda_m)
\]
is a \( C \)-partition of \( \Pi * \Lambda. \)
The only non-trivial task in proving this claim is to show that the coherence 3-morphism \( \Pi_n * \Lambda_1 \) (whose types are coherence 2-morphisms)
lifts to the unique coherence 3-morphism in \( \that, \) whose types are the standard lifts of the types of \( \Pi_n * \Lambda_1. \)

Thus let \( \bar \gamma \) resp. \( \bar \gamma' \) be the standard lift of \( \alpha_n * \beta_1 \) resp. \( \alpha'_n * \beta'_1. \)
Since \( \bar \gamma \) and \( \bar \gamma' \) are parallel coherence 2-morphisms, there is a unique coherence 3-morphism \( \hat \Gamma: \bar \gamma \to \bar \gamma' \) in \( \that. \)
Similarly let \( \bar \alpha_n, \bar \alpha'_n, \bar \beta_1, \bar \beta'_1 \) be the standard lifts of \( \alpha_n, \alpha'_n, \beta_1, \beta'_1 \) and
\[
\hat \Gamma_n: \bar \alpha_n \to \bar \alpha'_n,
\quad
\hat \Gamma_1: \bar \beta_1 \to \beta'_1
\]
be the associated unique coherence 3-morphisms in \( \that. \)
Then the following diagram, in which the vertical arrows are coherence-morphisms built from local associator constraint cells, commutes because \( \that \) is semi-free
\[
\begin{tikzcd}[column sep = large]
\bar \alpha_n * \bar \beta_1
\ar[d, "\sim" {anchor=north, sloped}]
\ar[r, "\hat \Gamma_n * \hat \Gamma_1"]
	& \bar \alpha'_n * \bar \beta'_1
	\ar[d, "\sim" {anchor=south, sloped}] \\
\bar \gamma
\ar[r, "\hat \Gamma"]
	& \bar \gamma'.
\end{tikzcd}
\]
Now let us write \( F \) for the virtually strict triequivalence
\[
\that \xrightarrow{\ftilde} \ttil \xrightarrow{\barf} \tbar \xrightarrow{[-]} \tcu.
\]
We know that \( F (\hat \Gamma_n) = \Pi \) and \( F (\hat \Gamma_1) = \Lambda. \)
If we additionally take into account that the vertical arrows in the commutative diagram above are mapped to identities by \( F, \) we have
\[
\Pi_n * \Lambda_1 = F (\hat \Gamma_n \hs \hat \Gamma_1) = F(\hat \Gamma).
\]
Therefore we have shown that \( \Pi_n * \Lambda_1 \) lifts to the unique coherence 3-morphism between the standard lifts of its source and target.

\item
Now we show that \( C \) is closed under \( \otimes. \)
For this, let \( \Pi \) and \( \Lambda \) be \( \otimes \)-composable 3-morphisms in \( C \) with \( C \)-partitions \( (\Pi_i: \alpha_i \to \alpha'_i)_n \) and \( (\Lambda_j: \beta_j \to \beta'_j)_m. \)

Now we consider the following diagram in \( \tcu, \) where we write \( (1)_n^* \) for the n-fold product \( 1 * 1 * \dots * 1: \)
\[
\begin{tikzcd}[column sep=huge, row sep=large]
(\alpha_i)_n^* \otimes (\beta_j)_m^*
\ar[r, "(\Pi_i)_n^* \otimes (\Lambda_j)_m^*"]
\ar[d, "\sim" {anchor=north, sloped}]
	& (\alpha'_i)_n^* \otimes (\beta'_j)_m^*
	\ar[d, "\sim" {anchor=south, sloped}] \\
\big( (1)_n^* * (\alpha_i)_n^* \big) \otimes \big( (\beta_j)_m^* * (1)_m^* \big)
\ar[r, "\big( (1)_m ^{*} * (\Pi_i)_n^* \big) \otimes \big( (\Lambda_j)_m^* * (1)_n ^{*} \big)" {yshift=4.5mm, name=s1}, "" {name=t1}]
\ar[d, "\sim" {anchor=north, sloped}]
\ar[from=s1, to=t1, dash]
	& \big( (1)_n^* * (\alpha_i)_n^* \big) \otimes \big( (\beta_j)_m^* * (1)_m^* \big)
	\ar[d, "\sim" {anchor=south, sloped}] \\
( 1 \otimes \beta_j)_m ^{*} * (\alpha_i \otimes 1)_n ^{*}
\ar[r, "(1 \otimes \Lambda_j)_m ^{*} * (\Pi_i \otimes 1)_n ^{*}"]
	& ( 1 \otimes \beta'_j)_m ^{*} * (\alpha'_i \otimes 1)_n ^{*}
\end{tikzcd}
\]
The vertical arrows in the upper square are of the form \( \Gamma_1 \otimes \Gamma_2, \) where \( \Gamma_1 \) and \( \Gamma_2 \) are \( * \)-composites of local unit constraints.
Due to the naturality of these constraints the upper square commutes.
Because local units are identities in the cubical tricategory \( \tcu, \) the vertical arrows in the upper square are identities.

The vertical arrows in the lower square can be constructed as a \( * \)-composite of interchangers of the form \( \phi ^{\otimes} _{(\alpha,\beta)(\alpha',\beta')}, \)
where \( \alpha \) or \( \beta' \) is a unit.
Then the lower square commutes due to the naturality of the interchangers.
As \( \otimes \) is a cubical functor, these interchangers are units and therefore the vertical arrows in the lower square are identities.

Hence we have shown that \( \Pi \otimes \Lambda = (1 \otimes \Lambda_j)_m ^{*} * (\Pi_i \otimes 1)_n ^{*}. \)
As we already established that \( C \) is closed under \( *, \)
it suffices to show that the morphisms \( 1 \otimes \Lambda_j \) and \( \Pi_i \otimes 1 \) are contained in \( C. \)

For \( i \) even, \( \Pi_i \otimes 1 \) is an identity by assumption \ref{list_i_even_C_part} and therefore contained in \( C. \)
Now let \( i \) be odd.
Then \( \alpha_i \) resp. \( \alpha'_i \) are coherence 2-morphisms of the form
\( \bar \alpha_i = \comseq{*}{(\bar \alpha _{ij})} _{j \in \set{1 \dots k}} \) resp. \( \bar \alpha'_i = \comseq{*}{(\bar \alpha' _{ij})} _{j \in \set{1 \dots k'}} \),
where \( \bar \alpha _{ij} \) resp. \( \bar \alpha' _{ij} \) are \( \tbar \)-basic coherence 2-morphisms and \( k \) resp. \( k' \) is minimal.
We know that \( \Pi_i \) lifts to a unique coherence isomorphism
\[
\hat \Gamma: \comseq{*}{(\alpha _{ij})} _{j \in \set{1 \dots k}} ^{\sigma^l} \to \comseq{*}{(\alpha' _{ij})} _{j \in \set{1 \dots k'}} ^{\sigma^l} \in \mmmor{\that},
\]
where \( \sigma ^{l} \) denotes the left bracketing.
We then consider the following diagram in \( \that. \)
\[
\begin{tikzcd}
\comseq{*}{(\alpha _{ij})} _{j \in \set{1 \dots k}} ^{\sigma^l} \otimes \bar 1
\ar[r, "\hat \Gamma \otimes 1"]
\ar[d]
	& \comseq{*}{(\alpha _{ij})} _{j \in \set{1 \dots k'}} ^{\sigma^l} \otimes \bar 1
	\ar[d] \\
\comseq{*}{(\alpha _{ij})} _{j \in \set{1 \dots k}} ^{\sigma^l} \otimes \comseq{*}{(\bar 1)} ^{\sigma ^{l}} _k
\ar[d]
	& \comseq{*}{(\alpha _{ij})} _{j \in \set{1 \dots k'}} ^{\sigma^l} \otimes \comseq{*}{(\bar 1)} ^{\sigma ^{l}} _{k'}
	\ar[d] \\
\comseq{*}{(\alpha _{ij} \otimes \bar 1)} _{j \in \set{1 \dots k}} ^{\sigma^l}
\ar[d]
	& \comseq{*}{(\alpha _{ij} \otimes \bar 1)} _{j \in \set{1 \dots k'}} ^{\sigma^l}
	\ar[d] \\
\comseq{*}{(\alpha _{ij} \otimes \tilde 1)} _{j \in \set{1 \dots k}} ^{\sigma^l}
\ar[r]
	& \comseq{*}{(\alpha _{ij} \otimes \tilde 1)} _{j \in \set{1 \dots k'}} ^{\sigma^l} \\
\end{tikzcd}
\]
All vertices are coherence 2-morphisms and all arrows are given as unique coherence 3-morphisms (note that \( \that \) is semi-free).

We know that \( \Pi_i \otimes 1 \) lifts to the upper horizontal morphism and we have to show that \( \Pi_i \otimes 1 \) lifts to the bottom horizontal morphism, which is the unique coherence 3-morphism between the standard lifts of the types of \( \Pi_i \otimes 1. \)
For this, it is sufficient to show that all vertical arrows in the diagram are mapped to identities by the virtually strict triequivalence
\( \that \xrightarrow{\ftilde} \ttil \xrightarrow{\barf} \tbar \xrightarrow{[-]} \tcu. \)

First we observe that all vertical arrows in the diagram can be constructed from constraint cells which are units in a cubical tricategory.
The vertical arrows in the upper square are constructed from local units.
The vertical arrows in the middle square are constructed from interchangers of the form \( \hat \phi ^{\otimes} _{(\alpha,\beta)(\gamma,1)}. \)
The vertical arrows in the lower square are constructed from unit comparison cells of \( \otimes, \) i.e. cells of the form \( \hat \phi ^{\otimes} _{(\hat f, \hat g)}. \)

Now the preservation of constraint cells by virtually strict functors and the fact that \( \tcu \) is cubical implies that the vertical arrows in the diagram are mapped to identities by the virtually strict functor
\(
\that \xrightarrow{\ftilde} \ttil \xrightarrow{\barf} \tbar \xrightarrow{[-]} \tcu.
\)
\end{itemize}
\end{proof}

\begin{satz} \label{prop_tprime_def}
Let \( \cat{T} \) be a tricategory and \( C \subseteq \mmmor{\tcu} \) the coherence 3-system from proposition \ref{prop_coherence_3_system_tcu}.
Then the tricategory \( \tprime := \cat{ \tcu}/\scriptstyle{C} \) is a semi-free cubical tricategory, in which the unique coherence isomorphism between parallel coherence 2-morphism is an identity.
\end{satz}

\begin{proof}
\( \tprime \) is cubical, because \( \tcu \) is cubical and because any constraint cell in \( \tprime \) is the image of a corresponding constraint cell in \( \tcu \)
under the virtually strict triequivalence \( [-]: \tcu \to \tprime = \cat{ \tcu}/\scriptstyle{C}. \)
Next, we observe that all constraint 3-cells in \( \tcu \) whose types are coherence 2-morphisms are contained in \( C. \)
Therefore all constraint 3-cells in \( \tprime \) whose types are coherence 2-morphisms are identities.
Any coherence 3-morphism whose types are coherence 2-morphisms is a \( (\otimes,*,\circ) \)-composite of such constraint 3-cells and hence an identity as well.

Parallel coherence 2-morphisms in \( \tprime \) lift to parallel coherence 2-morphisms in \( \that, \) where they are coherently isomorphic (=isomorphic via a coherence 3-morphism).
Thus they are coherently isomorphic in \( \tprime \) as well, since coherent isomorphy is preserved under virtually strict triequivalences.
\end{proof}

Before we continue with our strictification, we need a more explicit description of the tricategory \( \tprime. \)
The virtually strict triequivalence \( \that \to \tprime \) we have constructed so far acts trivially on 1-morphisms.
Therefore it makes sense to consider \( \that \)-basic 1-morphisms in \( \tprime. \)
Also the notions of \( \tbar \)- and \( \ttil \)-basic 2-morphisms make still sense in \( \tprime, \) if one restricts their meaning to non-coherent 2-morphisms
(i.e. 2-morphisms which are no coherence 2-morphisms).

\begin{terminology}~ \label{term_basic_morphisms_in_tprime}
\begin{enumerate}
\item
A 1-morphism in \( \tprime \) is called \( \that \)-basic, if it lifts to a (uniquely determined) \( \that \)-basic 1-morphism in \( \that. \)
\item
A 2-morphism in \( \tprime \) is called \( \tbar \)-basic, if it lifts to a (uniquely determined) non-coherent \( \tbar \)-basic 2-morphism in \( \tcu. \)
\item
A \( \tbar \)-basic 2-morphism in \( \tprime \) is called \( \ttil \)-basic, if it lifts to a (uniquely determined) 2-morphism in \( \ttil. \)
\end{enumerate}
\end{terminology}

\begin{lem}[characterization of 2-morphisms in \( \tprime \)] \label{lem_characterization_2_mor_tprime}
Let \( \alpha \in \mmor{\tprime} \) then there exists a unique odd \( n \in \mathbb{N} \) and a unique sequence \( (\alpha_i)_n \) in \( \mmor{\tprime} \), such that
\begin{enumerate}[(i)]
\item
For \( i \) odd \( \alpha_i \) is a coherence 2-morphism.
\item
For \( i \) even \( \alpha_i \) is a \( \tbar \)-basic 2-morphism.
\item
\( \alpha = \comseq{*}{(\alpha_i)}_n \).
\end{enumerate}
\end{lem}

\begin{terminology} \label{term_tprime_partition}
Let \( \alpha \) be a 2-morphism in \( \tprime. \)
Then we call the sequence \( (\alpha_i)_n \) from lemma \ref{lem_characterization_2_mor_tprime} the \( \tprime \)-partition of \( \alpha. \)
\end{terminology}

In the last step of our strictification we will quotient out a coherence-(2,3)-system from the tricategory \( \tprime \) by means of proposition \ref{prop_coherence_23_system}.
The coherence-(2,3)-system \( C_3 \subseteq \mmmor{\tprime} \) we need to define for that purpose contains in particular all components of \( a, l \) and \( r. \)
Though, in order to be able to show uniqueness of parallel morphisms in \( C_3, \) it is beneficial to define \( C_3 \) in a rather indirect way.
Roughly speaking we say that \( C_3 \) contains such coherence 3-morphisms which manipulate the bracketing of \( \tbar \)-basic 2-morphisms in a certain "coherent way".
That means that the morphisms in \( C_3 \) need to lift in a certain way to 3-morphisms in \( \oper{F} \tprime. \)
In order to make that precise we first need to introduce some terminology.

\begin{mydef}~
\begin{enumerate}
\item
For a \( \that \)-basic 1-morphism \( f \in \tprime \) we define \( \sfl f \) as the following 1-morphism in \( \oper{F} \tprime: \)
\[
\sfl f :=
\begin{cases}
1_a ^{\sms{F}}\text{ (the unit at \( a \) in \( \oper{F} \tprime \)),}
	& \text{if } f = 1_a \\
f \text{ (as a morphism in \( \oper{F} \tprime \)),}
	& \text{if } f \text{ is not a unit}.
\end{cases}
\]
For a general 1-morphism \( g = \comseq{\otimes}{(f_i)} _{i \in 1 \dots n} ^{\sigma} \in \mor{ \tprime}, \) where \( f_i \) are \( \that \)-basic 1-morphisms, we define
\[
\sfl g := \comseq{\fotimes}{(\sfl f_i)} _{i \in 1 \dots n} ^{\sigma} \in \mor{\oper{F} \tprime}.
\]

\item
We consider an arbitrary \( \tbar \)-basic 2-morphism
\[
\alpha =
(1 _{f_1} \otimes \dots \otimes 1 _{f_k}
\otimes \beta \otimes
1 _{f_{k+1}} \otimes \dots \otimes 1 _{f_n}) ^{\sigma}
\]
in \( \tprime, \) where \( f_i \) are \( \that \)-basic 1-morphisms and \( \beta \) is a \( \ttil \)-basic 2-morphism in \( \tprime. \)
Then we define \( \sfl{\alpha} \) as the following 2-morphism in \( \oper{F} \tprime: \)
\[
(\sfl{\alpha} =
1 ^{\sms{F}} _{\sfl f_1} \fotimes \dots \fotimes 1 ^{\sms{F}} _{\sfl f_k}
\fotimes \beta \fotimes
1 ^{\sms{F}} _{\sfl f_{k+1}} \fotimes \dots \fotimes 1 ^{\sms{F}} _{\sfl f_n}) ^{\sigma}
\]

\item
For \( g = \comseq{\otimes}{(f_i)}_{n}^{\sigma} \in \mor{\tprime}, \) with \( f_i \) \( \that \)-basic, we define
\[
\sfl 1_g = \comseq{\fotimes}{(1 _{\sfl f_i}^{\sms{F}})}_{n}^{\sigma} \in \mmor{\oper{F} \tprime}
\]
\end{enumerate}
\end{mydef}

\begin{bem}
Note that (2) and (3) are different cases of the "dot notation" for 2-morphisms,
as \( 1_g \in \mmor{\tprime} \) is by terminology \ref{term_basic_morphisms_in_tprime} not a \( \tbar \) basic 2-morphism \emph{in \( \tprime \).}
For a \( \tbar \)-basic 2-morphism \( \alpha = \beta \otimes 1_g \in \mmor{\tprime} \) we have \( \sfl \alpha = \sfl \beta \otimes \sfl 1_g. \)
\end{bem}

\begin{terminology}
Let \( \alpha_1 \) and \( \alpha_2 \) be \( \tbar \)-basic 2-morphisms.
We say \( \alpha_1 \) and \( \alpha_2 \) are coherence related if there exist coherence 2-morphisms \( \gamma, \gamma' \in \mmor{F \tprime} \)
and a coherence 3-morphism \( \Gamma \in \mmmor{\operatorname{F} \tprime } \) with
\[
\sfl \alpha_1 \xrightarrow{\phantom{x} \Gamma \phantom{x}} (\gamma * \sfl \alpha_2) * \gamma'.
\]
We call \( \epsilon (\Gamma) \) a coherence relator of \( \alpha_1 \) and \( \alpha_2. \)
\end{terminology}

\begin{lem}
Any two coherence relators of coherence related 2-morphisms \( \alpha_1 \) and \( \alpha_2 \) are equal.
\end{lem}

\begin{proof}
Suppose we have coherence 2-morphisms
\( \gamma_1, \gamma_2, \gamma_3, \gamma_4 \in \mmor{F \tprime} \) and coherence 3-morphisms
\[
\sfl \alpha_1 \xrightarrow{\phantom{x} \Gamma_1 \phantom{x}} (\gamma_1 * \sfl \alpha_2) * \gamma_2
\quad \text{and} \quad
\sfl \alpha_1 \xrightarrow{\phantom{x} \Gamma_2 \phantom{x}} (\gamma_3 * \sfl \alpha_2) * \gamma_4.
\]
As \( \gamma_1 \) and \( \gamma_3 \) resp. \( \gamma_2 \) and \( \gamma_4 \) are parallel coherence 2-morphisms in \( \oper{F} \tprime, \)
there exist unique coherence 3-morphisms \( \Gamma_3: \gamma_1 \to \gamma_3 \) and \( \Gamma_4: \gamma_2 \to \gamma_4. \)
By the coherence theorem, the diagram
\[
\begin{tikzcd}
\sfl \alpha_1
\ar[r, "\Gamma_1"]
\ar[d, "1"]
	& (\gamma_1 * \sfl \alpha_2) * \gamma_2
	\ar[d, "(\Gamma_3 * 1) * \Gamma_4"] \\
\sfl \alpha_1
\ar[r, "\Gamma_2"]
	& (\gamma_3 * \sfl \alpha_2) * \gamma_4
\end{tikzcd}
\]
commutes, and since \( \Gamma_3 \) and \( \Gamma_4 \) are mapped to identities by \( \epsilon, \) we have shown that \( \epsilon(\Gamma_1) = \epsilon(\Gamma_2). \)
\end{proof}

\begin{satz} \label{prop_tprime_coh_23_system}
Let \( \cat{T} \) be a tricategory, and \( \tprime \) be the associated tricategory from proposition \ref{prop_tprime_def}.
We consider a subset \( C' \subseteq \mmmor{\tprime} \) defined as follows:

Let \( \Pi: \alpha \to \beta \) be a 3-morphism in \( \tprime \)
and let \( (\alpha_i)_n \) and \( (\beta_i)_m \) be the \( \tprime \)-partitions (see terminology \ref{term_tprime_partition}) for \( \alpha \) and \( \beta. \)
Then \( \Pi \) is contained in \( C_3 \) if and only if \( n=m \) and there exists a sequence \( (\Pi_i)_n \) such that
\begin{enumerate}[(i)]
\item \label{list_source_C_part}
\( \Pi_i \) has source \( \alpha_i. \)
\item \label{list_i_odd_C_part}
For \( i \) odd, \( \Pi_i \) is an identity.
\item \label{list_i_even_C_part}
For \( i \) even, \( \Pi_i \) is a coherence relator of \( \alpha_i \) and \( \beta_i. \)
\item \label{list_C_part_decomposes_Pi}
\( \Pi = \comseq{*}{(\Pi_i)}_n. \)
\end{enumerate}
Then \( C_3 \) is a coherence-(2,3)-system.
\end{satz}

\begin{proof}

In the following we call the (uniquely determined) decomposition \( (\Pi_i)_n \) of \( \Pi \) described in the proposition the \( C_3 \)-partition of \( \Pi. \)
\begin{itemize}[-]
\item
Parallel 3-morphism in \( C_3 \) are equal, since a morphism in \( C_3 \) is uniquely determined by its types.
This follows directly from the uniqueness of the \( \tprime \)-partition and from the uniqueness of the coherence relators \( \Gamma_i. \)
\item
That units are contained in \( C_3 \) follows from the observation that units on \( \tbar \)-basic 2-morphisms are coherence relators.

\item
We show that \( C_3 \) is closed under \( \circ. \)
For that let \( \Pi' \) and \( \Pi \) be \( \circ \)-composable 3-morphisms in \( C_3, \) with \( C_3 \)-partitions \( (\Pi'_i)_n \) resp. \( (\Pi_i)_n. \)
Furthermore let \( (\alpha_i)_n, (\beta_i)_n \) and\( (\beta'_i)_n \) be the \( \tprime \)-partitions of \( s \Pi, t \Pi = s \Pi' \) and \( t \Pi'. \)

Furthermore let for \( i \) even, \( \gamma_i \) and \( \gamma'_i \) be the unique coherence 2-morphisms \( t \alpha _{i} \to s \beta_i \) and \( s \beta_i \to s \alpha _{i} \) in \( \tprime. \)
In other words for \( i \) even, \( \gamma_i \) and \( \gamma'_i \) are the coherence 2-morphisms in the the target morphism of the coherence relator \( \Pi_i: \alpha_i \to \gamma'_i * \beta_i * \gamma_i. \)

Now we define \( \Lambda_i := 1 _{\alpha_i} \) for \( i \) odd and \( \Lambda_i := (1 _{\gamma'_i} * \Pi' * 1 _{\gamma_i}) \circ \Pi_i \) for \( i \) even.
Then we claim \( (\Lambda_i)_n \) is the \( C_3 \)-partition of \( \Pi' \circ \Pi. \)

Condition \ref{list_source_C_part} and \ref{list_i_odd_C_part} of the proposition are trivially satisfied
and condition \ref{list_C_part_decomposes_Pi} namely the identity \( \comseq{*}{(\Lambda_i)}_n = \Pi' \circ \Pi \) is easy to check.
It is left to show that for \( i \) even, \( \Lambda_i := (1 _{\gamma'_i} * \Pi' * 1 _{\gamma_i}) \circ \Pi_i \) is a coherence relator.
For that let \( \mu'_i, \mu_i, \nu'_i, \nu_i \) be coherence 2-morphisms in \( \oper{F} \tprime \), such that the unique coherences
\[
\sfl \alpha_i \rightarrow (\mu'_i * \sfl \beta_i) * \mu_i
\quad \text{and} \quad
\sfl \beta_i \rightarrow (\nu'_i * \sfl \beta_i) * \nu_i
\]
are lifts of the coherence relators \( \Pi_i \) and \( \Pi'_i. \)

Now we consider the following commutative diagram:
\[
\begin{tikzcd}
\sfl \alpha_i
\ar[r]
\ar[d]
	& (\mu'_i * \sfl \beta_i) * \mu_i
	\ar[r]
		& \big( \mu'_i * ((\nu'_i * \sfl \beta_i) * \nu_i)) \big) * \mu_i
		\ar[d] \\
\sfl \alpha_i
\ar[rr]
		& & \big( (\mu'_i * \nu'_i) * \sfl \beta_i \big) * (\nu_i  * \mu_i).
\end{tikzcd}
\]
By the definition of a coherence relator, \( \epsilon \) sends the bottom morphism in the diagram to the coherence relator of \( \alpha_i \) and \( \beta'_i. \)
As the horizontal arrows are sent to identities by \( \epsilon \) and the top morphism is by construction a lift of \( (1 _{\gamma'_i} * \Pi' * 1 _{\gamma_i}) \circ \Pi_i = \Lambda_i, \)
the 3-morphism \( \Lambda_i \) is a coherence relator of \( \alpha_i \) and \( \beta'_i. \)

\item
\( C_3 \) is closed under \( *: \)
If \( (\Pi_i)_n \) and \( (\Lambda_i)_m \) are the \( C_3 \)-partitions of \( * \) composable 2-morphisms \( \Pi \) and \( \Lambda, \)
the sequence \( (\Pi_1, \dots ,\Pi _{n-1}, \Pi_n * \Lambda_1, \Lambda_2, \dots, \Lambda_m \) is a \( C_3 \)-partition of \( \Pi * \Lambda. \)

\item
Now we show that \( C_3 \) is closed under \( \otimes. \)
For this let \( \Pi \) and \( \Lambda \) be \( \otimes \)-composable 3-morphisms in \( C_3 \) with \( C_3 \)-partitions \( (\Pi_i)_n \) and \( (\Lambda_j)_m. \)

An analogous argument as in the proof of proposition \ref{prop_coherence_3_system_tcu} shows that \( \Pi \otimes \Lambda = (1 \otimes \Lambda_j)_m ^{*} * (\Pi_i \otimes 1)_n ^{*}. \)
Since we already know that \( C_3 \) is closed under \( *, \) it suffices to show that the morphisms \( 1 \otimes \Lambda_i \) and \( \Pi_i \otimes 1 \) are contained in \( C_3. \)
This is trivially true for odd \( i. \)
For even \( i \) we show that \( \Pi_i \otimes 1 \) is a coherence relator (and therefore contained in \( C_3 \)).
As \( \Pi_i \) is a coherence relator of \( \tbar \)-basic 2-morphisms - say \( \alpha \) and \( \alpha' \) - there exist coherence 2-morphisms \( \mu, \mu' \) and a coherence 3-morphism
\( \sfl \alpha \rightarrow (\mu * \sfl \alpha') * \mu' \) in \( \oper{F} \tprime, \) which is a lift of \( \Pi_i. \)
Now we consider the following composite of coherence 3-morphisms in \( \oper{F} \tprime, \) where \( v := s ^{2} (\Lambda). \)
\begin{align*}
\sfl \alpha \fotimes \sfl 1_v
\rightarrow \sfl \alpha \fotimes 1 ^{\sms{F}} _{\sfl v}
& \rightarrow \big( (\mu * \sfl \alpha') * \mu' \big) \fotimes 1 ^{\sms{F}} _{\sfl v} \\
& \rightarrow \big( (\mu * \sfl \alpha') * \mu' \big) \fotimes \big( ( 1 ^{\sms{F}} _{\sfl v} * 1 ^{\sms{F}} _{\sfl v}) * 1 ^{\sms{F}} _{\sfl v} \big) \\
& \rightarrow \big( (\mu \fotimes 1 ^{\sms{F}} _{\sfl v}) * (\sfl \alpha' \fotimes 1 ^{\sms{F}} _{\sfl v}) \big) * (\mu' \fotimes 1 ^{\sms{F}} _{\sfl v}) \\
& \rightarrow \big( (\mu \fotimes \sfl 1_v) * (\sfl \alpha' \fotimes \sfl 1_v) \big) * (\mu' \fotimes \sfl 1_v)
\end{align*}

The second arrow is a lift of \( \Pi \otimes 1 \) and all other arrows are sent to identities by \( \epsilon. \)
Therefore the whole composite is a lift of \( \Pi \otimes 1, \) but it is also a lift of the coherence relator of \( \alpha \otimes 1_v \) and \( \alpha' \otimes 1_v. \)

\item
We show \( C_3 \) contains inverses.
As \( C_3 \) is closed under \( \circ, \) contains units and parallel morphisms in \( C_3 \) are equal,
it suffices to prove that for any \( \Pi: \alpha \to \beta \) contained in \( C_3 \) there exists a morphism \( \Lambda: \beta \to \alpha \) that is contained in \( C_3 \) as well.
This follows easily from the observation that coherence relatedness is a symmetric property.
\end{itemize}

By now we have shown that \( C_3 \) is a coherence 3-system.
\( C_3 \) satisfies condition \eqref{list_C_3_is_C_2_complete} and \eqref{list_local_constraints_in_C_3} in definition \ref{def_coherence_23_system},
since in \( \tprime \) local constraints and constraint 3-morphisms between parallel constraint 2-morphisms are identities and therefore contained in \( C_3. \)
To show that \( C_3 \) is a coherence-2/3-system it is only left to show that condition \eqref{list_interchanger_in_23_coherence_system} is satisfied,
i.e. that for coherence 2-morphisms \( \gamma_1, \gamma_2 \) in \( \tprime \) the interchangers \( \phi ^{\otimes} _{(\alpha_1,\alpha_2),(\gamma_1,\gamma_2)} \) and \( \phi ^{\otimes} _{(\gamma_1,\gamma_2),(\alpha_1,\alpha_2)} \) are contained in \( C_3. \)
For that we consider the following diagram in \( \tprime, \) where the arrows labeled with "\( loc \)" are given by constraint cells which are composites of local constraint cells.
\[
\begin{tikzcd}[font=\scriptsize, column sep=large]
(\alpha_1 \otimes \alpha_2) * (\gamma_1 \otimes \gamma_2)
\ar[r, "{\phi ^{\otimes} _{(\alpha_1,\alpha_2),(\gamma_1,\gamma_2)}}"]
\ar[d, "loc" swap]
	& (\alpha_1 * \gamma_1) \otimes (\alpha_2 * \gamma_2) \\
\big( (1 * \alpha_1) \otimes (\alpha_2 * 1) \big) * \big( (1 * \gamma_1) \otimes (\gamma_2 * 1) \big)
\ar[d, "(\phi ^{\otimes}) ^{-1} \otimes (\phi ^{\otimes}) ^{-1}" swap]
	& \big( (1 * (\alpha_1 * 1)) * \gamma_1 \big) \otimes \big( (\alpha_2 * (1 * \gamma_2)) * 1 \big)
	\ar[u, "loc" swap] \\
\big( (1 \otimes \alpha_2) * (\alpha_1 \otimes 1) \big) * \big( (1 \otimes \gamma_2) * (\gamma_1 \otimes 1) \big)
\ar[d, "loc" swap]
	& \big( (1 * (\alpha_1 * 1)) \otimes (\alpha_2 * (1 * \gamma_2)) \big) * (\gamma_1 \otimes 1)
	\ar[u, "\phi ^{\otimes}" swap] \\
\big( (1 \otimes \alpha_2) * \big( (\alpha_1 \otimes 1) * (1 \otimes \gamma_2) \big) \big) * (\gamma_1 \otimes 1)
\ar[r, "(1 * \phi ^{\otimes}) * 1" swap]
	& \big( (1 \otimes \alpha_2) * \big( (\alpha_1 *  1) \otimes (1 * \gamma_2) \big) \big) * (\gamma_1 \otimes 1)
	\ar[u, "\phi ^{\otimes} * 1" swap]
\end{tikzcd}
\]
This diagram commutes by coherence.
All vertical arrows are identities, since they are either built from local constraints, or from the sort of interchangers which are identities because \( \tprime \) is cubical.
This implies
\[
\phi ^{\otimes} _{(\alpha_1,\alpha_2),(\gamma_1,\gamma_2)} = \phi ^{\otimes} _{(\alpha_1,1)(1,\gamma_2)}.
\]
For notational convenience we now set \( \beta:= \alpha_1 \) and \( \gamma:= \gamma_2. \)
Let \( (\beta_i)_n \) be the \( \tprime \)-partition of \( \beta. \)

Now we construct \( \phi ^{\otimes} _{(\beta,1),(1,\gamma)} \) as a composite of coherence morphisms that are either units in \( \tprime \)
or constructed from the interchangers \( \phi ^{\otimes} _{(\beta_i,1),(1,\gamma)}: \)

\begin{align*}
 {}&{} \big( \comseq{*}{(\beta_i)}_n ^{\sigma ^{l}} \otimes 1 \big) * (1 \otimes \gamma) \\
\rightarrow {}&{}
\big( \comseq{*}{(\beta_i)}_n ^{\sigma ^{l}} \otimes \comseq{*}{(1)}_n ^{\sigma ^{l}} \big) * (1 \otimes \gamma) \\
\rightarrow {}&{}
\comseq{*}{(\beta_i \otimes 1)}_n ^{\sigma ^{l}} * (1 \otimes \gamma) \\
= {}&{}
\big( \comseq{*}{(\beta_i \otimes 1)}_n ^{\sigma ^{l}} * (\beta_n \otimes 1) \big) * (1 \otimes \gamma) \\
\rightarrow {}&{}
\comseq{*}{(\beta_i \otimes 1)}_{n-1} ^{\sigma ^{l}} * \big( (\beta_n \otimes 1) * (1 \otimes \gamma) \big) \\
\rightarrow {}&{}
\comseq{*}{(\beta_i \otimes 1)}_{n-1} ^{\sigma ^{l}} * \big( (\beta_n * 1) \otimes (1 * \gamma) \big) \\
\rightarrow {}&{}
\comseq{*}{(\beta_i \otimes 1)}_{n-1} ^{\sigma ^{l}} * \big( (1 * \beta_n) \otimes (\gamma * 1) \big) \\
\rightarrow {}&{}
\comseq{*}{(\beta_i \otimes 1)}_{n-1} ^{\sigma ^{l}} * \big( (1 \otimes \gamma) * (\beta_n \otimes 1) \big) \\
\rightarrow {}&{}
\big( \comseq{*}{(\beta_i \otimes 1)}_{n-1} ^{\sigma ^{l}} * (1 \otimes \gamma) \big) * (\beta_n \otimes 1) \\
\rightarrow \dots \rightarrow {}&{}
(1 \otimes \gamma) * \comseq{*}{(\beta_i \otimes 1)}_n ^{\sigma ^{l}} \\
\rightarrow {}&{}
(1 \otimes \gamma) * \big( \comseq{*}{(\beta_i)}_n ^{\sigma ^{l}} \otimes \comseq{*}{(1)}_n ^{\sigma ^{l}} \big) \\
\rightarrow {}&{}
(1 \otimes \gamma) * \big( \comseq{*}{(\beta_i)}_n ^{\sigma ^{l}} \otimes 1 \big) \\
\rightarrow {}&{}
\big( 1 * \comseq{*}{(\beta_i)}_n ^{\sigma ^{l}} \big) \otimes (\gamma * 1) \\
\rightarrow {}&{}
\big( \comseq{*}{(\beta_i)}_n ^{\sigma ^{l}} * 1 \big) \otimes (1 * \gamma).
\end{align*}
Since the composite as a whole equals \( \phi ^{\otimes} _{(\beta,1),(1,\gamma)} \) by coherence,
\( \phi ^{\otimes} _{(\beta,1),(1,\gamma)} \) is an identity if the interchangers \( \phi ^{\otimes} _{(\beta_i,1),(1,\gamma)} \) are identities.

If \( i \) is odd, \( \beta_i \) is a coherence 2-morphism and \( \phi ^{\otimes} _{(\beta_i,1),(1,\gamma)} \) is an identity by proposition \ref{prop_tprime_def}.

Now let \( i \) be even.
For notational convenience we set \( \alpha:= \beta_i. \)
There exists a coherence 2-morphism \( \mu \) in \( \oper{F} \tprime, \) which is a lift of \( \gamma \) and whose types are the standard F-lifts of the types of \( \gamma. \)
Now consider the following 3-morphism in \( \oper{F} \that: \)

\begin{align*}
{}& (\sfl \alpha \otimes \sfl 1 _{t \gamma}) * (1 _{s \sfl \alpha} ^{\sms{F}} \otimes \mu) \\
\rightarrow {}&
\big( ( (1 ^{\sms{F}} * \sfl \alpha) * 1 ^{\sms{F}}) \otimes ((\mu * \sfl 1) * \mu ^{\adsq}) \big) * (1 ^{\sms{F}} \otimes \mu) \\
\rightarrow {}&
\big( ((1 ^{\sms{F}} \otimes \mu) * (\sfl \alpha \otimes \sfl 1)) * (1 ^{\sms{F}} \otimes \mu ^{\adsq}) \big) * (1 ^{\sms{F}} \otimes \mu) \\
\rightarrow {}&
\big( (1 ^{\sms{F}} \otimes \mu) * (\sfl \alpha \otimes \sfl 1) \big) * \big( (1 ^{\sms{F}} \otimes \mu ^{\adsq}) * (1 ^{\sms{F}} \otimes \mu) \big) \\
\rightarrow {}&
\big( (1 ^{\sms{F}} * \sfl \alpha) \otimes (\mu * \sfl 1) \big) * 1 ^{\sms{F}} \\
\rightarrow {}&
(\sfl \alpha * 1 ^{\sms{F}}) \otimes (\sfl 1 * \mu) \\
\end{align*}

The composite of the first two arrows is a lift of a coherence-relator and therefore its image under \( \epsilon \) is contained in \( C_3. \)
The remaining arrows are sent to identities under \( \epsilon \) and therefore the image of the whole composite under \( \epsilon \) is contained in \( C_3. \)
Since by coherence the image of the whole composite under \( \epsilon \) is equal to \( \phi ^{\otimes} _{(\alpha,1),(1,\gamma)}, \) we have shown \( \phi ^{\otimes} _{(\alpha,1),(1,\gamma)} \) is contained in \( C_3. \)
\end{proof}

\begin{satz} \label{theorem_tst_is_gray}
Let \( \cat{T} \) be a tricategory and \( C_3 \) be the coherence 3-system from proposition \ref{prop_tprime_coh_23_system}.
Then \( \tst := \tprime/\scriptstyle{C_3} \) is a Gray category.
\end{satz}

\begin{proof}
\( \tst \) is a cubical, because \( \tprime \) is cubical and because any constraint cell of \( \tst \) is the image under \( [-]: \tprime \to \tcu \) of a corresponding constraint cell of \( \tcu. \)

Components of the modifications \( \pi, \lambda, \rho, \mu \) are contained in \( C_3 \) (because \( C_3 \) is a coherence 2/3-system) and are therefore identities by proposition \ref{prop_T_mod_C23}.
For the same reason components of the units and counits of the adjoint equivalences \( a \dashv a ^{\adsq}, l \dashv l ^{\adsq}, r \dashv r ^{\adsq} \) are identities.
To show that these adjoint equivalences are identities it is therefore sufficient to show that the components of \( a,l,r \) are identities.
The components of \( a,l,r \) at 1-morphisms of \( \tprime \) are coherence 2-morphisms and therefore identities by proposition \ref{prop_T_mod_C23}.
Now we show that components of \( a \) at 2-morphisms of \( \tprime \) are contained in \( C_3. \)
For that let \( \alpha, \beta, \gamma \) be \( \otimes \)-composable 2-morphisms in \( \tprime, \) and let \( (\alpha_i)_n, (\beta_i)_m, (\gamma_i)_k \) be their respective \( \tprime \)-partitions.
Then we can decompose the morphism
\begin{align*}
a ^{\sms{F}} _{t \alpha, t \beta, t \gamma} \fstar
\big( (\comseq{\fstar}{(\alpha_i)}_{n}^{\sigma ^{\lambda}} \fotimes \comseq{\fstar}{(\beta_i)}_{m}^{\sigma ^{\lambda}}) \fotimes \comseq{\fstar}{(\gamma_i)}_{k}^{\sigma ^{\lambda}} \big) \\
\rightarrow
\big( \comseq{\fstar}{(\alpha_i)}_{n}^{\sigma ^{\lambda}} \fotimes (\comseq{\fstar}{(\beta_i)}_{m}^{\sigma ^{\lambda}} \fotimes \comseq{\fstar}{(\gamma_i)}_{k}^{\sigma ^{\lambda}}) \big)
\fstar a ^{\sms{F}} _{s \alpha, s \beta, s \gamma}
\end{align*}
in \( \oper{F} \tprime \) as a composite of coherence 3-morphisms which are either sent to identities by \( \epsilon, \) or built from the associators
\( a ^{\sms{F}} _{\alpha_i,1 ^{\sms{F}} ,1 ^{\sms{F}}}, \) \( a ^{\sms{F}} _{1 ^{\sms{F}} , \beta_i, 1 ^{\sms{F}}} \) or \( a ^{\sms{F}} _{1 ^{\sms{F}} ,1 ^{\sms{F}}, \gamma_i}. \)
Those are contained in \( C_3, \) since \( (\alpha_i \otimes 1) \otimes 1 \) is coherence related to \( \alpha_i \otimes (1 \otimes 1), \)
\( (1 \otimes \beta_i) \otimes 1 \) is coherence related to \( 1 \otimes (\beta_i \otimes 1) \) and
\( (1 \otimes 1) \otimes \gamma \) is coherence related to \( 1 \otimes (1 \otimes \gamma_i). \)
Therefore all components of \( a \) at 2-morphisms are in \( C_3 \) and hence identities in \( \tst. \)
The proof that the components of \( l \) an \( r \) at 2-morphisms are identities is analogous.
\end{proof}

Let us now summarize the overall result of the forgoing strictifications:

\begin{theorem}
For any tricategory \( \cat{T}, \) there is a virtually strict triequivalence
\[
[-] := \that \xrightarrow{\ftilde} \ttil \xrightarrow{\barf} \tbar \xrightarrow{[-]} \tcu \xrightarrow{[-]} \tprime \xrightarrow{[-]} \tst
\]
from the tricategory of formal composites \( \that \) to the Gray category \( \tst. \)
\end{theorem}

Now we will explain how the results of this section can be used to reduce calculations in a tricategory \( \cat{T} \) to calculations in the Gray category \( \tst. \)
More precisely we describe how an equation of pasting diagram in a local bicategory of \( \cat{T} \) can be replaced with an equivalent equation of pasting diagram in a local bicategory of \( \tst. \)

First we note that any composite 1- or 2-morphism in \( \cat{T} \) lifts along \( \ev \) to the 1- or 2-morphism in \( \that, \) that constitutes precisely that composite.
Similarly a pasting diagram in a local bicategory of \( \cat{T} \) lifts along \( \ev \) to a corresponding pasting diagram in a local bicategory of \( \that. \)
Since \( \ev \) is a virtually strict triequivalence, two pasting diagrams in \( \cat{T} \) are equal if and only if their corresponding lifts in \( \that \) are equal.
Finally, the lifted pasting diagrams in \( \that \) are equal if and only if they are equal under the virtually strict triequivalence \( [-]: \that \to \tst. \)
Therefore, in total we have achieved our goal of replacing an equation of pasting diagram in \( \cat{T} \) with an equivalent equation of pasting diagrams in \( \tst. \)
(Here equivalent means that the equation in \( \cat{T} \) holds if and only if the equation in \( \tst \) holds.)

The following considerations will clarify the whole procedure and its benefits:

There are many concepts defined internally to tricategories (or monoidal bicategories), whose data consist of certain 1-,2- and 3-morphisms.
This data is then required to satisfy some axioms, which have the form of equations involving the 3-morphisms of the data.
As we will demonstrate with the example of the concept of a biadjunction in a tricategory \( \cat{T} \), these axioms are equivalent to corresponding -much easier- axioms in \( \tst. \)

\begin{mydef} [Definition 2.1. in \cite{Gurski2012}] \label{def_biadjunction}
A biadjunction in a tricategory \( \cat{T} \) consist of the following data
\begin{itemize}
\item
1-morphisms \( f: a \to b \) and \( g: b \to a. \)
\item
2-morphisms \( \alpha: f \otimes g \to 1_b \) and \( \beta: 1_a \to g \otimes f. \)
\item
3-isomorphisms \( \Phi \) and \( \Psi \) as in the following diagrams
\[
\begin{tikzcd}
f
\ar[r, "r ^{\adsq}"]
\ar[rrrdd, "1" {swap, name=t1, pos=.7}]
	& f \otimes 1
	\ar[r, "1 \otimes \beta"]
		& f \otimes (g \otimes f)
		\ar[r, "a ^{\adsq}"]
			& (f \otimes g) \otimes f
			\ar[d, "\alpha \otimes 1"]
			\ar[to=t1, "\Phi", short=7mm, Rightarrow, swap] \\
			& & & 1 \otimes f
			\ar[d, "l"] \\
			& & & f
\end{tikzcd}
\]
\[
\begin{tikzcd}
g
\ar[r, "l ^{\adsq}"]
\ar[rrrdd, "1" {swap, name=t1, pos=.7}]
	& 1 \otimes g
	\ar[r, "\beta \otimes 1"]
		& (g \otimes f) \otimes g
		\ar[r, "a"]
			& g \otimes (f \otimes g)
			\ar[d, "1 \otimes \alpha"]
			\ar[to=t1, "\Psi", short=7mm, Rightarrow, swap] \\
			& & & g \otimes 1
			\ar[d, "r"] \\
			& & & g
\end{tikzcd}
\]
\end{itemize}

The data are subject to the axioms that demand that the pasting diagrams in Figure \ref{fig_pasting_1} and Figure \ref{fig_pasting_2} are identities.

\begin{figure}
\caption{First pasting} \label{fig_pasting_1}
\includegraphics[width=\textwidth]{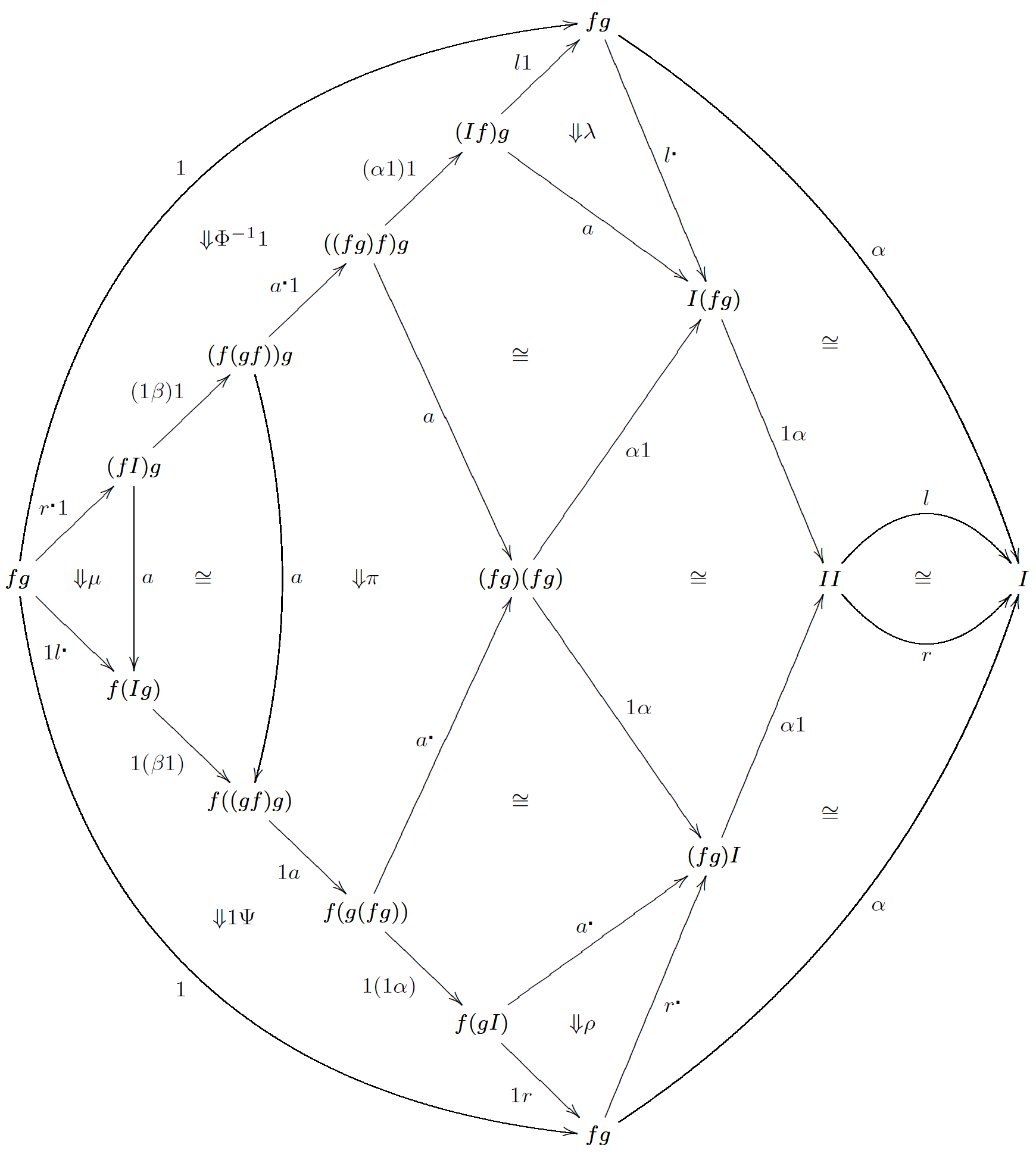}
\end{figure}

\begin{figure}
\caption{Second pasting} \label{fig_pasting_2}
\includegraphics[width=\textwidth]{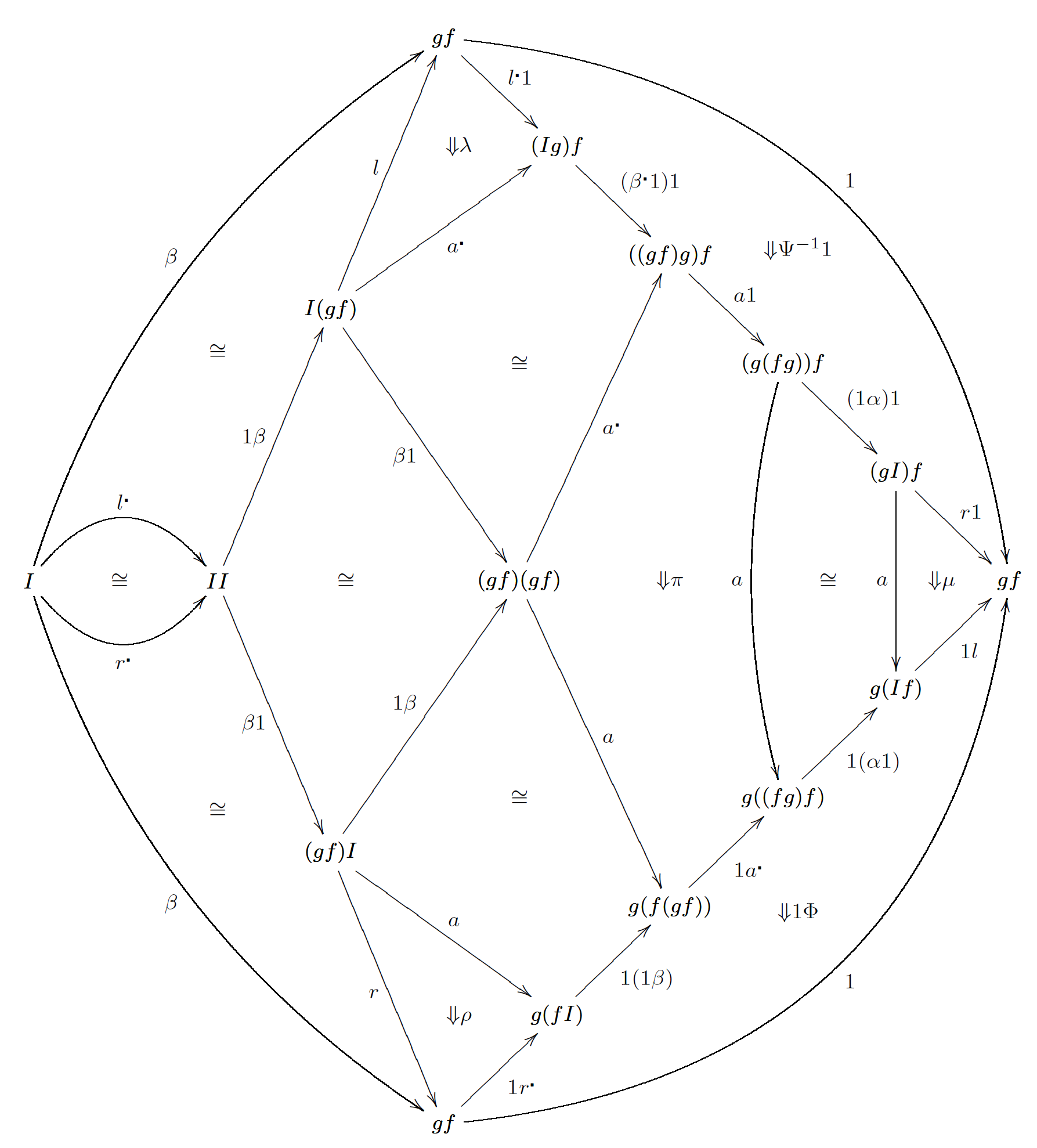}
\end{figure}
\end{mydef}

We now express the axioms for the adjunction as pasting diagrams in \( \tst. \)
For that purpose we need to start by lifting the data from definition \ref{def_biadjunction}:

We define the 2-morphism \( \hat \alpha:= (\alpha, f \hot g, 1_b), \; \hat \beta:= (\beta, 1_a, g \hot f) \) and the 3-morphism \( \hat \Phi, \hat \Psi \) by the diagrams
\[
\begin{tikzcd}
f
\ar[r, "\hat r ^{\adsq}"]
\ar[rrrdd, "\hat 1" {swap, name=t1, pos=.7}]
	& f \hot 1
	\ar[r, "\hat 1 \hot \hat \beta"]
		& f \hot (g \hot f)
		\ar[r, "\hat a ^{\adsq}"]
			& (f \hot g) \hot f
			\ar[d, "\hat \alpha \hot \hat 1"]
			\ar[to=t1, "\hat \Phi", short=7mm, Rightarrow, swap] \\
			& & & 1 \hot f
			\ar[d, "\hat l"] \\
			& & & f
\end{tikzcd}
\]
\[
\begin{tikzcd}
g
\ar[r, "\hat l ^{\adsq}"]
\ar[rrrdd, "\hat 1" {swap, name=t1, pos=.7}]
	& 1 \hot g
	\ar[r, "\hat \beta \hot \hat 1"]
		& (g \hot f) \hot g
		\ar[r, "\hat a"]
			& g \hot (f \hot g)
			\ar[d, "\hat 1 \hot \hat \alpha"]
			\ar[to=t1, "\hat \Psi", short=7mm, Rightarrow, swap] \\
			& & & g \hot 1
			\ar[d, "\hat r"] \\
			& & & g.
\end{tikzcd}
\]

The pasting diagrams in definition \ref{def_biadjunction} lift in a obvious way to respective pasting diagrams in \( \that. \)
Since \( \ev: \that \to \cat{T} \) is a virtually strict triequivalence, these lifted pasting diagrams are identities if and only if the pasting diagrams from definition \ref{def_biadjunction} are identities.
And since \( [-]: \that \to \tst \) is a virtually strict triequivalence, the lifted pasting diagrams in \( \that \) are identities if and only if they are mapped to identities under \( [-]. \)
Thanks to the semi-strictness of \( \tst, \) the image of the lifted diagrams under \( [-] \) simplifies tremendously to the following diagrams, where we replace \( \otimes \) by juxtaposition and label interchanger 3-cells with \( \sim. \)

\[
\begin{tikzcd}
 		& & {[f] [g]}
 		\ar[rd, "{[\hat \alpha]}"] \\
{[f] [g]}
\ar[rru, "1" {name=s1}, bend left=30]
\ar[rrd, "1" {name=s2, swap}, bend right=30]
\ar[r, "{1 [\hat \beta] 1}"]
	& {[f] [g] [f] [g]}
	\ar[ru, "{[\hat \alpha]11}" swap]
	\ar[rd, "{11[\hat \alpha]}"]
			& & 1
			\ar[ll, "\sim", phantom, description] \\
 		& & {[f] [g]}
 		\ar[ru, "{[\hat \alpha]}" swap]
 		\ar[from=s1, to=lu, Rightarrow, "{[\hat \Phi ^{-1}]1}", short=3mm, pos=.6]
 		\ar[from=lu, to=s2, Rightarrow, "{1[\hat \Psi]}", short=3mm, pos=.35]
\end{tikzcd}
\]
\[
\begin{tikzcd}
 		& & {[g] [f]}
 		\ar[rd, "{[\hat \beta]}", leftarrow] \\
{[g] [f]}
\ar[rru, "1" {name=s1}, bend left=30, leftarrow]
\ar[rrd, "1" {name=s2, swap}, bend right=30, leftarrow]
\ar[r, "{1 [\hat \alpha] 1}", leftarrow]
	& {[g] [f] [g] [f]}
	\ar[ru, "{[\hat \beta]11}" swap, leftarrow]
	\ar[rd, "{11[\hat \beta]}", leftarrow]
			& & 1
			\ar[ll, "\sim", phantom, description, leftarrow] \\
 		& & {[g] [f]}
 		\ar[ru, "{[\hat \beta]}" swap, leftarrow]
 		\ar[from=s1, to=lu, Rightarrow, "{[\hat \Psi ^{-1}]1}", short=3mm, pos=.6]
 		\ar[from=lu, to=s2, Rightarrow, "{1[\hat \Phi]}", short=3mm, pos=.35]
\end{tikzcd}
\]

These pasting diagrams in \( \tst \) can be expressed more intuitively by the corresponding rewriting diagrams in \( \tst. \)

\[
\begin{tikzpicture}[string={.7cm}{.7cm}, rounded corners=0, baseline={([yshift=-.5ex]current bounding box.center)}, math mode, label distance=-1.5mm, scale=.5, inner sep=.8mm]
\draw (-.2,3) to
	[d r] (.5,1) node (n1) [label=-90:\scriptstyle{[\hat \alpha]}] {} to
	[r r] (1.5,2) node (n2) [label=90:\scriptstyle{[\hat \beta]}] {} to
	[r r] (2.5,0) node (n3) [label=-90:\scriptstyle{[\hat \alpha]}] {} to
	[r u] (3.3,3);
\draw[green] (-.5,.4) rectangle (2.3,2.8);
\begin{scope}[xshift=7cm]
\draw (-.2,3) to
	[d r] (.5,0) node (n1) [label=-90:\scriptstyle{[\hat \alpha]}] {} to
	[r r] (1.5,2) node (n1) [label=90:\scriptstyle{[\hat \beta]}] {} to
	[r r] (2.5,1) node (n1) [label=-90:\scriptstyle{[\hat \alpha]}] {} to
	[r u] (3.3,3);
\draw[blue] (.8,2.8) rectangle (3.5,.4);
\end{scope}
\begin{scope}[xshift=3.45cm, yshift=-5.5cm]
\draw (0,3) to
	[d r] (1.5,0) node [label=-90:\scriptstyle{[\hat \alpha]}] {} to
	[r u] (3,3);
\draw[green] (-.4,1.3) rectangle (.5,2.5);
\draw[blue] (2.5,1.3) rectangle (3.4,2.5);
\end{scope}
\draw[->] (3.6,1.5) -- +(2.7cm,0) node [pos=.5, above, inner sep=1mm] {\sim};
\draw[->] (1.5,-.5) -- +(1.7cm,-2) node [pos=.25, below, inner sep=3mm] {\scriptstyle{[\hat \Phi]}};
\draw[->] (8.5,-.5) -- +(-1.7cm,-2) node [pos=.25, below, inner sep=3mm] {\scriptstyle{[\hat \Psi]}};
\end{tikzpicture}
\quad \quad \quad
\begin{tikzpicture}[string={.7cm}{.7cm}, rounded corners=0, baseline={([yshift=-.5ex]current bounding box.center)}, math mode, label distance=-1.5mm, scale=.5, inner sep=.8mm, yscale=-1]
\draw (-.2,3) to
	[d r] (.5,1) node (n1) [label=90:\scriptstyle{[\hat \beta]}] {} to
	[r r] (1.5,2) node (n2) [label=-90:\scriptstyle{[\hat \alpha]}] {} to
	[r r] (2.5,0) node (n3) [label=90:\scriptstyle{[\hat \beta]}] {} to
	[r u] (3.3,3);
\draw[green] (-.5,.4) rectangle (2.3,2.8);
\begin{scope}[xshift=7cm]
\draw (-.2,3) to
	[d r] (.5,0) node (n1) [label=90:\scriptstyle{[\hat \beta]}] {} to
	[r r] (1.5,2) node (n1) [label=-90:\scriptstyle{[\hat \alpha]}] {} to
	[r r] (2.5,1) node (n1) [label=90:\scriptstyle{[\hat \beta]}] {} to
	[r u] (3.3,3);
\draw[blue] (.8,2.8) rectangle (3.5,.4);
\end{scope}
\begin{scope}[xshift=3.45cm, yshift=-5.5cm]
\draw (0,3) to
	[d r] (1.5,0) node [label=90:\scriptstyle{[\hat \beta]}] {} to
	[r u] (3,3);
\draw[green] (-.4,1.3) rectangle (.5,2.5);
\draw[blue] (2.5,1.3) rectangle (3.4,2.5);
\end{scope}
\draw[->] (3.6,1.5) -- +(2.7cm,0) node [pos=.5, below, inner sep=1mm] {\sim};
\draw[->] (1.5,-.5) -- +(1.7cm,-2) node [pos=.25, above, inner sep=3mm] {\scriptstyle{[\hat \Psi]}};
\draw[->] (8.5,-.5) -- +(-1.7cm,-2) node [pos=.25, above, inner sep=3mm] {\scriptstyle{[\hat \Phi]}};
\end{tikzpicture}
\] \clearpage{}


\bibliography{bibliography.bib}

\begin{thebibliography}{10}

\bibitem{Gurski2006}
M.~N. Gurski, {\em An algebraic theory of tricategories}.
\newblock PhD thesis, University of Chicago, Department of Mathematics, 2006.

\bibitem{Bar2016}
K.~Bar, A.~Kissinger, and J.~Vicary, ``Globular: an online proof assistant for
  higher-dimensional rewriting,'' {\em arXiv preprint arXiv:1612.01093}, 2016.

\bibitem{Barrett2012}
J.~W. Barrett, C.~Meusburger, and G.~Schaumann, ``Gray categories with duals
  and their diagrams,'' {\em arXiv preprint arXiv:1211.0529}, 2012.

\bibitem{Gurski2013}
N.~Gurski, {\em Coherence in three-dimensional category theory}, vol.~201.
\newblock Cambridge University Press, 2013.

\bibitem{Leinster1998}
T.~Leinster, ``Basic bicategories,'' {\em arXiv preprint math/9810017}, 1998.

\bibitem{Kelly1982}
M.~Kelly, {\em Basic concepts of enriched category theory}, vol.~64.
\newblock CUP Archive, 1982.

\bibitem{Marsden2014}
D.~Marsden, ``Category theory using string diagrams,'' {\em arXiv preprint
  arXiv:1401.7220}, 2014.

\bibitem{Power1990}
A.~J. Power, ``A 2-categorical pasting theorem,'' {\em Journal of Algebra},
  vol.~129, no.~2, pp.~439--445, 1990.

\bibitem{MacLane2013}
S.~Mac~Lane, {\em Categories for the working mathematician}, vol.~5.
\newblock Springer Science \& Business Media, 2013.

\bibitem{nlab:transfor}
{nLab authors}, ``transfor.'' \url{http://ncatlab.org/nlab/show/transfor}, Aug.
  2018.
\newblock \href{http://ncatlab.org/nlab/revision/transfor/35}{Revision 35}.

\bibitem{Crans2003}
S.~E. Crans, ``Localizations of transfors,'' {\em K-theory}, vol.~28, no.~1,
  pp.~39--105, 2003.

\bibitem{Baez2004}
J.~C. Baez and A.~D. Lauda, ``Higher-dimensional algebra v: 2-groups,'' {\em
  Version}, vol.~3, pp.~423--491, 2004.

\bibitem{Gordon1995}
R.~Gordon, A.~J. Power, and R.~Street, {\em Coherence for tricategories},
  vol.~558.
\newblock American Mathematical Soc., 1995.

\bibitem{MO_question}
M.~Shulman, ``Is a weak functor which strictly preserves horizontal composition
  and which runs between strict bicategories automatically strict?.'' URL =
  {https://mathoverflow.net/q/307441}, 2018.

\bibitem{Bartlett2014}
B.~Bartlett, ``Quasistrict symmetric monoidal 2-categories via wire diagrams,''
  {\em arXiv preprint arXiv:1409.2148}, 2014.

\bibitem{Leinster2004}
T.~Leinster, {\em Higher operads, higher categories}, vol.~298.
\newblock Cambridge University Press, 2004.

\bibitem{Gurski2012}
N.~Gurski, ``Biequivalences in tricategories,'' {\em Theory and Applications of
  Categories}, vol.~26, no.~14, pp.~349--384, 2012.

\end{thebibliography}
\bibliographystyle{ieeetr}

\end{document}